\documentclass[a4paper,12pt,reqno]{amsart}
\usepackage[english]{babel}
\usepackage{amsmath}
\usepackage{amscd}
\usepackage{amsgen}
\usepackage{amsfonts}
\usepackage{amssymb}
\usepackage{amsthm}
\usepackage{comment}
\usepackage{url}
\usepackage{fixmath}
\usepackage[mathscr]{euscript}
\usepackage[all]{xy}
\usepackage[final]{epsfig}
\usepackage{latexsym}
\usepackage{graphics}
\usepackage{verbatim}
\usepackage{slashed}
\usepackage{stmaryrd}
\usepackage{sansmath}

\makeindex

\catcode`\@=11
\@namedef{subjclassname@2020}{%
  \textup{2020} Mathematics Subject Classification}
\catcode`\@=12

\newtheorem{theorem}{Theorem}
\newtheorem{prop}[theorem]{Proposition}
\newtheorem{lem}[theorem]{Lemma}
\newtheorem{cor}[theorem]{Corollary}
\newtheorem{conj}[theorem]{Conjecture}
\newtheorem{defn}[theorem]{Definition}
\newtheorem{remark}[theorem]{Remark}
\newtheorem{rem}[theorem]{Remark}
\newtheorem{example}[theorem]{Example}

\newtheorem{theo}{Theorem}
\newtheorem{corro}{Corollary}

\allowdisplaybreaks

\newcommand{\Z}{{\mathbb Z}}

\newcommand{\N}{{\mathbb N}}
\newcommand{\C}{{\mathbb C}}
\newcommand{\BB}{{\mathbb B}}

\newcommand{\Ric}{\operatorname{Ric}}       
\newcommand{\Hess}{\operatorname{Hess}}   
\newcommand{\Res}{\operatorname{Res}}      

\newcommand{\G}{{\mathcal G}}                    

\newcommand{\scal}{\operatorname{scal}}      
\newcommand{\tr}{\operatorname{tr}}            
\newcommand{\grad}{\operatorname{grad}}    
\newcommand{\res}{\operatorname{Res}}       
\newcommand{\id}{\operatorname{Id}}           
\newcommand{\var}{{\textit{var}}}                 

\newcommand{\CT}{\operatorname{CT}}        
\newcommand{\LT}{\operatorname{LT}}         
\newcommand{\SC}{\mathcal{S}}                   
\newcommand{\JF}{\mathcal{F}}                    
\newcommand{\Po}{\mathcal {P}}                  
\newcommand{\Sc}{\mathscr{S}}                   
\newcommand{\Con}{\mathcal{C}}                 
\newcommand{\LO}{{\mathcal L}}                  
\newcommand{\LA}{{\mathcal A}}                  
\newcommand{\He}{\mathcal H}                     
\newcommand{\D}{{\bf D}}                            
\newcommand{\QC}{{\bf Q}}                          
\newcommand{\PO}{{\bf P}}                           
\newcommand{\E}{{\mathcal E}}                    
\newcommand{\SCY}{{(\mathcal {SY})}}         
\newcommand{\W}{\mathcal {W}}                   
\newcommand{\NV}{\mathcal {N}}                   
\newcommand{\A}{{\mathcal A}}                    
\newcommand{\supp}{\operatorname{supp}}    
\newcommand{\lo}{\mathring{L}}                     
\newcommand{\loh}{\hat{\mathring{L}}}

\def\Rho{{\sf P}}
\def\T{{\mathcal T}}                             
\def\J{{\sf J}}                                       
\def\R{{\mathbb R}}
\def\B{{\mathcal B}}                               
\def\OB{{\mathcal O}}                            
\def\BB{{\mathfrak B}}                            

\def\st{\stackrel{\text{def}}{=}}

\numberwithin{theorem}{section} \numberwithin{equation}{section}

\title[Residue families, singular Yamabe problems and conformal Laplacians]
{Residue families, singular Yamabe problems and extrinsic conformal Laplacians}

\author[Andreas Juhl and Bent {\O}rsted] {Andreas Juhl and Bent {\O}rsted}

\address{Department of Mathematics of {\AA}rhus University, Ny Munkegade 118,
8000 {\AA}rhus, Denmark}

\email{orsted@math.au.dk}

\address{Humboldt-Universit\"at, Institut f\"ur Mathematik, Unter den Linden 6, 10099 Berlin,
Germany}
\address{IHES, Bures sur Yvette}

\email{ajuhl@math.hu-berlin.de}
\email{juhl.andreas@googlemail.com}


\begin{document}

\begin{abstract}
Let $(X,g)$ be a compact manifold with boundary $M^n$ and $\sigma$ a defining function 
of $M$. To these data, we associate natural conformally covariant polynomial one-parameter families 
of differential operators $C^\infty(X) \to C^\infty(M)$. They arise through a residue construction 
which generalizes an earlier construction in the framework of Poincar\'e-Einstein metrics and are 
referred to as residue families. Residue families may be viewed as curved analogs of conformal 
symmetry breaking differential operators. The main ingredient of the definition of residue families 
are eigenfunctions of the Laplacian of the singular metric $\sigma^{-2}g$. We prove that if $\sigma$ is 
an approximate solution of a singular Yamabe problem, i.e., if $\sigma^{-2}g$ has constant scalar 
curvature $-n(n+1)$, up to a sufficiently small remainder, these families can be written as compositions 
of certain degenerate Laplacians (Laplace-Robin operators). This result implies that the notions of extrinsic 
conformal Laplacians and extrinsic $Q$-curvature introduced in recent works by Gover and Waldron can 
naturally be rephrased in terms of residue families. This new spectral theoretical perspective allows 
easy new proofs of several results of Gover and Waldron. Moreover, it enables us to relate the 
extrinsic conformal Laplacians and the critical extrinsic $Q$-curvature to the scattering operator of 
the asymptotically hyperbolic metric $\sigma^{-2}g$ extending the work of Graham and Zworski. The 
relation to the scattering operators implies that the extrinsic conformal Laplacians are self-adjoint. 
We describe the asymptotic expansion of the volume of a singular Yamabe metric in terms of 
Laplace-Robin operators (reproving results of Gover and Waldron). We also derive new local 
holographic formulas for all extrinsic $Q$-curvatures (critical and sub-critical ones) in terms of 
renormalized volume coefficients, the scalar curvature of the background metric, and the asymptotic
expansions of eigenfunctions of the Laplacian of the singular metric $\sigma^{-2}g$. These results 
naturally extend earlier results in the Poincar\'e-Einstein case. Furthermore, we prove a new formula 
for the singular Yamabe obstruction $\B_n$. The simple structure of these formulas shows the 
benefit of the systematic use of so-called adapted coordinates. We use the latter formula for 
$\B_n$ to derive explicit expressions for the obstructions in low-order cases (confirming earlier results). 
Finally, we relate the obstruction $\B_n$ to the supercritical $Q$-curvature $\QC_{n+1}$.
\end{abstract}


\maketitle

\begin{center} \today \end{center}

\tableofcontents

\section{Introduction}\label{intro}

Conformal differential geometry studies natural geometric quantities 
associated to Riemannian (and pseudo-Riemannian) manifolds (such as curvature
invariants and natural differential operators) which transform nicely under
conformal changes of the metric. In recent years, it has developed in profound and
surprising ways, which are connected to the spectral theory of Laplace-type operators,
scattering theory, holographic principles (AdS-CFT correspondence), Cartan geometry,
non-linear partial differential equations, and representation theory. Particular
powerful tools are tractor calculus and the Poincar\'e-Einstein and ambient metrics 
in the sense of Fefferman-Graham \cite{CS, DGH, FG-Cartan, FG-final, BJ, KS, KS2}.

A central role in those parts of geometric analysis which are related to conformal differential
geometry play the conformally invariant powers of the Laplacian, which are known as
GJMS-operators \cite{GJMS} and the related Branson's $Q$-curvatures \cite{sharp}.
The GJMS-operators $P_{2N}(h)$ act on the space $C^\infty(M)$ on any Riemannian
manifold $(M,h)$. They are of the form $\Delta_h^N + \mbox{lower-order terms}$,
where the lower-order terms are given in terms of the metric and covariant
derivatives of its curvature. We recall that for general $h$, the operators
$P_{2N}(h)$ exist for all orders $2N$ if $n$ is odd but only for $2N \le n$ if $n$
is even. The GJMS-operators $P_{2N}$ generalize the conformal Laplacian (Yamabe
operator)
$$
   P_2 = \Delta - \left(\frac{n}{2}-1\right) \J, \quad  2 (n-1) \J = \scal
$$
and the Paneitz operator
$$
    P_4 = \Delta^2 - \delta ((n-2) \J - 4 \Rho) d
   + \left(\frac{n}{2}-2\right) \left(\frac{n}{2}\J^2 - 2 |\Rho|^2 - \Delta (\J)\right),
$$
where $\Rho$ is the Schouten tensor of $h$ (for the notation, we refer to Section
\ref{notation}). The scalar curvature quantities
$$
   Q_2 = \J \quad \mbox{and} \quad Q_4 = \frac{n}{2} \J^2 - 2 |\Rho|^2 - \Delta(\J)
$$
are the lowest-order cases of Branson's $Q$-curvatures. In particular, the quantity $Q_4$ plays a
central role in geometric analysis \cite{CBull, IMC}.

In \cite{J1, J2}, one of the authors analyzed the structure of GJMS-operators
and Branson's $Q$-curvatures of a Riemannian manifold $(M^n,h)$ through a theory of
conformally covariant polynomial one-parameter families of differential operators
$C^\infty(X) \to C^\infty(M)$, where $X^{n+1}$ is a manifold of one more dimension.
Such an extrinsic perspective on objects living on $M$ is sometimes referred
to as a {\em holographic} point of view following \cite{W}. These families of differential operators
can be regarded as {\em curved versions} of certain intertwining operators in
representation theory. Following the framework introduced in a series of works by T.
Kobayashi and his coauthors, these intertwining operators are now known as conformal
{\em symmetry breaking operators} $C^\infty(S^{n+1}) \to C^\infty(S^n)$. The notion
is motivated by the fact that they are equivariant only with respect to the subgroup
of the conformal group of $S^{n+1}$ consisting of diffeomorphisms that leave the
equatorial subsphere $S^n \hookrightarrow S^{n+1}$ invariant. Such intertwining
operators acting on functions have analogs acting on sections of homogeneous vector
bundles on spheres as well as in contexts where other groups replace the conformal group 
of the sphere. We refer the interested reader to \cite{KS, KS2, KKP, FJS, FW}.

In the curved case, the conformal covariance property takes the role of the intertwining property, 
and the definition of the operators uses a Poincar\'e-Einstein metric in the sense of Fefferman 
and Graham \cite{FG-final}. The curved versions of the symmetry breaking operators then are 
defined in terms of the asymptotic expansions of eigenfunctions of the Laplacian of a 
Poincar\'e-Einstein metric on $X$ and the resulting so-called renormalized volume coefficients. 
The latter quantities are curvature invariants of a metric on $M$, the study of which originally 
was motivated by the AdS/CFT-correspondence \cite{W, G-vol}. Since the construction of the curved
versions of the symmetry breaking differential operators can be expressed as a
residue construction, these were termed {\em residue families} in \cite{J1}. A key
feature of the theory of residue families is that each GJMS-operator $P_{2N}(h)$
comes along with a residue family (see \eqref{GJMS-res}), and one can effectively utilize the
family parameter to study the structure of these special values \cite{J1, J2}.

Another basic perspective on GJMS-operators and $Q$-curvatures of a metric $h$ on
$M$ was developed in \cite{GZ} by showing how the scattering operator
$\Sc(\lambda)$ of the Laplacian of a Poincar\'e-Einstein metric on $X$ in normal
form relative to $h$ naturally encodes both quantities. The scattering operator
$\Sc(\lambda)$ is a one-parameter meromorphic family of conformally
covariant pseudo-differential operators, which may be regarded as a curved version of
the Knapp-Stein intertwining operator for spherical principal series
representations. Thus, both $\Sc(\lambda)$ and residue families describe
GJMS-operators and $Q$-curvatures. But whereas the former are meromorphic families
of pseudo-differential operators, the latter are polynomial families of differential
operators.

In \cite{GW}, Gover and Waldron developed an alternative perspective on the above
holographic descriptions of GJMS-operators and $Q$-curvatures. In this approach, the 
conformal geometry of the metric $h$ on $M$ {\em and} the associated Poincar\'e-Einstein metric 
$g_+$ on $X$ play a central role. In particular, this leads to a deeper understanding of the role 
of the Einstein property of $g_+$ in constructions of GJMS-operators. The conformal tractor 
calculus on $X$ naturally associates a conformally covariant polynomial one-parameter family 
of second-order differential operators on functions on $X$ to a metric $g$ on $X$ {\em and} a 
defining function $\sigma$ of the hypersurface $\iota^*: M \hookrightarrow X$. Later, Gover 
and Waldron termed this operator the {\em Laplace-Robin operator} \cite{GW-LNY}.\footnote{In 
\cite{GW}, the Laplace-Robin operators are regarded as operators acting on densities.} We shall 
follow this terminology here. The Laplace-Robin operators degenerate on $M$ in the sense that the
coefficients of its leading parts vanish on $M$. The data $(g,\sigma)$ induce a
metric $\iota^*(g)$ on $M$ and a singular metric $\sigma^{-2} g$ on the complement
of $M$ in $X$. It is the latter metric that generalizes the Poincar\'e-Einstein
metric. Now, if $\sigma^{-2} g$ is Einstein, then certain compositions of
Laplace-Robin operators on $X$ reduce to GJMS-operators on $M$. Even if the Einstein
condition is violated, analogous compositions of (renormalized) Laplace-Robin
operators still reduce to conformally covariant operators on $M$. We emphasize that the
notion of conformal covariance in the latter statement concerns constructions that 
depend on a metric $g$ on $X$ {\em and} a defining function $\sigma$ of the boundary
$M$ of $X$, i.e., constructions which are invariant under conformal changes of both
$g$ and $\sigma$: $\hat{g} = e^{2\varphi} g$ and $\hat{\sigma} = e^\varphi \sigma$.
Note that $\sigma^{-2} g$ is invariant under such substitutions. Now a version of
the singular Yamabe problem asks to find (for a given metric $g$ on $X$) a defining
function $\sigma$ so that the scalar curvature of $\sigma^{-2}g$ is a negative
constant. This conformally invariant condition for pairs $(g,\sigma)$ actually
determines the defining function $\sigma$ by the background metric $g$ (at least to
some extent) and the embedding $\iota$. If $\sigma$ solves the singular Yamabe
problem, the above constructions of conformally covariant operators on $M$ 
only depend on the metric $g$ and the embedding $\iota$. In other words, these
operators live on $M$, their definition depends on the embedding of $M
\hookrightarrow X$, and they are conformally covariant with respect to conformal
changes of the background metric $g$ on $X$. Gover, Waldron,
and coauthors used this idea in a series of works to develop a theory of conformal invariants of
hypersurfaces. To a large extent, this theory rests on the conformal tractor
calculus \cite{GW-announce, GW, GW-LNY, GW-reno, GW-reno-b, GGHW, GW-Obata, AGW}.

In recent years, Laplace-Robin operators independently also appeared in other
contexts, albeit not under this name. As described above, their original and most
general definition comes from conformal tractor calculus. But already in \cite{GW},
it was observed that a special case was crucial in \cite{GZ} in the setting of
Poincar\'e-Einstein metrics. Some other special cases were found to be interesting in
representation theory. Clerc \cite{C} proved that the symmetry breaking differential
operators $C^\infty(\R^{n+1}) \to C^\infty(\R^n)$ introduced in \cite{J1} also arise
as compositions of some second-order intertwining operator for spherical principal
series representations on functions on $\R^{n+1}$. In \cite{FJO}, an attempt to
generalize the result of Clerc led to a description of the general residue families
of \cite{J1} as compositions of some second-order operators, which turned out to be
Laplace-Robin operators. In other words, the latter result could be viewed as a
curved version of the quoted result of Clerc. In \cite{FOS, FO}, the notions of
spectral {\em shift operators} and {\em Bernstein-Sato operators} are used in the
context of symmetry breaking operators on functions, differential forms, and spinors.
These results suggest extensions of the notion of Laplace-Robin operators on forms.
Such extensions on forms using tractor calculus have been developed in the
monograph \cite{GLW}.

The present work sheds new light on the above results of Gover and Waldron. We
connect the various strands of development by introducing an extension of the notion
of residue families. Following the principles of \cite{J1}, we prove a series of new
results, and - as by-products - we confirm and give alternative proofs of some
already known results. The treatment replaces tractor calculus and distributional calculus 
by more classical arguments. Among other things, we stress the role of the scattering operator 
extending \cite{GZ}. We hope that our methods enhance the understanding of the
subject.

We continue with a more detailed description of the main new results.

We consider a compact Riemannian manifold $(X,g)$ with boundary $M$. Let $\iota: 
M \hookrightarrow X$ be the canonical embedding. Let $h = \iota^*(g)$ be the
induced metric on $M$ and $\sigma \in C^\infty(X)$ a boundary defining function,
i.e., $\sigma \ge 0$ on $X$, $\sigma^{-1}(0)=M$ and $d\sigma|_M \ne 0$.

The data $(g,\sigma)$ define a Laplace-Robin operator \cite{GW-LNY}
\begin{equation*}
   L(g,\sigma;\lambda) \st (n+2\lambda-1) (\nabla_{\grad_g (\sigma)} + \lambda \rho)
   - \sigma (\Delta_g + \lambda \J), \; \lambda \in \C
\end{equation*}
on $C^\infty(X)$. Here $2n \J = 2n \J^g = \scal^g$ and $(n+1) \rho(g,\sigma) =
-\Delta_g (\sigma) - \sigma \J$. The significance of this operator rests on its
conformal covariance property 
\begin{equation}\label{CTL-basic}
   L(\hat{g}, \hat{\sigma}; \lambda) \circ e^{\lambda\varphi}
   = e^{(\lambda-1)\varphi} \circ L(g,\sigma;\lambda), \; \lambda \in \C
\end{equation}
for all conformal changes $(\hat{g},\hat{\sigma})=(e^{2\varphi}g,e^{\varphi}\sigma)$
with $\varphi \in C^\infty(X)$. This property implies that all
compositions
\begin{equation}\label{LR-basic}
   L_N(g,\sigma;\lambda)
  \st L(g,\sigma;\lambda\!-\!N\!+\!1) \circ \cdots \circ L(g,\sigma;\lambda), \; N \in \N
\end{equation}
are conformally covariant: $L_N(\lambda)$ shifts the conformal weight $\lambda$
into $\lambda-N$.

The notion of residue families has its origin in a far-reaching generalization of
Gelfand's distribution $r_+^\lambda$ \cite{Gelfand}. As the first generalization of
$r_+^\lambda$ we consider the distribution
$$
   \left\langle M(\lambda),\psi \right\rangle = \int_X \sigma^\lambda \psi dvol_g, \; \Re(\lambda) > -1
$$
with $\psi \in C_c^\infty(X)$. Later $\sigma$ will be a solution of the singular
Yamabe problem, i.e., $\sigma^{-2}g$ has constant scalar curvature
$-n(n+1)$.\footnote{Then $\sigma$ is not necessarily $C^\infty$ up to the boundary.}
Note that $\left\langle M(0),1 \right\rangle = \int_X dvol_g$. The one-parameter
family of distributions $M(\lambda)$ admits a meromorphic continuation to $\C$ with
simple poles in $-\N$. In order to describe the residues of $M(\lambda)$, we
introduce coordinates on $X$ near its boundary as follows. Let $\mathfrak{X} \st
\NV/|\NV|^2$ with $\NV = \grad (\sigma)$ and define a local diffeomorphism
$$
   \eta: I \times M \to X, \; (s,x) \mapsto \Phi^s_{\mathfrak{X}}(x), \quad I =  (0,\varepsilon)
$$
using the flow $\Phi^s_\mathfrak{X}$ of $\mathfrak{X}$. Such coordinates will be
called {\em adapted coordinates}. Then $\eta^*(\sigma) = s$ and $\sigma^\lambda$
pulls back to $s^\lambda$. Moreover, it holds
$$
    dvol_{\eta^*(g)} = v(s) ds dvol_h
$$
with some $v \in C^\infty(I \times M)$. Hence studying the residues of $M(\lambda)$
reduces to studying the residues of
$$
   \left \langle M(\lambda),\psi \right\rangle = \int_I \int_M s^\lambda v(s) \psi(s,x) ds dvol_h.
$$
Now Gelfand's formula
$$
   \Res_{\lambda=-N-1} \left(\int_0^\infty r^\lambda \psi dr \right) = \frac{\psi^{(N)}(0)}{N!}
$$
shows that
\begin{equation}\label{M-res-g}
    \Res_{\lambda=-N-1}  \left \langle M(\lambda),\psi \right \rangle = \frac{1}{N!} \int_M (v\psi)^{(N)}(0)  dvol_h.
\end{equation}
In particular, if $\psi = 1$ near the boundary, then
\begin{equation}\label{M-res}
    \Res_{\lambda=-N-1}  \left \langle M(\lambda),\psi \right \rangle = \int_M v_N dvol_h,
\end{equation}
where the expansion $v(s) = \sum_{j \ge 0} s^j v_j$ defines the coefficients
$v_j(g,\sigma) \in C^\infty(M)$. The above definitions immediately imply that the
residue of $\left \langle M(\lambda),1\right\rangle$ at $\lambda=-n-1$ is a
conformal invariant in the sense that
$$
    I(\hat{g},\hat{\sigma}) = I(g,\sigma), \quad I(g,\sigma) \st \int_M v_n dvol_g.
$$
where $\hat{g} = e^{2\varphi}g$ and $\hat{\sigma} = e^{\varphi} \sigma$. Moreover, the formula
$$
  \left\langle M(\lambda),1 \right\rangle = \int_X \sigma^{\lambda+n+1} dvol_{\sigma^{-2}g}, \; \Re(\lambda) > -1
$$
suggests to regard the finite part of $\left\langle M(\lambda),1\right\rangle$ at
$\lambda = -n-1$ as a {\em renormalized volume} of the singular metric
$\sigma^{-2}g$. In fact, this should be called the {\em Riesz renormalization} of
the volume of $\sigma^{-2} g$ \cite{Albin}. In Theorem \ref{RVE} and Theorem
\ref{RVE-g}, we shall return to the issue of renormalized volumes. We shall see that
the residues in \eqref{M-res} for $N \le n$ determine the singular terms in the {\em
Hadamard renormalization} of the volume of $\sigma^{-2}g$. Similar 
renormalization techniques have also been used in the context of Möbius invariant energies of
knots and their generalizations \cite{Bryl, oha-book, oha-residue}.

Now we continue with a definition of residue families. Here we assume that
$|d\sigma|^2_g = 1$ on $M$. This assumption implies that the metric $\sigma^{-2}g$ is
asymptotically hyperbolic. The data $(g,\sigma)$ give rise to a holomorphic
one-parameter family
\begin{equation}\label{MU}
   \lambda \mapsto \langle M_u(\lambda),\psi \rangle \st
  \int_X \sigma^\lambda u \psi dvol_g, \; \Re(\lambda) \gg 0
\end{equation}
of distributions $M_u(\lambda)$ on $X$. Here $\psi$ is a test function on $X$ with
support up to the boundary, and the additional datum $u$ is an eigenfunction of the
Laplacian of the singular metric $\sigma^{-2} g$:
$$
   -\Delta_{\sigma^{-2} g} u = \mu (n-\mu) u, \; \Re(\mu) = n/2, \, \mu \ne n/2
$$
with {\em boundary value} $f \in C^\infty(M)$. In particular, it holds $M(\lambda) =
M_1(\lambda)$ since $u=1$ is an eigenfunction of $\Delta_{\sigma^{-2}g}$ with
eigenvalue $0$ and boundary value $1$. Since $\sigma^{-2}g$ is asymptotically
hyperbolic, there is an eigenfunction $u$ for any $f \in C^\infty(M)$ so that in its
asymptotic expansion $f$ defines one of the leading terms. Now $M_u(\lambda)$ admits
a meromorphic continuation to $\C$ with simple poles in $\{ -\mu-N-1, \, |\, N \in
\N \}$. For $N \in \N_0$, we set
\begin{equation}\label{res-fam-basic}
   \D_N^{res}(g,\sigma;\lambda)
   \st N!(2\lambda\!+\!n\!-\!2N\!+\!1)_N \delta_N(g,\sigma;\lambda\!+\!n\!-\!N),
\end{equation}
where $(a)_N = a (a+1) \cdots (a+N-1)$ is the Pochhammer symbol and
$$
   \Res_{\lambda=-\mu-1-N} (\langle M_u(\lambda), \psi \rangle)
   = \int_M f \delta_N(g,\sigma;\mu)(\psi) dvol_h
$$
for a meromorphic family $\delta_N(g,\sigma;\mu)$ of differential operators
$C^\infty(X) \to C^\infty(M)$ of order $\le N$. $M_u(\lambda)$ should be regarded as
a further generalization of Gelfand's distribution $r_+^\lambda$ \cite{Gelfand}. The
factor in \eqref{res-fam-basic} guarantees that the {\em residue family}
$\D_N^{res}(\lambda)$ is polynomial in $\lambda$. Its degree equals $2N$. It easily
follows from these definitions that $\D_N^{res}(\lambda)$ satisfies a conformal
transformation law which is analogous to that of $L_N(\lambda)$. In view of
$M(\lambda) = M_1(\lambda)$, the residue formula \eqref{M-res-g} can be restated as
$$
   \int_M \D_N^{res}(g,\sigma;0)(\psi) dvol_h \sim \int_M (v \psi)^{(N)}(0) dvol_h
$$
(for $N < n$). The above definition generalizes the notion of residue families
introduced in \cite{J1}. For the details, we refer to Section \ref{res-fam}.

We also emphasize that while $L_N(\lambda)$ maps functions on $X$ to functions on
$X$, residue families map functions on $X$ to functions on the boundary $M$.


The families $L_N(g,\sigma;\lambda)$ and $D_N^{res}(g,\sigma;\lambda)$ are the main
objects of the present paper. These objects and most related discussions will only depend on 
approximations of the data $(g,\sigma)$ in a sufficiently small neighborhood of the boundary $M$.

As a preparation for the following discussion, we briefly recall the role of residue
families for Poincar\'e-Einstein metrics in \cite{J1, J2}. These polynomial families
of differential operators are defined through a residue construction as in
\eqref{res-fam-basic}, where $g_+ = r^{-2}(dr^2 + h_r)$ is a Poincar\'e-Einstein
metric on $X = (0,\varepsilon) \times M$ in normal form relative to a given metric
$h$ on $M$. Let $\iota: M \hookrightarrow X$ be the embedding $m \mapsto (0,m)$. The
theory in \cite{J1, J2} deals with approximations of Poincar\'e-Einstein metrics
which are completely {\em determined} by the metric $h$. In particular, for even
$n$, these define even-order residue families $D_{2N}^{res}(h;\lambda)$ of order $2N
\le n$. Since
\begin{equation}\label{GJMS-res}
   D_{2N}^{res}\left(h;-\frac{n}{2}+N\right) = P_{2N}(h) \iota^*,
\end{equation}
the residue family $D_{2N}^{res}(h;\lambda)$ can be regarded as a perturbation of
the GJMS-operator $P_{2N}(h)$. Moreover, residue families satisfy systems of
recursive relations that involve lower-order residue families and GJMS-operators.
These imply recursive relations among GJMS-operators and recursive
relations for $Q$-curvatures. In other words, residue families may be viewed as a
device to study the GJMS-operators and $Q$-curvatures on the boundary $M$.

The above more general construction deals with a general metric $g$ on $X$. It
induces a metric $h = \iota^*(g)$ on $M$, but of course is not determined by $h$. In
that case, the resulting residue families depend on the metric $g$ {\em and} a
boundary defining function $\sigma$. To get residue families, which are
determined only by the metric $g$ and the embedding $\iota$, we put an extra
condition on $\sigma$. At the same time, we restrict considerations to residue
families for $N \le n$. Similarly, as in the Poincar\'e-Einstein case, this restriction means
that it suffices to deal with finite approximations of true eigenfunctions. This
way, we obtain specific conformally covariant polynomial one-parameter families of
differential operators $C^\infty(X) \to C^\infty(M)$. However, we emphasize that 
the problem of describing all such conformally covariant families is more complicated \cite{J1, GP}.

Now, following \cite{GW-LNY}, the extra condition which we pose is that $\sigma$
solves a singular Yamabe problem. The version of that problem of interest here asks
to find a defining function $\sigma$ of $M$ so that $\scal(\sigma^{-2}g) = -n(n+1)$.
By \cite{AMO1, ACF} such $\sigma$ exist and are unique. However, in
general, $\sigma$ is not smooth up to the boundary. More precisely, the existence of
smooth solutions is obstructed by a conformally covariant scalar curvature invariant
$\B_n$ called the {\em singular Yamabe obstruction}. Although this lack of smoothness 
will only play a minor role for our main purposes, we also have 
some new results for the invariant $\B_n$. By the conformal transformation law of scalar 
curvature, the condition that $\sigma$ solves the singular Yamabe problem can be restated as
\index{$\SC(g,\sigma)$}
$$
    \SC(g,\sigma)  \st |d\sigma|_g^2 + 2 \sigma \rho \stackrel{!}{=} 1.
$$
This role of the functional $\SC(g,\sigma)$ also implies its conformal invariance:
$$
    \SC(\hat{g},\hat{\sigma}) = \SC(g,\sigma).
$$
The main results of the present work only require us to assume that $\sigma$ satisfies
the weaker condition
\begin{equation}\label{condition-Y}
    \SCY:  \hspace{0.5cm}  \SC(g,\sigma) = 1 + \sigma^{n+1} R_{n+1} \qquad \qquad
\end{equation}
with a smooth remainder term $R_{n+1}$. The restriction of the remainder term to
$\sigma = 0$ then defines the singular Yamabe obstruction $\B_n$. For more details
and background on the singular Yamabe problems, we refer to Section \ref{Yamabe}.

Now we are ready to state the first main result. Let $\iota: M \hookrightarrow X$ be
the embedding of $M$ into $X$.

\begin{theo}\label{MT-1}
Let $\N \ni N \le n$ and assume that the condition $\SCY$ is satisfied. Then
\begin{equation}\label{equivalence}
   \iota^* L_N(g,\sigma;\lambda) = \D_N^{res}(g,\sigma;\lambda), \; \lambda \in \C.
\end{equation}
\end{theo}

Some comments on this result are in order. The identity \eqref{equivalence} is to be
interpreted as an identity of operators acting on smooth functions with support in a
sufficiently small neighborhood of $M$. Although the assumptions guarantee that the
operators on both sides of \eqref{equivalence} only depend on the metric $g$ and the 
embedding $\iota$, we keep the notation of the general case indicating the dependence 
of the operators on $g$ and $\sigma$. If $g = r^2 g_+$ is a Poincar\'e-Einstein metric 
in normal form relative to the metric $h$ on $M$, then
the equality \eqref{equivalence} was proven in \cite{FJO}.\footnote{\cite{FJO} uses
a different normalization of residue families.} In particular, if $g$ is the
Euclidean metric on $\R^{n+1}$, then Theorem \ref{MT-1} reduces to the main result
of \cite{C} (for more details, we refer to \cite{FJO}).

Theorem \ref{MT-1} is a consequence of the following identity.

\begin{theo}\label{MT-2} For any boundary defining function $\sigma \in C^\infty(X)$, it
holds
\begin{equation*}\label{M-CF}
   L(g,\sigma;\lambda) + \sigma^{\lambda-1} \circ \left(\Delta_{\sigma^{-2}g} -
   \lambda(n\!+\!\lambda) \id \right) \circ \sigma^{-\lambda} = \lambda
   (n\!+\!\lambda) \sigma^{-1}(\SC(g,\sigma)-1) \id,\; \lambda \in \C.
\end{equation*}
\end{theo}

Again, Theorem \ref{MT-2} is to be interpreted as an identity near the boundary of $X$.

We apply Theorem \ref{MT-2} to prove the existence of a meromorphic continuation of the
family $M_u(\lambda)$ using a  Bernstein-Sato argument. We recall that for any polynomial $p$,
the classical Bernstein-Sato argument \cite{Bern} proves the existence of a meromorphic 
distribution-valued function $p^\lambda$ by using the functional equation
$$
     D(\lambda) (p^{\lambda+1}) = b(\lambda) p^\lambda
$$
for some polynomial $b(\lambda)$ and a family $D(\lambda)$ of differential operators. Here 
$\sigma^\lambda u$ takes the role of $p^\lambda$ and $L(\lambda)$ takes the role of 
the family $D(\lambda)$. The Bernstein-Sato argument yields formulas for the residues and 
proves Theorem \ref{MT-1}. 

The conjugation formula in Theorem \ref{MT-2} easily implies the commutator relations
\begin{equation}\label{sl2}
    L_N(g,\sigma;\lambda) \circ \sigma - \sigma \circ L_N(g,\sigma;\lambda\!-\!1)
  = N(2\lambda\!+\!n\!-\!N) L_{N-1}(g,\sigma;\lambda\!-\!1)
\end{equation}
if $\SC(g,\sigma)=1$. The relation \eqref{sl2} was discovered in \cite{GW},
where the special case $N=1$ is regarded as one commutator relation in a basic
$sl(2)$-structure. In turn, it follows that for $2\lambda = -n+N$ the operator
$L_N(g,\sigma;\lambda)$ reduces to a tangential operator $\PO_N(g,\sigma)$ on $M$.
By Theorem \ref{MT-1}, we obtain
\begin{equation}\label{PO-Dres}
    \PO_N(g,\sigma) \iota^* = \D_N^{res}\left(g,\sigma;\frac{-n+N}{2}\right).
\end{equation}
In other words, each operator $\PO_N(g,\sigma)$ on $M$ comes with a
polynomial one-parameter family. The definition of $\PO_N$ in terms of $L_N$ is due
to \cite{GW}. The spectral theoretic description \eqref{PO-Dres} of $\PO_N$ will be
shown to have a series of significant consequences.

The conformal covariance of the families $L_N(\lambda)$ implies that the operators
$\PO_N$ on $C^\infty(M)$ are conformally covariant in the sense that
$$
   e^{\frac{n+N}{2} \iota^*(\varphi)} \circ  \PO_N(\hat{g},\hat{\sigma})
  = \PO_N(g,\sigma) \circ e^{\frac{n-N}{2} \iota^* (\varphi)}, \; \varphi \in C^\infty(X).
$$
For $N \le n$, these conformally covariant operators are completely determined by
the metric $g$ and the embedding $M \hookrightarrow X$. Following \cite{GW-LNY}, we
call them {\em extrinsic} conformal Laplacians. The notion is motivated by the fact
that for even $N$ the leading term of $\PO_N(g,\sigma)$ is given by a constant
multiple of $\Delta_h^{N/2}$, where $h =\iota^*(g)$ (see Theorem \ref{LT-M}). If the
background metric $g$ is the conformal compactification $r^2 g_+$ of a
Poincar\'e-Einstein metric in normal form relative to $h$, then these even-order
operators actually are constant multiples of GJMS-operators of $h$. But, in general,
the operators $\PO_N(g,\sigma)$ depend on the metric $g$ in a neighborhood of $M$
and the embedding $M \hookrightarrow X$.

The following result makes the leading parts of all extrinsic conformal Laplacians
explicit. Its proof rests on the spectral theoretic interpretation \eqref{PO-Dres}
of $\PO_N$.

\begin{theo}\label{LT-M} It holds
\begin{equation}\label{LT-P-even-M}
   \PO_{2N} = (2N\!-\!1)!!^2 \Delta^N + LOT
\end{equation}
for $2N \le n$ and
\begin{equation}\label{LT-P-odd-M}
   \PO_N = (2N\!-\!2) (N\!-\!1)!  \sum_{r=0}^{\frac{N-3}{2}} m_N(r)
   \Delta^r \delta(\lo d) \Delta^{\frac{N-3}{2}-r} + LOT
\end{equation}
for odd $N$ with $n \ge N \ge 3$ with rational coefficients given by
\eqref{m-coeff}. Here $\Delta$ is the Laplacian of $h$, and $\lo$ is the trace-free
part of the second fundamental form $L$. LOT refers to terms of order $2N-2$ and
$N-3$, respectively.
\end{theo}

In the Poincar\'e-Einstein case, \cite[Theorem 4.5]{GW} identifies the operator
$\PO_{2N}$ with a constant multiple of the GJMS operator $P_{2N}$. An alternative
proof of that result is given in \cite[Theorem 5.6]{FJO}. In the general case,
formula \eqref{LT-P-even-M} is also stated in \cite[Theorem 8.1]{GW-LNY} although
the proof only refers to results in the Poincar\'e-Einstein case in \cite{GW}. Note
also that in the Poincar\'e-Einstein case, the operators $\PO_N$ for odd $N$ vanish.
Note that the order of $\PO_{2N}$ equals $2N$ and the order of $\PO_N$ for odd $N$
equals $N-1$ in the generic case.

Now we use the operators $\PO_N$ to define analogs of Branson's $Q$-curvatures.
Theorem \ref{MT-1} implies that for $N < n$ the identity
$$
   \PO_N(g,\sigma) (1) = \left(\frac{n-N}{2}\right) \QC_N(g,\sigma)
$$
defines a function $\QC_N(g,\sigma)$. These functions will be called the {\em
subcritical} extrinsic $Q$-curvatures. The {\em critical} extrinsic $Q$-curvature
$\QC_n(g,\sigma)$ can be defined either through a limiting argument from the
subcritical ones or more elegantly by
\begin{equation}\label{Q-crit-M}
   \QC_n^{res}(g,\sigma) = - \dot{\D}_n^{res}(g,\sigma;0)(1).
\end{equation}
We recall that the condition $\SCY$ guarantees that the quantities in\eqref{Q-crit-M} are determined 
by the metric $g$ and the embedding $M \hookrightarrow X$. In the Poincar\'e-Einstein case, these 
definitions reduce to constant multiples of Branson's $Q$-curvatures. If $\SC(g,\sigma)=1$, it follows
from Theorem \ref{MT-1} that the definition \eqref{Q-crit-M} is a special case of
the notion of $Q$-curvature defined in \cite[Theorem 3.9]{GW-reno}.\footnote{The
numbering in the published version differs from that in the arXiv version.} By
differentiating the conformal transformation law of $\D_n^{res}(g,\sigma;\lambda)$
(see Theorem \ref{CTL-RF}) at $\lambda=0$, it follows that
\begin{equation}\label{fundamental-M}
   e^{n\iota^*(\varphi)} \QC_n(\hat{g},\hat{\sigma})
   = \QC_n(g,\sigma) + \PO_n(g,\sigma)(\iota^*(\varphi))
\end{equation}
for $\varphi \in C^\infty(X)$.

As noticed above, the metric $\sigma^{-2}g$ is asymptotically hyperbolic if
$|d\sigma|^2_g=1$ on $M$. Associated to such metrics, there is a scattering operator
$\Sc(\lambda)$ acting on $C^\infty(M)$. It describes the asymptotic expansion of
eigenfunctions of the Laplacian of $\sigma^{-2}g$. In \cite{GZ}, Graham and Zworski
showed how the GJMS-operators $P_{2N}$ and the critical Branson $Q$-curvature $Q_n$
of a metric $h$ on $M$ are encoded in the scattering operator of a
Poincar\'e-Einstein metric with the conformal class of $h$ as conformal infinity.
The following result extends these interpretations to the present framework.

\begin{theo}\label{scatt-PQ-M}
Suppose that condition $\SCY$ is satisfied. Let $N \in \N$ with $2 \le N \le n$ and suppose
that $(n/2)^2-(N/2)^2$ is not in the discrete spectrum of $-\Delta_{\sigma^{-2}g}$. Then
$$
   \PO_N = 2 (-1)^N (N-1)! N! \Res_{\frac{n-N}{2}}(\Sc(\lambda)).
$$
The function $\Sc(\lambda)(1)$ is regular at $\lambda=0$ and its value at $\lambda=0$ is denoted
by $\Sc(0)(1)$. Then
$$
    \QC_n = 2 (-1)^n (n-1)! n! \Sc(0)(1).
$$
\end{theo}

Note that the scattering operator $\Sc(\lambda)$ is a pseudo-differential operator
with principal symbol
$$
   2^{2\lambda-n} \frac{\Gamma(\lambda-\frac{n}{2})}{\Gamma(-\lambda+\frac{n}{2})}|\xi|_h^{n-2\lambda}.
$$
This shows that the residues of $\Sc(\lambda)$ at $\lambda=\frac{n}{2}-N$ are caused
by the $\Gamma$-factors and that the residue of its pole at $\lambda =
\frac{n}{2}-N$ is an operator with principal part given by a constant multiple of $\Delta^N$. Moreover,
Theorem \ref{scatt-PQ-M} yields an identification of this residue. In particular,
these residues are elliptic operators. On the other hand, the above formula yields
no information on the structure of the residues of $\Sc(\lambda)$ at $\frac{n-N}{2}$
for {\em odd} $N$. This makes the second part of Theorem \ref{LT-M} interesting.

Similarly as in \cite{GZ}, the self-adjointness of $\Sc(\lambda)$ for $\lambda \in
\R$ combined with Theorem \ref{scatt-PQ-M} implies

\begin{corro}\label{PO-sa}
Let $\SCY$ be satisfied. Then the operators $\PO_N(g)$ are formally self-adjoint
with respect to the scalar product on $C^\infty(M)$ defined by $h$.
\end{corro}

For closed $M$, combining \eqref{fundamental-M} with the self-adjointness of $\PO_n$
and $\PO_n(1)=0$ shows that the integral
$$
   \int_M \QC_n(g) dvol_h
$$
is a global conformal invariant.

Next, we describe a {\em local} formula for the critical extrinsic $Q$-curvature.
It extends the holographic formula for Branson's critical $Q$-curvature proved in
\cite{GJ}. The formulation of that result requires two more ingredients:
renormalized volume coefficients $v_k$ and solution operators $\T_k(0)$. These data
are defined in local coordinates. We choose a diffeomorphism $\eta$ between the
product $[0,\varepsilon) \times M$ with coordinates $(s,x)$ and a neighborhood of
$M$ in $X$. $\eta$ is defined by a renormalized gradient flow of $\sigma$. It
satisfies $\eta^*(\sigma) = s$. These coordinates will be called {\em adapted
coordinates}. Then
$$
    \eta^* (dvol_{\sigma^{-2}g}) = s^{-n-1} v(s,x) ds dvol_h,
$$
and the expansion $v(s,x) = \sum_{k\ge 0} v_k(x) s^k$ defines the coefficients
$v_k$. Let $(v\J)_k$ be the coefficients in the analogous expansion of $v
\eta^*(\J)$, where $\J = \J^g = \scal^g/2n$. Secondly, the differential operators
$\T_k(0)$ describe the asymptotic expansions of harmonic functions of the Laplacian
of $\sigma^{-2} g$. The details of these definitions are given in Sections
\ref{AC-RVC}--\ref{res-fam}.

\begin{theo}[\bf Holographic formula for $\QC_n$]\label{Q-holo-crit-M}
Let $n$ be even. If $\SCY$ is satisfied, then it holds
\begin{equation*}\label{Q-holo-form-M}
   \QC_n(g,\sigma) = (n\!-\!1)!^2 \sum_{k=0}^{n-1} \frac{1}{n\!-\!1\!-\!2k}
   \T_k^*(g,\sigma;0) \left( (n\!-\!1)(n\!-\!k) v_{n-k} + 2k (v \J)_{n-k-2} \right),
\end{equation*}
where $\J = \J^g$. For $k=n-1$, the second term in the sum is defined as $0$.
\end{theo}

The proof of Theorem \ref{Q-holo-crit-M} generalizes a proof of the holographic
formula for $Q_n$ for Poincar\'e-Einstein metrics given in \cite{FJO}. The arguments
rest on Theorem \ref{MT-1} and an alternative description of residue families in
terms of the coefficients $v_k$ and the operators $\T_k(\lambda)$. We conjecture
that the result extends to odd $n$ (Conjecture \ref{van}). There are analogous
formulas for all subcritical extrinsic $Q$-curvatures which extend \cite{J3}. 

Theorem \ref{Q-holo-crit-M} implies the representation
$$
  \QC_n = n! (n\!-\!1)! v_n + \sum_{k=1}^{n-1} \T_k^*(0)(\cdots).
$$
Integration of this identity for closed $M$ implies the equality
\begin{equation}\label{GI}
   \int_M \QC_n dvol_h =  n! (n\!-\!1)!  \int_M v_n dvol_h
\end{equation}
of global conformal invariants. This extends a result of \cite{GZ} for
Poincar\'e-Einstein metrics and reproves a special case of \cite[Theorem
3.9]{GW-reno}. We shall also give a second proof of the identity \eqref{GI} which
utilizes residue families and extends to odd $n$. We emphasize again that the
integrands on both sides depend on the embedding of $M$ in $X$, i.e., these are
global conformal invariants of the embedding.

The integral on the right-hand side of \eqref{GI} appears as a coefficient of $\log
\varepsilon$ in the asymptotic expansion of the volume
\begin{equation}\label{reno-volume-M}
   vol_{\sigma^{-2}g}(\{\sigma > \varepsilon \}) = \int_{\sigma > \varepsilon} dvol_{\sigma^{-2}g}
\end{equation}
of the singular metric $\sigma^{-2}g$ for $\varepsilon \to 0$. This implies its
conformal invariance. Through that interpretation of the right-hand side of
\eqref{GI}, this identity is a special case of \cite[Theorem 3.9]{GW-reno}. The
singular terms in this expansion of the volume are given by the integrals of the
renormalized volume coefficients $v_k$ for $k <n$. All these integrals admit the
uniform expressions
\begin{equation}\label{HF-v}
   \int_{M} v_{k} dvol_h
   = (-1)^k \frac{(n\!-\!1\!-\!k)!}{ (n-1)! k!} \int_{M} \iota^* L_k(-n\!+\!k)(1) dvol_h
\end{equation}
in terms of Laplace-Robin operators. This reproves a special case of \cite[Proposition 3.5]{GW-reno}.

The above results for the expansion of the volumes \eqref{reno-volume-M} have analogs for general 
defining functions $\sigma$. In that case, the operators $L_k$ are replaced by composition $\tilde{L}_k 
= (L \circ \SC^{-1})_k$ and the coefficient of $\log \varepsilon$ is a constant multiple of the integral 
of a curvature quantity $\tilde{\QC}_n(g,\sigma)$ (Theorem \ref{RVE-g}, \eqref{tilde-Qn}). The total 
integral of $\tilde{\QC}_n$ is a conformal invariant of the pair $(g,\sigma)$ (Lemma \ref{Q-tilde}).

The integrals in \eqref{GI} are generalizations of the Willmore energy of a closed
surface in $\R^3$. In fact, for a surface $M^2 \hookrightarrow X^3$ with Gauss
curvature $K$, it holds
$$
   2 v_2 = \QC_2 = -K + \frac{1}{2} |\lo|^2
$$
(by \eqref{Q2-gen}, \eqref{v2-2} and $\J^h = K$).  But if $X$ is the flat $\R^3$,
then $H^2 - K = |\lo|^2$. Hence for closed $M \hookrightarrow \R^3$ the integral
$\int_M v_2 dvol_h$ is a linear combination of the Euler characteristic $\chi(M)$ and
the Willmore energy                  \index{$\W$ \quad Willmore energy}
$$
   \W =  \int_M H^2 dvol_h.
$$
The Willmore energy also contributes to the rigid string action introduced in \cite[(4)]{Pol},
\cite[(5)]{Klein}.

Finally, the usage of adapted coordinates leads to a new formula for the singular Yamabe
obstruction $\B_n$ (see Section \ref{Yamabe} for its definition). In order to state that formula,
we note that the pull-back $\eta^*(g)$ of $g$ takes the form $a^{-1} ds^2 + h_s$ with some
coefficient $a \in C^\infty ((0,\varepsilon) \times M)$ and a one-parameter family $h_s$ of
metrics on $M$. For
$$
   \mathring{v} \st  dvol_{h_s}/dvol_h,
$$
it holds
$$
   \frac{\mathring{v}'}{\mathring{v}} = \frac{1}{2} \tr (h_s^{-1} h_s').
$$
Let $\J$ and $\rho$ denote the pull-backs of $\J^g$ and $\rho(g,\sigma)$ by $\eta$, respectively. The
following result restates Theorem \ref{obstruction-ex}.

\begin{theo}[\bf The obstruction $\B_n$]\label{B-form}
If $\sigma$ satisfies $\SCY$, then
\begin{equation}\label{ob-magic-2}
   (n\!+\!1)! \B_n = -2\partial_s^n \left( \frac{\mathring{v}'}{\mathring{v}}\right)|_0
  + 4  \sum_{j=1}^n j \binom{n}{j} \partial_s^{j-1}(\rho)|_0 \partial_s^{n-j}
   \left( \frac{\mathring{v}'}{\mathring{v}}\right)|_0 - 2n \partial_s^{n-1}(\J)|_0
\end{equation}
for $n \ge 1$.
\end{theo}

For $n=1$, the identity \eqref{ob-magic-2} yields $\B_1 = 0$.

Moreover, the Taylor coefficients of $\rho$ obey a recursive relation in terms of the
Taylor coefficients of $\mathring{v}'/\mathring{v}$ and $\J$ (Proposition \ref{rec-rho}).

This result has several consequences of independent interest. For a flat
background metric, the scalar curvature term on the right-hand side of \eqref{ob-magic-2}
vanishes, and the recursive structure of all terms implies that for even $n$
\begin{equation*}
   (n\!+\!1)! \B_n =  - 2 (n\!-\!1)!!/(n\!-\!2)!! \Delta^\frac{n}{2} (H) + LOT
\end{equation*}
(Theorem \ref{B-linear}), where LOT refers to terms of lower differential order. Up
to the numerical coefficients, that description of the leading term of $\B_n$ was
first formulated in \cite{GW-announce} \cite[Theorem 5.1]{GW-LNY} (for general
backgrounds). We also use Theorem \ref{B-form} to reproduce the known explicit
formulas for $\B_2$ (for general backgrounds) and to calculate $\B_3$ for
(conformally) flat backgrounds confirming earlier results.

Moreover, Theorem \ref{Q-holo-crit-M} and Theorem \ref{B-form} play a key role in the proof of the
residue formula
\begin{equation*}
   \Res_{n=N-1} (\QC_N) = (-1)^{n-1} n! (n\!+\!1)! \frac{n}{2} \B_n
\end{equation*}
for even $n$ (Theorem \ref{QB-residue}), where we regard $\QC_N$ as a function of
$n$ which is singular at $n=N-1$. This formula relates $\B_n$ for even $n$ to the
super-critical $Q$-curvature $\QC_{n+1}$. It proves a conjecture in \cite[Remark
8.6]{GW-LNY}. We conjecture that this result extends to all dimensions.

\cite[Theorem 8.11]{GW-LNY} established a coordinate-free formula for the Yamabe obstructions 
$\B_n$ in terms of tractor calculus constructions. This formula involves an analog of the critical 
extrinsic conformal Laplacian acting on the so-called normal tractor of the hypersurface $M$ as a 
key ingredient.\footnote{This result coincides with Theorem 7.7 in \url{arXiv:150602723v4}.}

We finish this section with an outline of the content of the paper. We combine this
review with additional comments on the relations of the new results to the
literature. Section \ref{Yamabe} briefly recalls basic results on the version of the
singular Yamabe problem of interest here. In Section \ref{CFormula}, we establish
Theorem \ref{M-CF} and describe first consequences. This result is the basic link that
connects the Laplace-Robin operators of Gover and Waldron with the spectral theory
of the Laplacian of the singular metric $\sigma^{-2}g$. In Section \ref{SBO}, we
describe some of the representation theoretical aspects of the Laplace-Robin
operators. Section \ref{AC-RVC} introduces the notion of adapted coordinates. This
notion is basic for the whole work. Adapted coordinates give rise to the definition
of renormalized volume coefficients $v_k$, which enables us to prove local formulas
for $Q$-curvatures (Theorem \ref{Q-holo-crit-M}) and to prove the formula for the
obstruction $\B_n$ in Theorem \ref{B-form}. In this connection, we show that in
adapted coordinates, the function $\rho$ obeys an extremely beneficial ordinary differential
equation in the variable $s$, which implies recursive relations for the Taylor
coefficients of $\rho$. We emphasize that Graham \cite{Graham-Yamabe} utilizes a
different notion of renormalized volume coefficients which are defined in terms of
the normal exponential map. This notion seems to be less appropriate from the
present point of view. Section \ref{res-fam} introduces the notion of residue
families. Here we extend earlier definitions in the context of Poincar\'e-Einstein
metrics \cite{J1, J2, BJ}. The constructions are built on basic facts in the
spectral theory of asymptotically hyperbolic metrics, which are detailed in
\cite{GZ}. In particular, we use a version of a Poisson transform and introduce the
scattering operator $\Sc(\lambda)$. We derive a formula for residue families in
terms of the coefficients $v_k$ and the operators $\T_k(\lambda)$. Section
\ref{products} contains the proof of Theorem \ref{MT-1}. The identity \eqref{HF-v}
is a direct consequence. Here we also prove that these identities can be regarded as
consequences of a beautiful formula for the action of $L_N(-1)$ on the distribution
$\sigma^*(\delta) = \delta_M$ of $M$ (Theorem \ref{dist-volume}). This
distributional formula reproves a basic technical result in \cite{GW-reno}. In
Section \ref{ECPL}, we introduce the notion of extrinsic conformal Laplace operators
$\PO_N$. Here we derive the spectral theoretical interpretation \eqref{PO-Dres} of
$\PO_N$, use it to determine the leading terms of $\PO_N$ (Theorem \ref{LT-M}) and
recognize these operators as residues of the scattering operator (first part of
Theorem \ref{scatt-PQ-M}). These results extend results of \cite{GZ}. Section
\ref{Q-curvature} defines extrinsic $Q$-curvatures and establishes the second part
of Theorem \ref{scatt-PQ-M} which extends another result of \cite{GZ}. Moreover, we
supply two proofs of the equality \eqref{GI}. The integrated coefficients $v_k$ are
shown to describe the singular coefficients in the asymptotic expansion of
\eqref{reno-volume-M}. These expansions and their generalizations (Theorem
\ref{RVE-g}) reprove results of \cite{GZ, GW-reno}. In Section \ref{Q-hol}, 
we establish extensions of the holographic formulas for $Q$-curvatures in \cite{GJ,J3}. In
particular, we prove Theorem \ref{Q-holo-crit-M}. Section \ref{f-comm} contains
comments on further perspectives.

In the main body of the text, we usually suppress detailed calculations and the
discussion of examples and special cases. However, the reader can find this material in
Section \ref{comm-cal}. It starts with an overview of its own. The results presented
here may be used to gain a deeper understanding of the material. In particular,
this section contains full details on low-order Yamabe obstructions, low-order cases
of conformal Laplacians, extrinsic $Q$-curvatures, and renormalized volume
coefficients. All proofs are independent of the literature.

Hopefully, the attached list of symbols facilitates reading.

Finally, we like to emphasize that, although the present work is deeply inspired by
the pioneering works of Gover and Waldron, the current treatment is fully
independent and self-contained. We also stress again that our perspective is (via
the residue families and their applications) one of a spectral-theoretic nature
exploiting the structure of eigenfunctions of the Laplacian.

After the present work had been posted, the paper \cite{CMY} again discussed extrinsic 
conformal Laplacians from the perspective of scattering theory. It {\em defines} 
extrinsic conformal Laplacians in terms of the scattering operator of the singular 
Yamabe metric $\sigma^{-2}g$ by mimicking the known relations between 
GJMS-operators and the scattering operator of Poincar\'e-Einstein metrics \cite{GZ}. 
However, the relations between these definitions and the notions introduced by Gover 
and Waldron are established only here.

{\em Acknowledgments}. The work on this project started during a visit of the first
author at the University of {\AA}rhus in autumn 2019. A large part of the paper's
final version was written during a stay of the first author at IHES in early
2020. He is grateful to both organizations for financial support and very
stimulating atmospheres. Finally, we thank the anonymous referee who
provided valuable detailed comments on an earlier version of the manuscript.

\section{General notation}\label{notation}

\index{$\N$ \quad natural numbers}
\index{$\N_0$ \quad non-negative integers}
\index{$(a)_N$ \quad Pochhammer symbol}
\index{$\mathfrak{X}(X)$ \quad space of vector fields on $X$}

$\N$ is the set of natural numbers, and $\N_0$ is the set of non-negative integers. For
a complex number $a\in\C$ and an integer $N\in\N$, the Pochhammer symbol $(a)_N$ is
defined by $(a)_N \st a(a+1) \cdots (a+N-1)$. We set $(a)_0 \st 1$.

\index{$R$ \quad curvature tensor}
\index{$C^\infty(X)$}
\index{$C^{-\infty}(X)$ \quad space of distributions on $X$}

All manifolds are smooth. For a manifold $X$, $C^\infty(X)$ and $C_c^\infty(X)$
denote the respective spaces of smooth functions and smooth functions with compact
support on $X$. If $X$ is a manifold with boundary, then $C^\infty(X)$ is the space of 
functions that are smooth up to the boundary. $C^{-\infty}(X)$ is the space of distributions 
on $X$.  Let $\mathfrak{X}(X)$ be the space smooth vector fields on $X$. Metrics 
on $X$ usually are denoted by $g$. $dvol_g$ is the Riemannian volume element defined by $g$. The
Levi-Civita connection of $g$ is denoted by $\nabla_X^g$ or simply $\nabla_X$ for $X
\in \mathfrak{X}(X)$ if $g$ is understood. In these terms, the curvature tensor $R$
of the Riemannian manifold $(X,g)$ is defined by $R(X,Y)Z =\nabla_X \nabla_Y (Z) -
\nabla_Y \nabla_X (Z) - \nabla_{[X.Y]}(Z)$ for vector fields $X,Y,Z \in
\mathfrak{X}(X)$. We also set $\nabla_X (u) = \langle du,X \rangle$ for $X \in
\mathfrak{X}(X)$ and $u \in C^\infty(X)$.

\index{$dvol_g$ \quad volume element}
\index{$\Omega^p$ \quad space of $p$-forms}
\index{$\nabla^g$ \quad Levi-Civita connection of $g$}
\index{$\grad_g(u)$ \quad gradient field}
\index{$\delta^g$ \quad divergence operator}
\index{$\Delta_g$ \quad Laplacian of $g$}

For a metric $g$ on $X$ and $\sigma \in C^\infty(X)$, let $\grad_g(\sigma)$ be the
gradient of $\sigma$ with respect to $g$, i.e., it holds $g(\grad_g(\sigma),V) =
\langle d\sigma,V \rangle$ for all vector fields $V \in \mathfrak{X}(X)$. $g$
defines pointwise scalar products and norms $|\cdot|_g$ on $\mathfrak{X}(X)$ and on
forms $\Omega^*(X)$. Then $|\grad_g(\sigma)|_g^2 = |d\sigma|_g^2$. $\delta^g$ is the
divergence operator on differential forms or symmetric bilinear forms. On forms, it
coincides with the negative adjoint $-d^*$ of the enaböe  differential $d$ with
respect to the Hodge scalar product defined by $g$. Let $\Delta_g = \delta^g d$ be
the non-positive Laplacian on $C^\infty(X)$. On the Euclidean space $\R^n$, it
equals $\sum_i \partial_i^2$.

\index{$\Ric^g$ \quad Ricci tensor of $g$}
\index{$\scal^g$ \quad scalar curvature of $g$}
\index{$\scal(g)$ \quad scalar curvature of $g$}
\index{$\Rho^g$ \quad Schouten tensor of $g$}
\index{$\J^g$}

A metric $g$ on a manifold $X$ with boundary $M$ induces a metric $h$ on $M$.
Curvature quantities of $g$ and $h$ have the respective metric as an index if
required by clarity. In particular, the scalar curvature of the metric $g$ on $X$ is
denoted by $\scal^g$ or $\scal(g)$. $\Ric^g$ denote the Ricci tensor of $g$. On a
manifold $(M,g)$ of dimension $n$, we set $\J^g =\frac{1}{2(n-1)} \scal^g$ if $n \ge
2$ and define the Schouten tensor $\Rho^g$ of $g$ by $(n-2)\Rho^g = \Ric^g - \J^g g$
if $n \ge 3$.

\index{$L$ \quad second fundamental form}

Let $M$ be a hypersurface in $(X,g)$ with the induced metric $h$. The second
fundamental form $L$ of $M$ is defined by $L(X,Y) = -g(\nabla_X(Y),N)$ for vector
fields $X,Y \in \mathfrak{X}(M)$ and a unit normal vector field $N$. In particular,
if $X$ is a manifold with boundary $M$ with defining function $\sigma$ so that
$|d\sigma|_g=1$ on $M$, then we set $L(X,Y) = -g(\nabla_X(Y),\NV)$ with $\NV \st
\grad_g(\sigma)$. With these conventions, $L=g$ for the round sphere $S^n \subset
\R^{n+1}$ if $\sigma$ is the distance function of $S^n$. We set $n H = \tr_h(L)$ if
$M$ has dimension $n$. $H$ is the mean curvature of $M$. Let $\lo = L - H h$ be the
trace-free part of $L$. We sometimes identify $L$ with the shape operator $S$ defined by
$h(X,S(Y)) = L(X,Y)$. For a bilinear form $b$ on $T(M)$, we denote its trace-free
part with respect to a metric known by context by $\mathring{b}$ or $b_\circ$.

\index{$\NV$ \quad gradient of $\sigma$}
\index{$H$ \quad mean curvature}
\index{$\lo$ \quad trace-free part of $L$}
\index{$\iota$ \quad embedding}
\index{$b_\circ$ \quad trace-free part of $b$}

$\iota$ denotes various canonical embeddings such as $\iota: M \hookrightarrow X$. 
The symbol $\iota^*$ will be used for the induced pull-back of functions, forms, and metrics. 
For any diffeomorphism $f$, the symbol $f_* = (f^{-1})^*$ denotes the push-forward by $f$.
Differentiation with respect to the variable $\lambda$ will often be denoted by
$^\cdot$. In contrast, differentiation with respect to the variables $r$ and $s$
will usually be denoted by $'$. The symbol $\circ$ denotes compositions of
operators. The symbol $\sim$ indicates a proportionality.

\section{The singular Yamabe problem}\label{Yamabe}

Let $(X,g)$ be a compact manifold with boundary $M$ of dimension $n$. The problem to
ask for a defining function $\sigma$ of $M$ so that
\begin{equation}\label{syp}
    \scal (\sigma^{-2}g) = -n(n+1)
\end{equation}
is known as (a version of) the singular Yamabe problem \cite{LN}. The metric $\sigma^{-2} g$ 
is called a singular Yamabe metric if \eqref{syp} is true.\footnote{Later we shall often deal with 
a weaker condition.} The conformal transformation law of scalar curvature shows that
$$
    \scal(\sigma^{-2}g) = -n(n+1) |d\sigma|_g^2 + 2n \sigma \Delta_g(\sigma) + \sigma^2 \scal(g).
$$
Following \cite{GW}, we write this equation in the form
$$
   \scal(\sigma^{-2}g) = -n(n+1) \SC(g,\sigma),
$$
where
$$
   \SC(g,\sigma) = |d\sigma|_g^2 + 2 \rho \sigma
$$
(see Definition \eqref{LRO}).\footnote{In \cite[Section 2.2]{GW-reno}, the quantity
$\SC(g,\sigma)$ is termed the $\SC$-curvature of $(g,\sigma)$.} In these terms,
$\sigma$ is a solution of \eqref{syp} iff $\SC(g,\sigma)=1$. Such $\sigma$ exist
\cite{AMO1} and are unique \cite{ACF}. However, in general, $\sigma$ is not smooth up to
the boundary. The smoothness is obstructed by a locally determined conformally
invariant scalar function on $M$, called the singular Yamabe obstruction.
Moreover, the solution is smooth up to the boundary iff the obstruction vanishes
\cite{ACF}.

In order to describe the structure of $\sigma$ more precisely, we follow \cite{ACF}
and \cite{Graham-Yamabe}. We use geodesic normal coordinates (see Section
\ref{AC-RVC}). Let $r = d_M$ be the distance function of $M$ for the background metric $g$.
Then there are uniquely determined coefficients $\sigma_{(k)} \in C^\infty(M)$ for $2 \le k \le
n+1$ so that the smooth defining function                           \index{$\sigma_F$}
\begin{equation}\label{sigma-finite}
   \sigma_F \st  r + \sigma_{(2)} r^2 + \dots + \sigma_{(n+1)} r^{n+1}
\end{equation}
satisfies
\begin{equation}\label{Yamabe-finite}
   \SC(g,\sigma_F) = 1 + R r^{n+1}
\end{equation}
with a smooth remainder term $R$. We briefly describe how these coefficients are
recursively determined. In geodesic normal coordinates, the metric $g$ takes the form
$dr^2 + h_r$ with a one-parameter family $h_r$ of metrics on $M$. The condition
\eqref{Yamabe-finite} is equivalent to
$$
  |d\sigma_F|_g^2 - \frac{2}{n+1} \sigma_F \Delta_g(\sigma_F) - \frac{1}{n(n+1)} \sigma_F^2 \scal^g
  = 1 +  R r^{n+1}.
$$
We write the left-hand side of this equation in the form
\begin{align}\label{Y-F}
  & \partial_r(\sigma_F)^2 + h_r^{ij} \partial_i (\sigma_F) \partial_j (\sigma_F) \notag \\
  & - \frac{2}{n+1} \sigma_F \left (\partial_r^2 (\sigma_F) + \frac{1}{2} \tr (h_r^{-1} h_r') \partial_r (\sigma_F)
  + \Delta_{h_r} (\sigma_F) \right) - \frac{1}{n(n+1)} \sigma_F^2 \scal^g
\end{align}
and expand this sum into a Taylor series in the variable $r$. Then the vanishing of
the coefficient of $r^k$ for $k \le n$ is equivalent to an identity of the form
$$
   (k-1-n) \sigma_{(k+1)} = LOT,
$$
where $LOT$ involves only lower-order Taylor coefficients of $\sigma$. The latter relation
also shows that there is a possible obstruction to the existence of an improved
solution $\sigma_F'$ which contains a term $\sigma_{(n+2)} r^{n+2}$ and satisfies
$\SC(g,\sigma_F') = 1 + R r^{n+2}$. However, by setting          \index{$\LO_n$}
\begin{equation}\label{sol-log-d}
   \sigma = \sigma_F + \LO_n r^{n+2} \log r
\end{equation}
with an appropriate coefficient $\LO_n \in C^\infty(M)$ one may get a solution of
\begin{equation}\label{sol-log}
   \SC(g,\sigma) = 1 + O(r^{n+2} \log r).
\end{equation}
The coefficient $\LO_n$ is determined by the condition that the coefficient of
$r^{n+1}$ in the expansion of \eqref{Y-F} vanishes. But that coefficient equals
$$
   \left( r^{-n-1} (\SC(g,\sigma_F) - 1) \right)|_{r=0} - 2 \frac{n+2}{n+1} \LO_n.
$$
The first term exists by the construction of $\sigma_F$ and the second term is
generated by the action of the terms $\partial_r(\sigma)^2$ and $\sigma
\partial_r^2(\sigma)$ in \eqref{Y-F} on the log-term in $\sigma$.\footnote{In fact, $\partial_r (r^{n+2} \log r)
= r^{n+1} + \cdots$ and $\partial_r^2 (r^{n+2} \log r) = (2n+3) r^n + \cdots$.} Hence for
$$
   \LO_n \st \frac{1}{2} \frac{n+1}{n+2} \left( r^{-n-1} (\SC(g,\sigma_F) - 1) \right)|_{r=0}
$$
the condition \eqref{sol-log} is satisfied. Following \cite{ACF}, we define the
{\em singular Yamabe obstruction} by
\index{$\B_n$ \quad singular Yamabe obstruction}
\begin{equation}\label{B-def}
   \B_n \st \left( r^{-n-1} (\SC(g,\sigma_F) - 1) \right)|_{r=0} .
\end{equation}
In these terms, we see that with
\begin{equation}\label{obstruction-two}
   \LO_n = \frac{1}{2} \frac{n+1}{n+2} \B_n
\end{equation}
the improved $\sigma$ defined in \eqref{sol-log-d} satisfies \eqref{sol-log}. By
\cite{ACF}, the unique solution $\sigma$ of the singular Yamabe problem has an
expansion of the form
$$
   \sigma = r + \sigma_{(2)} r^2 + \dots + \sigma_{(n+1)} r^{n+1} + \LO_n r^{n+2} \log r + \dots.
$$
Graham \cite{Graham-Yamabe} calls $\LO_n$ the singular Yamabe obstruction. 

Since $\sigma_F$ is determined by $g$ (and the embedding $M \hookrightarrow X$), we regard
$\B_n$ as a functional of $g$ (and the embedding). It is a key result that $\B_n$ is a conformal invariant
of $g$. More precisely, we write $\hat{\B}_n$ for the obstruction defined by $\hat{g}=e^{2\varphi} g$
with $\varphi \in C^\infty(X)$. Then


\begin{lem}[\cite{ACF,Graham-Yamabe}]\label{B-CTL} $e^{(n+1) \iota^*(\varphi)} \hat{\B}_n = \B_n$.
\end{lem}

\begin{proof} Let $r$ be the distance function of $M$ for $g$. Let
$$
   S_N \st 1 + \sum_{j=2} r^j \sigma_{(j)}.
$$
Then for $N \le n+1$ the condition
\begin{equation}\label{scalar-exp}
   \SC (g,S_N) = 1 + 0 r + \cdots + 0 r^{N-1} + O(r^N) = 1 +  O(r^N)
\end{equation}
determines the coefficient $\sigma_{(N)}$ in terms of lower-order coefficients $\sigma_{(2)}, \dots, \sigma_{(N-1)}$.
In that case, we also write $S_N = S_N(g)$. Moreover, we recall that the coefficient of $r^k$ ($k \le N-1$) in the
expansion of $\SC(g,S_N)$ depends only on $\sigma_{(2)}, \dots, \sigma_{(k+1)}$.

Now we have the obvious relation
$$
   \SC(\hat{g},\hat{\sigma}) = \SC(g,\sigma)
$$
for any $\sigma \in C^\infty(X)$ and $\hat{\sigma} \st e^{\varphi} \sigma$.

We write the expansion \eqref{scalar-exp} (for $N \le n+1$) as an expansion in terms of
the distance function $\hat{r}$ of $M$ for $\hat{g}$. Hence
\begin{equation}\label{scalar-exp-hat}
     \SC(\hat{g},e^\varphi S_N(g)) = 1 + O(\hat{r}^N).
\end{equation}
This expansion determines the coefficients $\hat{\sigma}_{(2)}, \dots, \hat{\sigma}_{(N)}$ in the expansion
$$
   e^\varphi S_N(g) = \hat{r} + \hat{r}^2 \hat{\sigma}_{(2)} + \dots + \hat{r}^{N} \hat{\sigma}_{(N)} + \cdots
   = S_N(\hat{g}) + \mbox{higher-order terms}.
$$
In general, this expansion involves higher-order terms, i.e., $e^\varphi S_N(g) \ne S_N(\hat{g})$.
In particular, the remainder term in \eqref{scalar-exp-hat} depends on $\hat{\sigma}_{N+1}$. However,
for $N=n+1$, the coefficient of $\hat{r}^{n+1}$ in \eqref{scalar-exp-hat} does {\em not} depend
on $\hat{\sigma}_{(n+2)}$ since it appears with a prefactor $0$. Therefore,
\begin{align*}
     (\hat{r}^{-n-1} \SC(\hat{g},S_{n+1}(\hat{g})))|_{\hat{r}=0}
    & = (\hat{r}^{-n-1} \SC(\hat{g},e^\varphi S_{n+1}(g)))|_{\hat{r}=0} \\
    &= e^{-(n+1) \iota^*(\varphi)} (r^{-n-1} \SC(g,S_{n+1}(g)))|_{r=0}.
\end{align*}
This relation implies the assertion.
\end{proof}

In \cite{GW-LNY}, Gover and Waldron describe an elegant algorithm that recursively
determines the solution $\sigma$ of the singular Yamabe problem as a power series of some
boundary defining function $\sigma_0$ (like the distance function $\sigma_0 = d_M$). Note that
these power series are not Taylor series: their coefficients still live on the ambient space $X$.
This algorithm rests on the interpretation of the quantity $\SC(g,\sigma)$ as the squared length of
the scale tractor associated to $\sigma$.

For our purposes, it will be enough to consider smooth {\em approximate} solutions
of the singular Yamabe problem. This motivates the following definition.

\begin{defn}\label{Y-smooth} The defining function $\sigma \in C^\infty(X)$ of $M$ is said to
satisfy the condition $\SCY$ iff
$$
   \sigma = r  + \sigma_{(2)} r^2 + \dots + \sigma_{(n+1)} r^{n+1} + O(r^{n+2})
$$
and
$$
   \SC(g,\sigma) = 1 + R_{n+1} r^{n+1}
$$
for a smooth remainder term $R_{n+1}$. Equivalently, it holds
$$
   \SC(g,\sigma) = 1 + R_{n+1} \sigma^{n+1}
$$
for another smooth remainder term $R_{n+1}$. The restriction of either remainder
terms to $M$ is the singular Yamabe obstruction: $\B_n = \iota^* R_{n+1}$.
\end{defn}

We recall that the obstruction $\B_n$ satisfies
\begin{equation}\label{B-CI}
   e^{(n+1) \iota^*(\varphi)} \B_n(\hat{g}) = \B_n(g).
\end{equation}

Graham \cite{Graham-Yamabe} determined the first two non-trivial coefficients $\sigma_{(2)}$
and $\sigma_{(3)}$ in the expansion of $\sigma$ (Lemma \ref{sigma23}). An explicit formula
for the next coefficient $\sigma_{(4)}$ is given in \cite[(2.18)]{GG}. We reproduce these results
in a slightly different form in Section \ref{appY} and also display a formula for $\sigma_{(5)}$.


Graham \cite{Graham-Yamabe} shows that the obstruction $\B_1$ vanishes. The obstruction
for surfaces in a three-manifold is given by the formula                  \index{$\B_2$} \index{$\B_3$}
\begin{align}\label{B2}
   \B_2 & = -\frac{1}{3} (\delta^h \delta^h (\lo) + H |\lo|^2 + (\lo,\iota^*(\Rho^g))) \notag \\
   & = -\frac{1}{3} (\Delta_h (H) + \delta^h (\Ric^g (\NV,\cdot)) +  H |\lo|^2 + (\lo,\iota^*(\Rho^g)))
\end{align}
(\cite[Corollary 6.10]{GW-LNY}). The equivalence of both expressions follows from the
Codazzi-Mainardi equation. For the details of that argument and a derivation of these formulas from
Theorem \ref{B-form}, we refer to Section \ref{B2-details}. The first formula reproduces a result
in \cite[Theorem 1.3]{ACF}.

Explicit formulas for $\B_3$ were first derived in \cite{Ha} and \cite{GGHW} from a general universal tractor 
formula found in \cite{GW-LNY}.  In particular, it was proved that for a conformally flat background metric $g$ 
the obstruction $\B_3$ is given by the closed formula
\begin{equation}\label{B3-closed}
   6 \B_3 = 3 (\delta^h \delta^h + (\Rho^h,\cdot))((\lo^2)_\circ)  + |\lo|^4,
\end{equation}
where $b_\circ$ denotes the trace-free part of the symmetric bilinear form $b$. For general background 
metrics $g$, the formula for $\B_3$ contains additional terms defined by the Weyl tensor of $g$. For full details,
we refer to \cite{GGHW}. In the more recent work \cite{JO-Y}, these formulas for $\B_3$ were derived without
utilizing tractor calculus. In Section \ref{B3-flat-back}, we shall deduce \eqref{B3-closed} from the 
general formula in Theorem \ref{B-form}. Formula \eqref{B3-closed} manifestly implies the conformal 
invariance $e^{4 \varphi} \hat{\B}_3 = \B_3$ since the operator $b \mapsto \delta \delta (b) + (\Rho,b)$ is 
conformally covariant on trace-free symmetric bilinear forms on $M$.

Up to constant multiples, $\B_2$ and $\B_3$ are the respective variations of the functionals
$$
    \int_M |\lo|_h^2 dvol_h \quad \mbox{and} \quad \int_M (\lo,\JF)_h dvol_h
$$
(for the definition of the Fialkov tensor $\JF$ we refer to Section \ref{basic-Gauss}).
These are special cases of the variation formulas of the functional $\A \st \int_M v_n dvol_h$
(with respect to a one-parameter family of hypersurfaces $M \hookrightarrow X$) which were proved
in \cite[Theorem 3.1]{Graham-Yamabe} and \cite[Section 5]{GW-reno}. They state that the
variation of $\A$ is proportional to the obstruction $\B_n$. The equivalence of both results follows
from \eqref{obstruction-two}. This result may be regarded as an analog of the result that the metric
variation of the total critical $Q$-curvature of an even-dimensional closed manifold is given by the
corresponding Fefferman-Graham obstruction tensor \cite{GH}.

Poincar\'e-Einstein metrics are an important special class of singular Yamabe
metrics. In fact, if $g_+ = r^{-2} (dr^2 + h_r)$ is a Poincar\'e-Einstein metric in
normal form relative to $h=h_0$ \cite{FG-final}, then $\scal(g_+) = -n(n+1)$, the
background metric is $dr^2+h_r$ and the corresponding defining function is
$\sigma=r$. In this case, the singular Yamabe obstruction vanishes. We shall refer to this
case as the Poincar\'e-Einstein case.

\section{The conjugation formula}\label{CFormula}

In this section, we introduce Laplace-Robin operators (or degenerate Laplacians)
following \cite{GW-LNY}. We relate them to the spectral theory of the Laplacian of
singular metrics $\sigma^{-2}g$ and use this relation to prove basic properties of the
Laplace-Robin operators.

Let $X$ be a manifold of dimension $n+1$.

\index{$L(g,\sigma)$ \quad Laplace-Robin operator}
\index{$\rho(g,\sigma)$}

\begin{defn}[\bf Laplace-Robin operators]\label{LRO}
For any pair $(g,\sigma)$ consisting of a metric $g$ on $X$ and $\sigma \in
C^\infty(X)$, the one-parameter family
\begin{equation}\label{LR-OP}
   L(g,\sigma;\lambda) \st (n\!+\!2\lambda\!-\!1) (\nabla_{\grad_g(\sigma)} + \lambda \rho)
   - \sigma (\Delta_g + \lambda \J):  C^\infty(X) \to C^\infty(X)
\end{equation}
of differential operators is called the Laplace-Robin operator of the pair $(g,\sigma)$. Here
$\lambda \in \C$,
$$
   2 n \J \st \scal(g) \quad \mbox{and} \quad (n+1) \rho(g,\sigma) \st - \Delta_g (\sigma) - \sigma \J.
$$
Moreover, we set                                                            \index{$\SC(g,\sigma)$}
\begin{equation}\label{SC}
   \SC(g,\sigma)  \st  |d\sigma|_g^2 + 2 \sigma \rho.
\end{equation}
\end{defn}

Similarly, we define $L(g,\sigma;\lambda)$ for a manifold $X$ with boundary $M$. In
this case, $g$ and $\sigma$ are assumed to be smooth up to the boundary. Then
$L(\lambda)$ acts on the space $C^\infty(X)$ of smooth functions up to the boundary
and on the space $C^\infty(X^\circ)$ of smooth functions on the open interior
$X^\circ = X \setminus M$ of $X$.

From now on, we assume that $(X,g)$ is a compact manifold with boundary $M$ and
$\sigma$ is a defining function of $M$. We recall that $\sigma$ is a defining
function of $M$ if $\sigma^{-1}(0)=M$, $\sigma > 0$ on $X^\circ$ and $d\sigma|_M \ne
0$. Let $\iota: M \hookrightarrow X$ be the embedding and set $h \st \iota^*(g)$.

Then the operator $\iota^* L(g,\sigma;\lambda)$ degenerates to the first-order operator
\begin{equation}\label{bv-op}
   C^\infty(X) \ni u \mapsto (n\!+\!2\lambda\!-\!1) \iota^* \left[\nabla_{\grad_g(\sigma)}(u)
   - \frac{\lambda}{n\!+\!1}\Delta_g(\sigma) u \right] \in C^\infty(M).
\end{equation}
Since certain linear combinations of Dirichlet and Neumann boundary values are also known
as Robin boundary values, this naturally motivates the above notion of a Laplace-Robin operator.

If $\sigma^{-2} g$ has constant scalar curvature $-n(n+1)$, the boundary
operator \eqref{bv-op} reduces to the conformally covariant boundary operator
$$
   u \mapsto (n\!+\!2\lambda\!-\!1) \iota^* (\nabla_{\grad_g(\sigma)} - \lambda H) u,
$$
where $H$ is the mean curvature of $M$ (\cite[Section 3.1]{Gover-AE}, \cite[Lemma 2.3]{GW-reno}).

If $g_+ = r^{-2} g$ is Poincar\'e-Einstein in the sense that $\Ric(g_+) = - n g_+$, then 
$L(g,r;\lambda-n+1)$ equals the shift operator $S(g_+;\lambda)$ of \cite{FJO}.

\begin{theorem}[\bf Conjugation formula]\label{MCF}
Assume that $\sigma \in C^\infty(X)$ is a defining function of the boundary $M$ of $X$.
Then it holds
\begin{equation}\label{CF}
   L(g,\sigma;\lambda) + \sigma^{\lambda-1} \circ \left(\Delta_{\sigma^{-2}g} -
   \lambda(n\!+\!\lambda) \id \right) \circ \sigma^{-\lambda} = \lambda
   (n\!+\!\lambda) \sigma^{-1}(\SC(g,\sigma)-1) \id
\end{equation}
as an identity of operators acting on $C^\infty(X^\circ)$.
\end{theorem}

\begin{proof} Let $\Con(g,\sigma;\lambda)$ denote the operator
$$
   \sigma^{\lambda-1} \circ \left(\Delta_{\sigma^{-2}g} -
   \lambda(n\!+\!\lambda) \id \right) \circ \sigma^{-\lambda}.
$$
The relation
$$
   \Delta_{\sigma^{-2}g} = \sigma^2 \Delta_g - (n\!-\!1)  \sigma \nabla_{\grad_g(\sigma)}
$$
shows that
\begin{align*}
   \Con(g,\sigma;\lambda) (u) & = \sigma^{\lambda+1} \Delta_g (\sigma^{-\lambda} u)
   - (n\!-\!1) \sigma^\lambda \nabla_{\grad_g(\sigma)}(\sigma^{-\lambda} u) - \lambda(n\!+\!\lambda) \sigma^{-1} u \\
   & = \sigma \Delta_g (u) - (n\!+\!2\lambda\!-\!1) \nabla_{\grad_g(\sigma)}(u)
   + \Con(g,\sigma;\lambda)(1) u
\end{align*}
for $u \in C^\infty(X^\circ)$. Thus, it only remains to calculate the constant term $\CT(\Con)(\lambda) \st
\Con(\lambda)(1) \in C^\infty(X^\circ)$ of $\Con(\lambda)$. Note that
\begin{align*}
   \CT(\Con)(\lambda) & = \sigma^{\lambda-1}(\Delta_{\sigma^{-2}g}
   - \lambda(n\!+\!\lambda) \id )(\sigma^{-\lambda}) \\
   & = \sigma^{\lambda+1} \Delta_{g} (\sigma^{-\lambda})
   + \lambda (n\!-\!1) \sigma^{-1} |\grad_g(\sigma)|^2
   - \lambda(n\!+\!\lambda) \sigma^{-1}
\end{align*}
shows that $\CT(\Con)(\lambda)$ is a quadratic polynomial in $\lambda$. It is
obvious that $\CT(\Con)(0)=0$. Next, we determine the leading coefficient of that
polynomial. We choose orthonormal bases $\left\{\partial_i \right\}$ on the
tangent spaces of the level hypersurfaces $\sigma^{-1}(c)$ of $\sigma$. The sets
$\sigma^{-1}(c)$ are smooth manifolds if $c$ is sufficiently small. $\NV$ is
perpendicular to these hypersurfaces. Let $\left\{ \alpha, dx^i \right\}$ be the
dual basis of $\left\{ \NV, \partial_i \right\}$. Then
\begin{equation}\label{Laplace-basis}
   \Delta_g (u) = \frac{1}{|\NV|^2} \nabla_{\NV}^2 (u) + \Delta_{g_\sigma}(u) + \frac{1}{|\NV|} H_\sigma \nabla_\NV (u)
   - \frac{1}{|\NV|^2} \langle du, \nabla_{\NV}(\NV) \rangle,
\end{equation}
where $-H_\sigma = \langle \alpha,\nabla_{\partial_i} (\partial_i) \rangle$ and
$\Delta_{g_\sigma}$ denotes the tangential Laplacians for the induced metrics on
the leaves $\sigma^{-1}(c)$; for more details, see Section \ref{Laplace-expansion}.
Since $\Delta_{g_\sigma} (\sigma^{-\lambda}) = 0$, it follows that the coefficient of $\lambda^2$
in $\sigma^{\lambda+1} \Delta_{g}(\sigma^{-\lambda})$ is given by
$\sigma^{-1} |\NV|^{-2} \nabla_\NV(\sigma)^2  = \sigma^{-1} |\NV|^2$.
Hence the leading coefficient of the quadratic polynomial $\CT(\Con)(\lambda)$ equals
$$
   \sigma^{-1 }|\NV|^2 - \sigma^{-1}.
$$
Finally, we calculate
$$
   \CT(\Con)(-1) = \Delta_g(\sigma) - (n-1) \sigma^{-1} |\NV|^2 + (n-1) \sigma^{-1}.
$$
These arguments prove that
\begin{equation}\label{C-CT}
   \CT(\Con)(\lambda) = \lambda \left((n\!+\!\lambda) \sigma^{-1} (|\NV|^2-1) - \Delta_g(\sigma)\right).
\end{equation}
Hence
\begin{align*}
   \CT(L)(\lambda) + \CT(\Con)(\lambda) & = \lambda (n\!+\!2\lambda\!-\!1) \rho - \lambda \sigma \J
   + \lambda \left((n\!+\!\lambda) \sigma^{-1} (|\NV|^2-1) - \Delta_g(\sigma)\right) \\
   & = \lambda (n\!+\!\lambda) \sigma^{-1}(2 \rho \sigma + |\NV|^2-1) \\
   & = \lambda (n\!+\!\lambda) \sigma^{-1} (\SC(g,\sigma)-1)
\end{align*}
by the definition of $\rho$. This completes the proof.
\end{proof}

The above proof shows the identity \eqref{CF} only for functions with support near the boundary $M$. 
This will be enough for all later applications. For simplicity, we shall interpret \eqref{CF} and similar identities in this 
way without further mentioning.

\begin{remark}\label{CF-PE}
Let $g_+ = r^{-2} (dr^2 + h_r)$ be a Poincar\'e-Einstein metric in normal form
relative to the metric $h$ on $M$. It lives on the space $(0,\varepsilon) \times M$
and satisfies $\Ric(g_+) = -n g_+$. Hence $\scal(g_+)=-n(n+1)$. Then
$\SC(dr^2+h_r,r;\lambda) = \SC(g_+,1;\lambda)=1$, and the conjugation formula reads
$$
   - L(dr^2+h_r,r;\lambda) =  r^{\lambda-1} \circ (\Delta_{g_+} - \lambda(n+\lambda) \id ) \circ r^{-\lambda}.
$$
We refer to \cite[Section 3]{FJO} for the discussion of the relation between this conjugation formula
and a formula in \cite{GZ}.
\end{remark}

\begin{remark}\label{CF-simple}
The conjugation formula is equivalent to the identity
\begin{equation}\label{CF-s}
   L(g,\sigma;\lambda) + \sigma^{\lambda-1} \circ \Delta_{\sigma^{-2}g} \circ \sigma^{-\lambda}
  = \lambda (n\!+\!\lambda) \sigma^{-1} \SC(g,\sigma) \id
\end{equation}
of operators acting on $C^\infty(X^\circ)$.
\end{remark}

The conformal covariance of the Laplace-Robin operator is an immediate consequence of these identities. 
More precisely, we have

\begin{cor}\label{LR-CC} The Laplace-Robin operator satisfies
\begin{equation}\label{CT-ID}
   L(\hat{g}, \hat{\sigma}; \lambda) \circ e^{\lambda\varphi} = e^{(\lambda-1)\varphi} \circ L(g,\sigma;\lambda),
   \; \lambda \in \C
\end{equation}
for all conformal changes $(\hat{g},\hat{\sigma}) = (e^{2\varphi}g,e^{\varphi}\sigma)$, $\varphi \in
C^\infty(X)$.
\end{cor}

\begin{proof} It suffices to note that $\hat{\sigma}^{-2} \hat{g} = \sigma^{-2} g$ and
$\SC(\hat{g},\hat{\sigma}) = \SC(g,\sigma)$.
\end{proof}

Strictly speaking, the above arguments prove the conformal covariance of the
operator $L(g,\sigma;\lambda)$ for boundary defining $\sigma$ when acting on
$C^\infty(X^\circ)$. In \cite{GW}, the conformal covariance of the operator
$L(g,\sigma;\lambda)$ for {\em any} pair $(g,\sigma)$ acting on $C^\infty(X)$ follows from
its interpretation in terms of tractor calculus. For a direct proof, see \cite[Proposition 2.2]{FJO}.

The following consequence of the conjugation formula will be of central significance
in the rest of the paper. We continue to assume that $\sigma$ is a boundary defining function
and statements are valid near $M$.

\begin{cor}\label{main-c-c}
It holds
\begin{equation}\label{L-Delta}
   - L(g,\sigma;\lambda) = \sigma^{\lambda-1} \circ (\Delta_{\sigma^{-2}g} - \lambda(n+\lambda) \id )
   \circ \sigma^{-\lambda}
\end{equation}
for $\lambda \in \C$ iff $\SC(g,\sigma) = 1$. More generally, it holds
$$
  -L(g,\sigma;\lambda) = \sigma^{\lambda-1} \circ (\Delta_{\sigma^{-2}g}
  - \lambda(n+\lambda) \id ) \circ \sigma^{-\lambda} + O(\sigma^n)
$$
if $\sigma$ satisfies  $\SCY$. These identities are identities of operators acting on $C^\infty(X^\circ)$.
\end{cor}

For $N \in \N$, we define
\begin{equation}\label{LN}
   L_N(g,\sigma;\lambda) \st L(g,\sigma;\lambda\!-\!N\!+\!1) \circ \cdots \circ L(g,\sigma;\lambda).
\end{equation}
In these terms, iterated application of \eqref{L-Delta} implies

\begin{cor}\label{der-c-c} Assume that $\SC(g,\sigma)=1$. Then it holds
\begin{equation*}
    \sigma^{-\frac{n}{2}-N} \circ 
    \prod_{j=0}^{2N-1} \left(\Delta_{\sigma^{-2}g} + \left(\frac{n}{2}\!+\!N\!-\!j\right)
     \left(\frac{n}{2}\!-\!N\!+\!j\right) \id \right) \circ \sigma^{\frac{n}{2}-N}
     = L_{2N} \left(g,\sigma;-\frac{n}{2}\!+\!N\right)
\end{equation*}
for $2N \le n$. In particular, we have
\begin{equation}
    \sigma^{-n} \circ  
\prod_{j=0}^{n-1} \left(\Delta_{\sigma^{-2}g} + (n\!-\!j)j \id \right) = L_n (g,\sigma;0)
\end{equation}
\end{cor}

The conjugation formula also sheds new light on the fact that the formal adjoint of a
Laplace-Robin operator is another Laplace-Robin operator. 

\begin{cor}\label{L-adjoint} The Laplace-Robin operator satisfies
\begin{equation}\label{L-ad}
   L(g,\sigma;\lambda)^* = L(g,\sigma;-\lambda\!-\!n), \; \lambda \in \C,
\end{equation}
where $^*$ denotes the adjoint operator with respect to the Riemannian volume of
$g$. More precisely, it holds
\begin{equation}\label{adjoint}
   \int_X L(g,\sigma;\lambda)(\varphi) \psi dvol_g = \int_X \varphi L(g,\sigma;-\lambda\!-\!n)(\psi) dvol_g
\end{equation}
for $\varphi, \psi \in C_c^\infty(X^\circ)$. 
\end{cor}

\begin{proof} Let $\varphi,\psi \in C_c^\infty(X^\circ)$. The identity \eqref{CF-s} yields
\begin{align*}
   & \int_X L(g,\sigma;\lambda)(\varphi) \psi dvol_g \\
   & = - \int_X \sigma^{\lambda-1} \Delta_{\sigma^{-2}g} (\sigma^{-\lambda} \varphi) \psi dvol_g
   + \lambda (n\!+\!\lambda) \int_X \sigma^{-1} \SC(g,\sigma) \varphi \psi dvol_g.
\end{align*}
Note that $\lambda(n+\lambda)$ is invariant under the substitution $\lambda \mapsto
-\lambda-n$. We rewrite the first integral in terms of volumes with respect to the
metric $\sigma^{-2}g$ and apply the self-adjointness of $\Delta_{\sigma^{-2}g}$ with
respect to the volume of the metric $\sigma^{-2}g$. Using $dvol_{\sigma^{-2}g} =
\sigma^{-n-1} dvol_g$, we find
$$
   - \int_X \sigma^{-\lambda} \varphi  \Delta_{\sigma^{-2}g} (\sigma^{\lambda+n} \psi) dvol_{\sigma^{-2} g}
   = - \int_X \varphi \sigma^{-\lambda-n-1}  \Delta_{\sigma^{-2}g} (\sigma^{\lambda+n} \psi) dvol_g.
$$
Now another application of \eqref{CF-s} implies the assertion \eqref{adjoint}.
\end{proof}

For later applications, we need an extension of Corollary \ref{L-adjoint} to another class of test functions. 
The proof of the following result will be given in Section \ref{AC-RVC}.

\begin{prop}\label{adjoint-gen}
The identity \eqref{adjoint} continues to be true for $\psi \in C^\infty(X)$ and
$\varphi \in C^2(X)$ so that $\iota^*(\varphi)=0$.
\end{prop}

The proof of this result actually shows that 
\begin{align}\label{adjoint-g}
   &\int_X L(g,\sigma;\lambda)(\varphi) \psi dvol_g - \int_X \varphi L(g,\sigma;-\lambda\!-\!n)(\psi) dvol_g \notag \\
   & = (n\!+\!2\lambda) \int_M \iota^*(\varphi \psi |\NV|) dvol_h
\end{align}
for $\varphi , \psi \in C^2(X)$. The proof of this identity rests on a calculation of the left-hand side 
(see also \cite[Theorem 2.2]{GW-reno}).

\index{$L_N(g,\sigma;\lambda)$}

Finally, we derive some basic commutator relations. 

\begin{cor}\label{comm} For any $N \in \N$, it holds
$$
   L(g,\sigma;\lambda\!+\!N) \circ \sigma^N - \sigma^N \circ L(g,\sigma;\lambda)
   = N(n\!+\!2\lambda\!+\!N) \sigma^{N-1} \SC(g,\sigma) \id.
$$
In particular, it holds
\begin{equation}\label{sl2-comm}
   L(g,\sigma;\lambda\!+\!1) \circ \sigma - \sigma \circ L(g,\sigma;\lambda) =
   (n\!+\!2\lambda\!+\!1) \SC(g,\sigma) \id.
\end{equation}
Hence
\begin{align}\label{LN-gen-b}
  & L_N(g,\sigma;\lambda) \circ \sigma - \sigma \circ L_N(g,\sigma;\lambda\!-\!1)
  = N(n\!+\!2\lambda\!-\!N) L_{N-1}(g,\sigma;\lambda-1)  \notag \\
  & + \sum_{j=1}^N (n\!+\!2\lambda\!-\!2j\!+\!1) \underbrace{L(g,\sigma;\lambda\!-\!N\!+\!1)
  \circ \cdots \circ (\SC(g,\sigma)-1) \circ \cdots \circ L(g,\sigma;\lambda\!-\!1)}_{N factors}.
\end{align}
Moreover, if $\SC(g,\sigma)$ is nowhere zero, then for any $N \in \N$ it holds
\begin{equation}\label{LN-gen-tilde}
   \tilde{L}_N(g,\sigma;\lambda\!+\!1) \circ \sigma - \sigma \circ \tilde{L}_N(g,\sigma;\lambda)
  =  N(n\!+\!2\lambda\!-\!N\!+\!2) \tilde{L}_{N-1}(g,\sigma;\lambda),
\end{equation}
where
\begin{equation}\label{LR-general}    \index{$\tilde{L}(g,\sigma)$}
   \tilde{L}(g,\sigma;\lambda) \st L(g,\sigma;\lambda) \circ \SC(g,\sigma)^{-1}
\end{equation}
and
\begin{equation}\label{LN-tilde}      \index{$\tilde{L}_N(g,\sigma)$}
   \tilde{L}_N(g,\sigma;\lambda) \st
   \tilde{L}(g,\sigma;\lambda\!-\!N\!+\!1) \circ \cdots \circ \tilde{L}(g,\sigma;\lambda).
\end{equation}
\end{cor}

\begin{proof} The identity \eqref{CF-s} implies
\begin{align*}
   & L(g,\sigma;\lambda) \circ \sigma^N - \sigma^N \circ L(g,\sigma;\lambda-N) \\
   & = \lambda (n\!+\!\lambda) \sigma^{N-1} \SC(g,\sigma) \id - (\lambda\!-\!N)
   (n\!+\!\lambda\!-\!N) \sigma^{N-1} \SC(g,\sigma) \id \\
   & = N(n\!+\!2\lambda\!-\!N) \sigma^{N-1} \SC(g,\sigma) \id.
\end{align*}
This proves the first commutator relation. The remaining claims are
consequences. This completes the proof.
\end{proof}

The above commutator relations substantially simplify if $\SC=1$. 

Although we proved the identities in Corollary \ref{comm} as identities of operators acting on $C^\infty(X^\circ)$, 
they are also valid for operators acting on $C^\infty(X)$ (\cite[Section 3]{GW}).

\section{Symmetry breaking operators}\label{SBO}

In the present section, we discuss some representation theoretical aspects of the
results in Section \ref{CFormula}.

The simplest special case of the Laplace-Robin operator $L$ appears for the
hyperplane $M=\R^n$ in $X=\R^{n+1}$ with the flat Euclidean metric $g_0$. Let $M$ be
given by the zero locus of the defining function $\sigma_0 = x_{n+1}$. We shall also
write $\sigma_0=r$ and $g_+ = r^{-2} g_0$. Then $\J = \rho = 0$ and we obtain
\begin{equation}\label{L-flat}
   L(g_0,\sigma_0;\lambda) = (n+2\lambda-1) \partial_{n+1}
   - x_{n+1} \Delta_{\R^{n+1}}: C^\infty(\R^{n+1}) \to C^\infty(\R^{n+1}).
\end{equation}
An easy calculation shows the conjugation formula
$$
   L(g_0,r;\lambda) = r^{\lambda-1} \circ (-\Delta_{g_+} + \lambda(n+\lambda)) \circ r^{-\lambda}.
$$
It implies that the operator $L(g_0,r;\lambda)$ is an intertwining operator for
spherical principal series representations. Indeed, let $\gamma \in SO(1,n+1)$ be an
isometry of the hyperbolic metric $g_+=r^{-2} g_0$ acting on the upper-half space $r
> 0$. Then we calculate
\begin{align*}
   & L(g_0,r;\lambda) \left( \left(\frac{\gamma_*(r)}{r}\right)^{-\lambda} \gamma_*(u)\right) \\
   & = r^{\lambda-1} (-\Delta_{g_+} + \lambda(n+\lambda)) \left(r^{-\lambda}
   \left(\frac{\gamma_*(r)}{r}\right)^{-\lambda} \gamma_*(u)\right) \\
   & = r^{\lambda-1} (-\Delta_{g_+} + \lambda(n+\lambda)) (\gamma_*(r^{-\lambda} u)) \\
   & = r^{\lambda-1}\gamma_* (-\Delta_{g_+} + \lambda(n+\lambda))(r^{-\lambda} u)) \\
   & = \left(\frac{\gamma_*(r)}{r}\right)^{-\lambda+1} \gamma_* (r^{\lambda-1}
   (-\Delta_{g_+}+\lambda(n+\lambda))(r^{-\lambda} u) \\
   & = \left(\frac{\gamma_*(r)}{r}\right)^{-\lambda+1} \gamma_* L(g_0,r;\lambda)(u).
\end{align*}
In other words, it holds
\begin{equation}\label{rep-flat}
   L(g_0,r;\lambda) \circ \pi^0_{-\lambda}(\gamma) = \pi^0_{-\lambda+1}(\gamma) \circ L(g_0,r;\lambda)
\end{equation}
with  \index{$\pi_\lambda^0$ \quad non-compact model of spherical principal series representation}
$$
   \pi^0_\lambda(\gamma) \st \left( \frac{\gamma_*(r)}{r} \right)^\lambda \gamma_*.
$$
Note that
$$
   \frac{\gamma_*(r)}{r} = e^{\Phi_\gamma},
$$
where $\Phi_\gamma$ is the conformal factor of the conformal transformation induced
by $\gamma$ with respect to the Euclidean metric, i.e., $\gamma_*(g_0) =
e^{2\Phi_\gamma} g_0$. The representation $\pi_\lambda^0(\gamma)$ is actually
well-defined for all $\gamma \in SO(1,n+2)$ acting on $\R^{n+1}$ (viewed as the
boundary of hyperbolic space of dimension $n+2$). However, the intertwining property
\eqref{rep-flat} holds true only for the subgroup of $SO(1,n+1)$ leaving the
boundary $r=0$ of the upper half-space invariant. The fact that $L(g_0,r;\lambda)$
is an intertwining operator for a subgroup of the conformal group of the Euclidean
metric on $\R^{n+1}$ connects it with the theory of symmetry breaking operators. In
fact, it follows from the above that the compositions
$$
   D_N(\lambda) \st \iota^* L(\lambda-N+1) \circ \cdots \circ L(\lambda):
   C^\infty(\R^{n+1}) \to C^\infty(\R^n), \; N \in \N
$$
satisfy
$$
   D_N(\lambda) \circ \pi_{-\lambda}^0(\gamma) = \pi_{-\lambda+N}^{0 \prime}(\gamma)
   \circ D_N(\lambda), \; \gamma \in SO(1,n+1)
$$
and Clerc \cite{C} proved that $D_N(\lambda)$ coincides with the symmetry breaking
operator introduced in \cite[Chapter 5]{J1}.\footnote{$\pi_\lambda^{0 \prime}$
denotes the analogous representation on functions on the subspace $\R^n$.}

Similarly, let $M = S^n$ be an equatorial subsphere of $X=S^{n+1}$ with the round
metric $g$. Let $M$ be defined as the zero locus of the height function $\sigma=\He$
being defined as the restriction of $x_{n+2}$ to $S^{n+1}$. Then $\J =
\frac{n+1}{2}$ and           \index{$\He$ \quad height}
$$
   (n+1) \rho = - \Delta_{S^{n+1}} \He + \J \He = -(n+1) \He + \frac{n+1}{2} \He = - \frac{n+1}{2} \He
$$
using the fact that $\He$ is an eigenfunction of the Laplacian on the sphere
$S^{n+1}$. Thus, $\rho = \frac{1}{2} \He$ and we obtain
\begin{equation}\label{L-sphere}
   L(g,\He;\lambda) = (n+2\lambda-1) \nabla_{\grad (\He)} - \He \Delta_{S^{n+1}}
   + \lambda(\lambda+1) \He: C^\infty(S^{n+1}) \to C^\infty(S^{n+1}).
\end{equation}
A calculation shows that
$$
   L(g,\He;\lambda) = \He^{\lambda-1} \circ (-\Delta_{\He^{-2}g} + \lambda(n+\lambda)) \circ \He^{-\lambda}.
$$
Again, the operator $L(g,\He;\lambda)$ is an intertwining operator for spherical
principal series representations. Indeed, it holds
$$
   L(g,\He;\lambda) \circ \pi_{-\lambda}(\gamma) = \pi_{-\lambda+1}(\gamma) \circ L(g,\He;\lambda)
$$
for all $\gamma \in SO(1,n+1)$ acting on the upper hemisphere $\He > 0$ of
$S^{n+1}$. Here  \index{$\pi_\lambda$ \quad spherical principal series representation}
$$
   \pi_\lambda(\gamma) \st \left( \frac{\gamma_*(\He)}{\He} \right)^\lambda \gamma_*.
$$
The latter representations are well-defined for $\gamma \in SO(1,n+2)$
acting on $S^{n+1}$. Note that
$$
   \frac{\gamma_*(\He)}{\He} = e^{\Phi_\gamma},
$$
where $\Phi_\gamma$ is the conformal factor of the conformal transformation induced
by $\gamma$ with respect to the round metric $g$, i.e., $\gamma_*(g) =
e^{2\Phi_\gamma} g$. We also note that the operator \eqref{L-sphere} is equivalent
to the intertwining operator displayed in \cite[Proposition 7.9]{C}. We omit the
details of that calculation.

Finally, we observe that the above two models of the Laplace-Robin operator are conformally equivalent.
In fact, let $\kappa: S^{n+1} \to \R^{n+1}$ be the stereographic projection. Then
$$
   \kappa^* (\He) = \Phi x_{n+1} \quad \mbox{and} \quad \kappa^*(g)
   = \Phi^2 \sum_{i=1}^{n+1} dx_i^2 = \Phi^2 g_0
$$
with $\Phi = 2/(1+|x|^2)$ \cite[Section 2.2]{J1}. Hence
\begin{align*}
   \kappa^* L(g,\He;\lambda) \kappa_* & = L(\kappa^*(g),\kappa^*(\He);\lambda) \\
   & = L(\Phi^2 g_0,\Phi x_{n+1};\lambda) \\
   & = \Phi^{\lambda-1} L(g_0,x_{n+1};\lambda) \Phi^{-\lambda}
\end{align*}
using a very special case of the conformal invariance of the Laplace-Robin operator (Corollary \ref{LR-CC}). 
By combining this conjugation formula with the results in the later sections, it follows that the equivariant 
families $D_{2N}^c(\lambda): C^\infty(S^{n+1}) \to C^\infty (S^n)$ constructed in \cite[Section 5.2]{J1} 
can be regarded as residue families $\D_{2N}^{res}(g,\He;\lambda)$ (as defined in Section \ref{res-fam}).

\section{Adapted coordinates, renormalized volume coefficients and a formula for $\B_n$}\label{AC-RVC}


Let $X$ be compact with closed boundary $M$ and let $\sigma$ be a defining function of
$M$, i.e., $\sigma^{-1}(0)=M$, $\sigma > 0$ on $X \setminus M$ and $d\sigma|_M \ne 0$. Let
$\iota: M \hookrightarrow X$ be the embedding and $h = \iota^*(g)$.

We start with the definition of two different types of local coordinates of $X$ near the boundary:
{\em geodesic normal} coordinates and {\em adapted} coordinates.

\index{$\kappa$} 
\index{$w(r)$}
\index{$w_j$ \quad renormalized volume coefficients (normal coordinates)}
\index{$h_r$}

Geodesic normal coordinates are defined by the normal geodesic flow of the hypersurface $M$,
i.e., we consider a diffeomorphism of $I \times M$ (with a small interval $I = [0,\varepsilon)$) 
onto a neighborhood of $M$ in $X$, which is defined by        \index{$\Phi^r$ \quad geodesic flow}
$$
   \kappa: I \times M \ni (r,x) \mapsto \Phi^r(x) \in X,
$$
where $\Phi^r$ is the geodesic flow with initial speed given by a unit normal field on $M$. Then $\kappa^*(g)$ 
has the form $dr^2 + h_r$ for a one-parameter family $h_r$ on $M$. Let  \index{$u(r)$}    \index{$u_j$ \quad volume coefficients}
\begin{equation}\label{v-geo}
   u(r) \st dvol_{h_r}/dvol_h, \quad u(r) = \sum_{j \ge 0} r^j u_j, \quad u_j \in C^\infty(M).
\end{equation}
Now, if $\SC(g,\sigma) = 1$, then the volume form of the singular metric $\sigma^{-2} g$ has the form
\begin{align}\label{def-w}
    dvol_{\kappa^*(\sigma^{-2} g)} & = \sigma(r)^{-n-1} u(r) dr dvol_h \notag \\
    & = r^{-n-1} w(r) dr dvol_h.
\end{align}
for $\sigma(r) = \kappa^*(\sigma)$ and some $w \in C^\infty(I \times M)$. Moreover, we have expansions
\begin{equation}\label{RVC-A}
   w(r) = 1 + \sum_{j \ge 1} r^j w_j 
\end{equation}
with $w_j \in C^\infty(M)$ and
$$
   dvol_{\kappa^*(\sigma^{-2} g)} = \sum_{j \ge 0} r^{-n-1+j} w_j dr dvol_h.
$$
Following \cite{Graham-Yamabe}, the coefficients $w_j \in C^\infty(M)$ for $j \le n$ are called {\em singular Yamabe} 
renormalized volume coefficients. Note that the definition of the coefficients $w_j \in C^\infty(M)$ involves the Taylor 
expansion of $\sigma(r)$ in $r$. Special interest deserves the critical coefficient $w_n$ since for closed $M$, the 
total integral
$$
   \int_M w_n dvol_h
$$
is conformally invariant \cite[Proposition 2.1]{Graham-Yamabe}.

\index{$\eta$ \quad adapted coordinates}

Similarly, {\em adapted coordinates} are associated to the data $(g,\sigma)$ through a diffeomorphism
$$
   \eta: I \times M \ni (s,x) \mapsto \Phi_\mathfrak{X}^s(x) \in X
$$
onto a open neighborhood of $M$ in $X$ (with some small interval $I = [0,\varepsilon)$), where 
$\Phi_\mathfrak{X}^s$ denotes the flow of the vector field    \index{$\mathfrak{X}$}
\begin{equation}\label{X-field}
   \mathfrak{X} \st \NV /|\NV|^2, \quad  \NV = \grad_g(\sigma)
\end{equation}
with $\Phi^0_\mathfrak{X} = \id$. We shall also use the notation
$\mathfrak{X}_\sigma$ in cases where the dependence on $\sigma$ is important. Note that $|\mathfrak{X}|
= 1/|\NV|$. Then
$$
   (d/ds)(\sigma \circ \eta) = \langle d\sigma,
   \mathfrak{X} \rangle = \langle d \sigma, \NV \rangle / |\NV|^2 \stackrel{!}{=} 1
$$
and the differential of $\eta$ maps the vector field $\partial_s$ to the vector field
$\mathfrak{X}$ (see Section \ref{notation}). This implies the important relation
\begin{equation}\label{key-pb}
   \eta^*(\sigma) = s
\end{equation}
and the intertwining property
\begin{equation}\label{intertwine}
   \eta^* \circ \mathfrak{X} = \partial_s \circ \eta^*,
\end{equation}
where $\partial_s$ and $\mathfrak{X}$ are viewed as first-order differential operators. Therefore,
$$
   \eta^* \circ \mathfrak{X}^k = \partial_s^k \circ \eta^*,
$$
and by composition with $\iota^*$, we obtain
\begin{equation}\label{translate}
   \iota^* \partial_s^k \circ \eta^* = \iota^* \mathfrak{X}^k = \iota^* \left(|\NV|^{-2}
   \nabla_\NV\right)^k.
\end{equation}
Now if $\SC(g,\sigma)=1$, i.e., if $|\NV|^2=1-2\sigma \rho$, then it follows from
\eqref{translate} that the Taylor coefficients in the variable $s$ of any function
$\eta^*(u)$ with $u \in C^\infty(X)$ can be written as linear combinations of
iterated gradients $\iota^* \nabla_\NV^k (u)$ with coefficients that are polynomials
in the quantities $\iota^* \nabla_\NV^k(\rho) \in C^\infty(M)$. In particular, it
holds
\begin{equation}\label{trans-1-2}
   \iota^* \partial_s \circ \eta^* = \iota^* \nabla_\NV \quad \mbox{and} \quad \iota^* \partial^2_s
   \circ \eta^* = \iota^* \nabla_\NV^2 - 2 H \nabla_\NV.
\end{equation}
For more details, we refer to Section \ref{ho-normals}. If $\sigma$ satisfies only the weaker
condition $\SCY$, then $|\NV|^2 = 1 - 2\sigma \rho + O(\sigma^{n+1})$ and the same
conclusions are true for sufficiently small $k$.

Note that the metric $\eta^*(g)$ has the form
\begin{equation}\label{normal-adapted}
   \eta^*(|\NV|^{-2}) ds^2 + h_s
\end{equation}
with a one-parameter family $h_s$ of metrics on $M$ so that $h_0 = h = \iota^*(g)$. We shall refer to
\eqref{normal-adapted} as the {\em normal form} of $g$ in adapted coordinates. We expand
$h_s = \sum_{j \ge 0} h_{(j)} s^j$. It follows from \eqref{translate} that
$$
   \iota^* \partial_s^k (\eta^*(|\NV|^{-2})) = \iota^* (\nabla_\NV(|\NV|^{-2}))^k.
$$
Thus, if $\sigma$ satisfies $\SCY$, the Taylor coefficients of the coefficient
$\eta^*(|\NV|^{-2})$ in the variable $s$ are polynomials in the quantities
$\nabla_\NV^k(\rho)$.

In later calculations in adapted coordinates, we shall often use the same notation for quantities like 
$\rho$ and $\J$ and their pull-backs by $\eta$ without further mentioning.

\index{$v(s)$}
\index{$v_j$ \quad renormalized volume coefficients (adapted coordinates)}

Now \eqref{normal-adapted} implies
$$
   dvol_{\eta^*(g)} = \eta^*(|\NV|)^{-1} ds dvol_{h_s} = v(s) ds dvol_h
$$
for some $v(s) \in C^\infty(I \times M)$. Since the condition $\SCY$
implies $|\NV|=1$ on $M$, we get $v(0,x)=1$, and we have an expansion
\begin{equation}\label{RVC-B}
   v(s) = 1 + \sum_{j \ge 1} s^j v_j \quad \mbox{with $v_j \in C^\infty(M)$}.
\end{equation}
The coefficients $v_j$ for $j \le n$ also will be called {\em singular Yamabe} renormalized
volume coefficients. They describe the volume of the singular metric $\sigma^{-2} g$ through the expansion
$$
   dvol_{\eta^*(\sigma^{-2} g)} = s^{-n-1} v(s) ds dvol_h = \sum_{j \ge 0} s^{-n-1+j} v_j ds dvol_h.
$$
Again, special interest deserves the {\em critical} coefficient $v_n$ since
$$
  \int_M v_n dvol_h
$$
is conformally invariant for closed $M$. This follows from the equality
\begin{equation}\label{w=v}
   \int_M v_n dvol_h = \int_M w_n dvol_h
\end{equation}
which can be proved by the following argument of \cite{Graham-Yamabe}. The identity
$|d\sigma|^2_{\tilde{g}}=1$ for $\tilde{g} = |d\sigma|_g^2 g$ shows that $\sigma$ can be viewed
as the distance function $d_M^{\tilde{g}}$ of $M$ in the metric $\tilde{g}$. Hence it holds
$$
   vol_{\sigma^{-2}g} (\{ \sigma > \varepsilon \})
   = vol_{\sigma^{-2} g} (\{ d_M^{\tilde{g}}  > \varepsilon \}) =
   vol_{\tilde{\sigma}^{-2} \tilde{g}} (\{ d_M^{\tilde{g}} > \varepsilon \})
$$
with $\tilde{\sigma}= |d\sigma|_g \sigma$. By comparing the coefficients of $\log \varepsilon$ in the
expansions of both sides, we find
$$
   \int_M v_n(g) dvol_{\iota^*(g)} = \int_M w_n(\tilde{g}) dvol_{\iota^*(\tilde{g})}.
$$
Now the conformal invariance of the latter integral implies the equality \eqref{w=v}. Following 
\cite{Graham-Yamabe}, the integral                                            \index{$\A$ \quad anomaly} 
\begin{equation}\label{def-A}
   \A  \st  \int_{M^n} v_n dvol_h
\end{equation}
for a closed $M$ is called the {\em singular Yamabe energy} of $M$. The quantity $\A$ appears 
in the asymptotic expansion of the volume of the singular metric $\sigma^{-2}g$ (Theorem \ref{RVE}).

Now we continue with the

\begin{proof}[Proof of Proposition \ref{adjoint-gen}] We use adapted coordinates. It suffices to prove 
that the operator $\mathbf{L}(\lambda) \st L(\eta^*(g),s;\lambda)$ satisfies
$$
   \int_X \mathbf{L}(\lambda)(\varphi) \psi dvol_{\eta^*(g)}
   = \int_X \varphi \mathbf{L}(-\lambda-n)(\psi) dvol_{\eta^*(g)}
$$
if $X=[0,\varepsilon) \times M$, $\psi \in C_c^2([0,\varepsilon) \times M)$ and
$\varphi \in C^2([0,\varepsilon) \times M)$ so that $\varphi(0,x)=0$. Now, by definition  
$$
   \mathbf{L}(\lambda) = (n\!+\!2\lambda\!-\!1) (\eta^*(|\NV|^2) \partial_s + \lambda \eta^*(\rho))
   - s (\Delta_{\eta^*(g)} + \lambda \eta^*(\J)).
$$
In the following, we simplify the notation by writing $g$, $\rho$, and $\J$ instead of the pull-backs of these
quantities by $\eta$. Then                                                 \index{$a$} \index{$h_s$}
\begin{equation}\label{LR-adapted}
   \mathbf{L}(\lambda) = (n\!+\!2\lambda\!-\!1) (a \partial_s + \lambda \rho) - s (\Delta_g + \lambda \J),
\end{equation}
where $a \st |\NV|^2$ . In these terms, the background metric reads $g = a^{-1} ds^2 + h_s$
and we obtain $dvol_g = a^{-1/2} ds dvol_{h_s}$. Hence 
\begin{equation*}
   v = a^{-1/2} (\det (h_s)/\det(h))^{1/2}
\end{equation*}
and
\begin{equation}\label{vol-g}
   \frac{v'}{v} = -\frac{1}{2} \frac{a'}{a} + \frac{1}{2} \tr (h_s^{-1} h'_s),
\end{equation}
where $'$ denotes the derivative in the variable $s$. An easy calculation shows that 
\begin{equation}\label{Laplace-adapted}
    \Delta_g = a \partial_s^2 + \frac{a}{2} \tr (h_s^{-1} h'_s) \partial_s
    + \frac{1}{2} a' \partial_s - \frac{1}{2} (d\log a,d \cdot)_{h_s}
    + \Delta_{h_s}.
\end{equation}
Now we observe that
\begin{align*}
    & \int_X a \varphi' \psi dvol_g =  \int_X a \varphi' \psi v ds dvol_h \\
   & = - \int_X \varphi \left[a \psi' + a \frac{v'}{v} \psi + a' \psi \right] v ds dvol_h
   = - \int_X \varphi \left[a \psi' + a \frac{v'}{v} \psi + a' \psi \right] dvol_g
\end{align*}
and
$$
   \int_X s \Delta_g(\varphi) \psi dvol_g = \int_X \varphi \Delta_g(s\psi) dvol_g
$$
using Green's formula and the assumptions. The expression \eqref{Laplace-adapted} shows that
$$
    \Delta_g(s\psi) = s \Delta_g(\psi) + 2 a \psi' + \frac{a}{2} \tr (h_s^{-1} h'_s) \psi + \frac{1}{2} a' \psi.
$$
Hence
\begin{align*}
  \int_X \mathbf{L}(\lambda)(\varphi) \psi dvol_g
  & = -(n\!+\!2\lambda\!-\!1) \int_X \varphi \left( a \psi' + a \frac{v'}{v} \psi + a' \psi \right) dvol_g \\
  & - \int_X \varphi  \left(s \Delta_g(\psi) + 2 a \psi' + \frac{a}{2} \tr (h_s^{-1} h'_s) \psi
  + \frac{1}{2} a' \psi \right) dvol_g \\
  & + \int_X \varphi \psi (\lambda(n\!+\!2\lambda\!-\!1)\rho - \lambda s \J) dvol_g.
\end{align*}
On the other hand, we have
\begin{align*}
   & \int_X \varphi \mathbf{L}(-\lambda\!-\!n)(\psi) dvol_g \\
   & = \int_X \varphi \left(-(n\!+\!2\lambda\!+\!1) a \psi' - s \Delta_g (\psi)
   + (n\!+\!2\lambda\!+\!1)(\lambda\!+\!n) \rho \psi
   +  (\lambda\!+\!n) s \J \psi \right) dvol_g.
\end{align*}
It follows that the assertion is equivalent to the identity
\begin{align*}
   & -(n\!+\!2\lambda\!-\!1) a \frac{v'}{v} - (n\!+\!2\lambda\!-\!1) a' - \frac{a}{2} \tr (h_s^{-1} h'_s)
   - \frac{1}{2} a' + \lambda(n\!+\!2\lambda\!-\!1) \rho - \lambda s \J \\
  & = (n\!+\!2\lambda\!+\!1)(\lambda\!+\!n) \rho + (\lambda \!+\!n) s\J.
\end{align*}
By \eqref{vol-g}, this identity is equivalent to
$$
   -a  \frac{v'}{v} - a' = (n+1) \rho + s \J.
$$
The identities \eqref{vol-g} and \eqref{Laplace-adapted} also show that
\begin{equation}\label{Laplace-s}
   \Delta_g(s) = \frac{a}{2} \tr (h_s^{-1} h'_s) + \frac{1}{2} a' \stackrel{!}{=} a \frac{v'}{v} + a'.
\end{equation}
Thus, we have reduced the assertion to the identity
$$
   - \Delta_g(s) = (n+1) \rho + s \J.
$$
But this is just the definition of $\rho$. The proof is complete. 

The above arguments also prove the relation \eqref{adjoint-g}. In fact, partial integration and 
Green's formula yield the additional terms
$$
   (n\!+\!2\lambda\!-\!1) \int_M \iota^*(a \varphi \psi v) dvol_h + \int_M \iota^* (\varphi \psi |\NV|) dvol_h 
   = (n+2\lambda) \int_M \iota^*(\varphi \psi |\NV|) dvol_h
$$
since the unit normal field on $M$ is $|\NV| \partial_s$ and $v_0 = |\NV|^{-1}$
\end{proof}

Note that equation \eqref{Laplace-s} can be written in the form
\begin{equation}\label{bL2}
   v(s) \Delta_{g} (s) = \partial_s (v(s) a);
\end{equation}
we recall that $a = \eta^*(|\NV|^2)$. As a corollary of this formula, we obtain a useful
formula for $v(s)$ in terms of $\rho$ and $\J$.

\begin{lem}\label{rec}
If $\sigma$ satisfies $\SCY$, then it holds
\begin{equation}\label{bL}
   \frac{v'}{v} = \frac{-(n-1)\rho + 2 s \rho' - s \J}{1-2s\rho} + O(s^n).
\end{equation}
Here $\rho$ and $\J$ are identified with their pull-backs by $\eta$. If $\rho=0$, then
it holds $v'/v = - s \J + O(s^n)$. The latter case contains the Poincar\'e-Einstein case.
\end{lem}

\begin{proof} We write \eqref{bL2} in the form
$$
   v(s) (-(n+1) \rho - s \J) = \partial_s (v(s) a).
$$
Now the assumption implies $a = \SC - 2s \rho = 1- 2s \rho + O(s^{n+1})$. Hence
\begin{align*}
   v(s) (-(n+1) \rho - s \J) & = \partial_s (v(s) (1-2s \rho+ O(s^{n+1}))) \\
   & = v'(s) (1-2s\rho) - 2 v(s) \rho - 2 v(s) s \rho' + v O(s^n).
\end{align*}
Now simplification proves the claim. In the Poincar\'e-Einstein case, one easily shows that $\rho=0$ 
and $v'/v = - s \J$ (see Example \ref{rho-rec-PE}).
\end{proof}

Lemma \ref{rec} can be used to derive formulas for the coefficients $v_k$ with $k \le n$ in terms
of the Taylor coefficients of $\rho$ and $\J$. In particular, we obtain

\begin{cor}\label{v-rho}
For $\N \ni k \le n$, we have
$$
   v_k = -(n\!-\!2k\!+\!1) \iota^* \frac{1}{k!} \partial_s^{k-1}(\rho) + LOT,
$$
where LOT refers to terms with lower-order derivatives of $\rho$ and $\J$.
\end{cor}

\begin{proof}
We consider the coefficient of $s^{k-1}$ in the expansion of $v'/v$.
On the one hand, it equals $k v_k$. On the other hand, \eqref{bL} yields the expression
$$
   \iota^* \partial_s^{k-1} (\rho) \left( - \frac{n-1}{(k-1)!} + \frac{2}{(k-2)!} \right)
$$
for this coefficient. This implies the assertion.
\end{proof}

In particular, the critical coefficient $v_n$ involves the quantity $\iota^* \partial_s^{n-1}(\rho)$.

Conversely, a version of Lemma \ref{rec} implies a recursive formula for the Taylor coefficients
of $\rho$. For the discussion of that formula, we introduce the notation  \index{$\mathring{v}(s)$}
\begin{equation}\label{ring-v}
   \mathring{v}(s) \st  dvol_{h_s} / dvol_h.
\end{equation}
Then
$$
   \frac{\mathring{v}'}{\mathring{v}} = \frac{1}{2} \tr (h_s^{-1} h_s').
$$
We also recall that $a = 1- 2s \rho + O(s^{n+1})$ if $\sigma$ satisfies $\SCY$. The recursive formula
for the Taylor coefficients of $\rho$ will be a consequence of a first-order differential equation.

\begin{lem}[\bf Differential equation for $\rho$]\label{rec-2}
If $\sigma$ satisfies $\SCY$, then $\rho$ solves the differential equation
\begin{equation}\label{magic-rec}
   - s \rho' + n \rho + a \frac{\mathring{v}'}{\mathring{v}} + s \J = O(s^n)
\end{equation}
with the initial condition $\rho(0) = - H$.
\end{lem}

\begin{proof} The identity \eqref{vol-g} implies
\begin{equation}\label{v-deco}
    \frac{v'}{v} = - \frac{1}{2} \frac{a'}{a} + \frac{\mathring{v}'}{\mathring{v}}.
\end{equation}
We use this decomposition on the left-hand side of \eqref{bL} and multiply the resulting identity
with $a$. Then
$$
   (n-1) \rho - 2 s \rho' + s \J - \frac{1}{2} a' + a  \frac{\mathring{v}'}{\mathring{v}} = O(s^n).
$$
By $a' = - 2\rho - 2 s\rho' + O(s^n)$, this identity simplifies to \eqref{magic-rec}. By restriction
of \eqref{magic-rec} to $s=0$, we obtain $n \rho(0) + \tr (L) = 0$ using $h_{(1)} = 2L$
(see \eqref{h-adapted}). Hence $\rho(0) = - H$.
\end{proof}

Another proof of $\rho(0) = -H$ will be given in Lemma \ref{rho-01}.

By repeated differentiation of the identity \eqref{magic-rec} in the variable $s$, it follows that the
Taylor coefficients of $\rho$ can be determined recursively using the Taylor coefficients of $h_s$
and $\J$. More precisely, we obtain

\begin{prop}[\bf Recursive formula for Taylor coefficients of $\rho$]\label{rec-rho}
Assume that $\sigma$ satisfies $\SCY$. Then
\begin{equation}\label{rec-rho-full}
   (n-k) \partial_s^k(\rho)|_0 = -\partial_s^k \left( \frac{\mathring{v}'}{\mathring{v}}\right)|_0
   + 2 \sum_{j=1}^k j \binom{k}{j} \partial_s^{j-1}(\rho)|_0 \partial_s^{k-j}
   \left( \frac{\mathring{v}'}{\mathring{v}}\right)|_0 - k \partial_s^{k-1}(\J)|_0
\end{equation}
for $1 \le k \le n-1$.
\end{prop}

For a discussion of more details of such types of formulas in low-order cases, we refer to Section \ref{TC+B}.
In particular, we use \eqref{rec-rho-full} to derive explicit formulas for the first two derivatives of $\rho$
in $s$ at $s=0$.


\begin{example}\label{rho-rec-PE}
Let $g_+ = s^{-2} (ds^2 + h_s)$ be a Poincar\'e-Einstein metric. Assume that $g=ds^2+h_s$ is smooth
up to the boundary. In particular, the obstruction tensor vanishes. In that case, adapted coordinates coincide
with geodesic  normal coordinates. Now it holds $\rho=0$. We show that the vanishing of the Taylor coefficients
of $\rho$ (in the variable $s$) up to order $n-1$ recursively follows from \eqref{rec-rho-full}. In fact, assume
that we know that $\partial_s^j (\rho)|_0=0$ for $j=0,\dots,k-1$. Then the right-hand side of \eqref{rec-rho-full}
simplifies to
$$
    -\partial_s^k \left( \frac{\mathring{v}'}{\mathring{v}}\right)|_0 - k  \partial_s^{k-1}(\J)|_0.
$$
But
$$
   k \partial_s^{k-1}(\J)|_0 = \partial_s^k(s \J)|_0 = - \frac{1}{2} \partial_s^k (\tr (h_s^{-1}h_s')|_0
  = - \partial_s^k\left( \frac{\mathring{v}'}{\mathring{v}}\right)|_0
$$
by
\begin{equation}\label{J-trace}
    \J = -\frac{1}{2s} \tr (h_s^{-1}h_s')
\end{equation}
(which follows by combining the Einstein condition with the conformal transformation law for scalar curvature
- for the details see \cite[(6.11.8)]{J1}). Hence \eqref{rec-rho-full} implies $\partial_s^k(\rho)|_0=0$.
Alternatively, we could note that the relation \eqref{J-trace} transforms the differential equation
\eqref{magic-rec} into
$$
   -s \rho' + n\rho - 2s \rho \frac{\mathring{v}'}{\mathring{v}} = O(s^n)
$$
with the initial condition $\rho(0) = 0$. Then $\rho=0$ is the unique solution of this initial value problem.
\end{example}

For $k=n$, the coefficient on the left-hand side of \eqref{rec-rho-full} vanishes. This suggests the following
formula for the singular Yamabe obstruction.

\begin{theorem}[\bf The obstruction $\B_n$]\label{obstruction-ex}
Assume that $\sigma$ satisfies $\SCY$. Then
\begin{equation}\label{obstruction-magic}
   (n\!+\!1)! \B_n = -2\partial_s^n \left( \frac{\mathring{v}'}{\mathring{v}}\right)|_0
  + 4  \sum_{j=1}^n j \binom{n}{j} \partial_s^{j-1}(\rho)|_0 \partial_s^{n-j}
   \left( \frac{\mathring{v}'}{\mathring{v}}\right)|_0 - 2n \partial_s^{n-1}(\J)|_0.
\end{equation}
\end{theorem}


\begin{proof}
We start with general data $(g,\sigma)$. As before, we identify $|\NV|^2$ and $\SC$ with their respective
pull-backs by $\eta$. Then $\SC = |\NV|^2 + 2 s \rho = a + 2s\rho$. In these terms, the identity
\eqref{bL2} reads
$$
   v (-(n+1) \rho - s \J) = v' a + v a'.
$$
Hence
\begin{align*}
   a \frac{v'}{v} & = - (n+1) \rho - s \J - \partial_s(a) \\
   & = -(n+1) \rho - s \J + 2 \partial_s (s\rho) - \partial_s (\SC-1) \\
   & = -(n-1) \rho - s \J + 2 s \rho' - \partial_s(\SC-1).
\end{align*}
We decompose the left-hand side using \eqref{v-deco} and reorder. This gives
$$
   a \frac{\mathring{v}'}{\mathring{v}} - \frac{1}{2} a' + (n-1) \rho + s \J - 2 s \rho' + \partial_s(\SC-1) = 0.
$$
Now, using $a = (1-2s\rho) + (\SC-1)$, we obtain the relation
\begin{equation}\label{Basic-R}
 - s \rho'  + n \rho + a \frac{\mathring{v}'}{\mathring{v}} + s \J = -\frac{1}{2} \partial_s (\SC-1)
\end{equation}
which improves \eqref{magic-rec}. Now, assuming that $\sigma$ satisfies $\SCY$, differentiate
\eqref{Basic-R} $n$ times in $s$. By $\partial_s^j(a)|_0 = -2j \partial_s^{j-1}(\rho)|_0$
for $1 \le j \le n$ and
$$
    \partial_s^{n+1}(\SC -1)|_0 = (n+1)! \B_n,
$$
this proves the assertion.
\end{proof}

The basic relation \eqref{Basic-R} will be confirmed in a number of special cases with
$\SC = 1$ in Examples \ref{ball}--\ref{mixed}.


Proposition \ref{rec-rho} and Theorem \ref{obstruction-ex} should be compared with
\cite[Proposition 6.4]{GW-LNY}. The latter result establishes formulas for the restrictions of normal derivatives
$\nabla_\NV^k(\rho)$ of $\rho$ to $M$ and for the obstruction $\B_n$ in terms of lower-order normal derivatives
of $\rho$ and additional terms. The above results clarify the structure of all such additional terms.
Here it is crucial to work in adapted coordinates.

Note that the formula \eqref{obstruction-magic} shows that the obstruction $\B_n$
involves the Taylor coefficients $h_{(k)}$ of $h_s$ (in the normal form
\eqref{normal-adapted} of $g$ in adapted coordinates) for $k \le n+1$.

In Section \ref{B2-details}, we shall derive the classical formula for $\B_2$ (see
\eqref{B2} and \cite{ACF}) from \eqref{obstruction-magic}. Similarly, in Section
\ref{B3-flat-back} we evaluate the formula \eqref{obstruction-magic} for the obstruction $\B_3$ 
in case of a (conformally) flat background.


Finally, we apply the above results to determine the leading term of the obstruction $\B_n$
for an embedding $M^n \hookrightarrow \R^{n+1}$ if $n$ is even.

First, we note that Theorem \ref{obstruction-ex} and Proposition \ref{rec-rho} show
that  $\B_n$ is a functional of the second fundamental form $L$.

\begin{theorem}\label{B-linear} For a flat background metric and even $n$, it holds
\begin{equation}\label{Bn-deco}
   (n+1)! \B_n = c_n \Delta^\frac{n}{2} (H) + nl
\end{equation}
with
$$
   c_n = - 2 \frac{(n-1)!!}{(n-2)!!}
$$
and a non-linear functional $nl$ of $L$. For odd $n$, the obstruction $\B_n$ is non-linear in $L$.
\end{theorem}

The non-linear part in \eqref{Bn-deco} can also be described as a term of lower
differential order. In fact, we can write $\B_n$ as a sum of terms that are
homogeneous in $L$. In each such term, the sum of the number of derivatives and the
homogeneous degree in $L$ is $n+1$. But the non-linear terms in \eqref{Bn-deco}
consist of homogeneous terms of degree at least $2$. One should compare that version 
of the structural result for $\B_n$ with \cite[Theorem 5.1]{GW-LNY}.

Theorem \ref{B-linear} extends the following observations. By the second formula in
\eqref{B2}, $\B_2$ is the sum of a constant multiple of $\Delta(H)$ being linear in
$L$ and a term that is cubic in $L$ and does not contain derivatives. Similarly,
the first three terms in \eqref{B3-final2} are homogeneous of degree $2$ in $L$ and
each such term involves $2$ derivatives.

\begin{proof} Let $n$ be even. We extract from the formula
\begin{equation*}\label{ob-magic}
   (n\!+\!1)! \B_n = -2 \partial_s^n \left( \frac{\mathring{v}'}{\mathring{v}}\right)|_0
  + 4  \sum_{j=1}^n j \binom{n}{j} \partial_s^{j-1}(\rho)|_0 \partial_s^{n-j}
   \left( \frac{\mathring{v}'}{\mathring{v}}\right)|_0
\end{equation*}
the contributions which are {\em linear} in $L$. In the following, the symbol $nl$
indicates non-linear terms. First, we ignore in this sum all products with at least
two factors. Hence
$$
   (n+1)! \B_n = - 2  \partial_s^n \left( \frac{\mathring{v}'}{\mathring{v}}\right)|_0 + nl.
$$
Moreover, the expansion
$$
   \frac{\mathring{v}'}{\mathring{v}} = \frac{1}{2} \tr (h_s^{-1} h_s')
   = \frac{1}{2} \sum_{k \ge 1} ( k \tr (h_{(k)}) + nl ) s^{k-1}
$$
implies
$$
    \partial_s^n \left( \frac{\mathring{v}'}{\mathring{v}}\right)|_0 = \frac{1}{2} (n+1)! \tr (h_{(n+1)}) + nl.
$$
In order to evaluate $h_{(n+1)}$, we $(n-1)$-times differentiate in $s$ the determining relation
\eqref{R-adapted-2} for $h_s$. Then
$$
   \frac{1}{2}\partial_s^{n+1}(h_s)|_0 = - \Hess (\partial_s^{n-1}(s\rho))|_0) + nl
$$
using $g^{00} = a = 1 - 2s \rho$. Hence
$$
   \frac{1}{2} (n+1)! h_{(n+1)} = - (n-1) \Hess (\partial_s^{n-2}(\rho)|_0 )+ nl.
$$
These results imply
\begin{align}\label{Bn-leading}
   (n+1)! \B_n & = - (n+1)! \tr (h_{(n+1)}) + nl  \notag \\
   & = 2 (n-1) \Delta (\partial_s^{n-2}(\rho)|_0) + nl.
\end{align}
Now Proposition \ref{rec-rho} shows that
\begin{align*}
   (n-k) \partial_s^k(\rho)|_0 & = - \partial_s^k \left(\frac{\mathring{v}'}{\mathring{v}}\right)|_0 + nl \\
   & = - \frac{1}{2} (k+1)! \tr (h_{(k+1)}) + nl.
\end{align*}
But $(k-1)$-times differentiating in $s$ the determining relation for $h_s$, shows that
\begin{align*}
   \frac{1}{2} \partial_s^{k+1} (h_s)|_0 & = - \Hess(\partial_s^{k-1}(s \rho)|_0) + nl \\
   & = - (k-1) \Hess (\partial_s^{k-2}(\rho)|_0) +nl.
\end{align*}
Hence
\begin{equation}\label{rho-leading}
   (n-k) \partial_s^k(\rho)|_0 = (k-1) \Delta (\partial_s^{k-2}(\rho)|_0) + nl.
\end{equation}
Combining \eqref{Bn-leading} and \eqref{rho-leading} gives
\begin{align*}
   (n+1)! \B_n & = 2 (n-1) \Delta (\partial_s^{n-2}(\rho)|_0) + nl \\
   & = (n-1)(n-3) \Delta^2 (\partial_s^{n-4}(\rho)|_0) + nl \\
   & = \cdots = 2 (n-1)!!/(n-2)!! \Delta^\frac{n}{2}(\rho|_0) + nl.
\end{align*}
This implies the assertion using $\rho|_0 = -H$ (Lemma \ref{rho-01}). For odd $n$, the same
arguments show that $\B_n$ is a constant multiple of $\Delta^{\frac{n-1}{2}}(\partial_s(\rho)|_0) + nl$.
Since $\partial_s(\rho)|_0$ is a constant multiple of $|\lo|^2$ (Lemma \ref{rho-01}), this completes the proof.
\end{proof}

\begin{example} For $n=2$ and $n=4$, we find
$$
   3! \B_2 = - 2 \Delta(H) + nl  \quad \mbox{and} \quad
   5! \B_4 = - 3 \Delta^2 (H) + nl,
$$
respectively. The first decomposition fits with \eqref{B2}.
\end{example}

\begin{rem}\label{Bn-odd}
For a flat background and odd $n$, the proof of Theorem \ref{B-linear} shows that one contribution
to $\B_n$ is a constant multiple of $\Delta^{\frac{n-1}{2}}(|\lo|^2)$. But $\B_n$ has further contributions
of the same differential order, which are quadratic in $L$. For instance, \eqref{B3-final2} shows that in
addition to $\Delta(|\lo|^2)$, $\B_3$ contains the contributions $|dH|^2$ and $(\lo,\Hess(H))$ of differential
order $2$.
\end{rem}

We finish this section with a representation theoretical argument proving the vanishing of the obstruction $\B_n$
for the equatorial subsphere $S^n \hookrightarrow S^{n+1}$. In the following, we use the notation of 
Section \ref{SBO}.

First, assume that $M^n \hookrightarrow \R^{n+1}$ (with the flat metric $g_0$). Let $\B_n^M(g_0)$ be the singular
Yamabe obstruction of $M$. Let $\gamma \in SO(1,n+2)$ be a conformal diffeomorphism of $g_0$, i.e.,
$\gamma_*(g_0) = e^{2 \Phi_\gamma} g_0$. Then
$$
   e^{(n+1) \iota^* \Phi_\gamma} \B_n^M (\gamma_*(g_0)) = \B_n^M(g_0)
$$
by Lemma \ref{B-CTL}. This relation is equivalent to
\begin{equation}\label{B-diffeo}
   e^{(n+1) \iota^* \Phi_\gamma} \gamma_* (\B_n^{\gamma(M)}(g_0))  = \B_n^M(g_0), \; \gamma \in SO(1,n+2).
\end{equation}
In particular, all $\gamma \in SO(1,n+1) \hookrightarrow SO(1,n+2)$ leave invariant
the hypersurface $M^n = \R^n \hookrightarrow \R^{n+1}$, and it holds
$$
   \iota^* \left(\frac{\gamma_*(r)}{r}\right)^{n+1} \gamma_* (\B_n^M (g_0) ) = \B_n^M(g_0),
$$
i.e.,
$$
   \pi_{n+1}^0(\gamma) (\B_n^M(g_0)) = \B_n^M(g_0).
$$
But since the identical representation is not a subrepresentation of $\pi_{n+1}^0$, it follows that $\B_n^M(g_0)=0$.
A similar argument proves the vanishing of $\B_n$ for the equatorial subsphere $S^n \hookrightarrow S^{n+1}$.
By the analog of \eqref{B-diffeo} for hypersurfaces of $S^{n+1}$, the obstruction of any $\gamma(S^n)$ vanishes, too.


\section{Residue families}\label{res-fam}

\index{$M_u(\lambda)$}

In the present section, we associate residue families to any pair $(g,\sigma)$ which satisfies the
condition $\SCY$ (see Section \ref{Yamabe}). Residue families are defined in terms of the residues
of one-parameter families
$$
   \lambda \mapsto \langle M_u(\lambda),\psi \rangle = \int_X \sigma^\lambda u \psi dvol_g, \;
   \Re(\lambda) \gg 0
$$
of distributions on $X$, where $u$ are eigenfunctions of the Laplacian
$\Delta_{\sigma^{-2}g}$ of the singular metric $\sigma^{-2}g$. These residues
express the obstruction extending $u$ as a distribution up to the boundary of $X$.

The restriction of $\SC(g,\sigma)$ to the boundary $M$ equals $|\grad_g(\sigma)|_g^2 =
|d\sigma|^2_g$. Therefore, the condition $\SCY$ implies that
$|d\sigma|_g^2=1$ on $M$. That property is equivalent to the property that the
sectional curvatures of $\sigma^{-2}g$ tend to $-1$ at the boundary $M$, i.e., the
metric $\sigma^{-2}g$ is  asymptotically hyperbolic. Next, we recall
some basic results in the spectral theory of the Laplacian of asymptotically hyperbolic
metrics. For more details, we refer to \cite[Section 3]{GZ}. The spectrum of
$-\Delta_{\sigma^{-2}g}$ is the union of a finite pure
point spectrum $\sigma_{pp} \subset (0,(n/2)^2)$ and an absolutely continuous
spectrum $\sigma_{ac} = [(n/2)^2,\infty)$ of infinite multiplicity. The generalized
eigenfunctions with smooth functions on $M$ as boundary values are described
by a {\em Poisson operator}. This operator is a far-reaching
generalization of the well-known Poisson transform of Helgason \cite{He} which
relates generalized eigenfunctions of the commutative algebra of invariant
differential operators on a symmetric space of the non-compact type to
hyperfunctions on a naturally associated boundary. It is defined by an integral
transform. In the present situation, the family $\Po(\lambda)$ of Poisson operators
is meromorphic for $\Re(\lambda)<n/2$, $\lambda \ne n/2$ with poles in $\lambda$ iff
$\lambda(n-\lambda) \in \sigma_{pp}$ such that       \index{$\Po(\lambda)$ \quad Poisson operator}
$$
   (\Delta_{\sigma^{-2}g} + \lambda(n-\lambda)) \Po(\lambda) (f) = 0
$$
for any $f \in C^\infty(M)$. In contrast to Helgason's definition of a Poisson operator by
an integral transform, it is defined in terms of the resolvent of the Laplacian. In both theories,
the argument $f$ is seen in the leading terms of the asymptotic expansion of the
eigenfunction $\Po(\lambda)(f)$. To describe the asymptotic expansion of
eigenfunctions in the range of the Poisson operator, we choose coordinates on $X$
near the boundary. Indeed, there is a unique defining
function $\theta$ and a diffeomorphism $\tau$ mapping $[0,\varepsilon] \times M$
with coordinates $(t,x)$ to a neighborhood of $M$ in $X$ so that    \index{$\tau$} \index{$\theta$}
$$
    \tau^* (\sigma^{-2} g) = t^{-2} (dt^2 + h_t) \quad \mbox{and} \quad \;
    \iota^*(\theta^2 \sigma^{-2}g) = h_0 \stackrel{!}{=} h, \; \tau^*(\theta) = t.
$$
This is the normal form of an asymptotically hyperbolic metric with prescribed conformal infinity
as used in \cite{GZ}. The function $\theta$ is a solution of the eikonal equation $|d\theta|_{\theta^2
\sigma^{-2} g}=1$ near $M$ and the gradient flow of $\theta$ with respect to the
metric $\theta^2 \sigma^{-2} g$ defines the diffeomorphism $\tau$. Note that
$\iota^*(\theta^2 \sigma^{-2}) = 1$. In these terms, the eigenfunctions
$\Po(\lambda)(f)$ have the following properties.
\begin{itemize}
\item [(i)] $ \tau^* \Po(\lambda) (f) = t^{\lambda}  f + t^{n-\lambda} g + O(t^{n/2+1})$
for some $g \in C^\infty(M)$ if $\Re(\lambda)=n/2$, $\lambda \ne n/2$.\footnote{Of course, the
function $g$ should not be confused with the metric $g$.}
\item [(ii)] $\tau^*  \Po(\lambda) (f) = t^\lambda F + t^{n-\lambda} G$ with smooth $F$ and
$G$ on $[0,\varepsilon) \times M$ so that $\iota^*(F)=f$ if $\Re(\lambda) \le n/2$ with
$\lambda \not\in \{n/2-N/2 \,|\, N \in \N_0\}$ and $\lambda(n-\lambda) \notin \sigma_{pp}$.
\end{itemize}
The function $f$ is called the {\em boundary value} of $u = \Po(\lambda)(f)$. The
function $F$ in (ii) depends on $\lambda$ and has poles in $\lambda \in \{n/2-N/2 \,|\, N \in \N\}$.
But these poles cancel against poles of the second term $G= G(\lambda)$ in the decomposition of $\Po(\lambda)$.
If $\lambda$ is as in (ii), we define         \index{$\SC(\lambda)$ \quad scattering operator}
\begin{equation}\label{scatt-def}
   \Sc(\lambda)(f) = \iota^*(G).
\end{equation}
The operator $\Sc(\lambda)$ is called the {\em scattering operator} of the asymptotically
hyperbolic metric $\sigma^{-2}g$.\footnote{The substitution
$\lambda \mapsto n-\lambda$ maps the operator $\Sc(\lambda)$ to the scattering operator in \cite{GZ}.}
It is a family of pseudo-differential operators with principal
symbol being a constant multiple of $|\xi|^{n-2\lambda}$. It is meromorphic in $\Re(\lambda) < n/2$
with poles in the set $\{\frac{n-N}{2} \,|\, N \in \N\}$ and if $\lambda(n-\lambda) \in \sigma_{pp}$.
The poles in $\lambda = \frac{n-N}{2}$ are sometimes referred to as its {\em trivial} poles. Its nontrivial poles
in $\Re(\lambda) > n/2$ will not be of interest here. Now let $u = \Po(\mu)(f) $ be a solution of
$$
   -\Delta_{\sigma^{-2}g} u = \mu (n-\mu)u, \quad \Re(\mu) = n/2, \; \mu \ne n/2
$$
with {\em boundary value} $f \in C^\infty(M)$. Instead of the asymptotic expansions
of $u$ as above, we will consider asymptotic expansion in terms of powers of
$\sigma$. The following formal arguments describe the expansion of $u$ using the
Laplace-Robin operators $L(g,\sigma;\mu)$.\footnote{The arguments are only formal
since the products $\sigma^j f_j$ are not functions on $X$ without a specification of
coordinates.} In the Poincar\'e-Einstein case, the following algorithm is contained
in the proof of \cite[Proposition 4.2]{GZ}. We start with $f_0 = f$ and define $f_N
\in C^\infty(M)$ recursively by
$$
   N(n-2\mu-N) f_N = \iota^*\left(\sigma^{-N+1} L(-\mu) \left(\sum_{j=0}^{N-1} \sigma^j f_j \right)\right).
$$
But the definition of $L(\lambda)$ implies
$$
   L(\lambda) (\sigma^j f_j) = j(n+2\lambda-j) \sigma^{j-1} f_j + O(\sigma^j), \; \lambda \in \C.
$$
Hence
$$
   L(-\mu)  \left(\sum_{j=0}^N \sigma^j f_j\right) = O(\sigma^N)
$$
and the conjugation formula yields
$$
   -(\Delta_{\sigma^{-2}g} + \mu(n-\mu)) \left(\sigma^\mu  \sum_{j=0}^{N} \sigma^j f_j \right)
   = \sigma^{\mu+1} L(-\mu)  \left(\sum_{j=0}^{N} \sigma^j f_j \right) = \sigma^{\mu}
   O(\sigma^{N+1}).
$$
The coefficients $f_j$ are given by differential operators $f \mapsto \T_j(\mu)(f)$. The construction shows that
the operator $\T_N(\mu)$ has simple poles in the set
$$
   \left \{ \frac{n-j}{2} \; |\, j \le N \right \}.
$$
These are the poles that appeared above in (ii). Note that, using $\iota^*(\rho)=-H$, it easily follows
that $\T_1(\mu)=-\mu H$ (see also Lemma \ref{sol-1}). We also observe that $L(\mu)(1) \in C^\infty(X)$
is a multiple of $\mu$. By an easy induction, this implies that
\begin{equation}\label{T-1}
    \T_N(\mu)(1) = \frac{\mu}{N! (n\!-\!2\mu\!-\!1) \cdots (n\!-\!2\mu\!-\!N)} \QC_{N}(\mu)
\end{equation}
for some polynomial $\QC_{N}(\mu) \in C^\infty(M)$. In particular, the function $\T_n(\mu)(1)$
is regular at $\mu=0$.

Now, in order to justify the above arguments, we use adapted coordinates. Let $u$ be as above. Then
$\tau^*(u)$ has the form $t^\mu F + t^{n-\mu} G$ with $\iota^*(F)=f$ and $\iota^*(G) = \SC(\mu)(f)$.
Let $\zeta \st \tau^{-1} \circ \eta$. Then                                           \index{$\zeta$}
$$
   \zeta^*(t) = \eta^* \tau_* (t) = \eta^* \left(\sigma \frac{\tau_*(t)}{\sigma}\right)
  = s \eta^* \left(\frac{\tau_*(t)}{\sigma}\right)
$$
by $\eta^*(\sigma)=s$. It follows that the pull-back by $\zeta$ of $\tau^*(u)$ equals
$$
   \eta^*(u) = s^\mu \Omega^\mu \zeta^*(F) + s^{n-\mu} \Omega^{n-\mu} \zeta^*(G)
$$
with $\Omega \st \tau_*(t)/\sigma$. But $\iota^* \zeta^* (F) = \iota^*(F)$, $\iota^* \zeta^*(G) = \iota^*(G)$
and
\begin{equation}\label{Omega-rest}
   \iota^* (\Omega) = 1
\end{equation}
imply that the leading terms in the expansion of $\eta^*(u)$ are $\iota^*(F) = f$ and $\iota^*(G) = \Sc(\mu)(f)$.
In order to prove the restriction property \eqref{Omega-rest}, we recall that
$$
   \tau_* (t^{-2} (dt^2+h_t)) = \sigma^{-2} g.
$$
Hence
$$
   \Omega^2 g = dt^2 + h_t.
$$
In particular, $\iota^* (\Omega^2) h = h_0$. But in the construction of $(\tau,\theta)$ we required that $h_0=h$.
This proves \eqref{Omega-rest}. Thus, the asymptotic expansion of the eigenfunction $\eta^*(u)$ of
$\Delta_{\eta^*(\sigma^{-2}g)}$ takes the form         \index{$\T_j(\lambda)$ \quad solution operators}
\begin{equation}\label{asymp}
   \sum_{j \ge 0} s^{\mu+j} \T_j(\mu)(f)(x)
   + \sum_{j \ge 0} s^{n-\mu+j} \T_j(n\!-\!\mu) \Sc(\mu)(f)(x), \quad s \to 0,
\end{equation}
where $\T_j(\mu)$ are families of differential operators on $M$; we shall refer to these operators as
{\em solution operators}. One easily find that the order of $\T_j(\mu)$ is $\le 2 [\frac{j}{2}]$. The above
formal arguments show that the families $\T_N(\mu)$ are rational in $\mu$ with simple poles in the set
$$
   \left\{ \frac{n-j}{2} \; |\, j \le N \right\}.
$$
Of course, the coefficients $\T_j(\mu)$ can be determined recursively in terms of the Laplace-Robin operator
in adapted coordinates.  In the following, it will often suffice to work with a finite version of the expansion 
\eqref{asymp}.

\begin{rem}\label{sol-op} The solution operators $\T_j(\lambda): C^\infty(M) \to C^\infty(M)$ describe
formal asymptotic expansions of eigenfunctions (with smooth boundary value) of the Laplacian of asymptotically
hyperbolic metrics $\sigma^{-2}g$. Another type of asymptotic expansions of eigenfunctions appears in
\cite[Section 5]{GW}. In an even more general setting, these are expansions, say in powers of a defining function,
the coefficients of which are functions on the space $X$ but not on its boundary $M$. Comparing both types of
expansions would require additional expansions of the coefficients.
\end{rem}

Later, we shall use the fact that the scattering operator $\Sc(\lambda)$ for $\lambda \in \R$
is formally selfadjoint with respect to the scalar product on $C^\infty(M)$ defined by $h$.
For the convenience of the reader, we include a proof that directly derives this property from
the expansion \eqref{asymp} (without invoking the definition of $\Sc(\lambda)$ in terms of
expansions in power series of $t$) (compare with \cite[Proposition 3.3]{GZ}).

\begin{lem}\label{scatt-sd} $\Sc(\lambda)^* = \Sc(\lambda)$ for $\lambda \in \R$, $\lambda < n/2$
such that $\lambda \notin \left\{ \frac{n-N}{2} \,|\, N \in \N \right\}$ and $\lambda(n-\lambda) \notin \sigma_{pp}$.
\end{lem}

Note that the assumptions guarantee that the ladders $\lambda+\N$ and $n-\lambda +\N$ are disjoint.

\begin{proof} We recall Green's formula
\begin{equation}\label{Green}
   \int_X (du,dv)_g dvol_g + \int_X u \Delta_g(v) dvol_g
   = \int_{\partial X} u \star_g dv = \int_{\partial X} u i_{N(v) N} (dvol_g)
\end{equation}
on a compact Riemannian manifold $(X,g)$ with boundary $\partial X$. Here $N$ denotes a unit
normal field. Now let $\R \ni \lambda < n/2$ be as above. Let
$u_1$ and $u_2$ be real solutions of
$$
   -\Delta_{\eta^*(\sigma^{-2}g)} u = \lambda(n-\lambda) u
$$
on $(0,\varepsilon) \times M$ of the form $u = s^{\lambda} F + s^{n-\lambda} G$ with smooth
$F,G$. These are defined by the Poisson transforms of smooth boundary functions $f_1$ and $f_2$. Let
$$
   k \st \eta^*(\sigma^{-2}g) = s^{-2}( a^{-1}  ds^2 + h_s).
$$
Then $dvol_k = s^{-n-1} a^{-1/2} ds dvol_{h_s}$ and the restriction of
$\nu = s \sqrt{a} \partial_s$ to the boundary $s=\varepsilon$
defines a unit normal field. By \eqref{Green}, we find
\begin{align}\label{Green-a}
   & \int_{s > \varepsilon} ((du_1,du_2)_k - \lambda(n-\lambda) u_1 u_2) dvol_k \notag \\
   & = \varepsilon^{-n} \int_{s=\varepsilon} u_1 \partial_{\nu}(u_2) dvol_{h_\varepsilon}
   = \varepsilon^{-n+1} \int_{s=\varepsilon} u_1 \partial_s(u_2) \sqrt{a} dvol_{h_\varepsilon}.
\end{align}
The finite part of the expansion of the last integral in $\varepsilon$ is the coefficient of
$\varepsilon^{n-1}$ in the expansion of
$$
   \int_{s=\varepsilon} u_1 \partial_s(u_2) \sqrt{a} dvol_{h_\varepsilon}.
$$
By plugging in the expansions of $u_1$, $u_2$, $\sqrt{a}$ and
$dvol_{h_\varepsilon}$, we obtain the expression
$$
   \int_M \iota^* (\lambda F_2 G_1 + (n-\lambda) F_1 G_2) dvol_h
$$
for this coefficient. Since the left-hand side of \eqref{Green-a} is symmetric in
$u_1$ and $u_2$, the latter result equals
$$
   \int_M \iota^* (\lambda F_1 G_2+ (n-\lambda) F_2 G_1) dvol_h.
$$
It follows that
$$
  \int_M \iota^* (F_1 G_2) dvol_h = \int_M \iota^* (F_2 G_1) dvol_h.
$$
In terms of the expansion \eqref{asymp}, this means that
$$
    \int_M f_1 \SC(\lambda) (f_2) dvol_h = \int_M f_2 \SC(\lambda) (f_1) dvol_h,
$$
i.e., $\Sc(\lambda)^* = \Sc(\lambda)$.
\end{proof}


Now we are ready to define residue families.

\index{$\delta_N(g,\sigma;\mu)$}
\index{$\D_N^{res}(g,\sigma;\lambda)$ \quad residue family}

Let $\N \ni N \le n$ and assume that the condition $\SCY$ is satisfied. We consider an eigenfunction $u$
of $\Delta_{\sigma^{-2}} g$ with boundary value $f \in C^\infty(M)$ satisfying
\begin{equation}\label{eigen}
   \Delta_{\sigma^{-2}g} u + \mu (n-\mu) u = 0 \quad \mbox{for $\Re(\mu) = n/2$, $\mu \ne n/2$}.
\end{equation}
Such eigenfunctions have asymptotic expansions (as above) near the boundary. We consider the integral
\begin{equation}\label{Mu-L}
    \left\langle M_u(\lambda),\psi \right\rangle = \int_X \sigma^\lambda u \psi dvol_g
\end{equation}
with $\psi \in C_c^\infty(X)$ and $\lambda \in \C$. The function 
$\lambda \mapsto \left\langle M_u(\lambda),\psi \right\rangle$
is holomorphic if $\Re(\lambda) \gg 0$ and we regard $M_u(\lambda)$ as a holomorphic family of distributions on $X$.
If $\supp(\psi) \cap M = \emptyset$, then $M_u(\lambda)$ admits a holomorphic continuation to $\C$. $M_u(\lambda)$
generalizes  the meromorphic family of distributions $M(\lambda) = M_1(\lambda)$ discussed in Section \ref{intro}. Likewise
as $M(\lambda)$, the family $M_u(\lambda)$ admits a meromorphic continuation with simple poles in $\{-\mu - N - 1 \}$.
The proof of this fact is similar as for $u=1$ and follows by expanding the integrand near the boundary. In addition to these
poles, $M_u(\lambda)$ has simple poles in the set $\{\mu + n - N - 1\}$. However, this second ladder of poles will be
ignored in the following. The details are given in the proof of Theorem \ref{D-res-ex}. This proof also shows that the
residues have the form
\begin{equation}\label{delta-def}
   \Res_{\lambda=-\mu-1-N} \left( \int_X \sigma^\lambda u \psi dvol_g \right)
   = \int_M f \delta_N(g,\sigma;\mu)(\eta^*(\psi)) dvol_h
\end{equation}
with some meromorphic families $\delta_N(g,\sigma;\mu)$ of differential operators $C^\infty([0,\varepsilon)\times M)
\to C^\infty(M)$ of order $\le N$. The residues of $M_u(\lambda)$ are distributions on $X$ with support on the
boundary $M$ of $X$. They may be regarded as obstructions to extending $M_u(\lambda)$ as a distribution to $X$.

\begin{defn}[\bf Residue families]\label{D-res} Let $\N \ni N \le n$ and assume that the condition $\SCY$ is satisfied.
Then the one-parameter family
\begin{equation}\label{D-res-def}
   \D_N^{res}(g,\sigma;\lambda) = N!(2\lambda\!+\!n\!-\!2N\!+\!1)_N \delta_N(g,\sigma;\lambda\!+\!n\!-\!N), \; \lambda \in \C
\end{equation}
is called the residue family of order $N$. The family $\D_n^{res}(\lambda)$ will be called the critical residue family.
\end{defn}

Some comments are in order. 
The families $\delta_N(\mu)$ are defined for $\Re(\mu) = \frac{n}{2}$ (with $\mu \ne \frac{n}{2}$). 
Hence $\D_N^{res}(\lambda)$ is defined for $\Re (\lambda) = -\frac{n}{2}+N$ (with $\lambda \ne -\frac{n}{2}+N$). 
But the normalizing coefficient in \eqref{D-res-def} is chosen so that $\D_N^{res}(g,\sigma;\lambda)$ actually extends to 
a {\em polynomial}
family of order $N$ and degree $2N$ in $\lambda \in \C$. This fact will follow from Theorem \ref{residue-product}. The 
proof of the latter result actually contains a second proof of the existence of the meromorphic continuation of $\langle M_u(\lambda),\psi \rangle$. The method is a version of a Bernstein-Sato-type argument. It provides explicit knowledge 
of the position of poles of $M_u(\lambda)$ and formulas for the residues. However, the following result gives a 
more direct description of the operator $\delta_N(g,\sigma;\lambda)$ in terms of solution 
operators. The equivalence of both descriptions of residues will have interesting consequences.

\begin{theorem}\label{D-res-ex} Let $\N \ni N \le n$. Then
\begin{align}\label{D-res-sol}
   \D_N^{res}(g,\sigma;\lambda) & = N! (2\lambda\!+\!n\!-\!2N\!+\!1)_N \\
   & \times \sum_{j=0}^{N} \frac{1}{j!} \left[ \T_{N-j}^*(g,\sigma;\lambda\!+\!n\!-\!N) v_0 + \cdots +
   \T_0^*(g,\sigma;\lambda\!+\!n\!-\!N) v_{N-j} \right] \iota^* \partial_s^j. \notag
\end{align}
\end{theorem}

\begin{proof} The relation \eqref{key-pb} implies
\begin{equation}\label{M-adapted}
   \langle M_u(\lambda),\psi \rangle =  \int_X \sigma^\lambda u \psi dvol_g = \int_{(0,\varepsilon) \times M} s^\lambda \eta^*(u)
   \eta^*(\psi) dvol_{\eta^*(g)}, \; \Re(\lambda) \gg 0.
\end{equation}
Here $\eta^*(u) $ is an eigenfunction of the Laplacian of the metric $s^{-2}
\eta^*(g)$. In order to simplify the notation, we write the latter integral as
$$
   \int_{[0,\varepsilon) \times M} s^\lambda u \psi dvol_{\eta^*(g)} = \int_0^\infty
   \int_M s^\lambda u \psi v ds dvol_h
$$
with an appropriate eigenfunction $u$ and test functions $\psi$ on the space
$[0,\varepsilon) \times M$. Now we expand $v$ according to \eqref{RVC-B} and $u$
according to \eqref{asymp}. The classical formula \cite{Gelfand}
\begin{equation}\label{Gelfand}
   \Res_{\lambda=-N-1} \left( \int_0^\infty s^\lambda \psi(s) ds \right) = \frac{\psi^{(N)}(0)}{N!}, \; N \in \N_0
\end{equation}
for test functions $\psi \in C_c^\infty(\overline{\R_+})$ shows that
\begin{align*}
    & \Res_{\lambda=-\mu-1-N} \left( \int_X \sigma^\lambda u \psi dvol_g \right) \\
    & = \Res_{\lambda=-\mu-1-N} \left( \int_0^\infty \sum_{a=0}^N \sum_{j+k=a}
    \int_M s^{\lambda+\mu+a} \T_k(\mu)(f) v_j \psi dvol_h \right) \\
    & = \sum_{a=0}^N \frac{1}{(N-a)!} \sum_{j+k=a} \int_M \T_k(\mu)(f) v_j \iota^* \partial_s^{N-a}(\psi) dvol_h.
\end{align*}
Since $f$ is arbitrary, taking adjoints proves the assertion.
\end{proof}

\begin{remark}\label{rel-def}
Definition \ref{D-res} generalizes the notion of
residue families $D_N^{res}(h;\lambda)$ introduced in \cite{J1}. In that case,
$g = r^2 g_+$ is the conformal compactification of a Poincar\'e-Einstein metric
$g_+$. However, the definitions in \cite{J1, J2,  FJO} use a different normalizing
coefficient. That choice is motivated by the fact that in these references, the expansion of
$g$ involves only even powers of $r$. More precisely, if $g_+$ is in normal form
relative to $h$, i.e., $\iota^*(r^2g_+)=h$, then it holds
\begin{align*}
   \D^{res}_{2N}(g,r;\lambda) & = (-2N)_N
   \left(\lambda\!+\!\frac{n}{2}\!-\!2N\!+\!\frac{1}{2}\right)_N D^{res}_{2N}(h;\lambda), \\
   \D^{res}_{2N+1}(g,r;\lambda) & =
   -2(-2N\!-\!1)_{N+1}\left(\lambda\!+\!\frac{n}{2}\!-\!2N\!-\!\frac{1}{2}\right)_{N+1} D^{res}_{2N+1}(h;\lambda).
\end{align*}
If $\lambda$ is a zero of the prefactors in these factorization identities, then the family $\D_N^{res}(g,r;\lambda)$
vanishes and therefore hides the non-trivial operator $D_N^{res}(h;\lambda)$. In particular, if $2N+1=n$, then
$$
   \D_n^{res} (g,r;0) = 0.
$$
However, for general $g$ and $\sigma$ satisfying $\SCY$, the critical value
$\D_n^{res}(g,\sigma;0)$ for odd $n$ need not vanish. For the case $n=3$, we refer to
Section \ref{P3}.

The degrees of $D_{2N}^{res}(h;\lambda)$ and $D_{2N+1}^{res}(h;\lambda)$ both equal
$N$. Hence the above relations show that the respective degrees of
$\D_{2N}^{res}(g,r;\lambda)$ and $\D_{2N+1}^{res}(g,r,\lambda)$ are $2N$ and $2N+1$.
Since in the generic case $\D_N^{res}(\lambda)$ has degree $2N$, it follows that in
the Poincar\'e-Einstein case the degrees fall on half. This drop in degree reflects
the vanishing of curvature invariants in the Poincar\'e-Einstein case. For instance,
\eqref{D1-exp} and Lemma \ref{D2-exp} show that $\D_1^{res}(g,r;\lambda)$ and
$\D_2^{res}(g,r;\lambda)$ have respective degrees $1$ and $2$. In these cases, the
vanishing of $H$, $\lo$, and $\Rho_{00}$ are responsible for the drop in degree.
\end{remark}

The following result implies that residue families of order $N \le n$ are completely
determined by the metric $g$ and the embedding $M \hookrightarrow X$.

\begin{prop}\label{sigma-free}
Let $\N \ni N \le n$. Then $\D_N^{res}(g,\sigma;\lambda)$ is determined only by the
coefficients $\sigma_{(j)}$ for $j \le N+1$ in the expansion of $\sigma$ in geodesic
normal coordinates.
\end{prop}

\begin{proof}
The claim follows by evaluating the residue definition of residue families in terms
of geodesic normal coordinates. We use the diffeomorphism $\kappa$ to write
\begin{equation}\label{M-geodesic}
   \langle M_u(\lambda),\psi \rangle = \int_X \sigma^\lambda u \psi dvol_g
   = \int_{[0,\varepsilon)} \int_M \left( \frac{\kappa^*(\sigma)}{r} \right)^\lambda r^\lambda \kappa^*(u)
   \kappa^*(\psi) dr dvol_{h_r}
\end{equation}
for $\Re(\lambda) \gg 0$ and test functions $\psi$ with sufficiently small support.
The eigenfunction $\kappa^*(u)$ of $\Delta_{\kappa^*(\sigma)^{-2} g}$ has an
asymptotic expansion in $r$ of the form
$$
   \sum_{j \ge 0} r^{\mu+j} \T_j(\mu)(f) + \cdots,
$$
where the dots indicate an asymptotic expansion with exponents $n-\mu+j$. In that
expansion, the operators $\T_j(\mu)$ are determined by recursive relations. An
induction argument using the formula
$$
   \Delta_{\kappa^*(\sigma)^{-2} g} 
   = (\kappa^*(\sigma))^2 \Delta_{dr^2 + h_r} - (n-1) \kappa^*(\sigma) \nabla_{\grad(\kappa^*(\sigma))}
$$
shows that $\T_N(\mu)$ is determined only by the coefficients of $r^j$ for $j \le
N+1$ in the expansion of $\kappa^*(\sigma)$, i.e., by $\sigma_{(j)}$ for $j \le
N+1$. Now the residue of the left-hand side of \eqref{M-geodesic}
at $\lambda=-\mu-1-N$ is determined by
the coefficient of $r^{\lambda+\mu+N}$ in the expansion of the integrand. That
coefficient involves the operators $\T_j(\mu)$ with $j \le N$ and the coefficients
of $r^{j}$ for $j \le N$ in the expansion of $\kappa^*(\sigma)/r$. The latter are
determined by $\sigma_{(j)}$ for $j \le N+1$. All other ingredients of the integrand
do not depend on $\sigma$. The proof is complete.
\end{proof}

Since the coefficients $\sigma_{(j)}$ for $j \le n+1$ are determined by the metric $g$ and 
the embedding $\iota$, the residue families $\D_N^{res}(g,\sigma;\lambda)$ for $N \le n$ 
are completely determined by the metric $g$ and the embedding $\iota$, and it is justified 
to use the simplified notation to $\D_N^{res}(g;\lambda)$.

The definition of residue families can be extended to a wider setting. Let
$\sigma \in C^\infty(X)$ be a boundary defining function so that $|d\sigma|_g^2=1$
on the boundary $M$. Then the singular metric $\sigma^{-2}g$ is asymptotically
hyperbolic. This implies the existence of a Poisson operator and the existence of an
eigenfunction $u$ of $\Delta_{\sigma^{-2}g}$ with eigenvalue $-\mu(n-\mu)$ and
arbitrary given boundary value $f\in C^\infty(M)$. The asymptotic expansion of $u$
in terms of adapted coordinates can be stated as an asymptotic expansion in powers
of $\sigma$. For $N \in \N$, we define an operator
$$
   \delta_N(g,\sigma;\mu): C^\infty(X) \to C^\infty(M)
$$
by
$$
  \Res_{\lambda=-\mu-1-N} \left(\int_X \sigma^\lambda u \psi dvol_g \right)
  = \int_M f \delta_N(g,\sigma;\mu) (\psi) dvol_{\iota^*(g)}
$$
and let
$$
   \D_N^{res}(g,\sigma;\lambda) \st N! (2\lambda\!+\!n\!-\!2N\!+\!1)_N \delta_N(g,\sigma;\lambda\!+\!n\!-\!N).
$$
These general residue families are conformally covariant in the following sense.


\begin{theorem}\label{CTL-RF} The residue family $\D_N^{res}(g,\sigma;\lambda)$ is
conformally covariant in the sense that
$$
   \D_N^{res}(\hat{g},\hat{\sigma};\lambda) \circ e^{\lambda \varphi}
   = e^{(\lambda-N) \iota^* (\varphi)} \circ \D_N^{res}(g,\sigma;\lambda)
$$
for all conformal changes $(\hat{g},\hat{\sigma}) = (e^{2\varphi}g,e^{\varphi}\sigma)$, $\varphi \in C^\infty(X)$.
\end{theorem}


\begin{proof} Let $u \in \ker(\Delta_{\sigma^{-2}g}+\mu(n-\mu))$ be an eigenfunction with leading
term $f \in C^\infty(M)$ in its expansion into powers of $\sigma$. We calculate the
residue
$$
   \Res_{\lambda=-\mu-1-N} \left(\int_X \sigma^\lambda u \psi dvol_g \right)
$$
in two ways. On the one hand, it equals
$$
   \int_M f \delta_N(g,\sigma;\mu) (\psi) dvol_{\iota^*(g)}.
$$
On the other hand, for $\Re(\lambda) \gg 0$, we have
$$
   \int_X \sigma^\lambda u \psi dvol_g = \int_X \hat{\sigma}^\lambda u (e^{(-\lambda-n-1)\varphi} \psi) dvol_{\hat{g}}
$$
and the leading term in the $\hat{\sigma}$-expansion of $u$ equals $e^{-\mu
\iota^*(\varphi)}f$. Hence the residue equals
$$
   \int_M e^{-\mu \iota^*(\varphi)} f \delta_N(\hat{g},\hat{\sigma};\mu)(e^{(\mu-n+N)\varphi} \psi)
   dvol_{\iota^*(\hat{g})}.
$$
But $dvol_{\iota^*(\hat{g})} = e^{n \iota^* \varphi}
dvol_{\iota^*(g)}$. Since $f \in C^\infty(M)$ is arbitrary, we find
$$
   \delta_N(g,\sigma;\mu) = e^{(-\mu+n)\iota^*(\varphi)} \circ \delta_N(\hat{g},\hat{\sigma};\mu)
   \circ e^{(\mu+N-n)\varphi}.
$$
This result implies the assertion.
\end{proof}

As a special case, it follows that the compositions of residue families of order
$N \le n$ in the sense of Definition \ref{D-res} with $\eta^*$ (and $\sigma$ being a solution of
the Yamabe problem) are conformally covariant. Whereas the general residue families in
Theorem \ref{CTL-RF} depend on $g$ and $\sigma$, Proposition \ref{sigma-free} shows
that the special cases in Definition \ref{D-res} only depend on $g$ (and the embedding $\iota$).

Although Definition \ref{D-res} breaks the conformal covariance (by omitting $\eta^*$), for
those values of $\lambda$ for which residue families are tangential, the resulting operators on $M$
are still conformally covariant. This observation will play a central role in Section \ref{ECPL}. 
In the following sections, it will always be clear from the context which notion of residue families is 
being used.

\section{Residue families as compositions of $L$-operators}\label{products}

In the present section, we show that the composition of residue families as defined in Definition \ref{D-res}
with $\eta^*$ (defining adapted coordinates) can be identified with compositions of Laplace-Robin operators and
the restriction operator $\iota^*$.

We recall the notation $L_N(g,\sigma;\lambda) \st L(g,\sigma;\lambda\!-\!N\!+\!1) \circ \cdots \circ L(g,\sigma;\lambda)$ 
and set $L_0 = \id$ (see \eqref{LN}).

\begin{theorem}\label{residue-product}
Let $\N \ni N \le n$ and assume that $\sigma$ satisfies the condition $\SCY$. Then
$$
   \D_N^{res}(g,\sigma;\lambda) \circ \eta^* = \iota^* L_N(g,\sigma;\lambda).
$$
\end{theorem}


\begin{proof}
It suffices to prove that
\begin{equation}\label{red-main}
   \delta_N(g,\sigma;\lambda) \circ \eta^* = \frac{1}{N! (2\lambda\!-\!n\!+\!1)_N} \iota^*
   L(g,\sigma;\lambda\!-\!n\!+\!1) \circ \cdots \circ L(g,\sigma;\lambda\!-\!n\!+\!N).
\end{equation}
Let $u$ be an eigenfunction with boundary value $f \in C^\infty(M)$ satisfying
\eqref{eigen} with $\Re(\mu) = n/2$, $\mu \ne n/2$. In the following, it will be
convenient to use the notation
\begin{equation*}
   A((\lambda)_N) \st A(\lambda) \circ A(\lambda+1) \circ \cdots \circ A(\lambda+N-1)
\end{equation*}
for any $\lambda$-dependent family $A(\lambda)$ of operators. Then $L_N(\lambda) =
L((\lambda-N+1)_N)$. On the one hand, \eqref{delta-def} states that
\begin{equation}\label{eq:h1}
   \Res_{\lambda=-\mu-1-N} \left(\int_{X} \sigma^{\lambda} u \psi dvol_g \right)
   = \int_M f \delta_N(g,\sigma;\mu)(\eta^*(\psi)) dvol_h.
\end{equation}
Now we first assume that $\sigma$ satisfies the stronger assumption $\SC(g,\sigma)=1$. We apply Corollary \ref{main-c-c} to calculate
\begin{align}
   - L(g,\sigma;\lambda+1) (\sigma^{\lambda+1} u)
   & = \sigma^\lambda \left(\Delta_{\sigma^{-2}g} - (\lambda\!+\!1)(n\!+\!\lambda\!+\!1) \id \right) u \\
   & = \sigma^\lambda (-\mu(n\!-\!\mu) - (\lambda\!+\!1)(n\!+\!\lambda\!+\!1)) u \\
   & = -(\lambda\!+\!\mu\!+\!1)(\lambda\!-\!\mu\!+\!n\!+\!1) \sigma^\lambda u.
\end{align}
We regard this relation as a Bernstein-Sato-type functional equation. Hence for $\Re(\lambda) \notin -\frac{n}{2}-\N$ 
we obtain
\begin{align*}
   \sigma^\lambda u & = \frac{L(g,\sigma;\lambda\!+\!1)(\sigma^{\lambda+1} u)}
   {(\lambda\!+\!\mu\!+\!1)(\lambda\!-\!\mu\!+\!n\!+\!1)} \\
   & = \frac{L(g,\sigma;\lambda\!+\!1)L(g,\sigma;\lambda\!+\!2)(\sigma^{\lambda+2} u)}
   {(\lambda\!+\!\mu\!+\!1)(\lambda\!+\!\mu\!+\!2)(\lambda\!-\!\mu\!+\!n\!+\!1)(\lambda\!-\!\mu\!+\!n\!+\!2)} \\
   & = \dots = \frac{L(g,\sigma;(\lambda\!+\!1)_N) (\sigma^{\lambda+N} u)} {(\lambda\!+\!\mu\!+\!1)_N
   (\lambda\!-\!\mu\!+\!n\!+\!1)_N}.
\end{align*}
It follows that the integral
$$
   \lambda \mapsto \int_X \sigma^\lambda u \psi dvol_g, \; \Re(\lambda) \gg 0
$$
admits a meromorphic continuation to $\C$ with simple
poles in the set
$$
   -\mu-1-\N_0 \cup \mu-n-1-\N_0.
$$
More precisely, we get
\begin{equation}\label{int-N}
   \int_X \sigma^\lambda u \psi dvol_g = \frac{1}{(\lambda\!+\!\mu\!+\!1)_N (\lambda\!-\!\mu\!+\!n\!+\!1)_N}
   \int_X  L(g,\sigma;(\lambda\!+\!1)_N)(\sigma^{\lambda+N} u) \psi dvol_g
\end{equation}
for $\Re(\lambda) > - \frac{n}{2}-1$ and $N \ge 1$. In the following, it will be convenient to choose $\lambda$ so
 that $\Re(\lambda) > - \frac{n}{2}+1$. Now we note that a function in $C^\infty(X^\circ)$ with an asymptotic
expansion of the form $\sum_{j\ge 0} \sigma^{\nu + j} a_j$ with $\Re(\nu) > 2$ and $a_j \in C^\infty(M)$
satisfies the assumptions in Proposition \ref{adjoint-gen}. Thus, by a repeated application of Proposition
\ref{adjoint-gen}, the right-hand side of \eqref{int-N} equals
$$
   \frac{ 1}{(\lambda\!+\!\mu\!+\!1)_N (\lambda\!-\!\mu\!+\!n\!+\!1)_N} \int_X \sigma^{\lambda+N} u
   L(g,\sigma;(-\lambda\!-\!n\!-\!N)_N)(\psi) dvol_g.
$$
By the assumptions, the zeros of the product
$
  \mu \mapsto  (\lambda\!+\!\mu\!+\!1)_N (\lambda\!-\!\mu\!+\!n\!+\!1)_N
$
are simple for $\Re(\lambda) > - \frac{n}{2}+1$. Thus, using the residue formula
\begin{equation}\label{Delta0}
   \Res_{\lambda=-\mu-1} \left(\int_{X} \sigma^\lambda u \psi dvol_g \right)
   = \int_M f \iota^*(\psi) dvol_h,
\end{equation}
we find
\begin{align}\label{residue-2}
   & \Res_{\lambda=-\mu-1-N} \left(\int_{X} \sigma^\lambda u \psi dvol_g \right) \notag \\
   & = \frac{(-1)^N}{(-N)_N (2\mu\!-\!n\!+\!1)_N} \int_M f \iota^* L(g,\sigma;(\mu\!-\!n\!+\!1)_N)(\psi) dvol_h \notag \\
   & = \frac{1}{N!(2\mu\!-\!n\!+\!1)_N} \int_M f \iota^* L_N(g,\sigma;\mu\!-\!n\!+\!N)(\psi) dvol_h
\end{align}
for $N \ge 1$. Comparing this result with \eqref{eq:h1}, completes the proof of
\eqref{red-main} for $\Re(\mu) = n/2$, $\mu \ne n/2$. The assertion then follows by
meromorphic continuation. If $\sigma$ satisfies only the assumption $\SCY$, analogous
arguments show that the right-hand side of \eqref{int-N} contains an additional integral
$$
   \int_X u \psi R_{n+1} \sigma^{\lambda+n+1} dvol_g.
$$
Since $R_{n+1}$ is smooth up to the boundary and $N \le n$, this integral is regular
at $\lambda=-\mu-1-N$, i.e., does not contribute to the residue. The proof is
complete.
\end{proof}

\begin{cor}\label{factor}
Let $N \in \N$ with $N \le n$ and assume that $\sigma$ satisfies $\SCY$. Then
$$
   \D_N^{res}(g,\sigma;\lambda) = \D_{N-1}^{res}(g,\sigma;\lambda-1) \circ L(g,\sigma;\lambda).
$$
\end{cor}

\begin{remark}\label{trans2}
Theorem \ref{residue-product} identifies the composition of residue families with
$\eta^*$ with compositions of $L$-operators if $\sigma$ satisfies $\SCY$. By Theorem
\ref{D-res-ex}, residue families are linear combinations of compositions of
tangential operators and iterated normal derivatives $\iota^* \partial^k_s$. We may use 
formula \eqref{translate} to write their composition with $\eta^*$ in
terms of iterated gradients $\nabla_\NV^k$. This yields a formula for the composition
of residue families with $\eta^*$ in terms of iterated gradients and tangential
operators.
\end{remark}

For closed $M$, Theorem \ref{residue-product} implies formulas for integrated
renormalized volume coefficients in terms of compositions of Laplace-Robin
operators. First, we observe that, for $N \le n-1$, Theorem \ref{D-res-ex} implies
\begin{align*}
   \int_M \D_N^{res}(g,\sigma;-n\!+\!N)(1) dvol_h
   & = N! (-n\!+\!1)_N \sum_{j=0}^{N} \int_M \T_{N-j}^*(g,\sigma;0) (v_{j}) dvol_h \\
   & = (-1)^N \frac{(n\!-\!1)! N!}{(n\!-\!1\!-\!N)!} \sum_{j=0}^{N} \int_M v_{j} \T_{N-j}(g,\sigma;0)(1) dvol_h \\
   & = (-1)^N \frac{(n\!-\!1)! N!}{(n\!-\!1\!-\!N)!} \int_M v_{N} dvol_h.
\end{align*}
In the last equality, we used the fact that all coefficients except the leading one
in the expansion of the harmonic function $u=1$ vanish. Combining this identity with
Theorem \ref{residue-product} we obtain

\begin{cor}\label{RVC-subcritical}
Let $\N \ni N \le n-1$ and assume that $\sigma$ satisfies the condition $\SCY$. Then
\begin{equation}\label{HF-volume}
   \int_{M} v_N dvol_h
   = (-1)^N \frac{(n\!-\!1\!-\!N)!}{ (n-1)!N!} \int_{M} \iota^* L_{N}(g,\sigma;-n\!+\!N)(1) dvol_h.
\end{equation}
\end{cor}

We shall see later in Theorem \ref{RVE} that this reproves the special case $\tau=1$ of
\cite[Proposition 3.5]{GW-reno}. The critical case $N = n$ will be discussed in Section \ref{Q-curvature}.

These identities for integrated renormalized volume coefficients admit a natural interpretation as special
cases of an interesting identity for distributions. In order to describe that point of view, we smoothly extend
$g$ and $\sigma$ to a sufficiently small neighborhood $\tilde{X}$ of $M$ so that $|\NV| \ne 0$ on $\tilde{X}$ (this
is always possible \cite{PV}). It will be convenient to assume that $\tilde{X} = \eta (I \times M)$ with a sufficiently small interval
$I=(-\varepsilon,\varepsilon)$ around $0$. For any $u \in C^\infty(\R)$, the pull-back $\sigma^*(u) \in C^\infty(X)$
defines a current by
$$
   \left\langle \sigma^*(u),\psi dvol_g \right\rangle = \int_{\tilde{X}} \sigma^*(u) \psi dvol_g, \; \psi \in C_c^\infty(\tilde{X}).
$$
By approximating the delta distribution $\delta$ at $0$ by test functions, we obtain a current $\sigma^*(\delta)$. We recall that
the pull-back $\sigma^*(u)$ of a distribution $u$ on the real line by $\sigma$ exists since the differential of $\sigma$ is
surjective. The pull-back operation itself then is continuous on distributions (and currents) \cite[Theorem 6.1.2]{Ho1}.
Since $|\NV|=1$ on $M$, it holds
\begin{equation}\label{Ho-simple}
    \left\langle \sigma^*(\delta),\psi dvol_g \right\rangle = \int_M \iota^*(\psi) dvol_h
\end{equation}
by an extension of \cite[Theorem 6.1.5]{Ho1}. We also use the notation $\delta_M$ for the latter distribution and call it the
delta distribution of $M$.  \index{$\delta_M$ \quad delta distribution of $M$} We define the action of a differential operator $D$
on currents $u$ on $\tilde{X}$ by
$$
   \langle D (u), \psi dvol_g \rangle = \langle u, D^* (\psi) dvol_g \rangle, \; \psi \in C_c^\infty(\tilde{X}).
$$
Here the formal adjoint $D^*$ of $D$ is determined by the relation
\begin{equation}\label{pair}
   \int_{\tilde{X}} D(\varphi) \psi dvol_g = \int_{\tilde{X}} \varphi D^*(\psi) dvol_g, \; \varphi \in C^\infty(\tilde{X}),
   \psi \in C_c^\infty(\tilde{X}).
\end{equation}

\begin{theorem}\label{dist-volume} Assume that $N \le n-1$ and assume that $\sigma$ satisfies $\SCY$. Then
\begin{equation}\label{Shift-delta}
   L(g,\sigma;-N) \circ \cdots \circ L(g,\sigma;-1) (\sigma^*(\delta)) = a_N\mathfrak{X}^N (\sigma^*(\delta)),
\end{equation}
where $a_N = (n\!-\!1)!/(n\!-\!1\!-\!N)!$ and $\mathfrak{X} = \NV / |\NV|^2$ is defined in Section \ref{AC-RVC}.
\end{theorem}

\begin{proof} Corollary \ref{L-adjoint} implies that $L_N(\lambda)$ acts on $\sigma^*(\delta)$ by
\begin{equation}\label{dual-1}
   \langle L_N(\lambda)(\sigma^*(\delta)),\psi dvol_g \rangle
   = \langle \sigma^*(\delta), L_N(-n\!-\!\lambda\!+\!N\!-\!1) (\psi) dvol_g \rangle.
\end{equation}
Hence \begin{equation*}
   \langle L_N(\lambda)(\sigma^*(\delta)), \psi dvol_g \rangle = \int_M \iota^* L_N(-n\!-\!\lambda\!+\!N\!-\!1) (\psi) dvol_h.
\end{equation*}
Thus, Theorem \ref{residue-product} yields
$$
   \langle L_{N}(\lambda)(\sigma^*(\delta)), \psi dvol_g \rangle
   = \int_M \D_{N}^{res}(-\lambda\!-\!n+\!N\!-\!1)(\eta^*(\psi)) dvol_h.
$$
Note that this formula implies that $\langle L_{N}(\lambda)(\sigma^*(\delta)), \psi dvol_g \rangle$
only depends on the first $N$ terms in the expansion of $\sigma$. Now, by Theorem
\ref{D-res-ex}, the latter integral equals
\begin{align*}
   & N!(-2\lambda\!-\!n\!-\!1)_N \\
   & \times \sum_{j=0}^{N} \frac{1}{(N\!-\!j)!}
   \int_M [\T_{j}^*(-\lambda\!-\!1) \circ v_0 + \cdots + \T_0^*(-\lambda\!-\!1) \circ v_j ]
   (\iota^* \partial_s^{N-j}(\eta^*(\psi))) dvol_h.
\end{align*}
Hence using partial integration, we obtain
\begin{align*}
   & \langle L_{N}(\lambda)(\sigma^*(\delta)), \psi dvol_g \rangle = N!(-2\lambda\!-\!n\!-\!1)_N \\
   & \times \sum_{j=0}^{N} \frac{1}{(N\!-\!j)!} \int_M [\T_{j}(-\lambda\!-\!1)(1) + \cdots + \T_0(-\lambda\!-\!1)(1) v_{j}]
   \iota^* \partial_s^{N-j}(\eta^*(\psi)) dvol_h.
\end{align*}
Now let $\lambda = -1$. Since $\T_{j}(0)(1) = 0$ for $j \ge 1$ and $\T_0 = \id$, it follows that
\begin{align}\label{int-L}
   \langle L_N(-1)(\sigma^*(\delta)), \psi dvol_g \rangle
   & = N!(-n\!+\!1)_N \sum_{j=0}^{N} \frac{1}{(N\!-\!j)!} \int_M v_{j} \iota^* \partial_s^{N-j}(\eta^*(\psi)) dvol_h \notag \\
   & = (-n\!+\!1)_N \int_M \iota^* \partial_s^{N} (v \eta^*(\psi)) dvol_h.
\end{align}
On the other hand, partial integration shows
\begin{align*}
   \langle \mathfrak{X}^N(u),\psi dvol_g \rangle  & = \int_{\tilde{X}} \mathfrak{X}^N(u) \psi dvol_g \\
   & = \int_{I \times M} \eta^* \mathfrak{X}^N(u) \eta^*(\psi) v ds dvol_h \\
   & = \int_{I \times M} \partial_s^N \eta^*(u) \eta^*(\psi) v ds dvol_h
   & \mbox{(by \eqref{intertwine})} \\
   & = (-1)^N \int_{I \times M} \eta^*(u) v^{-1} \partial_s ^N(\eta^*(\psi) v) v ds dvol_h \\
   & = (-1)^N \int_{\tilde{X}} u \eta_*(v^{-1} \partial_s^N (v \eta^*(\psi)) dvol_g
\end{align*}
for $u \in C^\infty(\tilde{X})$ and $\psi \in C^\infty_c (\tilde{X})$. Hence $(\mathfrak{X}^N)^* (\psi)
= (-1)^N \eta_*(v^{-1} \partial_s^N (v \eta^*(\psi))$ and
\begin{equation}\label{fund-dist}
   \langle \mathfrak{X}^N(\sigma^*(\delta)),\psi dvol_g \rangle
   = (-1)^N \int_M \iota^* \partial_s^N (v \eta^*(\psi)) dvol_h.
\end{equation}
The proof is complete.
\end{proof}

Note that, for $\psi = 1$ near $M$, the arguments in the above proof show that
\begin{align*}
   \int_M \iota^* L_N(-n\!+\!N) (1) dvol_h & = \frac{(n\!-\!1)!}{(n\!-\!1\!-\!N)!} \langle \mathfrak{X}^N
   (\sigma^*(\delta)), dvol_g \rangle \\ & = (-1)^N \frac{(n\!-\!1)!}{(n\!-\!1\!-\!N)!}\int_M \iota^* \partial_s^N (v) dvol_h
\end{align*}
for $N \le n-1$. This proves that the identity \eqref{HF-volume} is a special case of Theorem \ref{dist-volume}.

\begin{remark}\label{dist-form}
Theorem \ref{dist-volume} results from our attempt to understand the distributional formula
in \cite[Proposition 3.2]{GW-reno}. In \cite{GW-reno}, this distributional identity is the key to prove
formulas like \eqref{HF-volume}. Here we follow a reverse logic.
\end{remark}

The following shift-property of residue families either follows by combining Theorem
\ref{residue-product} with Corollary \ref{comm} or directly from the residue
definition of residue families.

\begin{lem}\label{D-reduction}
Let $1 \le N \le n$ and assume that $\sigma$ satisfies the condition $\SCY$. Then
$$
   \D_N^{res}(g,\sigma;\lambda) \circ s = N (2\lambda\!+\!n\!-\!N) \D_{N-1}^{res}(g,\sigma;\lambda\!-\!1).
$$
\end{lem}

\begin{proof}
We use the residue definition of residue families (Definition \ref{D-res}).
In particular, $u$ is an eigenfunction of the Laplacian of $\sigma^{-2}g$ for the
eigenvalue $-\mu(n-\mu)$. Now the calculation
\begin{align*}
   \int_M f \delta_N(\mu) (\eta^*(\sigma \psi)) dvol_h
   & = \Res_{\lambda=-\mu-1-N} \left( \int_X \sigma^\lambda u (\sigma \psi) dvol_g \right) \\
   & = \Res_{\lambda=-\mu-1-N} \left( \int_X \sigma^{\lambda+1} u \psi dvol_g \right) \\
   & = \Res_{\lambda=-\mu-N} \left( \int_X \sigma^\lambda u \psi dvol_g \right) \\
   & = \int_M f \delta_{N-1}(\mu) (\eta^* (\psi)) dvol_h
\end{align*}
shows that $\delta_N(\mu) \circ \eta^* \circ \sigma = \delta_{N-1}(\mu) \circ
\eta^*$, i.e., $\delta_N(\mu) \circ s = \delta_{N-1}(\mu)$. The claim is a direct
consequence. Finally, we note that the relation $\delta_N(\mu) \circ s =
\delta_{N-1}(\mu)$ also is an immediate consequence of Theorem \ref{D-res-ex}.
\end{proof}

\section{Extrinsic conformal Laplacians}\label{ECPL}

If $\SC(g,\sigma) \ne 0$ and $N \in \N$, the commutator relation \eqref{LN-gen-tilde} in Corollary
\ref{comm} shows that the composition
$$
   \tilde{L}_N(g,\sigma) \st  \tilde{L}_N \left(g,\sigma;\frac{-n\!+\!N}{2}\right) =
   \tilde{L}\left(g,\sigma;\frac{-n\!-\!N}{2}\!+\!1\right) \circ \cdots
   \circ \tilde{L}\left(g,\sigma;\frac{-n\!+\!N}{2}\right)
$$
is a tangential operator, i.e., it holds $\tilde{L}_N \circ \sigma = \sigma \circ \tilde{L}'_N$ for
some operator $\tilde{L}'_N$. Thus, the operator $\tilde{L}_N(g,\sigma)$ on $C^\infty(X)$
induces an operator $\tilde{\PO}_N(g,\sigma)$ on $C^\infty(M)$ according to
\begin{equation}\label{conformal-power}
   \iota^* \tilde{L}_N(g,\sigma) = \tilde{\PO}_N(g,\sigma) \iota^*.
\end{equation}
This observation is a special case of \cite[Theorem 4.1]{GW}. If additionally $\sigma$ satisfies
the condition $\SCY$ and $N \le n$, then the resulting operator on $C^\infty(M)$ will be denoted by
$\PO_N(g,\sigma)$. In the latter case, one may also directly apply \eqref{LN-gen-b}.

Now, following Gover and Waldron, we define

\index{$\PO_N(g,\sigma)$}

\begin{defn}[\bf Extrinsic conformal Laplacians]
Let $\N \ni N \le n$.  Assume that $\sigma$ satisfies the condition $\SCY$. Then the
operators $\PO_N(g,\sigma)$ are called extrinsic conformal Laplacians. The operator $\PO_n(g,\sigma)$
is called the critical extrinsic conformal Laplacian.
\end{defn}

The notion is justified by the fact that for even $N$ these operators generalize the
GJMS-operators $P_{2N}$ which are of the form $\Delta^N + LOT$. More precisely, let
$g_+$ be a Poincar\'e-Einstein metric in normal form relative to $h$. Then
\begin{equation}\label{P-GJMS}
   \PO_{2N}(r^2g_+,r) = (2N\!-\!1)!!^2 P_{2N}(h).
\end{equation}
In the general case, the operator $\PO_{2N}(g,\sigma)$ is of the form
$$
   (2N\!-\!1)!!^2 \Delta^N_M + LOT.
$$
Here the lower order terms depend on the embedding $M \hookrightarrow X$
(Proposition \ref{LT-P-even}). For odd $N$, the leading part of $\PO_N(g)$ is not
given by a power of the Laplacian but involves $\lo$ (Proposition \ref{LT-P-odd}).
Note that $\D_1^{res}(\frac{-n+1}{2})=0$ shows that $\PO_1 = 0$. Note also that the
vanishing
$$
   L(g,\sigma;0)(1) = 0
$$
implies
\begin{equation}\label{crit-van}
   \PO_n(g,\sigma)(1) = 0.
\end{equation}

\begin{theorem}\label{left-factor} Let $\N \ni N \le n$. Assume that $\sigma$ satisfies $\SCY$.
Then
\begin{equation}\label{D-res-power}
   \D_N^{res}\left(g,\sigma;\frac{-n\!+\!N}{2}\right) = \PO_N(g,\sigma) \iota^*.
\end{equation}
More generally, the factorization identities
$$
   \D_N^{res}\left(g,\sigma;\frac{-n\!-\!k}{2}\!+\!N\right) = \PO_k(g,\sigma) \circ
   \D_{N-k}^{res}\left(g,\sigma;\frac{-n\!-\!k}{2}\!+\!N\right)
$$
for $1\le k \le N$ hold true.
\end{theorem}

\begin{proof} Theorem \ref{residue-product} implies
$$
   \D_N^{res}\left(\frac{-n\!-\!k}{2}\!+\!N\right)
   = \iota^* L\left(\frac{-n\!-\!k}{2}\!+\!1\right) \circ \cdots \circ L\left(\frac{-n\!-\!k}{2}\!+\!N\right).
$$
We decompose this product as
$$
   \left( \iota^* L\left(\frac{-n\!-\!k}{2}\!+\!1\right) \circ \cdots \circ
   L\left(\frac{-n\!+\!k}{2}\right) \right) \circ \left(L\left(\frac{-n\!+\!k}{2}+1\right) \circ \cdots \circ
   L\left(\frac{-n\!-\!k}{2}\!+\!N\right) \right).
$$
By the definition of $\PO_k$ and Theorem \ref{residue-product}, this composition equals
$$
   \PO_k \circ \D_{N-k}^{res}\left(\frac{-n\!-\!k}{2}\!+\!N\right).
$$
The proof is complete.
\end{proof}

Note that $\PO_1=0$ implies the vanishing property
$\D_N^{res}\left(\frac{-n-1}{2}\!+\!N\right)=0$. The trivial zero of
$D_N^{res}(\lambda)$ at $\lambda=-\frac{n+1}{2}\!+\!N$ is actually one of the zeros
in the prefactor in the definition \eqref{D-res-def}. It appears here as a trivial
zero since the solution operator $\T_1(\lambda)$ is regular (Lemma \ref{sol-1}).

Theorem \ref{left-factor} extends factorization formulas in the setting of
Poincar\'e-Einstein metrics. In that case, the left factors are GJMS-operators on
the boundary. For details, we refer to \cite{J1, J2}.

Finally, we notice an interesting direct formula for the critical extrinsic conformal Laplacian 
$\PO_n(g,\sigma)$ in terms of the Laplacian of the singular Yamabe metric $\sigma^{-2} g$. 
By composition with $\iota^*$, Corollary \ref{der-c-c} shows that
$$
    \PO_n(g,\sigma) \iota^*(f) = \iota^* \left( \sigma^{-n} \prod_{j=0}^{n-1} 
    \left(\Delta_{\sigma^{-2}g} + (n\!-\!j)j \right) \right) (f)
$$
for any $f \in C^\infty(X)$. We omit the analogous formulas in the subcritical cases.

Next, we provide a spectral theoretical description of the extrinsic conformal Laplacians.

\begin{theorem}\label{power-spec}
Let $N \in \N$ with $2 \le N \le n$. Assume that $\sigma$ satisfies the condition
$\SCY$. Then
$$
   \PO_N = (-1)^{N-1} 2 (N\!-\!1)! N! \Res_{\frac{n-N}{2}}(\T_N^*(\lambda)).
$$
\end{theorem}

\begin{proof}
On the right-hand side of formula \eqref{D-res-sol} in Theorem \ref{D-res-ex}, all
solution operators except $\T_N(\lambda)$ are regular at $\lambda=\frac{n-N}{2}$.
The result follows by combining that observation with the identity
\eqref{D-res-power} in Theorem \ref{left-factor}.
\end{proof}

We use the spectral theoretical interpretation of the operators $\PO_N$ to separate
their leading parts.

\begin{prop}\label{LT-P-even} Let $\N \ni N \le n$. Assume that $\sigma$ satisfied $\SCY$. Then
$$
   \LT( \Res_{\lambda=\frac{n}{2}-N}(\T_{2N}(\lambda))
  = - \frac{1}{2^{2N} (N-1)! N!} \Delta^N_h.
$$
Hence
$$
   \LT(\PO_{2N}) = (2N-1)!!^2 \Delta^N_h.
$$
The remaining terms are of order $\le N-2$.
\end{prop}

As a preparation for the proof, we observe that formula \eqref{Laplace-adapted} and the identity
$$
   \Delta_{\sigma^{-2}g} = \sigma^2 \Delta_g - (n\!-\!1) \sigma \nabla_{\grad(\sigma)}
$$
imply that the Laplacian of the metric $\eta^*(\sigma^{-2}g) = s^{-2} \eta^*(g)$ takes the form
\begin{align}\label{L-op}
   & s^2 a \partial_s^2 + \frac{1}{2} s^2 a \tr (h_s^{-1} h'_s) \partial_s
   + \frac{1}{2} s^2 a'  \partial_s - (n\!-\!1) s a \partial_s - \frac{1}{2} s^2 (d \log a, d \cdot)_{h_s} + s^2 \Delta_{h_s},
\end{align}
where $a = \eta^*(|\NV|^2)$.

\begin{remark}\label{Laplace-PE}
In the special case $|\NV|=1$, the formula \eqref{L-op} simplifies to
$$
   s^2 \partial_s^2 + \frac{1}{2} s^2 \tr (h_s^{-1} h'_s) \partial_s
   - (n\!-\!1) s \partial_s + s^2 \Delta_{h_s}.
$$
In particular, this reproduces the formula for the Laplacian of a Poincar\'e-Einstein metric \cite[(2.18)]{FJO}.
\end{remark}

\begin{remark}\label{CF-proof2}
By combining the formula \eqref{L-op} with \eqref{LR-adapted}, we can give another proof of the conjugation
formula \eqref{CF}. Here the identities \eqref{vol-g} and \eqref{Laplace-s} are useful. We omit the details.
\end{remark}

The solution operators $\T_N(\lambda)$ are determined by the ansatz
$$
   \sum_{N\ge 0} s^{\lambda+N} \T_N(\lambda)(f), \; \T_0 = \id
$$
for an approximate solution $u$ of the equation $-\Delta_{s^{-2} \eta^*(g)} u =
\lambda(n\!-\!\lambda) u$ with boundary value $f$. By formula \eqref{L-op}, this means that the sum
\begin{align}\label{L-A-Sol}
   &  a \sum_{N \ge 0} (\lambda\!+\!N)(\lambda\!+\!N\!-\!1) \T_N(\lambda) s^{\lambda+N} \notag \\
   + & \frac{a}{2} \sum_{N \ge 0} (\lambda\!+\!N) \tr (h_s^{-1} h'_s) \T_N(\lambda) s^{\lambda+N+1}  \notag \\
   + &  \frac{a'}{2} \sum_{N \ge 0}  (\lambda\!+\!N) \T_N(\lambda) s^{\lambda+N+1} \notag \\
   -  & (n\!-\!1) a \sum_{N \ge 0}  (\lambda\!+\!N) \T_N(\lambda) s^{\lambda+N} \notag \\
   -  & \frac{1}{2} \sum_{N \ge 0} (d \log a,d \T_N(\lambda))_{h_s} s^{\lambda+N+2} \notag \\
   + & \sum_{N \ge 0} \Delta_{h_s} \T_N(\lambda) s^{\lambda+N+2}
\end{align}
coincides with the sum
\begin{equation}\label{L-A-Sol-2}
  -\lambda(n\!-\!\lambda) \sum_{N\ge 0} \T_N(\lambda) s^{\lambda+N}.
\end{equation}
In order to compare coefficients of powers of $s$ in \eqref{L-A-Sol} and
\eqref{L-A-Sol-2}, we also insert the expansions of $a$ and of $h_s$. Note that the
equality of the coefficients of $s^\lambda$ in \eqref{L-A-Sol} and \eqref{L-A-Sol-2}
is trivially satisfied using $\iota^*a = 1$. The equality of coefficients of $s^{\lambda+N}$ in
\eqref{L-A-Sol} and \eqref{L-A-Sol-2} yields a recursive formula for
$\T_N(\lambda)$. For the coefficient $\T_1(\lambda)$, we find $\T_1(\lambda) = -H
\id$ (Lemma \ref{sol-1}). An easy induction shows that the order of
$\T_{2N}(\lambda)$ and $\T_{2N+1}(\lambda)$ is $2N$.

Now we give the proof of Proposition \ref{LT-P-even}.

\begin{proof}[Proof of  Proposition \ref{LT-P-even}]
Let $N$ be even. Then \eqref{L-A-Sol} implies that
$$
   N(2\lambda-n+N) \LT(\T_N(\lambda)) + \Delta_h \LT(\T_{N-2}(\lambda)) = 0.
$$
It follows that the leading term of $\T_{2N}(\lambda)$ is given by
\begin{equation}\label{LT-T}
    \prod_{j=1}^N \frac{1}{2j (n-2\lambda-2j)} \Delta_h^N
   = \frac{1}{N! 2^{2N}} \frac{\Gamma(\frac{n}{2}-\lambda-N)}{\Gamma(\frac{n}{2}-\lambda)} \Delta_h^N.
\end{equation}
Note that this observation fits with \cite[Equation (4.7)]{GZ}. Hence
$$
     \LT( \Res_{\lambda=\frac{n}{2}-N}(\T_{2N}(\lambda))  = - \frac{1}{2^{2N} (N-1)! N!} \Delta_h^N.
$$
The second claim follows from that result by combining it with Theorem \ref{LT-P-even}.
\end{proof}

For odd $N$, the operator $\PO_N$ is of order $N-1$ for general metrics. In the
following result, we separate from the residue
$\Res_{\lambda=\frac{n-N}{2}}(\T_N(\lambda))$ and from $\PO_N$ respective
self-adjoint {\em leading terms} $\LT(\cdot)$ so that the remaining terms are of
order $N-3$ for general metrics.

\begin{prop}\label{LT-P-odd-n} Let $3 \le N \in \N$ be odd. Assume that $\sigma$ satisfies $\SCY$. Then
$$
   \LT( \Res_{\lambda=\frac{n-N}{2}}(\T_N(\lambda))
   =  \frac{1}{N (N\!-\!2)!} \sum_{r=0}^{\frac{N-3}{2}} m_N(r) \Delta^r \delta(\lo d)  \Delta^{\frac{N-3}{2}-r}
$$
with
\begin{equation}\label{m-coeff}
   m_N(r) \st \binom{N-1}{r} \prod_{j=1}^{\frac{n-3}{2}} (N-j) \prod_{j=1}^r \frac{1}{(N-2j)}
   \prod_{j=1}^{\frac{n-3}{2}-r} \frac{1}{(N-2j)}.
\end{equation}
Hence
\begin{equation}\label{LT-P-odd}
   \LT(\PO_N) = (2N\!-\!2) (N\!-\!1)!  \sum_{r=0}^{\frac{N-3}{2}} m_N(r) \Delta^r \delta(\lo d) \Delta^{\frac{N-3}{2}-r}.
\end{equation}
\end{prop}

\begin{proof} Let $N$ be odd. Note that $a = 1 + 2 H s + \cdots$ and
$\tr (h_s^{-1} h'_s ) = 2 \tr(L) + \cdots$ by Lemma \ref{rho-01} and \eqref{h-adapted}.
Comparing the coefficients of $s^{\lambda+N}$ in \eqref{L-A-Sol} and \eqref{L-A-Sol-2} yields the relation
\begin{align*}
   & N (2\lambda\!-\!n\!+\!N) \T_N(\lambda) \\
   & + 2 (\lambda\!+\!N\!-\!1) (\lambda\!+\!N\!-\!2) H \T_{N-1}(\lambda) + (\lambda\!+\!N\!-\!1) \tr (L)  \T_{N-1}(\lambda) \\
   & + (\lambda\!+\!N\!-\!1) H  \T_{N-1}(\lambda)  - 2(n\!-\!1)(\lambda\!+\!N\!-\!1)  \T_{N-1}(\lambda) \\
   & - (dH,d\T_{N-3}(\lambda))_h \\
   & + \Delta_h \T_{N-2}(\lambda) + \Delta_h'  \T_{N-3}(\lambda) = 0,
\end{align*}
up to operators of order $\le N-3$. Here $\Delta_{h_s} = \Delta_h + s \Delta_h' + \cdots$. Simplification shows that
\begin{align}\label{RR-odd}
   & N(2\lambda\!-\!n\!+\!N) \T_N(\lambda) + (\lambda\!+\!N\!-\!1)(2\lambda\!+\!2N\!-\!n\!-\!1) H \T_{N-1}(\lambda)
  - (dH,d\T_{N-3}(\lambda))_h \notag \\
   & + \Delta \T_{N-2}(\lambda) + \Delta' \T_{N-3}(\lambda) = 0,
\end{align}
up to operators of order $\le N-3$.\footnote{From now on, $\Delta$ is the Laplacian of $h$.}
The leading terms of the solution operators $\T_{N-1}(\lambda)$ and $\T_{N-3}(\lambda)$
are multiplies of powers of $\Delta$ (see \eqref{LT-T}). Moreover, we recall the variation formula
\begin{equation}\label{Laplace-var-g}
   (d/dt)|_0(\Delta_{g+sh}) = - (\Hess_g (\cdot), h)_g - (\delta^g (h),d \cdot)_g + \frac{1}{2} (d \tr_g(h),d\cdot)_g
\end{equation}
(\cite[Proposition 1.184]{Besse}). Hence     \index{$\Delta'$}
\begin{align}\label{Delta-var}
   \Delta'_h \st (d/dt)|_0(\Delta_{h+2sL}) & = - (\Hess_h(\cdot),2L)_h - 2(\delta^h (L),d \cdot)_h + n (dH,d\cdot)_h \notag \\
   & = - 2\delta^h ( L d) + n (dH,d \cdot)_h \notag \\
   & = - 2 H \Delta_h + (n-2) (dH,d\cdot)_h - 2 \delta^h (\lo d).
\end{align}
In other words, the first variation $\Delta_h'$ of the Laplacian with respect to the variation $h_s$ of $h$
is given by
\begin{equation}\label{var-1}
   -2H \Delta + (n-2) (dH,d\cdot)_h - 2 \delta (\lo d).
\end{equation}
It follows that $N(2\lambda\!-\!n\!+\!N)\T_N(\lambda)$ is a linear combination of terms of the form
\begin{equation}\label{H-terms}
   H \Delta^{\frac{N-1}{2}} \quad \mbox{and} \quad (dH,d\Delta^{\frac{N-3}{2}}),
\end{equation}
terms
\begin{equation}\label{L-terms}
   \Delta^r \delta ( \lo d) \Delta^{\frac{N-3}{2}-r}, \; r=0,\dots,\tfrac{N-3}{2}
\end{equation}
and terms of order $\le N-3$. In order to determine the coefficients of the terms
\eqref{L-terms} in $\T_N(\lambda)$, we let $\mathring{\T}_N(\lambda)$ be the sum of
these contributions to $\T_N(\lambda)$. Then \eqref{RR-odd} implies
\begin{equation}\label{RR-odd-L}
   a_N(\lambda) \mathring{\T}_N(\lambda) + \Delta \mathring{\T}_{N-2}(\lambda)
  -  2 (-1)^{\frac{N-3}{2}} \prod_{j=1}^\frac{N-3}{2} \frac{1}{a_{2j}(\lambda)} \delta (\lo d) \Delta^{\frac{N-3}{2}}
  = 0,
\end{equation}
where $a_N(\lambda) \st N(2\lambda-n+N)$. That recursive relation is solved by
$$
   (-1)^\frac{N-3}{2} a_N(\lambda) \mathring{\T}_N(\lambda) =
   2 \sum_{r=0}^{\frac{N-3}{2}} \prod_{j=1}^{\frac{N-3}{2}-r} \frac{1}{a_{2j}(\lambda)}
   \prod_{j=1}^{r} \frac{1}{a_{N-2j}(\lambda)} \Delta^r \delta (\lo d) \Delta^{\frac{N-3}{2}-r}.
$$
Hence
\begin{align*}
   & N \Res_{\lambda=\frac{n-N}{2}}(\mathring{\T}_N(\lambda)) \\
   & = 2^{-\frac{N-3}{2}} \sum_{r=0}^{\frac{N-3}{2}} \frac{1}{(\frac{N-3}{2}\!-\!r)! r!}
   \prod_{j=1}^{\frac{N-3}{2}-r} \frac{1}{(N-2j)}
   \prod_{j=1}^r \frac{1}{(N-2j)} \Delta^r \delta (\lo d) \Delta^{\frac{N-3}{2}-r}.
\end{align*}
Now the inverse of the coefficient of the term for $r=0$ equals
$$
   2^{\frac{N-3}{2}} \left(\frac{N-3}{2}\right)! \prod_{j=1}^{\frac{N-3}{2}} (N-2j) = (N-2)!.
$$
Therefore, we obtain
$$
   \Res_{\lambda=\frac{n-N}{2}} \left(\T_N(\lambda)\right) = \frac{1}{N (N-2)!} \sum_{r=0}^{\frac{N-3}{2}}
   m_N(r) \Delta^r \delta (\lo d) \Delta^{\frac{N-3}{2}-r},
$$
up to contributions by the terms in \eqref{H-terms} (containing $H$) and lower-order
terms. However, the terms in \eqref{H-terms} do not contribute. This is a consequence of
the conformal covariance of $\PO_N$. The proof is complete.
\end{proof}


The formula  \eqref{LT-P-odd} also makes clear that, if $\lo$ vanishes, the operator
$\PO_N$ is of order $< N-2$. For a discussion of the special cases of $\PO_3$ and
$\PO_5$, we refer to Lemma  \ref{LT-P3} and Lemma \ref{LT-P5}. In particular, the
proof of Lemma \ref{LT-P3} confirms the vanishing of the terms $H \Delta$ and
$(dH,d\cdot)$.

Next, we relate the operators $\PO_N$ to the scattering operator $\Sc(\lambda)$
generalizing a result of \cite{GZ}.

\begin{theorem}\label{LS}
Let $N \in \N$ with $2 \le N \le n$. Assume that $\sigma$ satisfies $\SCY$ and
that $(n/2)^2 - (N/2)^2 \notin \sigma_{pp}$. Then
\begin{equation}\label{P-S}
   \PO_N = (-1)^N 2 (N-1)! N! \Res_{\frac{n-N}{2}}(\Sc(\lambda)).
\end{equation}
\end{theorem}


\begin{proof}
The assumptions guarantee that the scattering operator is well-defined and that
$\lambda(n-\lambda) \notin \sigma_{pp}$ for $\lambda=\frac{n-N}{2}$.
If $\Re(\lambda) < \frac{n}{2}$ so that
$\lambda(n-\lambda) \notin \sigma_{pp}$ and $\lambda \notin \frac{n}{2} - \N$,
the Poisson transform $\Po(\lambda)(f)$ yields an eigenfunction $u$ of the Laplacian
of the metric $\eta^*(\sigma^{-2}g) = s^{-2} \eta^*(g)$  with boundary value
$f \in C^\infty(M)$ and with an asymptotic expansion of the form
\begin{equation}
   \sum_{j\ge 0} s^{\lambda+j} \T_j(\lambda)(f)
  + \sum_{j \ge 0} s^{n-\lambda+j} \T_j(n-\lambda) \Sc(\lambda)(f).
\end{equation}
Although the families $\T_N(\lambda)$ and $\Sc(\lambda)$ have simple poles at
$\lambda=\frac{n-N}{2}$, the Poisson transform $\Po(\lambda)(f)$ is holomorphic at
$\lambda=\frac{n-N}{2}$ \cite[Section 3]{GZ}. That means that
$$
   \frac{\Res_{\frac{n-N}{2}}(\T_N(\lambda)) s^{\lambda+N}}{\lambda - \frac{n-N}{2}}
   + \frac{\Res_{{\frac{n-N}{2}}}(\Sc(\lambda)) s^{n-\lambda}}{\lambda-\frac{n-N}{2}}
$$
is regular at $\lambda = \frac{n-N}{2}$. Hence
$$
   \Res_{\frac{n-N}{2}}(\T_N(\lambda)) + \Res_{{\frac{n-N}{2}}}(\Sc(\lambda)) = 0
$$
and the asymptotic expansion of $u$ involves a $\log$-term
$$
   2 \Res_{\frac{n-N}{2}}(\T_N(\lambda))(f) s^{\frac{n+N}{2}} \log (s).
$$
Now the claim follows from Theorem \ref{power-spec}.
\end{proof}

\begin{remark} In \cite{GZ}, the scattering operator is defined as $\Sc(n-\lambda)$ and the
GJMS-operators $P_{2N}$ have leading part $(-\Delta)^N$. Now
\begin{align*}
   P_{2N}(h) & = \frac{1}{(2N\!-\!1)!!^2} \PO_{2N}(r^2 g_+,r)  & \mbox{(by \eqref{P-GJMS})} \\
   & = \frac{2(2N\!-\!1)! (2N)!} {(2N\!-\!1)!!^2} \Res_{\frac{n}{2}-N}(\Sc(\lambda)) & \mbox{(by Theorem \ref{LS})} \\
   & = 2^{2N} N!(N\!-\!1)! \Res_{\frac{n}{2}-N}(\Sc(\lambda)).
\end{align*}
This shows that Theorem \ref{LS} extends \cite[Theorem 1]{GZ}.
\end{remark}

\begin{cor}\label{sa} Let $\N \ni N \le n$. Assume that $\sigma$ satisfied $\SCY$. Then the operators
$\PO_N(g,\sigma)$ are formally self-adjoint as operators on $C^\infty(M)$ with respect to the scalar product defined by $h$.
\end{cor}

\begin{proof} Lemma \ref{scatt-sd} shows that $\Sc(\lambda)$ is self-adjoint on 
$\R \setminus \left\{\frac{n-N}{2}\,|\, N \in \Z \right\}$. Since $\Sc(\lambda)$ is meromorphic, its 
residue at $\lambda = \frac{n-N}{2}$ is self-adjoint, too. Then Theorem \ref{LS} proves the assertion.
\end{proof}


\section{Extrinsic $Q$-curvatures and renormalized volume coefficients}\label{Q-curvature}

\numberwithin{theorem}{section} \numberwithin{equation}{section}

The zeroth-order terms of the GJMS-operators $P_{2N}$ led to the notion of 
Branson's $Q$-curvature \cite{sharp}. In the present section, we use residue families to
extend the notion of Branson's $Q$-curvatures to the framework of the singular Yamabe 
problem. This extends the discussion of $Q$-curvatures in \cite{J1}. The resulting curvature 
quantities will be called extrinsic $Q$-curvatures.\footnote{In \cite{GW-reno}, the critical extrinsic 
$Q$-curvature is called the {\em extrinsically coupled $Q$-curvature}.} We relate the integrated
critical renormalized volume coefficient $v_n$ to the integrated critical extrinsic $Q$-curvature.
Moreover, we discuss the Hadamard renormalization of the volume of singular 
Yamabe metrics. Here the total critical extrinsic $Q$-curvature plays an important role. 
Although the treatment is inspired by \cite{GW, GW-reno}, our arguments differ and may continue to 
illuminate these topics.

By Theorem \ref{left-factor}, it holds $\PO_n(g,\sigma) \iota^* =
\D_n^{res}(g,\sigma;0)$ if $\sigma$ satisfies $\SCY$. Now, following the philosophy
of \cite{J1}, it is natural to consider the pair $(\D_n^{res}(g,\sigma;0),\dot{\D}_n^{res}(g,\sigma;0)(1))$.

\index{$\QC_n(g,\sigma)$ \quad critical extrinsic $Q$-curvature}

\begin{defn}[\bf Critical extrinsic $Q$-curvature]
Assume that $\sigma$ satisfies $\SCY$. Then the function
\begin{equation}\label{Q-critical}
   \QC_n(g,\sigma) \st - \dot{\D}_n^{res}(g,\sigma;0)(1) \in C^\infty(M)
\end{equation}
is called the critical extrinsic $Q$-curvature of $g$.
\end{defn}

Since $L(g,\sigma;0)(1) = 0$, the identification of residue families with
products of $L$-operators (Theorem \ref{residue-product}) implies that
\begin{equation}\label{QL}
   \QC_n(g,\sigma) = - \iota^* L(g,\sigma;-n\!+\!1) \circ \cdots \circ L(g,\sigma;-1)
   \circ \dot{L}(g,\sigma;0)(1).
\end{equation}
Moreover, the definition of $L$ yields
\begin{equation}\label{b}
   \dot{L}(g,\sigma;0)(1) = (n-1) \rho - \sigma \J.
\end{equation}
Therefore, we obtain
\begin{equation}\label{QL-2}
   \QC_n(g,\sigma) = - \iota^* L(g,\sigma;-n\!+\!1) \circ \cdots \circ L(g,\sigma;-1)
    ((n-1) \rho - \sigma \J).
\end{equation}

Next, we define subcritical versions of $\QC_n$. Let $N < n$. The definition of $L$
implies that
\begin{equation}\label{aa}
   L\left(g,\sigma;\frac{-n+N}{2}\right)(1) = \left(\frac{n-N}{2}\right) (-(N\!-\!1) \rho + \sigma \J).
\end{equation}
Hence the function $\PO_N(g,\sigma)(1)$ is of the form
$$
   \left(\frac{n-N}{2}\right) \QC_N(g,\sigma)
$$
with a scalar curvature quantity $\QC_N(g,\sigma) \in C^\infty(M)$. It follows that
\begin{equation}\label{a}
   \QC_N(g,\sigma) = -\iota^* L\left(g,\sigma;\frac{-n\!-\!N}{2}\!+\!1\right) \circ \cdots
   \circ L\left(g,\sigma;\frac{-n\!+\!N}{2}\!-\!1\right) ((N\!-\!1)\rho - \sigma \J)
\end{equation}
We shall call these quantities {\em subcritical extrinsic $Q$-curvatures}. In terms
of residue families, these definitions are equivalent to the following definition.

\index{$\QC_N(g,\sigma)$ \quad extrinsic $Q$-curvature}

\begin{defn}[\bf Subcritical extrinsic $Q$-curvature]
Let $\N \ni N < n$. Assume that $\sigma$ satisfies $\SCY$.  The functions
$\QC_N(g,\sigma) \in C^\infty(M)$ which are determined by the equation
\begin{equation}\label{Q-sub-residue}
   \D_N^{res}\left(g,\sigma;\frac{-n+N}{2}\right) (1) = \left(\frac{n-N}{2}\right) \QC_N(g,\sigma)
\end{equation}
are called subcritical extrinsic $Q$-curvatures.
\end{defn}

Since the residue families $\D_N^{res}(g,\sigma;\lambda)$ for $N \le n$ are
completely determined by $g$ (and $\iota$) (Proposition \ref{sigma-free}), the quantities
$\QC_N(g,\sigma)$ for $N \le n$ are also completely determined by $g$ (and $\iota$). 
Therefore, we shall also use the notation $\QC_N(g)$. 

The identities \eqref{QL-2} and \eqref{a} show that the critical extrinsic
$Q$-curvature is a limiting case $\lim_{n \to N}$ of the subcritical extrinsic
$Q$-curvatures (continuation in dimension).


The subcritical $Q$-curvature $\QC_N$ is directly linked to the solution operator $\T_N(\lambda)$
through the relation
\begin{equation}\label{QT}
    \QC_N = (-1)^N \QC_N\left(\frac{n-N}{2}\right),
\end{equation}
where the polynomial $\QC_N(\lambda)$ is defined in \eqref{T-1}. In fact, \eqref{T-1} implies
$$
   \Res_{\frac{n-N}{2}}(\T_N)(1) = -\frac{\frac{n-N}{2}}{2 (N-1)! N!} \QC_N\left(\frac{n-N}{2}\right).
$$
On the other hand, Theorem \ref{power-spec} shows that
$$
   \Res_{\frac{n-N}{2}}(\T_N)(1) = - (-1)^{N} \frac{1}{2 (N-1)!N!} \PO_N (1)
   = - (-1)^N \frac{\frac{n-N}{2}}{2 (N-1)!N!} \QC_N.
$$
The identity \eqref{QT}  follows by combining both relations. Continuation in $n$ also gives the relation
\begin{equation}\label{QT-critical}
    \QC_n = (-1)^n \QC_n(0)
\end{equation}
in the critical case.

\begin{remark}\label{Q-GJMS}
The extrinsic $Q$-curvatures $\QC_N$ are related to Branson's $Q$-curvatures
$Q_{2N}$ as follows. We use the convention that $Q_{2N}$ for even $n$ and $2N <n$ is
defined by
$$
   P_{2N} (g) (1) = \left({\frac{n}2}-N\right) (-1)^N Q_{2N}(g),
$$
where $P_{2N} = \Delta^N + LOT$ is the GJMS-operator of order $2N$ with $\Delta$
denoting the non-positive Laplacian. Now let $g_+$ be a Poincar\'e-Einstein metric
in normal form relative to $h$. Then \eqref{P-GJMS} implies
$$
   \QC_{2N}(r^2 g_+,r) = (-1)^N(2N-1)!!^2 Q_{2N}(h).
$$
\end{remark}

\begin{remark}
Formula \eqref{QL} for the critical extrinsic $Q$-curvature $\QC_n$ is closely related to the
definition of the critical extrinsic $Q$-curvature used in \cite{GW-reno}. In fact, under the
assumption $\SC(g,\sigma)=1$, \cite[(3.14)]{GW-reno} defines $\QC_n(g,\sigma)$ by
$$
   (-1)^{n} \iota^* L(g,\sigma;-n+1) \circ \cdots \circ L(g,\sigma;-1) \circ L_{\log}(g,\sigma;1) (\log (1)).
$$
Here $L_{\log}(g,\sigma;\lambda)$ is a version of $L$ which acts on log-densities
by\footnote{For the definition of the notion of log-densities, we refer to \cite{GW}.}
$$
   L_{\log}(g,\sigma;\lambda)(u) \st (n-1) (\nabla_{\grad(\sigma)}(u) + \lambda \rho)
   - \sigma (\Delta_g (u) + \lambda \J).
$$
In particular, we have $L_{\log}(g,\sigma;\lambda)(\log(1)) = \lambda ((n-1) \rho -\sigma \J)$ and hence
$$
   L_{\log}(g,\sigma;1)(\log(1)) = (n-1) \rho - \sigma \J = \dot{L}(g,\sigma;0)(1)
$$
using \eqref{b}.
\end{remark}

Differentiation of the conformal transformation law for the critical residue family
$\D_n^{res}(g,\sigma;\lambda)$ at $\lambda=0$ yields the following result.


\begin{theorem}\label{CTL-Q} Assume that $\sigma$ satisfies $\SCY$. Then
\begin{equation}\label{CTL-Q2}
   e^{n \iota^*(\varphi)} \QC_n(\hat{g},\hat{\sigma}) = \QC_n(g,\sigma) + \PO_n(g,\sigma)(\iota^*(\varphi))
\end{equation}
for all conformal changes $(\hat{g},\hat{\sigma}) = (e^{2\varphi}g,e^\varphi \sigma)$,
$\varphi \in C^\infty(X)$.
\end{theorem}

\begin{proof} Theorem \ref{CTL-RF} implies
$$
   e^{-(\lambda-n)\iota^*(\varphi)} \circ \D_n^{res}(\hat{g},\hat{\sigma};\lambda) =
   \D_n^{res}(g,\sigma;\lambda) \circ e^{-\lambda \varphi}.
$$
By differentiating this identity with respect to $\lambda$ at $\lambda=0$, we obtain
$$
   -e^{n \iota^*(\varphi)} \iota^*(\varphi) \circ \D_n^{res}(\hat{g},\hat{\sigma};0)
   + e^{n \iota^*(\varphi)} \circ \dot{\D}_n^{res}(\hat{g},\hat{\sigma};0) =
   \dot{\D}_n^{res}(g,\sigma;0) - \D_n^{res}(g,\sigma;0) \circ \varphi.
$$
Hence
$$
   -\iota^*(\varphi) \D_n^{res}(g,\sigma;0)(1) + e^{n \iota^*(\varphi)}
   \dot{\D}_n^{res}(\hat{g},\hat{\sigma};0)(1) = \dot{\D}_n^{res}(g,\sigma;0)(1) -
   \D_n^{res}(g,\sigma;0) (\varphi).
$$
Now $\D_n^{res}(g,\sigma;0) =\PO_n(g,\sigma) \iota^*$ and \eqref{crit-van} imply the
assertion.
\end{proof}

For closed $M^n$, integration of\eqref{CTL-Q2} implies
$$
   \int_{M} \QC_n(\hat{g},\hat{\sigma}) dvol_{\hat{h}}
   = \int_{M} \QC_n(g,\sigma) dvol_h + \int_{M} \PO_n(g,\sigma)(\iota^*(\varphi)) dvol_h.
$$
Later we shall prove that $\PO_n$ is self-adjoint. Hence the second integral on the
right-hand side equals $\int_M \PO_n (1) \iota^*(\varphi) dvol_h$. By
\eqref{crit-van}, this integral vanishes. This shows

\begin{cor}\label{CI-TQ}
Let $M^n$ be closed. Assume that $\sigma$ satisfies $\SCY$. Then the integral
$$
   \int_M \QC_n(g) dvol_h
$$
is conformally invariant as a functional of $g$ and the embedding $M \hookrightarrow X$. 
\end{cor}

An alternative argument proving this invariance will be given in Theorem \ref{LQ}.

A generalization of Corollary \ref{CI-TQ} to general $\sigma$ will be discussed in Lemma \ref{Q-tilde}.

Next, we recall that Theorem \ref{LS} relates the operators $\PO_N$ to the
scattering operator $\Sc(\lambda)$. In particular, it holds
$$
   \PO_n = (-1)^n 2 (n-1)! n! \Res_0(\Sc(\lambda)).
$$
Since $\PO_n(1)=0$, it follows that the function $\Sc(\lambda)(1)$ is regular at
$\lambda=0$. Its value at $\lambda=0$ will be denoted by $\Sc(0)(1)$.

\begin{theorem}\label{Q-S} Assume that $\sigma$ satisfies $\SCY$. Then
$$
    \QC_n = 2 (-1)^n (n-1)! n! \Sc(0)(1).
$$
\end{theorem}

\begin{proof} We consider the coefficient of $s^n$ in the asymptotic expansion \eqref{asymp} of
the eigenfunction $\Po(\lambda)(1)$ for $\lambda=0$. Since $\Po(0)(1)=1$, that
coefficient vanishes. On the other hand, it is given by
$$
   \T_n(0)(1) + \Sc(0)(1);
$$
we recall that the function $\T_n(\lambda)(1)$ is regular at $\lambda=0$. Hence
$\T_n(0)(1) = - \Sc(0)(1)$. But
\begin{equation}\label{T-Q}
   \T_n(0)(1) = -\frac{1}{2 (n-1)! n!} \QC_n(0) = - (-1)^n \frac{1}{2 (n-1)!n!} \QC_n
\end{equation}
using \eqref{T-1} and \eqref{QT-critical}. This implies the assertion.
\end{proof}

\begin{remark}
Theorem \ref{Q-S} extends \cite[Theorem 2]{GZ} for the scattering operator of Poincar\'e-Einstein
metrics. In fact,
\begin{align*}
   Q_n & = (-1)^\frac{n}{2} \frac{1}{(n-1)!!^2} \QC_n & \mbox{(by Remark \ref{Q-GJMS})} \\
   & = (-1)^\frac{n}{2} \frac{2 (n-1)! n!}{(n-1)!!^2} \Sc(0)(1) & \mbox{(by Theorem \ref{Q-S})} \\
   & = (-1)^\frac{n}{2} 2^n \left(\frac{n}{2}\right)! \left(\frac{n}{2}-1\right)! \Sc(0)(1).
\end{align*}
In that case, the crucial relation between the values of $\Sc(\lambda)(1)$
and $\T_n(\lambda)(1)$ at $\lambda=0$ is provided by \cite[Proposition 3.7]{GZ}.
\end{remark}

The following result (\cite[Theorem 3.9]{GW-reno} for $\tau=1$ and $\SC=1$) is an analog of the 
identity \eqref{HF-volume} in the critical case.

\begin{theorem}\label{LQ}
Let $M$ be closed and assume that $\sigma$ satisfies $\SCY$. Let $n\ge 2$ be even. Then it holds the equality
\begin{equation}\label{vQ}
   \LA = \int_M v_n(g) dvol_h = \frac{(-1)^{n} }{(n-1)! n!} \int_M \QC_n(g) dvol_h
\end{equation}
of conformal invariants.
\end{theorem}

The quantity $\LA$ is sometimes referred to as an anomaly. This is motivated by the fact
that, in the Poincar\'e-Einstein case, the function $v_n$ is the infinitesimal conformal anomaly of the
renormalized volume \cite{G-vol}.

We shall give two proofs of that result. The arguments in the first proof will also play a role in Section 
\ref{Q-hol}. The second proof resembles the proof of the analogous result in the subcritical cases.

\begin{proof} We work in adapted coordinates. In particular, the notation will not distinguish
between objects on $X$ and their pull-backs by $\eta$. First, we note that
\begin{equation}\label{help-2}
   \iota^* \partial_s^{n-1} (v \dot{L}(0)(1)) = \iota^* \partial_s^n (v).
\end{equation}
Since $\dot{L}(0)(1) = (n-1) \rho - s \J$ (see \eqref{b}), this local identity is a special case 
of the local identity \eqref{Myst-a}. The current assumption implies $|\NV|^2 = 1 - 2 s \rho + O(s^{n+1})$ 
and the same arguments as in the proof of Theorem \ref{RVE-g} yield the assertion. Now we integrate \eqref{help-2}.
It holds
$$
    \int_M  \iota^* \partial_s^n (v) dvol_h = n! \int_M v_n dvol_h
$$
and the integral of the left-hand side of \eqref{help-2} equals
\begin{align*}
  & (-1)^{n-1} \langle \mathfrak{X}^{n-1} (\sigma^*(\delta)), \dot{L}(0)(1) dvol_g \rangle & \mbox{(by \eqref{fund-dist})} \\
  & = c_n \langle L(-n\!+\!1) \circ \cdots \circ L(-1) (\sigma^*(\delta)), \dot{L}(0)(1) dvol_g \rangle
  & \mbox{(by Theorem \ref{dist-volume})} \\
  & = c_n \langle \sigma^*(\delta), L(-n\!+\!1) \circ \cdots \circ L(-1) \circ \dot{L}(0)(1) dvol_g \rangle &
  \mbox{(by \eqref{dual-1})} \\
  & = c_n \int_M \iota^* L(-n\!+\!1) \circ \cdots \circ L(-1) \circ \dot{L}(0)(1) dvol_h & \mbox{(by \eqref{Ho-simple})} \\
  & = -c_n \int_M \QC_n dvol_h & \mbox{(by \eqref{QL})}
\end{align*}
with $c_n = (-1)^{n-1}/(n\!-\!1)!$. This completes the proof.
\end{proof}

\begin{remark}\label{dot-v}
In the Poincar\'e-Einstein case, it holds $v' = v \dot{L}(0)(1)$. This immediately proves \eqref{help-2}.
\end{remark}

The proof of Theorem \ref{LQ} rests on the local identity \eqref{help-2}. This identity will also play a role 
in Section \ref{Q-hol}. We continue with a

\begin{proof}[Second proof of Theorem \ref{LQ}]
Theorem \ref{D-res-ex} yields the identity
\begin{equation}\label{critical-c}
   \int_M \D_n^{res}(\lambda)(1) dvol_h
   = n! (2\lambda\!-\!n\!+\!1)_n \sum_{j=0}^n \int_M v_j \T_{n-j}(\lambda)(1) dvol_h.
\end{equation}
We split the sum on the right-hand side as
$$
   \sum_{j=1}^{n-1} \int_M v_j \T_{n-j}(\lambda)(1) dvol_h
  + \int_M v_n dvol_h + \int_M \T_n(\lambda)(1) dvol_h.
$$
Now we differentiate \eqref{critical-c} at $\lambda=0$. Since, $\T_j(0)(1)=0$ for $j=1,\dots,n-1$, we find
$$
   \int_M \dot{\D}_n^{res}(0)(1) dvol_h
   = 2 (-1)^{n-1} (n-1)! n! \left( \int_M v_n dvol_h + \int_M \T_n(0)(1) dvol_h \right)
$$
using that $\T_n(\lambda)(1)$ is regular at $\lambda=0$. Now \eqref{Q-critical} and \eqref{T-Q} imply
$$
   -2 \int_M \QC_n dvol_h = 2 (-1)^{n-1} (n-1)!n! \int_M v_n dvol_h.
$$
This implies \eqref{vQ}.
\end{proof}

We finish this section with a discussion of renormalized volumes of singular metrics $\sigma^{-2} g$.
First, we combine the above results to prove

\begin{theorem}\label{RVE}
Let $M$ be closed and assume that $\sigma$ satisfied $\SCY$. Then the volume
$$
    vol_{\sigma^{-2}g}(\{\sigma > \varepsilon \}) = \int_{\sigma > \varepsilon} dvol_{\sigma^{-2}g}
$$
admits the expansion
$$
    \sum_{k=0}^{n-1} \frac{c_k}{n-k} \varepsilon^{-n+k} - \LA \log \varepsilon + V + o(1),  \; \varepsilon \to 0,
$$
where
$$
   c_k = \int_M v_k(g) dvol_h \quad \mbox{and} \quad \LA = \int_M v_n(g) dvol_h.
$$
$V$ is called the renormalized volume. The coefficients in the expansion are natural functionals of the
metric background $g$, which can be written in the form
$$
   c_k = (-1)^k \frac{(n\!-\!1\!-\!k)!}{(n\!-\!1)! k!} \int_M \iota^* L_k(-n\!+\!k) (1) dvol_h
$$
for $k=0,\dots,n-1$ and
\begin{equation}\label{anomaly-SY}
   \LA = (-1)^{n-1} \frac{1}{(n\!-\!1)! n!} \int_M \iota^* \dot{L}_n(0) (1) dvol_h.
\end{equation}
\end{theorem}

\begin{proof} Using $\eta^*(dvol_{\sigma^{-2} g}) = dvol_{s^{-2} \eta^*(g)} = s^{-n-1} v(s) ds dvol_h$,
we obtain the asymptotic expansion
$$
    \int_{\sigma > \varepsilon} dvol_{\sigma^{-2}g}
   = \sum_{k=0}^{n-1} \frac{1}{n-k} \varepsilon^{-n+k} \int_M v_k dvol_h
   - \log \varepsilon \int_M v_n dvol_h + V + o(1).
$$
Now the expressions for the coefficients in terms of Laplace-Robin operators follow from Corollary 
\ref{RVC-subcritical} and Theorem \ref{LQ}.
\end{proof}

The above definition of the renormalized volume is also known as the Hadamard renormalization \cite{Albin}.

In the proof of Theorem \ref{RVE}, we deduced the formulas for the coefficients $c_k$ from the relation between
residue families and iterated Laplace-Robin operators (Theorem \ref{residue-product}). We recall that this relation
requires assuming that $\sigma$ solves the singular Yamabe problem. However, the only consequences of Theorem
\ref{residue-product} which are relevant in this context already follow from a study of the residues of the meromorphic
continuation of the integral
\begin{equation}\label{basic-int}
    \lambda \mapsto \int_X \sigma^\lambda dvol_g, \; \Re(\lambda) > -1.
\end{equation}
Therefore it is of interest to include a discussion of an extension of Theorem \ref{RVE} to general $\sigma$ which only
rests on the study of that integral. First, we note that the coefficients $c_k$ in Theorem \ref{residue-product} are related
to the residues of \eqref{basic-int}: $c_k = \int_M v_k dvol_h$. Moreover, the Hadamard renormalization $V$ of the
volume of $\sigma^{-2}g$ is related to the Riesz renormalization of the volume of $\sigma^{-2} g$ which is defined
as the constant term in the Laurent series of \eqref{basic-int} at $\lambda=-n-1$ \cite{Albin}.

We first observe that Remark \ref{CF-simple} implies
\begin{equation*}
   \SC(g,\sigma)^{-1} \circ L(g,\sigma;\lambda) + \sigma^{\lambda-1} \circ \SC(g,\sigma)^{-1} \circ \Delta_{\sigma^{-2}g}
   \circ \sigma^{-\lambda} = \lambda (n\!+\!\lambda) \sigma^{-1} \id
\end{equation*}
in a sufficiently small neighborhood of the boundary, where $\SC(g,\sigma) \ne 0$. Hence
\begin{equation}\label{BS-shift}
   \SC(g,\sigma)^{-1} L(g,\sigma;\lambda) (\sigma^{\lambda})
  = \lambda (n\!+\!\lambda)  \sigma^{\lambda-1}.
\end{equation}
This is a Bernstein-Sato-type functional equation. Let $\chi \in C_c^\infty(X)$ be a cut-off function of 
the boundary $M$ so that $\SC \ne 0$ on the support of $\chi$. In the following, we shall simplify the notation by
suppressing the arguments $(g,\sigma)$ of $L$ and $\SC$. The second integral on the right-hand side of the 
decomposition
$$
   \int_X \sigma^\lambda dvol_g = \int_X \chi \sigma^\lambda dvol_g + \int_X (1-\chi) \sigma^\lambda dvol_g
$$
is holomorphic on $\C$. Now \eqref{BS-shift} implies
$$
   \int_X \chi \sigma^\lambda dvol_g = \frac{1}{(\lambda\!+\!1) (n\!+\!\lambda\!+\!1)}
   \int_X  \chi \SC^{-1} L(\lambda\!+\!1)(\sigma^{\lambda+1}) dvol_g
$$
for $\Re(\lambda) > -1$. Partial integration using Proposition \ref{adjoint-gen} shows that for $\Re(\lambda) \gg 0$ the 
latter integral equals 
$$
   \int_X \sigma^{\lambda+1} L(-\lambda\!-\!n\!-\!1)(\chi \SC^{-1}) dvol_g.
$$
Next, we note that $v_0 = \iota^* |\NV|^{-1}$ implies the residue formula
\begin{equation}\label{res-formula}
    \Res_{\lambda=-1} \left( \int_X \sigma^\lambda \psi dvol_g \right) = \int_M \iota^* \frac{1}{|\NV|} \psi dvol_h.
\end{equation}
Now combining the above result with the residue formula \eqref{res-formula} yields
$$
   \Res_{\lambda=-2} \left( \int_X \sigma^\lambda dvol_g \right)
   = - \frac{1}{n\!-\!1} \int_M \iota^* \frac{1}{|\NV|} L(-n\!+\!1)(\SC^{-1}) dvol_h.
$$
More generally, arguments as in the proof of Theorem \ref{residue-product} resting on a repeated application 
of the functional equation \eqref{BS-shift} and Proposition \ref{adjoint-gen} show that
\begin{equation}\label{vol-pi}
   \int_X \chi \sigma^\lambda dvol_g = \frac{1}{(\lambda\!+\!1)_k (\lambda\!+\!n\!+\!1)_k}
   \int_X \sigma^{\lambda+k} \tilde{L}_k (-\lambda\!-\!n\!-\!1)(\chi) dvol_g,
\end{equation}
for $\Re(\lambda) \gg 0$, and it follows that
\begin{equation}\label{res-vol-g}
    \Res_{\lambda=-k-1}  \left( \int_X \sigma^\lambda dvol_g \right)
    = \frac{1}{(-n\!+\!1)_k k!} \int_M \iota^* \frac{1}{|\NV|} \tilde{L}_k (-n\!+\!k) (1) dvol_h
\end{equation}
for $k < n$ using the relation $\iota^* \tilde{L}_k(\lambda)(\chi) = \iota^* \tilde{L}_k(\lambda)(1)$. 
The assumption $k < n$ guarantees that the prefactor on the right-hand side of \eqref{vol-pi} is regular 
at $\lambda = -k-1$. Here we use the notation
$$
    \tilde{L}(\lambda) \st L(\lambda) \circ \SC^{-1} \quad \mbox{and} \quad  \tilde{L}_k(\lambda)
    \st \tilde{L}(\lambda\!-\!k\!+\!1) \circ \cdots \circ \tilde{L}(\lambda)
$$
(see \eqref{LR-general} and \eqref{LN-tilde}).

\begin{example}\label{res-example} If $\SC =1$, then $\iota^* \NV=1$ and
\begin{align*}
   \Res_{\lambda=-2} \left( \int_X \sigma^\lambda dvol_g \right)
    & = - \frac{1}{n\!-\!1} \int_M \iota^* L(-n\!+\!1)(1) dvol_h \\
    & = -(n-1) \int_M \iota^* \rho dvol_h = (n-1) \int_M H dvol_h
\end{align*}
using $\iota^* \rho = - H$ (Lemma \ref{rho-01}). This fits with the fact that $v_1 = (n-1)H$ (Example \ref{v1}). 
Similarly, if $\sigma = d_M$ is the distance function of $M$, then $\SC = 1+2\sigma \rho$ implies 
$\iota^* L(-n+1)(\SC^{-1}) = (n-1)^2 \iota^* \rho + 2(n-1) \iota^* \rho = (n-1)(n+1) \iota^* \rho$. Hence
\begin{align*}
   \Res_{\lambda=-2} \left( \int_X \sigma^\lambda dvol_g \right) = \int_M \iota^* \Delta_g(\sigma) dvol_h 
   = n \int_M H dvol_h.
\end{align*}
\end{example}

These results prove the first part of the following theorem.

\begin{theorem}\label{RVE-g} Let $M$ be closed and $\sigma$ be a defining function of $M$. Then the volume
$$
   vol_{\sigma^{-2}g}(\{\sigma > \varepsilon \}) = \int_{\sigma > \varepsilon} dvol_{\sigma^{-2}g}
$$
admits the expansion
$$
   \sum_{k=0}^{n-1} \frac{c_k}{n-k} \varepsilon^{-n+k} - \LA \log \varepsilon + V + o(1),  \; \varepsilon \to 0,
$$
where
$$
   c_k = \int_M v_k dvol_h
   = (-1)^k \frac{(n\!-\!1\!-\!k)!}{(n\!-\!1)! k!} \int_M \iota^* \frac{1}{|\NV|} \tilde{L}_k(-n\!+\!k)(1) dvol_h
$$
for $k=0,\dots,n-1$. Moreover, it holds
\begin{equation}\label{anomaly-g}
    \LA = (-1)^{n-1} \frac{1}{(n\!-\!1)! n!} \int_M \iota^* \frac{1}{|\NV|} \tilde{L}_{n-1}(-1)
    \left( \dot{L}(0)(1) \SC^{-1} + (d \SC^{-1}, d\sigma)_g \right) dvol_h.
\end{equation}
\end{theorem}

Theorem \ref{RVE-g} is due to \cite{GW-reno}. For a discussion of the relations between the current arguments and the 
proofs in these references, we refer to Remark \ref{interpretation-currents}.

\begin{proof} We work in adapted coordinates. The form of the expansion follows as in the proof of Theorem 
\ref{RVE}. Since $c_k = \int_M v_k dvol_h$ is the residue of $\int_X\sigma^\lambda dvol_g$ at $\lambda=-k-1$,
the asserted formula for $c_k$ follows by a calculation of these residues using the functional equation \eqref{BS-shift}. 
It only remains to prove the formula for the anomaly $\LA$. Note that the proof for $c_k$ with $k \le n-1$ does not 
extend to the present case since for $k=n$ the right-hand side of \eqref{vol-pi} has a double pole at $\lambda=-n-1$. 
To bypass this difficulty, we first prove the local relation\footnote{This relation can be interpreted as a local version 
of the global relation in \cite[Lemma 3.8]{GW-reno}.}
\begin{equation}\label{Myst-a}
    \iota^* \partial_s^{n-1} \left(v \left[((n-1) \rho - s\J) \SC^{-1} + |\NV|^2 \partial_s(\SC^{-1})\right] \right) 
    = \iota^* \partial_s^n(v)
\end{equation}
near $M$. Let
$$
   I \st \iota^* \partial_s^{n-1} (v \left[(n-1) \rho - s\J \right] \SC^{-1}).
$$
The definition $\rho = -\frac{1}{n+1}(\Delta (s) + s \J)$ implies that
\begin{equation}\label{help-1}
   I = - \frac{n-1}{n+1} \left(\iota^* \partial_s^{n-1} (v \Delta (s)  \SC^{-1}) 
   + 2n \iota^* \partial_s^{n-2} (v \J \SC^{-1}) \right).
\end{equation}
Now we consider the first term on the right-hand side of \eqref{help-1}.  By (6.17), it holds
$$
   v \Delta(s) = \partial_s (v |\NV|^2).
$$
Hence we obtain
$$
    \iota^* \partial_s^{n-1} (v \Delta (s) \SC^{-1}) 
   = \iota^* \partial_s^n (v |\NV|^2 \SC^{-1}) - \iota^* \partial_s^{n-1} (v |\NV|^2 \partial_s(\SC^{-1})).
$$
But $\SC = |\NV|^2 + 2 s \rho$ implies $|\NV|^2 = \SC -2s \rho$. Hence 
$$
    \iota^* \partial_s^{n-1} (v \Delta (s) \SC^{-1})  = 
    \iota^* \partial_s^n (v) - 2 n  \iota^* \partial_s^{n-1} (v \rho \SC^{-1}) 
    - \iota^* \partial_s^{n-1} (v |\NV|^2 \partial_s(\SC^{-1})).
$$
Now another application of the definition of $\rho$ shows that the middle term on the right-hand side of the last equation
equals
$$
   \frac{2n}{n+1} \iota^* \partial_s^{n-1} ((v \Delta(s) + v s \J)\SC^{-1}).
$$
Simplification of the resulting equation yields 
\begin{align}\label{red-id}
    & -\frac{n-1}{n+1} \iota^* \partial_s^{n-1} (v \Delta (s) \SC^{-1}) \notag \\
    & = \iota^* \partial_s^n (v) + \frac{2n(n-1)}{n+1} \iota^* \partial_s^{n-2} (v \J) 
     - \iota^* \partial_s^{n-1} (v |\NV|^2 \partial_s(\SC^{-1})).
\end{align}
By combining this result with \eqref{help-1}, we have proved
$$
   I =  \iota^* \partial_s^n (v) - \iota^* \partial_s^{n-1} (v |\NV|^2 \partial_s(\SC^{-1})).
$$
This implies \eqref{Myst-a}. Now the relation \eqref{Myst-a} shows that 
\begin{align*}
    \int_M v_n dvol_h = \frac{1}{n!} \int_M \iota^* \partial_s^n(v) dvol_h 
    = \frac{1}{n!} \int_M  \iota^* \partial_s^{n-1} \left( v \E \right) dvol_h,
\end{align*} 
where
\begin{equation}\label{E}
   \E \st ((n-1) \rho - s\J) \SC^{-1} + |\NV|^2 \partial_s(\SC^{-1}).
\end{equation}
Hence                   
$$
   \int_M v_n dvol_h = \frac{(n-1)!}{n!} \Res_{\lambda=-n} \left(\int_X \sigma^\lambda \E dvol_g \right) 
   = \frac{1}{n} \Res_{\lambda=-n} \left(\int_X \sigma^\lambda \E dvol_g \right).
$$
Now we apply the functional equation \eqref{BS-shift}. We assume that $\Re(\lambda) \gg 0$ and let $\chi \in 
C_c^\infty(X)$ be a cut-off function as in the discussion following \eqref{BS-shift}. First, the relation
$$
    \sigma^\lambda = \SC^{-1} \frac{1}{(\lambda\!+\!1)(n\!+\!\lambda\!+\!1)} L(\lambda\!+\!1)(\sigma^{\lambda+1}) 
$$
implies 
\begin{align*}
    \int_X \chi \sigma^\lambda \E dvol_g 
    & = \frac{1}{(\lambda\!+\!1)(n\!+\!\lambda\!+\!1)} \int_X \chi \SC^{-1} L(\lambda\!+1\!)(\sigma^{\lambda+1}) \E dvol_g \\
    & = \frac{1}{(\lambda\!+\!1)(n\!+\!\lambda\!+\!1)} \int_X \sigma^{\lambda+1} 
    \tilde{L}(-n\!-\!\lambda\!-\!1) (\chi \E) dvol_g
\end{align*}
by partial integration using Proposition \ref{adjoint-gen}.  We continue to apply this argument and find  
$$
     \int_X  \chi  \sigma^\lambda \E dvol_g = \frac{1}{(\lambda\!+\!1)_{n-1} (n\!+\!\lambda\!+\!1)_{n-1}} 
     \int_X \sigma^{\lambda+n-1} \tilde{L}_{n-1}(-n\!-\!\lambda\!-\!1)(\chi \E) dvol_g.
$$
Hence the residue formula \eqref{res-formula} yields 
\begin{align*}
   \Res_{\lambda=-n} \left(\int_X \sigma^\lambda \E dvol_g \right) &= \frac{1}{(-n\!+\!1)\dots(-1) (n\!-\!1)!} 
   \int_M \iota^* \tilde{L}_{n-1}(-1)(\chi \E) dvol_h \\
   & = (-1)^{n-1} \frac{1}{(n\!-\!1)! (n\!-\!1)!}  \int_M \iota^* \frac{1}{|\NV|} \tilde{L}_{n-1}(-1)(\E) dvol_h.
\end{align*}
Thus, we find
$$
   \int_M v_n dvol_h = (-1)^{n-1} \frac{1}{n! (n-1)!} \int_M \iota^*  \frac{1}{|\NV|} \tilde{L}_{n-1}(-1)(\E) dvol_h.
$$
This proves the formula for the anomaly using $\dot{L}(0)(1) = (n-1) \rho - s \J$ (see \eqref{b}) and the fact that 
$\grad(\sigma)$ in adapted coordinates equals $|\NV|^2 \partial_s$ (see \eqref{intertwine}).
\end{proof}

\begin{rem}\label{interpretation-currents} The formulas for the coefficients $c_k$ in Theorem \ref{RVE} and 
Theorem \ref{RVE-g} are special cases of \cite[Proposition 3.5]{GW-reno} (for $\tau=1$). Similarly, the formulas 
for $\LA$ are special cases of \cite[Theorem 3.9]{GW-reno}. But note that in \cite{GW-reno} there are no 
local coefficients $v_k$ (defined in adapted coordinates). The present proofs differ from those in \cite{GW-reno}. 
Whereas the above proofs rest on the conjugation formula and a Bernstein-Sato-type argument, the latter 
rest on a certain distributional calculus (see also Remark \ref{dist-form}). Since the expansion in Theorem 
\ref{RVE-g} can be written in the form
$$
    \sum_{k=0}^{n-1} \frac{1}{n-k} \left\langle \sigma^*(\delta^{(k)}),dvol_g \right\rangle \varepsilon^{-n+k}
   - (-1)^{n-1}\left\langle \sigma^*(\delta^{(n)}),dvol_g \right\rangle  \log \varepsilon  + V + o(1)
$$
using the formula
\begin{equation}\label{BR}
     \left\langle \sigma^*(\delta^{(k)}),dvol_g \right\rangle = (-1)^k k! \int_M v_k dvol_h
\end{equation}
for the currents $\sigma^*(\delta^{(k)})$, this proves the equivalence of Theorem \ref{RVE-g} and
\cite[Theorem 3.1]{GW-reno} combined with \cite[Proposition 3.5]{GW-reno} (for $\tau =1$). 
We continue with a proof of the relation \eqref{BR}. First, we extend $g$ and $\sigma$ smoothly
to a sufficiently small neighborhood $\tilde{X}$ of $M$ as in the discussion after Corollary \ref{RVC-subcritical}.
Then $\sigma^*(\delta^{(k)})$ is defined as a continuous functional on smooth volume forms on $\tilde{X}$ with
compact support.\footnote{Since $\sigma^*(\delta)$ has compact support, we can pair it with any smooth volume form.}
Next, we note the commutation rule
\begin{equation}\label{B-1}
    \sigma^*(u') = \mathfrak{X}_\sigma (\sigma^*(u)), \; u \in C^\infty(\R)
\end{equation}
where
$$
    \mathfrak{X}(\tilde{X}) \ni \mathfrak{X}_\sigma \st \NV/|\NV|^2, \; \NV = \grad_g(\sigma).
$$
Hence
\begin{equation}\label{B-2}
   \sigma^*(u^{(k)}) = \mathfrak{X}^k_\sigma (\sigma^*(u))
\end{equation}
for $k \in \N$. Now any $\varphi \in C^\infty(\tilde{X})$ defines a current on $\tilde{X}$ by
$$
    \langle \varphi, \psi dvol_g \rangle = \int_{\tilde{X}} \varphi \psi dvol_g, \; \psi \in C_0^\infty(\tilde{X}).
$$
An extension of \eqref{B-2} to $u=\delta$ yields an analogous relation for currents; it suffices to approximate $\delta$ by
test functions in the weak topology. The relation \eqref{B-2} for $u=\delta$ implies
\begin{align*}
    \left \langle \sigma^*(\delta^{(k)}), \psi dvol_g \right \rangle
    & = \left\langle \mathfrak{X}_\sigma^k (\sigma^*(\delta)), \psi dvol_g \right \rangle \\
    & = \left \langle \sigma^*(\delta), (\mathfrak{X}_\sigma^*)^k (\psi) dvol_g \right \rangle,
\end{align*}
where the adjoint operator $\mathfrak{X}^*_\sigma$ is defined by
$$
   \int_{\tilde{X}} \mathfrak{X}_\sigma(\varphi) \psi dvol_g = \int_{\tilde{X}} \varphi \mathfrak{X}_\sigma^*(\psi) dvol_g,
   \; \psi \in C_c^\infty(\tilde{X}).
$$
Now we calculate (using adapted coordinates and partial integration)
\begin{align*}
    \int_{\tilde{X}} \mathfrak{X}^k_\sigma (\varphi) \psi dvol_g
    & = \int_{I \times M} \eta^*(\mathfrak{X}_\sigma^k(\varphi)) \eta^*(\psi) v ds dvol_h \\
    & = \int_{I \times M} \partial_s^k (\eta^*(\varphi)) \eta^*(\psi) v ds dvol_h \qquad \mbox{(by \eqref{intertwine})} \\
    & = (-1)^k \int_{I \times M} \eta^*(\varphi) v^{-1} \partial_s^k(\eta^*(\psi) v) v ds dvol_h \qquad \mbox{(by partial integration)}
    \\
    & = (-1)^k \int_X \varphi \eta_* (v^{-1} \partial_s^k (v \eta^*(\psi))) dvol_g.
\end{align*}
This shows that $\mathfrak{X}_\sigma^* (\psi) = \eta_* (v^{-1} \partial_s^k (v \eta^*(\psi)))$. Next, we observe that
$$
   \left \langle \sigma^*(\delta), \psi dvol_g \right\rangle = \int_M \iota^* \frac{\psi}{|\NV|} dvol_h.
$$
This formula is an extension of \cite[Theorem 6.1.5]{Ho1} (for a flat background).\footnote{With appropriate interpretation, 
this identity coincides with  \cite[(2.16)]{GW-reno}. The current reference to Hörmander replaces the reference for this result 
which has been used in \cite{GW-reno}. This reference actually only utilizes formal arguments for domains in flat space $\R^n$.} 
By 
$\iota^*(v)  = v_0 = \frac{1}{|\NV|}$, these results imply
\begin{align*}
    \left\langle \sigma^*(\delta^{(k)}),\psi dvol_g \right\rangle & = (-1)^k \int_M \iota^* \partial_s^k (v \eta^*(\psi)) dvol_h.
\end{align*}
In particular, for $\psi = 1$ we find \eqref{BR}.
\end{rem}

By \eqref{anomaly-g}, the anomaly $\LA$ is proportional to 
$$
   \int_M  \iota^*  \frac{1}{\sqrt{\SC}} \tilde{L}_{n-1}(-1) (\E) dvol_h 
$$
with $\E = \dot{L}(0)(1) \SC^{-1} + (d \SC^{-1},d\sigma)_g$ (see \eqref{E}).  In the singular Yamabe case, 
this integral reduces to 
$$
   \int_M \iota^* L_{n-1}(-1) \dot{L}(0)(1) dvol_h = \int_M \iota^* \dot{L}_n(0)(1) dvol_h = - \int_M \QC_n dvol_h.
$$
This motivates Gover and Waldron \cite{GW-reno} to regard the function
\begin{equation}\label{tilde-Qn}
    \tilde{\QC}_n(g,\sigma) \st \iota^* \frac{1}{\sqrt{\SC}} \tilde{L}_{n-1}(g,\sigma;-1) (\E) 
\end{equation}
as a {\em generalized} critical extrinsic $Q$-curvature. We show that this quantity shares basic properties with 
$\QC_n$. In this context, we consider conformal changes $(\hat{g},\hat{\sigma}) = (e^{2\varphi} g, e^\varphi \sigma)$
and we let $L = L(g,\sigma)$ and $\hat{L} = L(\hat{g},\hat{\sigma})$. Similarly $\hat{\cdot}$ also denotes other functionals 
of $(g,\sigma)$ for such conformal changes. Now differentiating the conformal transformation law 
$$
    e^{-(\lambda-1) \varphi} \circ \hat{L}(\lambda) = L(\lambda) \circ e^{-\lambda \varphi}
$$
at $\lambda = 0$, gives
$$
    e^\varphi \dot{\hat{L}}(0)(1) = \dot{L}(0)(1) - L(0)(\varphi)
$$
using $L(0)(1)=0$. Hence
$$
   e^\varphi \hat{\E} = \E - \SC^{-1} L(0)(\varphi) + \sigma (d\SC^{-1},d\varphi),
$$
and the conformal covariance of $\tilde{L}(\lambda)$ implies
\begin{align*}
    e^{n \iota^*(\varphi)} \hat{\tilde{\QC}}_n & = \iota^* \frac{1}{\sqrt{\SC}} e^{n \varphi} \hat{\tilde{L}}_{n-1}(-1) (\hat{\E}) 
    =  \iota^* \frac{1}{\sqrt{\SC}}\tilde{L}_{n-1}(-1) (e^\varphi \hat{\E}) \\
    & = \iota^* \frac{1}{\sqrt{\SC}} \tilde{L}_{n-1}(-1) (\E - \SC^{-1} L(0)(\varphi) + \sigma (d\SC^{-1},d\varphi)) \\
    & = \tilde{\QC}_n - \iota^* \frac{1}{\sqrt{\SC}} \tilde{L}_{n-1}(-1) (\SC^{-1} L(0)(\varphi) - \sigma (d\SC^{-1},d\varphi)).
\end{align*}
This proves the conformal transformation law
\begin{equation}\label{CTL-tilde}
   e^{n \iota^*(\varphi)} \hat{\tilde{\QC}}_n = \tilde{\QC}_n - \iota^* \tilde{\PO}_n(\varphi)
\end{equation}
with
$$
   \tilde{\PO}_n \st \frac{1}{\sqrt{\SC}} \tilde{L}_{n-1} (-1) ( \SC^{-1} L(0)(\cdot) - \sigma (d\SC^{-1},d\cdot)).
$$
Note that the latter operator again is conformally covariant: $e^{n\varphi} \tilde{\PO}_n(\hat{g},\hat{\sigma}) 
= \tilde{\PO}_n(g,\sigma)$. This immediately follows from the conformal covariance of $L$ and the invariance of $\SC$. 
The following result generalizes Corollary \ref{CI-TQ} and clarifies the content of \cite[Proposition 3.11]{GW-reno}. 

\begin{lem}\label{Q-tilde} For closed $M^n$, the integral
$$
    \int_M \tilde{\QC}_n (g,\sigma) dvol_h 
$$
is invariant under conformal changes of the pair $(g,\sigma)$. 
\end{lem}

\begin{proof} We give two proofs. The first proof uses an argument of \cite{Graham-Yamabe}. Since the integral 
of $\tilde{\QC}_n$ is proportional to the anomaly $\LA$ in Theorem \ref{RVE-g}, it suffices to prove that $\LA(\hat{g},\hat{\sigma}) = \LA(g,\sigma)$. 
We recall that $\hat{\sigma}^{-2} \hat{g} = \sigma^{-2} g$ and prove that the expansion of the difference
$$
    \int_{\hat{\sigma} > \varepsilon} dvol_{\sigma^{-2} g} - \int_{\sigma > \varepsilon} dvol_{\sigma^{-2} g}
$$
does not contain a $\log \varepsilon$ term. Note that $\{\hat{\sigma} > \varepsilon \} 
= \{\sigma > \varepsilon e^{-\varphi} \}$. Now, using adapted coordinates, we find that the above difference equals
\begin{align*}
     & \int_{\varepsilon}^{\varepsilon e^{-\varphi}} s^{-n-1} \left( \int_M v(s) dvol_h \right) ds \\
     & = \sum_{k=0}^{n-1} \varepsilon^{-n+k} \int_M \frac{v_k}{n-k} \left(1-e^{(n-k)\varphi}\right) dvol_h 
     - \int_M \varphi v_n dvol_h + o(1).
\end{align*}
This expansion does not contain a $\log \varepsilon$ term. 

The second proof rests on the conformal transformation law \eqref{CTL-tilde}. It shows that it suffices to prove 
that $\int_M \iota^* \tilde{\PO}_n(\varphi) dvol_h = 0$ for all $\varphi \in C^\infty(X)$. Now the 
generalization\footnote{With appropriate interpretations as in Remark \ref{interpretation-currents}, the 
relation \eqref{extension} follows from \cite[Proposition 3.3]{GW-reno}.}
\begin{equation}\label{extension}
    \int_M \iota^* \frac{1}{|\NV|} \tilde{L}_{n-1}(-1)(\psi) dvol_h \sim \int_M \iota^* \partial_s^{n-1}(v \psi) dvol_h
\end{equation}
of 
\begin{align*}
  \langle \sigma^*(\delta), L_{n-1}(-1) (\psi) dvol_g \rangle \sim \int_M \iota^* \partial_s^{n-1} (v \psi) dvol_h
\end{align*}
(see \eqref{int-L}) shows that the assertion is equivalent to
\begin{align}\label{van2}
   \int_M \iota^* \partial_s^{n-1} \left( v \left[\SC^{-1} L(0)(\varphi) - s (d\SC^{-1},d\varphi)_g \right]  \right)  dvol_h= 0.
\end{align}
We recall that $L(0)(\varphi) = (n-1) a \partial_s (\varphi) - s \Delta_g(\varphi)$ with $a = |\NV|^2$ and calculate 
\begin{align*}
    & \iota^* \partial_s^{n-1} \left( v f L(0)(\varphi) \right) \\ 
    & = (n-1)  \iota^* \partial_s^{n-1}  (v f a \varphi') - (n-1)  \iota^* \partial_s^{n-2} (v f \Delta_g (\varphi)) \\
    & = (n-1) \iota^* \partial_s^{n-2} ( v' f a \varphi' + v f' a \varphi' + v f a' \varphi' + v f a \varphi'') \\
    & - (n-1) \iota^* \partial_s^{n-2}\left(v f \left (a \varphi'' + a \frac{1}{2} \tr (h_s^{-1} h_s') \varphi'
    + \frac{1}{2} a' \varphi' - \frac{1}{2} (d \log a,d\varphi)_{h_s} + \Delta_{h_s}(\varphi) \right)\right)
\end{align*}
using \eqref{Laplace-adapted}. Simplification of that result gives
$$
    (n-1) \iota^* \partial_s^{n-2} \left( v a f' \varphi' + v f \frac{1}{2} (d \log a, d \varphi)_{h_s} 
    - v f \Delta_{h_s} (\varphi)\right).
$$
Here we used that $v = a^{-\frac{1}{2}} dvol_{h_s}/dvol_h =  a^{-\frac{1}{2}} \mathring{v}$ implies
$$
    \frac{\mathring{v}'}{\mathring{v}} = \frac{1}{2} \tr (h_s^{-1} h_s')  = \frac{v'}{v} + \frac{1}{2} \frac{a'}{a}
$$
(see \eqref{vol-g}). Thus, we obtain
\begin{align*}
    & \int_M \iota^* \partial_s^{n-1} \left( v f  L(0)(\varphi) \right) dvol_h \\
    & = (n-1)  \iota^* \partial_s^{n-2} \left( \int_M  v \left(a f' \varphi' 
    + f \frac{1}{2} (d \log a, d \varphi)_{h_s} - f \Delta_{h_s}(\varphi)\right) dvol_h \right).
\end{align*}
Now partial integration on $M$ yields 
\begin{align*}
    \int_M v f \Delta_{h_s}(\varphi) dvol_h  &  = \int_M  a^{-\frac{1}{2}} f \Delta_{h_s}(\varphi) dvol_{h_s} 
    = - \int_M (d (a^{-\frac{1}{2}} f, d\varphi)_{h_s} dvol_{h_s} \\
    & = - \int_M v a^{\frac{1}{2}}  (d (a^{-\frac{1}{2}} f), d\varphi)_{h_s} dvol_h \\
   & = - \int_M v (d f,d\varphi)_{h_s} dvol_h + \frac{1}{2} \int_M v f (d \log a,d\varphi)_{h_s} dvol_h
\end{align*}
using $a^{\frac{1}{2}}  (d (a^{-\frac{1}{2}} f), d\varphi)_{h_s} 
= - \frac{1}{2} f (d \log a, d\varphi)_{h_s} + (d f,d\varphi)_{h_s}$. Differentiating 
this relation by $s$ implies
\begin{align*}
     \int_M \iota^* \partial_s^{n-1} \left( v f  L(0)(\varphi) \right) dvol_h 
     & = (n-1) \iota^* \partial_s^{n-2} \left( \int_M v a f' \varphi' + v (df,d\varphi)_{h_s} dvol_h \right) \\ 
     & = (n-1) \iota^* \partial_s^{n-2} \left( \int_M v (df,d\varphi)_g dvol_h \right).
\end{align*}
For $f = \SC^{-1}$, this identity implies the vanishing of \eqref{van2}. 
\end{proof}

Note that similar arguments show that
$$
    L(-n) ( f \sigma^*(\delta^{(n-1)})) = L(0)(f) \sigma^*(\delta^{(n-1)}) 
$$
for $f \in C^\infty(X)$ (see \cite[Proposition 3.4]{GW-reno}). We omit the details. 

\begin{example} A calculation shows that in adapted coordinates
\begin{equation}\label{TQ}
   \tilde{\QC}_2(g,\sigma) = \iota^* (|\NV|^{-5} ( \rho^2 - \rho \SC' - 2 (\SC')^2) + |\NV|^{-3} (\J - \rho' + \SC'')).
\end{equation}
In particular, if $\iota^* \SC'$ and $\iota^* \SC''$ vanish (singular Yamabe case), then $\QC_2 
= \iota^* (\rho^2 - \rho' + \J)$. By $\iota^* (\rho) = -H$ and $\iota^* (\rho') = \Rho_{00} + |\lo|^2$ 
(Lemma \ref{rho-01}), we get $\QC_2 = - \Rho_{00} - |\lo|^2 + \iota^*(\J)$. Now the Gauss equation 
$\iota^* \J = \J^h + \Rho_{00} + \frac{1}{2} |\lo|^2 - H^2$ implies
$$
   \QC_2 = \J^h - \frac{1}{2} |\lo|^2.
$$
This is the known formula for the critical extrinsic $Q$-curvature in dimension $n=2$ (Example \ref{Q2}).

If $\sigma= d_M$ is the distance function, then $|\NV|=1$ and geodesic normal coordinates are adapted coordinates. 
Now $\SC = 1 + 2 s \rho$ implies $\iota^*(\SC') = 2  \iota^*(\rho)$ and $\iota^*(\SC'') = 4 \iota^*(\rho')$. Thus, we get
$$
   \tilde{\QC}_2 = \iota^* (\J + 3 \rho' - 9 \rho^2).
$$
But $\iota^*(\rho) = - \frac{2}{3} H$ and $3 \iota^*(\rho') = - \iota^* \partial_s (\Delta_g (s) + s \J) 
= - \iota^* (2 H' + \J) = \iota^* (\Ric_{00} - \J) + |L|^2 = \Rho_{00} + |L|^2$  (using $2 H' = -|L|^2 - \Ric_{00}$ 
- see \eqref{var-form}) show that $\tilde{\QC}_2 = \Ric_{00} + |\lo|^2 - 2 H^2$. This confirms the formula 
$-2 \LA = \int_M u_2 dvol_h$ (with $u_2$ as in \eqref{u2}).
\end{example}

Finally, we use similar arguments as above to establish analogous formulas for the integrated renormalized
volume coefficients $w_k$, which are defined in terms of geodesic normal coordinates. Let $d_M$ be the 
distance function of $M$.

\begin{theorem}\label{w-holo} 
Let $M$ be closed. Assume that $\sigma$ satisfies $\SCY$. Then
$$
   \int_M w_k dvol_h  = (-1)^k \frac{(n\!-\!1\!-\!k)!}{(n\!-\!1)! k!}
   \int_M \iota^* L_k(g,\sigma;-n\!+\!k) \left( \left(\frac{d_M}{\sigma}\right)^{n-k}\right) dvol_h
$$
for $0 \le k \le n-1$ and
$$
    \int_M w_n dvol_h = \frac{(-1)^{n-1} }{(n\!-\!1)! n!} \int_M  \iota^* \dot{L}_n(g,\sigma;0)(1) dvol_h.
$$
\end{theorem}

\begin{proof} We compare two different calculations of the residues of the family
$$
   \lambda \mapsto \int_X \sigma^\lambda \psi dvol_g, \; \Re(\lambda) > -1
$$
for appropriate test functions $\psi$. On the one hand, arguments as above using the functional equation 
\eqref{BS-shift} prove the relation 
\begin{equation}\label{res-cal-1}
   \Res_{\lambda=-k-1} \left( \int_X \sigma^\lambda \psi dvol_g \right)
   = \frac{1}{k!(-n\!+\!1)_k} \int_M \iota^* L_k(g,\sigma;-n\!+\!k)(\psi) dvol_h.
\end{equation}
On the other hand, we use geodesic normal coordinates and asymptotic expansions of the resulting
integrand. By $\kappa^*(d_M)=r$ and
$$
    \kappa^* (dvol_g) = \left( \frac{\kappa^*(\sigma)}{r}\right)^{n+1} w(r) dr dvol_h
$$
(see \eqref{def-w}), we obtain
\begin{align*}
  \int_X \sigma^\lambda \psi dvol_g & = \int_X d_M^\lambda \left(\frac{\sigma}{d_M}\right)^\lambda \psi dvol_g \\
  & = \int_{[0,\varepsilon)} \int_M r^\lambda \left( \frac{\kappa^*(\sigma)}{r}\right)^{\lambda+n+1} \kappa^*(\psi)
  w(r) dr dvol_h
\end{align*}
for test functions $\psi$ with appropriate support. The classical formula \eqref{Gelfand}
implies that
$$
  \Res_{\lambda=-k-1} \left( \int_X \sigma^\lambda \psi dvol_g \right) = \frac{1}{k!} \int_M
  \left( \left(\frac{\kappa^*(\sigma)}{r}\right)^{-k+n} \kappa^*(\psi) w(r)\right)^{(k)}(0) dvol_h.
$$
For the test function $\psi_k = (d_M/\sigma)^{n-k} \chi$ (with an appropriate
cut-off function $\chi$), the latter result yields
\begin{equation}\label{res-cal-2}
    \Res_{\lambda=-k-1} \left( \int_X \sigma^\lambda \psi_k dvol_g \right) = \int_M w_k dvol_h.
\end{equation}
Now comparing \eqref{res-cal-1} and \eqref{res-cal-2} proves the first assertion. The assertion in the critical case 
follows from $\int_M w_n dvol_h = \int_M v_n dvol_h$ (see \eqref{w=v}) and Theorem \ref{RVE}.
\end{proof}

The first part of Theorem \ref{w-holo} is a special case of \cite[Proposition 3.5]{GW-reno}.

\section{Holographic formul{\ae} for extrinsic $Q$-curvatures}\label{Q-hol}

We work in adapted coordinates. In particular, $\J=\J^g$ is identified with $\eta^*(\J)$.

The following result is a local version of Theorem \ref{LQ}.

\begin{theorem}\label{Q-holo-crit}
Let $n$ be even and assume that $\sigma$ satisfies $\SCY$. Then
\begin{equation}\label{Q-holo-form}
   \QC_n (g) = (-1)^n (n\!-\!1)!^2 \sum_{k=0}^{n-1} \frac{1}{n\!-\!1\!-\!2k}
   \T_k^*(g;0) \left( (n\!-\!1)(n\!-\!k) v_{n-k} + 2k (v \J)_{n-k-2} \right).
\end{equation}
For $k=n-1$, the second term on the right-hand side is defined as $0$.
\end{theorem}

The assumption that $n$ is even guarantees that the fractions on the right-hand side
are well-defined. For odd $n$, we refer to Conjecture \ref{van}.

\begin{remark}
Assume that $g_+ = r^{-2}(dr^2 + h_r)$ is an even Poincar\'e-Einstein metric in normal
form relative to $h$ and let $g = r^2 g_+ = dr^2 + h_r$. Then
$$
   \J^g= - \frac{1}{2r} \tr (h_r^{-1} h'_r) = - \frac{1}{r} \frac{v'}{v}(r)
$$
by \cite[(6.11.8)]{J1} or Lemma \ref{rec}, and the second term on the right-hand side
of \eqref{Q-holo-form} yields $-2k (n-k) v_{n-k}$. Therefore, the sum simplifies to
$$
   \sum_{k=0}^{n-1} (n-k) \T_k^*(0) (v_{n-k})
$$
and the assertion reduces to the main result of \cite{GJ} by noting that it only
contains contributions for even $k$.
\end{remark}

Note that \eqref{Q-holo-form} can be written in the form
$$
  \QC_n = (-1)^{n} n! (n\!-\!1)! v_n + \sum_{k=1}^{n-1} \T_k^*(0)(\cdots).
$$
Since $\T_k(0)(1) = 0$ for $k \ge 1$, integration of this identity (for closed $M$)
reproduces Theorem \ref{LQ}, and the holographic formula provides a formula for the
lower-order terms.

There is a generalization of Theorem \ref{Q-holo-crit} to subcritical $Q$-curvatures.

\begin{theorem} \label{Q-gen-holo}
Let $n$ be even and $\N \ni N < n$. Assume that $\sigma$ satisfies $\SCY$. Then
\begin{align}\label{Q-holo-form-gen}
   \QC_N(g) & = (-1)^{N} (N\!-\!1)!^2 \sum_{k=0}^{N-1} \frac{1}{(2N\!-\!n\!-\!1\!-\!2k)} \notag \\
   & \times \T_k^*\left(g;\frac{n-N}{2}\right) \left( (N\!-\!1)(N\!-\!k) v_{N-k} + (n\!-\!N\!+\!2k) (v \J)_{N-k-2} \right).
\end{align}
For $k=N-1$, the second term on the right-hand side is defined as $0$.
\end{theorem}

Again, the assumption that $n$ is even guarantees that the fractions on the right-hand side are well-defined.
For odd $n$, we refer to Conjecture \ref{van}.

In the even Poincar\'e-Einstein case, the formula is non-trivial only for even $N$. In that case,
the solution operators in \eqref{Q-holo-form-gen} act on
$$
   (N\!-\!1)(N\!-\!k) v_{N-k} - (n\!-\!N\!+\!2k) (N\!-\!k) v_{N-k}
  = (2N\!-\!n\!-\!1\!-\!2k)(N\!-\!k) v_{N-k}
$$
and the sum simplifies to
$$
   \sum_{k=0}^{N-1} \T_k^*\left(g;\frac{n-N}{2}\right) (N-k) v_{N-k}.
$$
Here only terms with even $k$ contribute. Thus the formula reduces to the main result of \cite{J3}.

The following conjecture implies that for odd $n$ the respective terms with the singular
fractions in the sums \eqref{Q-holo-form} and \eqref{Q-holo-form-gen} do not contribute.

\begin{conj}\label{van} For odd $n\ge 3$, it holds
\begin{equation}\label{van-id}
   \tfrac{n+1}{2} v_{\frac{n+1}{2}} + (v \J)_{\frac{n-3}{2}} = 0.
\end{equation}
Moreover, the holographic formulas \eqref{Q-holo-form} and \eqref{Q-holo-form-gen}
are valid also for odd $n$.
\end{conj}

\begin{rem}\label{co-ex} For $n=3$ and $n=5$, the respective relations \eqref{van-id} read 
\begin{align*}
    2v_2 + \J_0 & = 0,  \\
    3 v_3 + \J'_0 + v_1 \J_0 & = 0.
\end{align*}
For details, we refer to the discussion in Examples \ref{v1}--\ref{v3-ex}. These identities are 
consequences of the relation \eqref{bL}. More generally, the relation \eqref{bL} implies that any 
coefficient $v_N$ can be written as a linear combination of products of derivatives of $\rho$ and 
$\J$ at $s=0$. The calculations of the resulting formulas for odd $n \le 11$ confirm the relation 
\eqref{van-id} in these special cases. For the discussion of the holographic formula for $\QC_3$ in 
general dimensions, we refer to Example \ref{Q3}.
\end{rem}

Now we present the proof of Theorem \ref{Q-holo-crit}.

\begin{proof} We evaluate the quantity $\dot{\D}_n^{res}(0)(1)$ using the factorization formula
$$
   \D_n^{res}(\lambda) = \D_{n-1}^{res}(\lambda-1) L(\lambda)
$$
(Corollary \ref{factor}). In view of $L(0)(1)=0$, it follows that
\begin{equation}\label{start-crit}
   \dot{\D}_n^{res}(0)(1) = \D_{n-1}^{res}(-1) \dot{L}(0)(1).
\end{equation}
Now we apply the representation formula
$$
   \D_{n-1}^{res}(-1) = (-1)^{n-1} (n\!-\!1)!^2 \sum_{j=0}^{n-1} \frac{1}{(n\!-\!1\!-\!j)!}
   \left[ \T_j^*(0) v_0 + \cdots + \T_0^*(0) v_j\right] \iota^* \partial_s^{n-1-j}
$$
(Theorem \ref{D-res-ex}). We find
$$
   \dot{\D}_n^{res}(0)(1) = (-1)^{n-1} (n-1)!^2 \sum_{j=0}^{n-1} \frac{1}{(n\!-\!1\!-\!j)!}
   \left[ \T_j^*(0) v_0 + \cdots + \T_0^*(0) v_j\right] \iota^* \partial_s^{n-1-j} (\dot{L}(0)(1)).
$$
In the latter sum, the operator $\T_k^*(0)$ acts on the sum
$$
   \left(v_0 \frac{1}{(n\!-\!1\!-\!k)!} \iota^* \partial_s^{n-1-k} + \cdots + v_{n-1-k}
   \iota^* \right)(\dot{L}(0)(1)).
$$
But this quantity equals the $(n\!-\!1\!-\!k)$'th
Taylor coefficient $(v \dot{L}(0)(1))_{n-1-k}$ of $v \dot{L}(0)(1)$. Now the
identity
\begin{equation}\label{reduction}
   \iota^* \partial_s^k (v \dot{L}(0)(1)) = -\frac{n\!-\!1}{(n\!-\!1\!-\!2k)} \iota^*
   \partial_s^{k+1}(v) - \frac{2k(n\!-\!1\!-\!k)}{(n\!-\!1\!-\!2k)} \iota^* \partial_s^{k-1}(v\J)
\end{equation}
for $0 \le k \le n-1$ (which extends \eqref{help-2})\footnote{For $k=0$, the second term on the
right-hand side is defined as $0$.} yields
\begin{align*}
   (v \dot{L}(0)(1))_{n-1-k} & = \frac{1}{(n\!-\!1\!-\!k)!} \iota^* \partial_s^{n-1-k} (v \dot{L}(0)(1)) \\
   & = \frac{n\!-\!1}{(n\!-\!1\!-\!2k)(n\!-\!1\!-\!k)!} \iota^* \partial_s^{n-k} (v) + \frac{2k
   (n\!-\!1\!-\!k)}{(n\!-\!1\!-\!2k)(n\!-\!1\!-\!k)!} \iota^* \partial_s^{n-2-k}(v\J) \\
   & = \frac{(n\!-\!1)(n\!-\!k)}{(n\!-\!1\!-\!2k)} v_{n-k} + \frac{2k}{(n\!-\!1\!-\!2k)} (v\J)_{n-2-k}
\end{align*}
for $0 \le k \le n-1$. This implies the claim. It only remains to prove the identity \eqref{reduction}.
Note that $\dot{L}(0)(1) = (n-1)\rho-s\J$. Using the definition of $\rho$, we obtain
\begin{equation}\label{d-delta}
   \iota^* \partial_s^k (v\dot{L}(0)(1)) = - \frac{n-1}{n+1} \left( \iota^* \partial_s^k(v \Delta (s))
   + \frac{2kn}{n-1} \iota^* \partial_s^{k-1}(v\J)\right).
\end{equation}
Now the relation \eqref{bL2} implies
\begin{align*}
   \iota^* \partial_s^k(v \Delta (s)) & = \iota^* \partial_s^{k+1}(v |\NV|^2) \\
   & = \iota^* \partial_s^{k+1}(v) - 2 (k+1) \iota^* \partial_s^k(v \rho) \\
   & = \iota^* \partial_s^{k+1}(v) + \frac{2(k+1)}{n+1} \iota^* \partial_s^k(v \Delta(s) + s v\J)
\end{align*}
for $0 \le k \le n-1$ using $|\NV|^2 = 1-2s\rho + O(s^{n+1})$. We rewrite this identity in the form
$$
   \frac{n-2k-1}{n+1} \iota^* \partial_s^k(v \Delta (s))
   = \iota^* \partial_s^{k+1}(v) + \frac{2k(k+1)}{n+1} \iota^* \partial_s^{k-1}(v \J).
$$
By substituting this result into \eqref{d-delta}, we obtain \eqref{reduction}. Note
that the above arguments extend those in the proof of Theorem \ref{LQ}.
\end{proof}

Finally, we sketch a {\em proof of Theorem \ref{Q-gen-holo}}. For $N < n$, we have
$$
   \D_N^{res} \left(\frac{-n+N}{2}\right) (1) = \left(\frac{n-N}{2}\right) \QC_N
$$
by \eqref{Q-sub-residue}. Hence
\begin{align*}
   \QC_N & =
   \D_N^{res}\left(\frac{-n+N}{2}\right) (1)  \left(\frac{n-N}{2}\right)^{-1} \\
   & = \D_{N-1}^{res}\left(\frac{-n+N}{2}-1\right) L\left(\frac{-n+N}{2}\right) (1)
   \left(\frac{n-N}{2}\right)^{-1} & \mbox{(by Corollary \ref{factor})} \\
   & = \D_{N-1}^{res}\left(\frac{-n+N}{2}-1\right) (-(N-1) \rho + s \J) & \mbox{(by \eqref{aa})}.
\end{align*}
This formula generalizes \eqref{start-crit}. Now we proceed as in the proof of
Theorem \ref{LQ}. In particular, Theorem \ref{D-res-ex} and the above identity imply
that
\begin{align*}
   \QC_N & = (-1)^{N} (N\!-\!1)!^2 \\ & \times \sum_{j=0}^{N-1} \frac{1}{j!}
   \left[ \T_{N-1-j}^*\left(\frac{n\!-\!N}{2}\right) v_0 + \cdots + \T_0^*\left(\frac{n\!-\!N}{2}\right) v_j \right]
   \iota^* \partial_s^j ((N\!-\!1)\rho-s\J).
\end{align*}
But in the latter double sum the operator $\T_k^*\left(\frac{n-N}{2}\right)$ acts on
the $(N-1-k)$'th Taylor coefficient of $((N-1)\rho-s\J)v$. A calculation using the definition of $\rho$
yields the extension
\begin{equation}\label{reduction-2}
   \iota^* \partial_s^k (v ((N\!-\!1)\rho - s\J)) = - \frac{N\!-\!1}{(n\!-\!1\!-\!2k)}
   \iota^* \partial_s^{k+1}(v) - \frac{k(n\!-\!2k\!+\!N\!-\!2)}{(n\!-\!1\!-\!2k)} \iota^*
   \partial_s^{k-1}(v\J)
\end{equation}
of \eqref{reduction} for $k=0,\dots,N-1$.\footnote{For $k=0$, the second term on the right-hand side
is defined as $0$.} Hence
\begin{equation*}
   (v ((N\!-\!1)\rho-s\J))_{N-1-k} = - \frac{(N\!-\!1)(N\!-\!k)}{(n\!-\!2N\!+\!2k\!+\!1)} v_{N-k} -
\frac{n\!-\!N\!+\!2k}{(n\!-\!2N\!+\!2k\!+\!1)} (v\J)_{N-2-k}
\end{equation*}
for $k=0,\dots,N-1$. This implies the assertion. \hfill $\square$

We finish this section with an application to the singular Yamabe obstruction $\B_{n}$
(see Section \ref{Yamabe}). We recall that our discussion of extrinsic conformal Laplacians
$\PO_N$ and extrinsic $Q$-curvatures $\QC_N$ is restricted to the range $N \le n$.
Already the first super-critical $\QC$-curvature $\QC_{n+1}$ is not well-defined.
Calculations in low-order cases point to the origin of its non-existence: $\QC_N$
has a pole in $n=N-1$. The following result interprets its residue.

\begin{theorem}\label{QB-residue} Let $n$ be even and $N \ge 3$. Then it holds
\begin{equation}\label{QB-F}
   \Res_{n=N-1} (\QC_N) = (-1)^{n-1} n! (n\!+\!1)! \frac{n}{2} \B_n.
\end{equation}
\end{theorem}

\begin{proof} Assume that $N \le n$. Theorem \ref{Q-holo-form} and
Theorem \ref{Q-gen-holo} imply that
$$
   \QC_N = (-1)^N (N\!-\!1)! N! \frac{N\!-\!1}{2N\!-\!n\!-\!1} v_N + \cdots,
$$
where the hidden terms are regular at $n=N-1$. However, the term $v_N$ has a
simple pole at $n=N-1$. More precisely, it follows from Corollary \ref{v-rho} that
$$
   v_N = -(n\!+\!1\!-\!2N) \frac{1}{N!} \partial^{N-1}_s (\rho)|_0 + \cdots,
$$
where the hidden terms are regular at $n=N-1$, and Proposition \ref{rec-rho} explains the origin
of the pole of $\partial^{N-1}_s (\rho)|_0$. But a comparison of Proposition \ref{rec-rho} and
Theorem \ref{obstruction-ex} shows that\footnote{See also \cite[Proposition 6.4]{GW-LNY}.}
\begin{equation}\label{B-rho}
   \B_{n} = \frac{2}{(n+1)!} \Res_{n=N-1} (\partial_s^{N-1}(\rho))|_0.
\end{equation}
The result follows from these facts.
\end{proof}

For a detailed discussion of the relation between $\B_2$ and $\QC_3$, we refer to Section \ref{Q-low}.

Combining Theorem \ref{QB-residue} with $\Res_{n=N-1} (\PO_N) \sim \Res_{n=N-1}
(\QC_N)$ allows us to derive the conformal invariance $e^{(n+1) \iota^*(\varphi)}
\hat{\B}_{n} = \B_{n}$ of the obstruction from the conformal covariance of $\PO_N$.

The above observations resemble the result that the residues
$$
   \Res_{n=4}(P_6) = -16 (\delta (\B d) - (\B,\Rho))
$$
and
$$
   \Res_{n=6}(P_8) = -48 (\delta(\OB_6 d) - (\OB_6,\Rho))
$$
of super-critical GJMS-operators $P_6$ and $P_8$ are conformally covariant. Here
$\B$ is the Bach tensor, and $\OB_6$ is the Fefferman-Graham obstruction tensor in
dimension $n=6$. For more details, we refer to \cite[Section 11.3]{J2}.

\section{Comments on further developments}\label{f-comm}

We recall that $\PO_n \iota^* = \D_n^{res}(0)$. For even $n$, this critical extrinsic conformal Laplacian
is elliptic. For odd $n$, the leading part of this operator is determined by $\lo$. The lower-order terms are
not known in general. For $n=3$, the explicit formula in Proposition \ref{P3F} shows that $\PO_3$ vanishes iff
$\lo=0$, i.e., iff $M$ is totally umbilic. The vanishing of $\PO_n$ is a conformally invariant condition. Now
if $\PO_n(g) = 0$, we define
\begin{equation}
   \PO_n'(g) = \dot{\D}_n^{res}(g;0) \quad \mbox {and} \quad \QC_n'(g) = \frac{1}{2} \ddot{\D}_n^{res}(g;0).
\end{equation}
Then $\PO_n': C^\infty(X) \to C^\infty(M)$ is a conformally covariant operator
$$
   e^{n \iota^*(\varphi)} \PO_n'(\hat{g}) = \PO_n'(g), \; \varphi \in C^\infty(X)
$$
such that the pair
$$
   (\PO_n', \QC_n')
$$
satisfies the fundamental identity
\begin{equation}\label{FI-prime}
   e^{n \iota^*(\varphi)} \QC_n'(\hat{g}) = \QC_n'(g) - \PO_n'(g)(\varphi) + \iota^*(\varphi) \PO_n'(g)(1)
\end{equation}
for all $\varphi \in C^\infty(X)$. This identity follows by twice differentiation of the conformal
transformation law of the critical residue family $\D_n^{res}(\lambda)$ at $\lambda=0$. Note that
$\PO_n'(g)(1) = - \QC_n(g)$. It is interesting
to analyze this construction further. First of all, it is a question of independent interest to characterize
the vanishing of $\PO_n(g)$ for odd $n$. Is it true that $\QC_n(g)(1) = 0$? The operator $\PO_n$ may be regarded
as a boundary operator for a conformally covariant boundary value problem for the critical GJMS-operator
$P_{n+1}$ on $X$. For $n=3$, such boundary value problems were recently analyzed in \cite{Case}. A
conformally covariant boundary operator of third order together with a $Q$-curvature like scalar curvature quantity
was discovered in \cite{CQ} in connection with the study of Polyakov formulas on four-dimensional manifolds with
boundary. Later it was studied from various perspectives. For details, we refer to \cite{Grant, J1, Case, GP} and
the references in these works.

It would be interesting to develop an extension of the present theory in higher
codimension situations again using solutions of singular Yamabe problems. We briefly
describe some aspects of a special case of such a theory. Let $S^m \hookrightarrow
S^n$ be an equatorial subsphere of $S^n$, $m \le n-1$. It is well-known that the
complement of $S^m$ in $S^n$ with the round metric $g$ is conformally equivalent to
the product $\mathbb{H}^{m+1} \times S^{n-m-1}$ with the respective hyperbolic and
round metric of constant scalar curvature $-(m+1)(m+2)$ and $(n-m-1)(n-m)$ on the
factors. The conformal factor $\sigma \in C^\infty(S^n\setminus S^m)$ can be defined
in terms of a Knapp-Stein intertwining operator on $S^n$ applied to the
delta-distribution of $S^m$ \cite[Section 2.4]{J1}. More explicitly, assume that
$S^m$ is defined by the equations $x_{m+2} = \cdots = x_{n+1}=0$. Then $\sigma$ is
the restriction of $(\sum_{j=m+2}^{n+1} x_j^2)^{1/2}$ to $S^n$. A stereographic
projection yields an isometry of $\sigma^{-2} g$ and
$$
   (x_{m+1}^2 + \cdots + x_{n}^2)^{-1} \sum_{i=1}^n dx_i^2
  = r^{-2} \sum_{i=1}^{m+1} dx_i^2 + g_{S^{n-m-1}},
$$
where $r^2 = \sum_{i=m+1}^{n} x_i^2$. The above $\sigma$ is a generalization of the
height function $\He$ in Section \ref{SBO} (the case $m=n-1$). Now, for any
eigenfunction $\tau$ of the Laplacian on $S^{n-m-1}$ (spherical harmonics), there is
an intertwining operator
$$
   D_{N,\tau}(\lambda): C^\infty(S^n) \to C^\infty(S^{m})
$$
for the subgroup $SO(1,m+1) \subset SO(1,n+1)$ leaving $S^m$ invariant (symmetry breaking operator).
These families can be constructed in terms of the residues of the family
$$
   \lambda \mapsto \int_{S^n} \sigma^\lambda u \psi dvol_{S^n},
$$
where $u$ is an eigenfunction of $\Delta_{\sigma^{-2}g}$ on the complement of $S^m$
which is compatible with $\tau$. The conformal factor $\sigma$ is a solution of the
singular Yamabe problem on the complement of $S^m$. In fact, the scalar curvature of
$\sigma^{-2}g$ is $-(n+1)(2m-n+2)$. It is negative iff $m > (n-2)/2$, i.e., iff the
dimension of the hyperbolic space exceeds the dimension of the sphere. This is a
special case of \cite{LN}, where it is proved that if $M \subset S^n$ is a smooth
submanifold of dimension $m$, a solution of the singular Yamabe problem with
negative scalar curvature exists iff $m  > (n-2)/2$. The analogous result for $S^n$
replaced by an arbitrary $X$ is in \cite{AMO1}. For a description of the asymptotic
expansions of solutions of the singular Yamabe problem in the negative case, we refer
to \cite{M}. Related representation theoretical aspects of the case $S^m
\hookrightarrow S^n$ are studied in \cite{MO}. Here the spectral decomposition of
spherical principal series representations of $O(1,n+1)$ under restriction to
$O(1,m+1) \times O(n-m)$ is made fully explicit.

The conformal invariance of the integral
$$
    \int_M \QC_n(g) dvol_h
$$
for closed $M$ (see \eqref{GI}) generalizes the conformal invariance of the total Branson
$Q$-curvature
$$
    \int_M Q_n(h) dvol_h.
$$
The classification of scalar Riemannian curvature quantities of a manifold $(M,h)$
which - like $Q_n$ - upon integration give rise to a conformal invariant has been the subject
of the Deser-Schwimmer conjecture \cite{DS}. S. Alexakis has achieved this classification in a series of 
works (see \cite{Alex} and its references). The present context suggests asking
for an analogous classification of scalar Riemannian invariants of a manifold $(X,g)$ which, upon
integration over closed submanifolds, yield conformal invariants of $g$. For first results in that direction,
we refer to \cite{mondino}. For related results around ${\bf Q}_4$, we refer to \cite{J5} and \cite{BGW}.


\section{Calculations and further results}\label{comm-cal}

\numberwithin{theorem}{subsection} \numberwithin{equation}{subsection}

In the present section, we collect formulas and computational details used in the
main body of the text. In addition, we illustrate various aspects of the general theory
by low-order examples. This often gives additional insight into the situation's nature and 
complexity. The results may also serve as material for future research. Moreover, we 
add a few further results.

For the reader's convenience, we start with an outline of the content of this section.

In Section \ref{basic-id}, we recall basic forms of the hypersurface Gauss equations,
recall the transformation laws of the second fundamental form and of some derived constructions
under conformal changes of the metric, derive a basic formula for the Laplacian, which plays a
key role in the proof of the conjugation formula and clarify the relation between
high-order iterated normal derivatives $\nabla_\NV^k$ and their analogs in adapted coordinates.

In Section \ref{model}, we determine the first three terms in the expansion of a
general metric $g$ in geodesic normal and adapted coordinates. These results are
fundamental for the later calculations. We also illustrate the results in several 
model cases with constant curvature and vanishing trace-free
part of $L$. These model cases may serve as valuable test examples of identities of the general theory.

In Section \ref{appY}, we first determine the first four terms in the asymptotic expansion of
solutions of the singular Yamabe problem in geodesic normal coordinates. Then, we use these results to 
calculate the obstructions $\B_2$ and $\B_3$ in these terms. Finally, we prove a formula for the obstructions
in terms of a formal residue of the super-critical coefficient $\sigma_{(n+2)}$. The explicit
form of the first few coefficients $\sigma_{(k)}$ enables us to confirm this formula in low-order cases.

In Section \ref{reno-vol-low}, we derive explicit formulas for the first three renormalized volume
coefficients $v_k$ (in adapted coordinates).

In Section \ref{TC+B}, we derive explicit formulas for the first two normal derivatives of $\rho$.
The discussion illustrates the efficiency of the recursive relation for the Taylor coefficients of $\rho$
expressed in Proposition \ref{rec-rho}. We further use these results in Section \ref{Q-low} for a
detailed discussion of $\QC_2$ and $\QC_3$.

In Section \ref{B2-details}, we show that the well-known formula for the obstruction $\B_2$
naturally follows from Theorem \ref{B-form}. Moreover, in Section \ref{B3-flat-back}, we evaluate 
the special case $n=3$ of Theorem \ref{B-form} for a flat background metric. We find that the 
formula coincides with the one derived in Section
\ref{appY} as well with a formula of Gover and Waldron. In addition, we verify that in the
conformally flat case, the result fits with the formula for $\B_3$ established in \cite{GGHW}.

In Section \ref{var-aspects}, we illustrate the role of the obstruction in the variational formula for
the singular Yamabe energy \cite{Graham-Yamabe}, \cite{GW-reno} in low-order cases. In \cite{GW-reno}
and \cite{GGHW}, the authors developed a new variational calculus. In contrast
to these references, here we only use classical style arguments as in \cite{Will, CM, Huisken, Graham-Yamabe}, 
for instance.

In Section \ref{P3}, we provide direct proofs of the conformal covariance of $\PO_3$
and the fundamental conformal transformation law of $\QC_3$.

Section \ref{sol-low} is devoted to a derivation of explicit formulas for the first two
solution operators $\T_1(\lambda)$ and $\T_2(\lambda)$ and the resulting first two
residue families $\D_1^{res}(\lambda)$ and $\D_2^{res}(\lambda)$. In addition, we
determine the leading term of $\PO_3$ through the leading term of the third solution
operator $\T_3(\lambda)$.

In the last section, we describe low-order renormalized volume coefficients $w_k$ in terms of 
Laplace-Robin operators. These results imply low-order cases of Theorem \ref{w-holo}.

Throughout this section, we apply some additional conventions. We use
indices $i,j$ for tensorial objects on $M$ and $0$ for $\NV$ when viewed as a normal
vector of $M$. In particular, $h_{ij}$ are the components of the metric $h$ on the boundary
and $\Rho_{00} = \Rho_X (\NV,\NV)$ is the restriction to $M$ of the Schouten tensor of
$g$ for the normal vector $\NV$. Similar conventions are used in adapted coordinates.

Sometimes, it will be convenient to distinguish curvature quantities of $g$ and $h$ not
by superscripts but by adding a bar to those of $g$ and leaving those of $h$ unbared.
Then $\bar{R}$, $\overline{\Ric}$ and $\overline{\scal}$ are the curvature tensor, the Ricci
tensor and the scalar curvature of $g$, respectively. For instance, $\Ric^g(\NV,\cdot)$
and $\overline{\Ric}_0$ are the same $1$-forms on $M$.

\index{$\bar{R}$}
\index{$\overline{\Ric}$}
\index{$\bar{\Rho}$}
\index{$\overline{\scal}$}

As before, we often use the same notation for a function on $X$ and its pull-back
by $\kappa$ or $\eta$ without mentioning it. A prime denotes derivatives in the variable $r$ and $s$. 
The restriction of a function $f(s)$ to $s=0$ is also denoted by $f_0$.
This notation often replaces $\iota^* (f)$.  For instance, we write $\rho'_0$ for the restriction 
$\iota^*(\partial_s (\rho))$ of $\partial_s(\rho)$ to $s=0$. If confusion is excluded, we 
sometimes even omit the symbols indicating restriction.


\subsection{Some basic identities}\label{basic-id}

\subsubsection{Gauss equations}\label{basic-Gauss}

Let $M^n$ be a hypersurface in $(X^{n+1},g)$ with the induced metric $h=\iota^*(g)$.
Then it holds
\begin{align}
   \Ric^g_{ij} - \Ric^h_{ij} & = R^g_{0ij0} + L^2_{ij} - n H L_{ij}, \label{GRicci} \\
   \scal^g - \scal^h & = 2 \Ric^g_{00} + |L|^2 - n^2 H^2 \label{G1}
\end{align}
on $M$. In the bar-notation, \eqref{G1} reads
$$
    \iota^* \overline{\scal} - \scal = 2 \overline{\Ric}_{00} +|L|^2 - n^2 H^2.
$$
The following result follows from the Gauss equation \eqref{G1}.

\begin{lem}\label{g12} For $n \ge 2$, it holds
\begin{equation}\label{G2}
   \iota^* \J^g - \J^h  = \Rho^g_{00} + \frac{1}{2(n\!-\!1)} |\lo|^2 - \frac{n}{2} H^2.
\end{equation}
\end{lem}

\begin{proof} The Gauss equation \eqref{G1} is equivalent to
$$
   2n \iota^* \J^g - 2(n\!-\!1) \J^h = 2((n\!-\!1) \Rho^g_{00} + \J^g) + |L|^2 - n^2 H^2.
$$
Hence
$$
   2(n\!-\!1) \iota^* \J^g - 2(n\!-\!1) \J^h = 2(n\!-\!1)\Rho^g_{00} + |\lo|^2 - n(n-1) H^2.
$$
This implies the assertion.
\end{proof}

\index{$\JF$ \quad Fialkow tensor}

Next, let
\begin{equation}\label{Fialkow-tensor}
   \JF \st \iota^* (\Rho^g) - \Rho^h + H \lo + \frac{1}{2} H^2 h.
\end{equation}
$\JF$ is a conformally invariant symmetric bilinear form, i.e., it holds
$\hat{\JF} = \JF$ (\cite[Theorem 6.22.11]{J1}.\footnote{In \cite{GW-LNY, GW-reno}, the tensor
$\JF$ is called the Fialkow tensor following \cite{V} (in turn being inspired by \cite{J1}). For more details on
the relation to Fialkow's classical work, we refer to \cite[Section 6.23]{J1}.}  Moreover, it satisfies the
fundamental relation
\begin{align}\label{FW-relation}
    (n\!-\!2) \JF_{ij} & = (n\!-\!2) \left(\iota^* \Rho^g_{ij} - \Rho^h_{ij} + H \lo_{ij}
    + \frac{1}{2}  H^2 h_{ij} \right) \notag \\
    & \stackrel{!}{=} W_{0ij0} + \lo^2_{ij} - \frac{|\lo|^2}{2(n\!-\!1)} h_{ij} 
\end{align}
(\cite[Lemma 6.23.3]{J1}), where $W$ is the Weyl tensor of $g$. We recall that
\index{$W$ \quad Weyl tensor}
\begin{equation}\label{RW-deco}
   R = W - \Rho \owedge g,
\end{equation}
where the Kulkarni-Nomizu product of the bilinear forms $b_1$ and $b_2$ is defined by
\begin{align*}
  & (b_1 \owedge b_2) (X,Y,Z,W)  \\
  & \st b_1 (X,Z) b_2 (Y,W) - b_1 (Y,Z) b_2(X,W) + b_1(Y,W) b_2(X,Z) - b_1(X,W) b_2(Y.Z).
\end{align*}

\index{$\owedge$ \quad Kulkarni-Nomizu product}

\subsubsection{Conformal change and the second fundamental form}\label{CTLL}

Let $N$ be a fixed unit normal field of $M$ defining $L$. Let $\hat{L}$ denote the second fundamental form
of $M$ with respect to the metric $\hat{g} = e^{2\varphi} g$. Then it holds
$$
    e^{-\varphi} \hat{L} = L + \nabla_N(\varphi) h.
$$
As a consequence, we find
$$
   e^{\varphi} \hat{H} = H + \nabla_n(\varphi).
$$
Both relations combine into the conformal invariance property
\begin{equation}\label{CTL-L}
    e^{-\varphi} \loh = \lo
\end{equation}
of the trace-free part $\lo = L - H h$ of $L$. It follows that $\loh^2 = \lo^2$, where 
$(\lo^2) _{ij} \st h^{ab} \lo_{ia} \lo_{jb}$. For $|L|^2 = \tr_h (L^2)
= h^{ij} (L^2)_{ij} = h^{ij} h^{ab} L_{ia} L_{jb}$, we find
$e^{2\varphi} |\loh|^2 = |\lo|^2$.\footnote{Of course, the norm on the left-hand side is taken with respect to $\hat{h}$.} 
Hence 
\begin{equation}\label{tf-square}
   (\loh^2)_\circ = (\lo^2)_\circ.
\end{equation}

\subsubsection{Some formulas for the Laplacian}\label{Laplace-expansion}

Here we discuss some useful formulas for the Laplacian. In particular, we prove the identity 
\eqref{Laplace-basis}.

\begin{lem}\label{LB-M}
Let $\sigma$ be a defining function of $M$ and $\partial_0 = \NV = \grad_g(\sigma)$.
We assume that $|\partial_0|_g =1$ on $M$. Then it holds
\begin{equation}\label{LBM}
    \Delta_g (u) = \partial_0^2 (u) + \Delta_h (u) + n H \partial_0 (u)
    - \langle du, \nabla_{\partial_0}(\partial_0) \rangle
\end{equation}
on $M$.
\end{lem}

\begin{proof}
Let $\left \{\partial_i \right \}$ be an orthonormal basis of the tangent spaces of
the level surfaces $\sigma^{-1}(c)$ (for small $c$) and let $\partial_0 =
\grad_g(\sigma)$. By assumption, these form an orthonormal basis on $M$. Let
$\left\{ dx^i \right\}$ together with $\left\{dx^0\right\}$ be the dual basis. We
calculate $\iota^* \Delta_g (f)$ using $\Delta = \tr (\nabla^2)$. First of all, we
have $\nabla (u) = \partial_i (u) dx^i + \partial_0 (u) dx^0$. Hence\footnote{As
usual, we sum over repeated indices.}
\begin{align*}
   \nabla^2(u) & = \partial_{ij}(u) dx^i \otimes dx^j + \partial_i(u) \nabla_{\partial_j}(dx^i) \otimes dx^j \\
   & + \partial_{i0}(u) dx^i \otimes dx^0 + \partial_i(u) \nabla_{\partial_0} (dx^i) \otimes dx^0 \\
   & + \partial_{0j}(u) dx^0 \otimes dx^j + \partial_0(u) \nabla_{\partial_j}(dx^0) \otimes dx^j \\
   & + \partial_{00}(u) dx^0 \otimes dx^0 + \partial_0(u) \nabla_{\partial_0}(dx^0) \otimes dx^0.
\end{align*}
Now taking traces gives
\begin{align*}
  \tr (\nabla^2)(u) & = \sum_{i=0}^n \partial_{ii}(u)
   + \partial_i(u) \langle \nabla_{\partial_j}(dx^i), \partial_j \rangle
   + \partial_i(u) \langle \nabla_{\partial_0} (dx^i), \partial_0 \rangle
   + \partial_0(u) \langle \nabla_{\partial_j}(dx^0), \partial_j \rangle \\
   & + \partial_0^2(u) + \partial_0(u) \langle \nabla_{\partial_0}(dx^0), \partial_0 \rangle
\end{align*}
on $M$. Thus, we obtain
\begin{align*}
  \tr (\nabla^2)(u) & = \sum_{i=0}^n \partial_{ii}(u)
   + \partial_i(u) \langle \nabla_{\partial_j}(dx^i), \partial_j \rangle
   - \partial_i(u) \langle dx^i, \nabla_{\partial_0} (\partial_0) \rangle
   - \partial_0(u) \langle dx^0, \nabla_{\partial_j}(\partial_j) \rangle \\
   & + \partial_0^2(u)
   - \partial_0(u) \langle dx^0, \nabla_{\partial_0}(\partial_0) \rangle.
\end{align*}
But since $\nabla^g_{\partial_i}(\partial_j) = \nabla^h_{\partial_i}(\partial_j)$ on
$M$, the last display simplifies to
\begin{align*}
   \Delta_g (u) & = \partial_0^2(u) + \Delta_h (u)
   - \partial_0(u) \langle dx^0, \nabla_{\partial_j}(\partial_j) \rangle \\
   & - \partial_i(u) \langle dx^i, \nabla_{\partial_0}
   (\partial_0) \rangle - \partial_0(u) \langle dx^0, \nabla_{\partial_0}(\partial_0) \rangle
\end{align*}
on $M$. By $L (\partial_i,\partial_j) = - \langle \nabla_{\partial_i} (\partial_j),
dx^0 \rangle $, we obtain
$$
   \Delta_g (u) = \partial_0^2(u) + \Delta_h (u) + n  H \partial_0(u) - \langle
   du,\nabla_{\partial_0}(\partial_0) \rangle
$$
on $M$. This completes the proof.
\end{proof}

\begin{remark}\label{Laplace-N} Without the assumption that $|\partial_0|=1$ on $M$, an extension 
of the above arguments shows that
\begin{equation*}
   \Delta_g (u) = \frac{1}{|\partial_0|^2} \partial_0^2 (u) + \Delta_h (u)
    + n H \frac{1}{|\partial_0|} \partial_0 (u)
    - \frac{1}{|\partial_0|^2} \langle du, \nabla_{\partial_0}(\partial_0) \rangle
\end{equation*}
on $M$. Similar arguments prove \eqref{Laplace-basis}.
\end{remark}

\begin{remark}\label{Laplace-standard}
In the situation of Lemma \ref{LB-M}, assume that $\sigma$ is the distance function of $M$. Then
$\nabla_\NV (\NV) = 0$ and we recover the well-known formula
\begin{equation*}
    \Delta_g = \nabla_{\NV}^2 + \Delta_h + n H \nabla_{\NV}
\end{equation*}
on $M$.
\end{remark}

The following result reproves \cite[Lemma B2]{GW-LNY}.

\begin{cor}\label{LRF}
If $\sigma$ satisfies $\SCY$, then
\begin{equation}\label{LR}
   \Delta_g = \nabla_\NV^2 + \Delta_h  + (n-1) H \nabla_\NV
\end{equation}
on $M$.
\end{cor}

\begin{proof}
By Lemma \ref{LB-M}, it suffices to prove that $\nabla_\NV(\NV) = H \NV$ on $M$.
For any $X \in \mathfrak{X}(X)$, we calculate
$$
    \langle \nabla_\NV(\NV),X\rangle = \Hess_g(\sigma)(\NV,X) = \Hess_g(\sigma)(X,\NV)
   = \langle \nabla_X(\NV),\NV \rangle = 1/2 \langle d (|\NV|^2),X \rangle.
$$
Now the assertion follows from $\SCY$ using $\rho = -H$ on $M$ (Lemma \ref{rho-01}).
\end{proof}

\subsubsection{Iterated normal derivatives}\label{ho-normals}

The identity
$$
   \iota^* \partial_s^k \circ \eta^* = \iota^* (|\NV|^{-2} \nabla_\NV)^k
$$
(see \eqref{translate}) relates iterated normal derivatives $\iota^* \partial_s^k
\circ \eta^*$ with respect to $s$ to iterated weighted gradients $\iota^*
(|\NV|^{-2} \nabla_\NV)^k$ of $\sigma$. Moreover, if $\SC(g,\sigma)=1$, i.e., if
$|\NV|^2 = 1- 2 \sigma \rho$, any iterated weighted gradient $\iota^* (|\NV|^{-2}
\nabla_\NV)^k$ can be written as a composition with $\iota^*$ of a linear
combination of iterated gradients $\nabla^j_\NV$ for $j \le k$ and polynomials in
the curvature quantities $\iota^* \nabla_\NV^j(\rho) \in C^\infty(M)$ for $j \le
k-1$. This follows by an easy induction using $\nabla_\NV (\sigma) = |\NV|^2$. In
particular, we obtain the following low-order formulas.

\begin{example}\label{HON} If $\SC(g,\sigma)=1$, then it holds
\begin{equation*}
   \iota^* \partial_s \eta^* = \iota^* \nabla_\NV \quad \mbox{and} \quad 
   \iota^* \partial_s^2 \eta^* = \iota^* (\nabla_\NV^2 + 2 \rho \nabla_\NV).
\end{equation*}
\end{example}

\begin{proof} We calculate
$$
   \iota^* \partial_s^2 \eta^* = \iota^* (|\NV|^{-2} \nabla_\NV)^2
   = \iota^* \nabla_\NV (1 + 2 \sigma \rho) \nabla_\NV
   = \iota^* (\nabla_\NV^2 + 2 \nabla_\NV(\sigma) \rho \nabla_\NV)
$$
using $|\NV|^2 = 1-2\sigma \rho$. Now $\iota^* \nabla_\NV (\sigma) = \iota^* (|\NV|^2) = 1$
implies the second identity. 
\end{proof}

These formulas can easily be inverted to express iterated normal gradients in terms of iterated
normal derivatives with respect to $s$. In particular, we obtain the following identities.

\begin{example}\label{trans-low-2} If $\SC(g,\sigma)=1$, then it holds
\begin{equation*}
   \iota^* \nabla_\NV = \iota^* \partial_s \eta^* \quad \mbox{and} \quad 
   \iota^* \nabla_\NV^2 = (\iota^* \partial_s^2 - 2 \rho_0 \iota^* \partial_s ) \eta^*.
\end{equation*}
\end{example}

The above discussion obviously generalizes to the case that $\sigma$ only satisfies the condition $\SCY$
with a non-trivial remainder.

\subsection{Expansions of the metric. Model cases}\label{model}

We start by discussing the normal forms of a given metric $g$ in geodesic normal coordinates
and in adapted coordinates (as defined in Section \ref{AC-RVC}). The formulas for these normal forms
contain respective families $h_r$ and $h_s$ of metrics on $M$. In order to simplify notation,
we shall use the same notation for the coefficients of their Taylor series in $r$ and $s$. It always will
be clear from the context which coefficients are meant. We derive formulas for the first few Taylor
coefficients of $h_r$ and $h_s$. A series of geometrically intuitive examples follow the discussion.

\index{$h_r$}
\index{$h_s$}
\index{$h_{(k)}$ \quad coefficients of $h_r$ or $h_s$}

\begin{prop}\label{h-low-order}
We expand the family $h_r$ in the normal form $dr^2 + h_r$ of $g$ in geodesic normal
coordinates as
$$
   h_r = h + h_{(1)} r + h_{(2)} r^2 + \cdots.
$$
The coefficients $h_{(k)}$ can be expressed in terms of the curvature of the metric $g$, its
covariant derivatives, and the second fundamental form $L$. The expansion starts
with
\begin{equation}\label{h-geodesic}
    (h_r)_{ij} = h_{ij} + 2L_{ij} r + ((L^2)_{ij} - R_{0ij0}) r^2 + (h_{(3)})_{ij} r^3 + \cdots
\end{equation}
with
\begin{equation}\label{h-cubic}
  3 (h_{(3)})_{ij}  =  - \nabla_{\partial_r} (R)_{0ij0} - 2 L_i^k R_{0jk0} - 2 L_j^k R_{0ik0}.
\end{equation}
Here $(L^2)_{ij} = L_{ik} L^k_j = L_{ik} L_{js} h^{ks}$.

Next, assume that $\sigma$ satisfies $\SCY$. We expand the family $h_s$ in the
normal form \eqref{normal-adapted} as
$$
   h_s = h + h_{(1)} s + h_{(2)} s^2 + \cdots.
$$
The coefficients $h_{(k)}$ can be expressed in terms of the curvature of the metric
$g$, its covariant derivatives, and the second fundamental form $L$. The expansion
starts with
\begin{equation}\label{h-adapted}
   (h_s)_{ij} = h_{ij} + 2 L_{ij} s + ((L \lo)_{ij} - R_{0ij0}) s^2 + (h_{(3)})_{ij} s^3  + \cdots
\end{equation}
with
\begin{align}\label{h-adapted-cubic}
  3 (h_{(3)})_{ij}  & =  - \nabla_{\partial_s} (R)_{0ij0} - 2 \lo_i^k R_{0jk0} - 2 \lo_j^k R_{0ik0} \notag \\
  & + \Hess_{ij}(H) - H R_{0ij0} - 3 H (L \lo)_{ij} + 2 L_{ij} \rho_0'.
\end{align}
Here $(L \lo)_{ij} = L_{ik} \lo_{js} h^{ks}$ and $\rho_0' = \Rho_{00} + |\lo|^2/(n-1)$.
\end{prop}

The notation $\rho_0'$ will be justified in Lemma \ref{rho-01}.

\begin{proof} We recall the standard formulas
$$
   R^s_{ijk} = \Gamma_{jk}^l \Gamma_{il}^s - \Gamma_{ik}^l \Gamma_{jl}^s
   + \partial_i (\Gamma_{jk}^s) - \partial_j (\Gamma_{ik}^s),  \quad R_{ijkl} = R_{ijk}^s g_{sl}
$$
and                                     \index{$\Gamma_{ij}^k$ \quad Christoffel symbol}
$$
   \Gamma_{ij}^m = \frac{1}{2} g^{km}  (\partial_i (g_{jk}) + \partial_j (g_{ik}) - \partial_k (g_{ij}))
$$
for the components of the curvature tensor of a metric $g$.\footnote{As usual, we
sum over repeated indices and denote derivatives by a lower index.} For the metric
$g = dr^2 + h_r$, we find $\Gamma_{ij}^0 = -\frac{1}{2} \partial_r(h_r)_{ij}$. Hence
$L_{ij} = - \Gamma_{ij}^0 = \frac{1}{2} (h_{(1)})_{ij}$ on $M$. This proves $h_{(1)}
= 2L$. In order to verify the formula for $h_{(2)}$ in \eqref{h-geodesic}, we
calculate the components $R_{0ij0}$ of the curvature tensor of the metric $dr^2 +
h_r$. We decompose $R_{0jk}^0$ as
$$
   R_{0jk}^0 = \left( \Gamma_{jk}^l \Gamma_{0l}^0 - \Gamma_{0k}^l \Gamma_{jl}^0
  +  \partial_r (\Gamma_{jk}^0) -  \partial_j (\Gamma_{0k}^0) \right)
  + \Gamma_{jk}^0 \Gamma_{00}^0 - \Gamma_{0k}^0 \Gamma_{j0}^0,
$$
where the summations run only over tangential indices. Now, for the metric $g = dr^2 + h_r$, we
find the Christoffel symbols
$$
   \Gamma_{ij}^0 = - \frac{1}{2} g_{ij}', \quad \Gamma_{0j}^0 = 0,
   \quad \Gamma_{00}^0 = 0, \quad \Gamma_{0k}^l = \frac{1}{2} g^{rl} g_{kr}',
$$
where $^\prime$ denotes $\partial_r$. It follows that
\begin{equation}\label{curv-geo}
   R_{0jk}^0 = \frac{1}{4} g^{rl} g_{kr}' g_{jl}'  - \frac{1}{2} g_{jk}''.
\end{equation}
We evaluate this formula at $r=0$. Using $g_{jk}' = 2 L_{jk}$, we find
$$
   R_{0jk0} = R_{0jk}^0 = h^{rl} L_{kr} L_{jl}  - (h_{(2)})_{jk} = (L^2)_{jk} - (h_{(2)})_{jk}.
$$
This implies the formula for $h_{(2)}$ in \eqref{h-geodesic}. Next, we prove the
formula for the cubic term. First, we note that
$$
     \nabla_{\partial_r}(R)_{0jk0}  = \partial_r (R_{0jk0})
    - R(\partial_r,\nabla_{\partial_r}(\partial_j),\partial_k,\partial_r)
    - R(\partial_r,\partial_j,\nabla_{\partial_r}(\partial_k),\partial_r)
$$
using $\nabla_{\partial_r}(\partial_r) = 0$. Now \eqref{curv-geo} implies
\begin{align*}
   \partial_r (R_{0jk0})  = \partial_r (R_{0jk}^0 g_{00}) = \partial_r (R_{0jk}^0)
   =  \frac{1}{4} (g^{rl})'  g_{kr}' g_{jl}' + \frac{1}{4} g^{rl} g_{kr}'' g_{jl}'
  + \frac{1}{4} g^{rl} g_{kr}' g_{jl}''- \frac{1}{2} g_{jk}'''.
\end{align*}
Evaluation of this formula for $r=0$ yields
\begin{align*}
    \partial_r (R_{0jk0}) & = - 2 L^{rl} L_{kr} L_{jl} + h^{rl} (h_{(2)})_{kr} L_{jl} + h^{rl} L_{kr} (h_{(2)})_{jl}
    - 3 (h_{(3)})_{jk} \\
   & = - L_j^r R_{0kr0} - L_k^l R_{0jl0} - 3 (h_{(3)})_{jk}.
\end{align*}
The two remaining terms in the formula for $\partial_r (R_{0jk0})$ for $r=0$ are
$$
   - R_{0jl0} L_k^l - R_{0lk0} L_j^l.
$$
Thus, we find
$$
   \nabla_{\partial_r}(R)_{0jk0} = -2 L_j^r R_{0kr0} - 2 L_k^l R_{0jl0} - 3 (h_{(3)})_{jk}.
$$
This proves \eqref{h-cubic}.

Similarly, the formula for $h_s$ in \eqref{h-adapted} and \eqref{h-adapted-cubic}
follows from a calculation of the Christoffel symbols and the components $R_{0ij0} =
R_{0ij}^0 g_{00}$ of the curvature tensor of the metric
$$
    \eta^*(g) = \eta^*(|\NV|^2)^{-1} ds^2 + h_s = a^{-1} ds^2 + h_s
$$
with $a = \eta^*(|\NV|^2)$. In order to simplify the notation, we shall write $g$
instead of $\eta^*(g)$ and $\rho$ instead of $\eta^*(\rho)$. For the above metric,
we find the Christoffel symbols\footnote{As usual the prime means derivative in $s$.}
$$
   \Gamma_{ij}^0 = - \frac{1}{2} g^{00} g_{ij}',
   \quad  \Gamma_{0j}^0 = \frac{1}{2} g^{00} \partial_j (g_{00}),
   \quad \Gamma_{00}^0 = \frac{1}{2} g^{00} g_{00}',
   \quad \Gamma_{0k}^l = \frac{1}{2} g^{rl} g_{kr}'.
$$
Here $g^{00} = a = 1- 2s\rho$ (by assumption) and $g_{00} = a^{-1} = 1 + 2s\rho +
\cdots$. In particular, $g^{00}=1$, $g_{00}' = -2H$ and $(g^{00})' = 2H$ on $M$.
Here we used that $\rho = -H$ on $M$ (Lemma \ref{rho-01}). Hence $L_{ij} = -
\Gamma_{ij}^0 = \frac{1}{2} (h_{(1)})_{ij}$. The components $R_{0jk0}$ of the
curvature tensor are given by
$$
    R_{0jk0} = R_{0jk}^0 g_{00}
$$
with
\begin{align}\label{R-adapted}
   R_{0jk}^0 & = \frac{1}{2} \Gamma_{jk}^l g^{00} \partial_l (g_{00})
   + \frac{1}{4} g^{00} g^{rl} g_{kr}' g_{jl}'
   - \frac{1}{2} ((g^{00})' g_{jk}' + g^{00} g_{jk}'')  \notag\\
   & - \frac{1}{2} ( \partial_j (g^{00}) \partial_k(g_{00}) + g^{00} \partial^2_{kj}(g_{00})) \notag \\
   & - \frac{1}{4} (g^{00})^2  g_{jk}' g_{00}' - \frac{1}{4} (g^{00})^2 \partial_j(g_{00}) \partial_k(g_{00}).
\end{align}
Evaluation of this formula for $s=0$ gives
$$
    R_{0jk0} = h^{rl} L_{kr} L_{jl} - (h_{(2)})_{jk} - 2H L_{jk} + H L_{jk}
   = (L^2)_{jk} - H L_{jk} - (h_{(2)})_{jk}.
$$
But $L^2 - H L = L (L - H h) = L \lo$. Hence $R_{0jk0} = (L \lo)_{jk} -
(h_{(2)})_{jk}$. This implies the formula for $h_{(2)}$ in \eqref{h-adapted}.
Finally, we prove the formula for the cubic term. Here we utilize the expansions
$$
    g_{00} = (1-2 \rho s)^{-1} = 1 + 2\rho s + 4 \rho^2 s^2 + \cdots
   = 1 + 2\rho_0 s + (2\rho_0'+4\rho_0^2) s^2 + \cdots
$$
and
$$
  g^{00} = 1 -2\rho s = 1- 2\rho_0 s - 2\rho_0' s^2 + \cdots.
$$
We proceed as above and calculate $\nabla_{\partial_s} (R)_{0jk0}$ for $s=0$. In the
present case, it holds
\begin{align*}
    \nabla_{\partial_s}(R)_{0jk0} & = \partial_s (R_{0jk0})
    - R(\partial_s,\nabla_{\partial_s}(\partial_j),\partial_k,\partial_s)
   - R(\partial_s,\partial_j,\nabla_{\partial_s}(\partial_k),\partial_s) \\
    & - R(\nabla_{\partial_s}(\partial_s),\partial_j,\partial_k,\partial_s)
   - R(\partial_s,\partial_j,\partial_k,\nabla_{\partial_s}(\partial_s));
\end{align*}
note that $\nabla_{\partial_s}(\partial_s)$ does not vanish in general. First, we
calculate the term $\partial_s (R_{0jk0})$ for $s=0$ using \eqref{R-adapted}. We
obtain
\begin{align*}
   \partial_s (R_{0jk0}) & = \partial_s (R_{0jk}^0) - 2 H R_{0jk}^0 \\
   & = -\Gamma_{jk}^l \partial_l(H) + \frac{1}{2} H h^{rl} g_{kr}' g_{jl}'
   + \frac{1}{4} (g^{rl})' g_{kr}'  g_{jl}'
   + \frac{1}{4} h^{rl} g_{kr}'' g_{jl}'
   + \frac{1}{4} h^{rl} g_{kr}' g_{jl}''  \\
   & + 4 \rho_0' L_{jk} - 2 H g_{jk}'' - \frac{1}{2} g_{jk}''' + \partial^2_{kj}(H) \\
   & + \frac{1}{2} H g_{jk}'' - \rho_0' g_{jk}' - 2 H R_{0jk0}
\end{align*}
using the expansions of $g_{00}$ and $g^{00}$. We further simplify that sum by using
the known formula for the first derivative of $h_s$ and $\Hess_{jk} =
\partial^2_{jk} - \Gamma_{jk}^l \partial_l$. Then
\begin{align*}
   \partial_s (R_{0jk0}) & = \Hess_{jk}(H) + 2 H L_k^l L_{jl} - 2 L^{rl} L_{kr} L_{jl}
  +  L_j^r (h_{(2)})_{kr} + L_k^l (h_{(2)})_{jl} \\
   & + 4  L_{jk} \rho_0' - 4 H (h_{(2)})_{jk} - 3 (h_{(3)})_{jk} + H (h_{(2)})_{jk} - 2 L_{jk} \rho_0' \\
   & =  \Hess_{jk}(H) + 2 H (L^2)_{jk} - 2 (L^3)_{jk} +  L_j^r (h_{(2)})_{kr} + L_k^l (h_{(2)})_{jl} \\
   & -3 H (h_{(2)})_{jk} + 2  L_{jk} \rho_0' - 3 (h_{(3)})_{jk} - 2 H R_{0jk0}
\end{align*}
on $M$. The four remaining contributions to $\nabla_{\partial_s}(R)_{0jk0}$ for $s=0$ equal
\begin{align*}
   & - \Gamma_{0j}^l R_{0lk0} - \Gamma_{0k}^l R_{0jl0} - \Gamma_{00}^0 R_{0jk0} - \Gamma_{00}^0 R_{0jk0}  \\
  & =  - \frac{1}{2} g^{rl} g_{jr}' R_{0lk0} - \frac{1}{2} g^{rl} g_{kr}' R_{0jl0} - g^{00} g_{00}'  R_{0jk0} \\
   & = - L_j^l  R_{0lk0} - L_k^l R_{0jl0}  + 2 H R_{0jk0} \\
   & = -\lo_j^l  R_{0lk0} - \lo_k^l R_{0jl0}.
\end{align*}
Hence
\begin{align*}
   \nabla_{\partial_s}(R)_{0jk0} & = \Hess_{jk}(H) + 2 H (L^2)_{jk} - 2 (L^3)_{jk} \\
   & + L_j^r ((L \lo)_{kr} - R_{0kr0}) + L_k^l ((L \lo)_{jl} - R_{0jl0})  - 3 H ((L \lo)_{jk} - R_{0jk0}) \\
   & - \lo_j^l  R_{0lk0} - \lo_k^l R_{0jl0} + 2 L_{jk} \rho_0' - 3 (h_{(3)})_{jk} - 2 H R_{0jk0} \\
   & = \Hess_{jk}(H) - L_j^r  R_{0kr0} - L_k^l R_{0jl0} - 3 H (L \lo)_{jk} + 3 H R_{0jk0} \\
   &  -\lo_j^l  R_{0lk0} - \lo_k^l R_{0jl0} + 2  L_{jk} \rho_0' - 3 (h_{(3)}) _{jk} - 2 H R_{0jk0} \\
   & = \Hess_{jk}(H)  - 2 \lo_j^r  R_{0kr0} - 2 \lo_k^r R_{0jr0} - 3 H (L \lo)_{jk} - H R_{0jk0} \\
   & + 2 L_{jk}  \rho_0' - 3 (h_{(3)}) _{jk}
\end{align*}
by the formula for $h_{(2)}$. This proves \eqref{h-adapted-cubic}.
\end{proof}

Formula \eqref{h-geodesic} was given in \cite[(2.4)]{Graham-Yamabe} and the formula
\eqref{h-cubic} for the cubic coefficient was displayed in \cite[(2.11)]{GG}.\footnote{In our
conventions, $R$ and $L$ have opposite signs as in \cite{Graham-Yamabe, GG}.}

The proof of Proposition \ref{h-low-order} shows that $h_{(2)}$ depends only on $\rho_0$. Similarly, $h_{(3)}$
only depends on $\rho_0$ and $\rho_0'$. More generally, it is easy to see that $h_{(k)}$
only depends on $\partial_s^j(\rho)|_0$ for $j \le k-2$.

\begin{cor}\label{h3-trace}
Formula \eqref{h-adapted-cubic} for the cubic term $h_{(3)}$ of the expansion of $g$
in adapted coordinates yields
\begin{equation}\label{h12-trace-n}
   \tr (h_{(1)}) = 2 n H, \quad \tr(h_{(2)}) = |\lo|^2 - \Ric_{00}
\end{equation}
and
\begin{equation}\label{h3-trace-n}
   3 \tr (h_{(3)}) = - \nabla_{\partial_s}(\Ric)_{00} - 4 \lo^{ik} R_{0ik0} + \Delta H
   - H \Ric_{00} - 3 H|\lo|^2 + 2 n H \rho_0'.
\end{equation}
\end{cor}

If $g_+ = r^{-2} g = r^{-2} (dr^2 + h_r)$ is a Poincar\'e-Einstein metric in normal
form relative to $h = h_0$, then $h_{(1)}=0$ and $h_{(2)} = -\Rho^h$. Comparing this
result with \eqref{h-geodesic}, implies that $R^g_{0ij0} = \Rho^h_{ij}$. Hence
$\Ric^g_{00} = \J^h$.

We illustrate the above results in some simple model cases.

\begin{example}\label{ball}
Let $S^n \subset \R^{n+1}$ be defined by $|x|=1$. Let $g_0$ be the Euclidean metric
on $\R^{n+1}$. Then the function $\sigma = (1-|x|^2)/2$ yields the Poincar\'e metric
$$
   \sigma^{-2} g_0 = \frac{4}{(1-|x|^2)^2} \sum_{i=1}^{n+1} dx_i^2
$$
on the unit ball $|x|<1$. Now $\NV = \grad(\sigma) = - \sum_i x_i \partial_i$. Hence the
normalized gradient field $\mathfrak{X} = \NV/|\NV|^2$ is given by $\mathfrak{X} =
-\frac{1}{|x|^2} \sum_i x_i \partial_i$. It follows that the map
$$
   \eta: (-1/2,1/2) \times S^n \ni (s,x) \mapsto \sqrt{1-2s} x \in \R^{n+1}
$$
defines the adapted coordinates. Note that $\eta^*(\sigma)=s$. Hence we obtain
$$
   \eta^*(g_0) = \frac{1}{1-2s} ds^2 + (1-2s) h,
$$
where the round metric $h$ on $S^n$ is induced by $g$. In particular, $h_s = h -
2sh$, i.e., $h_{(1)} = -2h$ and $h_{(2)}=0$ The coefficient $h_{(1)}$ is to be
interpreted as $2L$ being defined by the unit normal field $-\partial_r$. The
vanishing of the quadratic and the cubic term is confirmed by the general formulas.
Note also that $\rho = 1$. It follows that $v(s)=(1-2s)^{\frac{n-1}{2}}$ and
$\mathring{v}(s) = (1-2s)^{\frac{n}{2}}$. These results confirm the relation \eqref{Basic-R}
using $a=1-2s$. Moreover, the identity \eqref{van-id} in Conjecture \ref{van} is trivially satisfied.
\end{example}

\begin{example}\label{subsphere}
Let $S^n \subset S^{n+1}$ be an equatorial subsphere. Let $g$ be the round metric on
$S^{n+1}$. Then the height-function $\sigma = \He \in C^\infty(S^{n+1})$ defines the
metric $\He^{-2} g$ on both connected components of the complement of the zero locus
of $\He$. It is isometric to the Poincar\'e-metric on the unit ball (see Section
\ref{f-comm}). The map
$$
   \eta: (-1,1) \times S^n \ni (s,x) \mapsto (\sqrt{1-s^2}x,s) \in S^{n+1}
$$
defines the adapted coordinates. Note that $\eta^*(\sigma)=s$. Hence we obtain
$$
   \eta^*(g) = \frac{1}{1-s^2} ds^2 + (1-s^2) h,
$$
where the round metric $h$ on $S^n$ is induced by $g$. In particular, $h_{(1)} = 0$
and $h_{(2)}= -h$. The coefficient $h_{(1)}$ vanishes since $L=0$ and the
coefficient $h_{(2)}$ is to be interpreted as $-R_{0ij0}$. The vanishing of the
cubic term is confirmed by the general formula. Note that $\He$ is an eigenfunction
of the Laplacian on $S^{n+1}$ with eigenvalue $-(n+1)$ and $\J = \frac{n+1}{2}$.
Hence $\rho=\frac{1}{2} \He$ and $\eta^*(\rho) = \frac{1}{2} s$. Finally, we have
$v(s)=(1-s^2)^{\frac{n-1}{2}}$ and $\mathring{v}(s) = (1-s^2)^{\frac{n}{2}}$.
These results confirm the relation \eqref{Basic-R} using $a = 1-s^2$. Moreover,
the identity \eqref{van-id} in Conjecture \ref{van} reduces to the trivial relation
$$
    \binom{\frac{n-1}{2}}{\frac{n+1}{4}} (-1)^{\frac{n+1}{4}}
   + \binom{\frac{n-1}{2}}{\frac{n-3}{4}} (-1)^{\frac{n-3}{4}} = 0.
$$
\end{example}

In the above two examples, either the curvature of the background metric or the
second fundamental form vanishes. We finish this section with the discussion of a
model case with non-trivial curvature and non-trivial second fundamental form.

\begin{example}\label{mixed}
Let $\mathbb{H}^{n+1}$ be the upper half-space with the hyperbolic metric $g_+ =
r^{-2} (dr^2 + dx^2)$. For $c>0$, we let $X = \{r \ge c\}$ with boundary $M = \{ r=c
\}$ and background metric $g=g_+$. The metric $g$ restricts to $h = c^{-2} dx^2$ on
$M$. The defining function $\sigma = 1 - c/r$ is smooth up to the boundary $M$. It
solves the singular Yamabe problem since $\sigma^{-2} g_+ = (r-c)^{-2} (dr^2 +
dx^2)$ has scalar curvature $-n(n+1)$. The map
$$
   \eta: (0,1) \times \R^n \ni (s,x) \mapsto \left(\frac{c}{1-s},x\right) \in X
$$
defines the adapted coordinates. In fact, $\eta^*(\sigma) = s$. We obtain the normal form
$$
   \eta^*(g) = \frac{1}{(1-s)^2} ds^2 + \frac{(1-s)^2}{c^2} dx^2
$$
of $g$ in adapted coordinates. In particular, $h_{(1)} = -2h$, $h_{(2)} = h$ and
$h_{(3)} = 0$. These results fit with the general formulas in Proposition
\ref{h-low-order} since $L=-h$ and $R_{0ij0} = -h_{ij}$. The vanishing
of $h_{(3)}$ follows using $\lo = 0$, $H = -1$ and $\Rho_{00} = - 1/2$. A
calculation using the formula $\Delta_{g_+} = r^2 \partial_r^2 - (n-1) r \partial_r
+ \Delta_{\R^n}$ for the Laplacian of $g_+$ yields
$$
   \rho = \frac{1}{2} + \frac{c}{2r}.
$$
Hence $\eta^*(\rho) = 1- s/2$. Note that $\NV = \grad_g(\sigma) = -c \partial_r$,
$|\NV|^2 = r^{-2} c^2$ and $\eta^*(|\NV|^2) = (1-s)^2 \stackrel{!}{=} 1-
2s\eta^*(\rho)$. Finally, we have $v(s) = (1-s)^{n-1}$ and $\mathring{v}(s) =
(1-s)^n$. By $\J = - \frac{n+1}{2}$, the identity \eqref{van-id} in Conjecture
\ref{van} reduces to the trivial relation
$$
    \binom{n-1}{\frac{n+1}{2}} (-1)^{\frac{n+1}{2}}
    - \binom{n-1}{\frac{n-3}{2}} (-1)^{\frac{n-3}{2}} = 0.
$$
Finally, these results confirm the relation \eqref{Basic-R} using $a=(1-s)^2$.
\end{example}


\subsection{Approximate solutions of the singular Yamabe problem. The residue formula}\label{appY}


In the present section, we determine the first few terms in the expansion
$\sigma_F = r + \sigma_{(2)} r^2 + \dots$ of a solution of the equation
\begin{equation}\label{start-sol}
   \SC(g,\sigma_F) = 1+O(r^{n+1}).
\end{equation}
We also describe the obstruction $\B_n$ in terms of a formal residue of the supercritical
term $\sigma_{n+2}$.

By \eqref{Y-F}, equation \eqref{start-sol} takes the form
\begin{align*}
  & \partial_r(\sigma_F)^2 + h_r^{ij} \partial_i (\sigma_F) \partial_j (\sigma_F) \notag \\
  & - \frac{2}{n+1} \sigma_F \left (\partial_r^2 (\sigma_F) + \frac{1}{2} \tr (h_r^{-1} h_r') \partial_r (\sigma_F)
  + \Delta_{h_r} (\sigma_F) + \bar{\J} \sigma_F \right) = 1+ O(r^{n+1}).
\end{align*}
Here and in the following, we use the bar notation for curvature quantities of $g$.

We shall formulate the results in terms of the volume coefficients $u_k$ of the metric $g$
in geodesic normal coordinates (see \eqref{v-geo}) using
\begin{equation}\label{volume-geodesic}
   \frac{u'(r)}{u(r)} = \frac{1}{2} \tr (h_r^{-1} h_r').
\end{equation}
The following results describe the first three Taylor coefficients of a solution $\sigma_F$ of the 
singular Yamabe problem in general dimensions.

\index{$\sigma_{(2)}$, $\sigma_{(3)}$, $\sigma_{(4)}$}

\begin{lem}\label{sigma-v} It holds
\begin{align*}
    \sigma_{(2)} & = \frac{1}{2n} u_1,  \\
    \sigma_{(3)} & =  \frac{2}{3(n-1)} u_2 - \frac{1}{3n} u_1^2 + \frac{1}{3(n-1)} \bar{\J}
\end{align*}
and
\begin{align}\label{sigma4-g}
   \sigma_{(4)} & = \frac{3}{4(n-2)} u_3 - \frac{9n^2-20n+7}{12n(n-1)(n-2)} u_1 u_2
   + \frac{6n^2-11n+1}{24n^2(n-2)} u_1^3 \notag \\
   & + \frac{2n-1}{6n(n-1)(n-2)} u_1 \bar{\J} + \frac{1}{4(n-2)} \bar{\J}' + \frac{1}{4(n-2)}  \Delta (\sigma_{(2)}).
\end{align}
\end{lem}

Alternatively, the above formulas can be derived from the description of the solution $\sigma$ in \cite[Appendix]{GGHW}. 
In fact, these formulas describe expansions of $\sigma$ into power series of any defining function with coefficients that live
on the background space $X$. The calculation then requires expanding these coefficients into power series of the distance
function. We omit the details.

The volume coefficients $u_j$ may be expressed in terms of the Taylor coefficients of $h_r$. Such
relations follow from \eqref{volume-geodesic} by Taylor expansion in $r$ and resolving the resulting relations
for $u_j$. We find
\begin{align*}
    u_1 & = \frac{1}{2} \tr (h_{(1)}), \\
    u_2 & = \frac{1}{8} (\tr (h_{(1)})^2 + 4 \tr (h_{(2)}) - 2 \tr (h_{(1)}^2))
\end{align*}
and
\begin{align*}
     u_3 & = \frac{1}{48} (\tr (h_{(1)})^3 + 12 \tr (h_{(1)}) \tr (h_{(2)}) + 24 \tr (h_{(3)})
    - 6 \tr(h_{(1)}) \tr (h_{(1)}^2) \\ & - 24 \tr(h_{(1)} h_{(2)}) + 8 \tr (h_{(1)}^3)).
\end{align*}
These formulas are valid in general dimensions. The expressions for $h_{(j)}$ (for $j \le 3$) in
Proposition \ref{h-low-order} imply
\begin{align*}
   u_1 & = n H, \\
   2 u_2 & =  -\overline{\Ric}_{00} - |L|^2 + n^2 H^2 = \overline{\Ric}_{00} + \scal - \overline{\scal}
\end{align*}
(by the Gauss equation) or equivalently
\begin{equation}\label{u2}
   2 u_2 =  -\overline{\Ric}_{00} - |\lo|^2 + n(n-1) H^2.
\end{equation}
Moreover, we find
\begin{equation}\label{u3}
   6 u_3 = - \bar{\nabla}_0 (\overline{\Ric})_{00} + 2 (L,\bar{\G}) - 3 n H \overline{\Ric}_{00}
  + 2 \tr(L^3) - 3n H |L|^2 + n^3 H^3,
\end{equation}
where $\bar{\G}_{ij} \st \bar{R}_{0ij0}$, or equivalently                        \index{$\bar{\G}$}
\begin{align*}
   6 u_3 & = -\bar{\nabla}_0 (\overline{\Ric})_{00} + 2 (\lo,\bar{\G}) - (3n-2) H \overline{\Ric}_{00} \\
   & + 2 \tr (\lo^3) - 3(n-2) H |\lo|^2 + n(n-1)(n-2) H^3.
\end{align*}
For these formulas for $u_1,u_2,u_3$, see also \cite[(2.14)]{GG}.

Note that
\begin{equation}\label{det-2}
   - |L|^2 + n^2 H^2
   = \sum_{ij} \begin{vmatrix}  L_{ii}  & L_{ij}  \\ L_{ji}  & L_{jj}  \end{vmatrix} = 2 \sigma_2(L)
\end{equation}
and                            \index{$\sigma_k(L)$ \quad elementary symmetric polynomial}
\begin{equation}\label{det-3}
   2 \tr(L^3) - 3n H |L|^2 + n^3 H^3
   = \sum_{ijk} \begin{vmatrix}  L_{ii}  & L_{ij}  & L_{ik} \\ L_{ji}  & L_{jj} & L_{jk} \\
   L_{ki} & L_{kj} & L_{kk}   \end{vmatrix} = 6 \sigma_3(L)
\end{equation}
in an orthonormal basis. Here $\sigma_k(L)$ denotes the $k$-th elementary symmetric polynomial
in the eigenvalues of the shape operator. The identities \eqref{det-2} and \eqref{det-3} are special
cases of Newton's identities relating elementary symmetric polynomials to power series.

In these terms, we have
\begin{align*}
    2 u_2 & = -\overline{\Ric}_{00} + 2 \sigma_2(L) \\
    6 u_3 & = -\bar{\nabla}_0 (\overline{\Ric})_{00} + 2 (\lo,\bar{\G}) - (3n-2) H \overline{\Ric}_{00}
   + 6 \sigma_3(L).
\end{align*}
Note also that versions of the Gauss identity express the quantities $\sigma_2(L)$ and $\sigma_3(L)$
in terms of the curvatures of $g$ and of the induced metric on $M$ \cite[Corollary 3.2]{AGV}.
Finally, we note that \eqref{det-3} vanishes in $n=2$ - this identity is equivalent to $\tr(\lo^3) = 0$ in $n=2$.

The above results imply

\begin{lem}\label{sigma23} It holds $2 \sigma_{(2)} = H$ and
\begin{equation}\label{sigma3}
    6 \sigma_{(3)}  = -2 \bar{\Rho}_{00} - \frac{2}{n-1} |\lo|^2
   = 2( \J - \iota^* \bar{\J}) - \frac{1}{n-1} |\lo|^2 - n H^2
\end{equation}
for $n \ge 2$.
\end{lem}

The second equality follows by combining the first equality with \eqref{G2}.

The results of Lemma \ref{sigma23} are contained in \cite[(2.6) and Section 4]{Graham-Yamabe}. See also
\cite[(2.16)--(2.19)]{GG}.\footnote{These references use a different convention for $L$ and $H$.} 

The above formulas for $v_j$ are equivalent to the corresponding formulas in \cite[Theorem 3.4]{AGV}. We
also refer to \cite[Theorem 9.22, Problem 9.1]{Gray} for the corresponding results in higher codimensions.
However, the methods of proof in these references are different.

For $n=1$, the above results easily imply
\begin{equation}\label{B1=0}
   \B_1 = (r^{-2} (\SC(S_2)-1))|_0 = -2u_2 - \bar{\J}_0 = \overline{\Ric}_{00} - \bar{\J}_0 = 0
\end{equation}
using $\lo=0$ (see also Remark \ref{B1}). This result is well-known \cite{Graham-Yamabe}.


It also is of interest to explicate the above formulas for flat backgrounds. In fact, it follows from the identity
\begin{equation}\label{volume-flat}
   u(r) = \det (\id + r L)
\end{equation}
(see \cite[Section 3.4]{Gray}) for a flat background that the formulas for $\sigma_{(k)}$ (for $k \le 4$)
can be expressed in terms of $L$. Newton's identities imply
\begin{align*}
   u_1 & = n H, \\
   u_2 & = \frac{1}{2} (n (n-1) H^2 - |\lo|^2), \\
   u_3 & = \frac{1}{6} ( H^3 n(n-1)(n-2) - 3 (n-2) H |\lo|^2 + 2 \tr(\lo^3))
\end{align*}
and a direct calculation yields the following result.

\begin{lem}\label{sigma-flat} The expansion of a solution of the Yamabe problem for $M^n \hookrightarrow \R^{n+1}$
has the form
$$
  \sigma_F = r + \frac{r^2}{2} H - \frac{r^3}{3(n-1)} |\lo|^2 + r^4 \sigma_{(4)} + \cdots
$$
with the coefficient
\begin{align*}
   \sigma_{(4)}  & = \frac{1}{24(n-2)} \left(6 \tr(\lo^3) + \frac{7n-11}{n-1} H |\lo|^2 + 3 \Delta(H)\right).
\end{align*}
\end{lem}

Note that the coefficients in formula \eqref{sigma4-g} for $\sigma_{(4)}$ have a simple pole in $n=2$. The following result
calculates the formal residue at $n=2$.

\begin{cor}\label{sigma4-res} It holds
$$
   \res_{n=2}(\sigma_{(4)}) = \frac{3}{4} u_3 + \frac{1}{32} u_1^3 - \frac{1}{8} u_1 u_2
  + \frac{1}{4} u_1 \bar{\J} + \frac{1}{4} \bar{\J}' + \frac{1}{4} \Delta (\sigma_{(2)}).
$$
For a flat background, we obtain
$$
     \res_{n=2}(\sigma_{(4)}) = \frac{1}{8} (H |\lo|^2 + \Delta (H)).
$$
\end{cor}

Note that for a flat background it holds $u_3 = 0$ in $n=2$. We also recall that $\tr(\lo^3) = 0$ for $n=2$.
Alternatively, one can use formula \cite[(2.18)]{GG} for $\sigma_{(4)}$ to confirm this residue formula for
general backgrounds.

Corollary \ref{sigma4-res} and the following result are special cases of the residue formula in Lemma \ref{res-form-B}.

\begin{cor}\label{B-res} It holds
$
   \res_{n=2}(\sigma_{(4)}) = - \frac{3}{8} \B_2.
$
\end{cor}

\begin{proof} We recall that the obstruction $\B_2$ is defined by
$$
   \B_2 = (r^{-3} (\SC(S_3)-1))|_0,
$$
where $S_3= r + \sigma_{(2)} r^2 + \sigma_{(3)} r^3$. A calculation yields
\begin{equation*}\label{B2-new-g}
   \B_2 = - 2 u_3 - \frac{1}{12} u_1^3 + \frac{1}{3} u_1 u_2
   - \frac{2}{3} u_1\bar{\J} - \frac{2}{3} \bar{\J}' - \frac{2}{3} \Delta (\sigma_{(2)}).
\end{equation*}
We omit the details. We recall that the term $u_3$ vanishes in $n=2$ in the flat case but not in
the curved case.
\end{proof}

In Section \ref{B2-details}, we shall derive another formula for $\B_2$ from Theorem \ref{B-form}.
Although the equivalence of both formulas is non-trivial, we leave the check of consistency to the reader.

Similarly, we may either directly evaluate the definition $\B_3 = (r^{-4} (\SC(S_4)-1))|_0$ or
use the residue formula \eqref{res-f} to derive a formula for $\B_3$ from the residue of
$\sigma_{(5)}$ at $n=3$. The results read as follows.

\begin{lem}\label{B3-general-back} It holds
\begin{align*}\label{B3-start}
   \B_3 & = - 2u_4 + \frac{1}{2} u_1 u_3 + \frac{1}{3} u_2^2 - \frac{7}{18} u_1^2 u_2 + \frac{2}{27} u_1^4
   - \frac{1}{3} \bar{\J} u_2 - \frac{5}{12} \bar{\J}' u_1 - \frac{1}{4} \bar{\J}'' \notag \\
   & - \frac{1}{2} \Delta (\sigma_{(3)}) - \frac{1}{3} u_1 \Delta(\sigma_{(2)})
   - \frac{1}{2} \Delta' (\sigma_{(2)}) + |d\sigma_{(2)}|^2,
\end{align*}
where
\begin{equation*}
  6 \sigma_{(2)} = u_1 \quad \mbox{and} \quad 9 \sigma_{(3)} = 3 u_2 - u_1^2 + \frac{3}{2} \bar{\J}
\end{equation*}
and $\Delta' = (d/dt)|_0 (\Delta_{h+2sL})$ (see \eqref{var-1}). For a flat background, it holds $\bar{\J}=0$,
the term $u_4$ vanishes, and we obtain
\begin{equation}\label{B3-flat-algo}
   12 \B_3 = \Delta (|\lo|^2) + 6 H \tr (\lo^3)  + |\lo|^4 + 3 |dH|^2  - 6 H \Delta(H) - 3 \Delta' (H).
\end{equation}
\end{lem}

The above formula for $\B_3$ also follows from the following result through the residue formula
\begin{equation}\label{res-5}
    \res_{n=3}(\sigma_{(5)}) = - \frac{2}{5} \B_3.
\end{equation}
The evaluation of that formula rests on the following result for the coefficient $\sigma_{(5)}$.

\begin{lem}\label{sigma-5} In general dimensions, it holds
\begin{align*}
    \sigma_{(5)} & =
   - \frac{n+1}{10(n-3)} |d\sigma_{(2)}|^2 \\
   & + \frac{1}{5(n-3)} \Delta'(\sigma_{(2)})
   + \frac{1}{5(n-3)} \Delta(\sigma_{(3)}) + \frac{3n-1}{20(n-3)(n-2)n} \Delta(\sigma_{(2)}) u_1 \\
   & + \frac{1}{10(n-3)} \bar{\J}''  + \frac{n-1}{4(n-3)(n-2)n} \bar{\J}' u_1
   + \frac{2(3n-5)}{15(n-3)(n-1)^2} \bar{\J} u_2 \\
   & - \frac{4n-3}{20 (n-2)(n-1)n} \bar{\J} u_1^2 + \frac{1}{30(n-1)^2} \bar{\J}^2 \\
   & + \frac{48n^4-247n^3+387n^2-179n+3}{60(n-3)(n-2)(n-1)n^2} u_1^2 u_2
   - \frac{2(3n^2-11n+10)}{15(n-3)(n-1)^2} u_2^2 \\
   & - \frac{24n^4-110n^3+133n^2-24n-3}{120 (n-3)(n-2)n^3} u_1^4
   - \frac{16n^2-53n+27}{20(n-3)(n-2)n} u_1 u_3 + \frac{4}{5(n-3)} u_4.
\end{align*}
\end{lem}

We omit the details of the proof. An alternative formula for $\sigma_{(5)}$ is given in \cite[Appendix]{GGHW}. 
Here the same comments as after Lemma \ref{sigma-v} apply.

\begin{remark}\label{sigma-residue-formula}
Lemma \ref{sigma-5} implies
$$
    \res_{n=2}(\sigma_{(5)}) = - \frac{1}{2} u_1 \res_{n=2} (\sigma_{(4)}) = \frac{3}{4} H \B_2.
$$
This relation extends the identities
\begin{align*}
    \res_{n=1}(\sigma_{(4)}) & = - \frac{1}{2} u_1 \res_{n=1}(\sigma_{(3)}), \\
    \res_{n=1}(\sigma_{(3)}) & = - \frac{1}{2} u_1 \res_{n=1}(\sigma_{(2)}) = 0.
\end{align*}
More generally, we conjecture the factorization identities
$$
    \res_{n=k-3}(\sigma_{(k)}) = - \frac{1}{2} u_1 \res_{n=k-3} (\sigma_{(k-1)})
$$
for $k \ge 4$.
\end{remark}

Combining \eqref{B3-flat-algo} with the variation formula $\Delta'(u) = -2 (\Hess(u),L) - 3 (du,dH)$ (see
\eqref{var-1}) shows that
$$
   12 \B_3 = \Delta (|\lo|^2) + 12 |dH|^2 + 6 (\Hess(H),\lo)  + |\lo|^4  + 6 H \tr(\lo^3).
$$
In Section \ref{B3-flat-back}, we shall alternatively derive that result from Theorem \ref{B-form}.

A direct proof that, for conformally flat backgrounds, Lemma \ref{B3-general-back} is equivalent to Lemma
\ref{B3-CF} will be given in a separate work.

Finally, we establish the residue formula for the singular Yamabe obstruction in full generality.

\begin{lem}\label{res-form-B} It holds
\begin{equation}\label{res-f}
   \res_{n=k-2}(\sigma_{(k)}) =-\frac{k-1}{2k} \B_{k-2}, \; k \ge 4.
\end{equation}
\end{lem}

\begin{proof} Let $S_k$ be the Taylor polynomial $r + r^2 \sigma_{(2)} + \cdots + r^k \sigma_{(k)}$
of degree $k$. Then a calculation shows that
\begin{equation}\label{S-deco}
   \SC(S_k) = \SC(S_{k-1}) +  r^{k-1} \frac{2k (n-k+2)}{n+1} \sigma_{(k)} + O(r^k).
\end{equation}
Assume that $S_{k-1}$ is the $(k-1)$-th approximate solution of the singular Yamabe problem, i.e.,
$\SC(S_{k-1}) = 1+ O(r^{k-1})$. Then $S_k$ is the $k$-th approximate solution if in the expansion
$$
   \SC(S_k) = 1 + r^{k-1} \left (\frac{2k (n-k+2)}{n+1} \sigma_{(k)} + \cdots \right) + O(r^k)
$$
with an unknown coefficient $\sigma_{(k)}$ the coefficient of $r^{k-1}$ vanishes, i.e., if
$\SC(S_k) = 1 + O(r^k)$. This can be solved for $\sigma_{(k)}$ if $k =2,3,\dots,n+1$.\footnote{This algorithm
yields Lemma \ref{sigma-v}.} In the case $k=n+2$, we only have
$$
   \SC(S_{n+1}) = 1 + O(r^{n+1}),
$$
and the restriction of the latter remainder is the obstruction $\B_n$. Now \eqref{S-deco} implies
$$
   \frac{k-1}{2k} (r^{-k+1} (\SC(S_{k-1})-1)|_0 +  \res_{n=k-2} (\sigma_{(k)})
  =  \frac{k-1}{2k}  (r^{-k+1}(\SC(S_k)-1)|_0
$$
if $S_{k-1}$ is the $(k-1)$-th approximate solution of the singular Yamabe problem. Then the left-hand
side is well-defined and the right-hand side vanishes. This implies the assertion.
\end{proof}

\subsection{Renormalized volume coefficients (adapted coordinates)}\label{reno-vol-low}

Here we consider the first three coefficients in the expansion of $v(s)$ as defined in \eqref{RVC-B} 
in adapted coordinates. As usual we identify $\rho$ with $\eta^*(\rho)$. We also recall 
that derivatives with respect to $s$ are denoted by a prime.

We first derive formulas for $v_1$ and $v_2$ from Proposition \ref{h-low-order}.

\begin{example}\label{v-low-order} Note that
$$
   a^{-1/2} = (1-2s\rho)^{-1/2} = 1 + s \rho_0 + s^2 (\frac{3}{2} \rho_0^2 + \rho'_0) + \cdots.
$$
Hence the factorization $v(s) = a^{-1/2} \mathring{v}(s)$ and the identities
\begin{equation*}
   \mathring{v}_1 = \frac{1}{2} \tr (h_{(1)}) \quad \mbox{and} \quad 
   \mathring{v}_2 = \frac{1}{8} \tr (h_{(1)})^2 +  \frac{1}{2} \tr (h_{(2)}) - \frac{1}{4} \tr (h_{(1)}^2))
\end{equation*}
show that the expansion of $v(s)$ starts with
\begin{align*}
   & 1 + \frac{1}{2} (2 \rho_0 + \tr (h_{(1)})) s \\
   & + \frac{1}{2} \left(\tr (h_{(2)}) - \frac{1}{2} |h_{(1)}|^2
   + \frac{1}{4} \tr(h_{(1)})^2 + \rho_0 \tr (h_{(1)}) + 3 \rho_0^2 + 2 \rho'_0 \right) s^2 + \cdots.
\end{align*}
Now \eqref{h-adapted} and $\rho_0 = -H$ (Lemma \ref{rho-01}) imply that the linear
coefficient equals $(n-1)H$. Hence $v_1=(n-1)H$. This result fits with the formula
\cite[(4.6)]{GW-reno} for its integral. Moreover, using Lemma \ref{rho-01}, we
obtain
$$
   \frac{1}{2} \left( |\lo|^2 - \overline{\Ric}_{00} - 2 |L|^2 + n^2 H^2
   - 2n H^2 + 3 H^2 + 2 \bar{\Rho}_{00} + 2 \frac{|\lo|^2}{n-1}\right)
$$
for the quadratic coefficient. By $\overline{\Ric}_{00} = (n-1) \bar{\Rho}_{00} + \bar{\J}_0$ and
$|L|^2 =|\lo|^2 + n H^2$, the latter sum simplifies to
\begin{equation}\label{v2}
   v_2 = \frac{1}{2} \left(-\frac{n-3}{n-1} |\lo|^2 - (n-3) \bar{\Rho}_{00} + (n-1)(n-3) H^2 - \bar{\J}_0 \right).
\end{equation}
In particular, \eqref{v2} shows that $2 v_2 = - \bar{\J}_0$ for $n=3$. This proves the relation \eqref{van-id} in
Conjecture \ref{van} for $n=3$.
\end{example}

In the following examples, we demonstrate how the identity \eqref{bL} can be used to calculate the 
renormalized volume coefficients $v_k$ (for $k=1,2,3$). This alternative method does not require 
the calculation of composition of $L$-operators and does not use explicit formulas for the Taylor coefficients of $h_s$.

\begin{example}\label{v1}
The restriction of \eqref{bL} to $s=0$ implies $v_1=-(n-1)\rho_0$. Now the fact
$\rho_0 = \iota^* \rho = -H$ (Lemma \ref{rho-01}) implies $v_1 = (n-1)H$.
\end{example}

\begin{example}\label{v2-ex}
We restrict the derivative of \eqref{bL} in $s$ to $s=0$. Then
$$
   2v_2 - v_1^2 = -(n\!-\!3) \rho'_0 - \bar{\J}_0 - 2(n\!-\!1) \rho^2_0.
$$
Now we combine this result with the value of $v_1$ (Example \ref{v1}) and Lemma \ref{rho-01}
to conclude that
\begin{align}\label{v2n}
   -2v_2 & = \bar{\J}_0 + (n\!-\!3) \left(\bar{\Rho}_{00} + \frac{1}{n\!-\!1} |\lo|^2 \right) - (n\!-\!3)(n\!-\!1) H^2 \notag \\
   & = \bar{\J}_0 + (n\!-\!3) \rho_0' - (n\!-\!3)(n\!-\!1) H^2.
\end{align}
This result fits with \eqref{v2} and with the formula \cite[(4.8)]{GW-reno} for its
integral. Using \eqref{G2}, we finally obtain
\begin{equation}\label{v2n-2}
   -2v_2 = \J + (n\!-\!2) \bar{\Rho}_{00} + \frac{2n\!-\!5}{2(n\!-\!1)} |\lo|^2 - (n\!-\!2)(n\!-\!3/2) H^2.
\end{equation}
In particular, for $n=2$, we have
\begin{equation}\label{v2-2}
   -2v_2 = \J - \frac{1}{2} |\lo|^2.
\end{equation}
For closed $M$, the total integral of this quantity is a conformal invariant (by
Gauss-Bonnet). Note that $2v_2=\QC_2$ (by Example \ref{Q2}) which confirms Theorem
\ref{LQ} for $n=2$.
\end{example}

\begin{example}\label{v3-ex}
The equality of the coefficients of $s^2$ in \eqref{bL} yields the identity
$$
  3v_3 - 3 v_1 v_2 + v_1^3 = - \frac{n\!-\!5}{2} \rho''_0 - \bar{\J'}_0 - 4(n\!-\!2) \rho'_0 \rho_0
  - 2 \rho_0 \bar{\J}_0 - 4(n\!-\!1) \rho_0^3.
$$
Hence, using $\rho_0 = -H$, we obtain
$$
   3v_3 = 3 v_1 v_2 - v_1^3 - \frac{n\!-\!5}{2} \rho''_0 - \bar{\J}'_0 + 4(n\!-\!2) \rho'_0 H
   + 2 H \bar{\J}_0 + 4(n\!-\!1) H^3.
$$
Combining this formula with the formulas for $v_1$ and $v_2$ in Example \ref{v1} and Example
\ref{v2-ex} gives
$$
  6 v_3 = (n\!-\!5)(n\!-\!3)(n\!-\!1) H^3 - (n\!-\!5)(3n\!-\!5) H \rho'_0
  - (3n\!-\!7) H \bar{\J}_0 - (n\!-\!5) \rho''_0 - 2 \bar{\J}'_0.
$$
In particular, this formula implies
$$
   6 v_3 = - 8 H \bar{\J}_0 - 2 \bar{\J}'_0.
$$
if $n=5$. This proves the relation \eqref{van-id} in Conjecture \ref{van} for $n=5$. Finally, we note that
$\iota^* \nabla_{\NV}^2(\rho)$ corresponds to $\rho_0'' - 2 \rho_0 \rho_0'$ (see Example \ref{trans-low-2}).
Hence we can rewrite the latter formula as
\begin{align}\label{v3-inter}
   6 v_3 & = (n\!-\!5)(n\!-\!3)(n\!-\!1) H^3 - (3n\!-\!7) H \iota^* \bar{\J} \notag \\
   & - (n\!-\!5)(3n\!-\!7) H \iota^* \nabla_\NV (\rho)- (n\!-\!5) \iota^* \nabla_\NV^2(\rho)
   - 2 \iota^* \nabla_\NV(\bar{\J}).
\end{align}
In particular, for $n=3$ we get
$$
   6 v_3 = - 2 H \iota^* \bar{\J} + 4 H \iota^* \nabla_\NV (\rho) + 2 \iota^* \nabla_\NV^2(\rho)
   - 2 \iota^* \nabla_\NV( \bar{\J} ).
$$
This quantity actually equals a multiple of $\QC_3$, up to a divergence term; for a discussion of $\QC_3$ 
we refer to Example \ref{Q3}. By \cite[Appendix B]{GW-reno}, the result \eqref{v3-inter} implies 
that\footnote{There seems to be a misprint in the contribution of the term $H \iota^* \J$.}
$$
   -\iota^* L(-n+1)L(-n+2)L(-n+3)(1) = 6 (n-1)(n-2)(n-3) v_3,
$$
up to a divergence term. Therefore, the result confirms the formula for $c_3$ in
Theorem \ref{RVE}. In order to express the sum in \eqref{v3-inter} in terms of
standard curvature terms, it remains to determine $\iota^* \nabla_\NV^k (\rho)$ for
$k=1,2$. Explicit formulas for these terms will be derived in Section \ref{TC+B} (Lemma \ref{rho-01}, 
Lemma \ref{rho-two}). The case $k=2$ was first treated in \cite[Lemma 6.8]{GW-LNY} (see
Remark \ref{GW-rho-2}). It is only here where we need the full information of Proposition
\ref{h-low-order}.
\end{example}

Of course, the renormalized volume coefficients $v_k$, which are defined in terms of adapted coordinates, 
are to be distinguished from the renormalized volume coefficients $w_k$, which are defined in terms of 
geodesic normal coordinates \cite{Graham-Yamabe}. By \eqref{def-w}, the latter ones are defined by the 
relation
$$
   w(r) = (1 + \sigma_{(2)} r +  \sigma_{(3)} r^2 + \cdots)^{-(n+1)} u(r).
$$
In particular, we find $w_1 = -(n+1) \sigma_{(2)} + u_1 = \frac{n-1}{2} H$. The above relation implies
$$
   w_2 = 6 \sigma_{(2)}^2 - 3 \sigma_{(3)} - 3 \sigma_{(2)} u_1 + u_2
$$
for $n=2$, and its evaluation yields
$$
    2 w_2 = - \J + \frac{1}{2} |\lo|^2.
$$
Note that $w_1 = v_1$ (for $n=1$) and $w_2 = v_2$ (for $n=2$). In general, the coefficients 
$w_n$ and $v_n$ differ by a non-trivial total divergence. In particular, we find
$$
   w_3 = -20 \sigma_{(2)}^3 + 20 \sigma_{(2)} \sigma_{(3)} - 4 \sigma_{(4)} + 10 \sigma_{(2)}^2 u_1 
   - 4 \sigma_{(3)} u_1 - 4 \sigma_{(2)} u_2 + u_3
$$ 
for $n=3$, and an evaluation yields
\begin{align*}
    6 v_3 & = 6 w_3 + \Delta(H).
\end{align*}
Explicit formulas for $w_1, w_2$ (in general dimensions) were derived in \cite[(4.5)]{Graham-Yamabe}. 
In Section \ref{w-LR}, these coefficients will be described in terms of $L$-operators. 

\subsection{Low-order Taylor coefficients of $\rho$}\label{TC+B}

Assuming that $\sigma$ satisfies the condition $\SCY$, 
we derive formulas for the first few Taylor coefficients of $\rho$ in the variable $s$. 

We first use Lemma \ref{sigma23} to derive formulas for the restrictions of $\rho$ and
$\nabla_\NV(\rho)$ to $M$. The following result reproves part of \cite[Lemma 6.6]{GW-LNY}.

\begin{lem}\label{rho-01} Let $n \ge 2$. Then $\iota^* \rho = - H$ and
\begin{equation}\label{rho-1}
   \iota^* \nabla_\NV(\rho) = \Rho_{00} + \frac{|\lo|^2}{n-1}.
\end{equation}
\end{lem}

\begin{proof} We calculate in geodesic normal coordinates. We expand the defining relation
\begin{equation}\label{rho-def}
   \Delta_g (\sigma) = -(n\!+\!1) \rho - \sigma \J
\end{equation}
of $\rho$ into a power series of $r$. The Laplacian takes the form $\partial_r^2 + \frac{1}{2} \tr
(h_r^{-1} h'_r) \partial_r + \Delta_{h_r}$. Hence the restriction of \eqref{rho-def} to $r=0$ implies
$2 \sigma_{(2)} + n H = -(n+1)\rho_0$, and Lemma \ref{sigma23} yields the first assertion. Next,
we restrict the derivative of \eqref{rho-def} in $r$ to $r=0$. Then
$$
   6 \sigma_{(3)} + \frac{1}{2} \partial_r (\tr(h_r^{-1} h'_r))|_0
   + \tr(h_r^{-1} h'_r)|_0 \sigma_{(2)} = -(n\!+\!1) \partial_r(\rho)|_0 - \J_0.
$$
But \eqref{h-geodesic} implies the identities
$$
   \tr(h_r^{-1} h'_r)|_0 = 2 \tr (L) \quad \mbox{and}
   \quad \partial_r (\tr(h_r^{-1} h'_r))|_0 = - 2|L|^2 - 2 \Ric_{00}.
$$
Hence the above relation transforms into
$$
   6 \sigma_{(3)} - |L|^2 - \Ric_{00} + 2 \tr(L) \sigma_{(2)}
   = -(n\!+\!1) \partial_r(\rho)|_0 - \J_0.
$$
We combine this result with \eqref{G1} for $\Ric_{00}$, \eqref{sigma3} for
$\sigma_{(3)}$ and $|L|^2 = |\lo|^2 + n H^2$ to obtain
$$
   \partial_r(\rho)|_0 = \iota^* (\J^g) - \J^h + \frac{|\lo|^2}{2(n\!-\!1)} + \frac{n}{2} H^2.
$$
Now \eqref{G2} implies the assertion.
\end{proof}

The equation \eqref{rho-1} justifies the notation $\rho_0'$ in Proposition \ref{h-low-order}.

Next, we provide an alternative proof of Lemma \ref{rho-01}. This illustrates the efficiency
of the differential equation for $\rho$ (Lemma \ref{rec-2}), the resulting recursive formula
in Proposition \ref{rec-rho} and the formula for the obstruction (Theorem \ref{obstruction-ex})
in low-order cases.

First of all, the restriction of \eqref{magic-rec} to $s=0$ yields $n \rho(0) +
\frac{1}{2} \tr (h_{(1)}) = 0$. Using $h_{(1)} = 2L$ by \eqref{h-adapted}, we get
$\rho(0) = -H$. This reproves the first part of Lemma \ref{rho-01}. Next, the
formula \eqref{rec-rho-full} for $k=1$ yields
$$
   (n-1) \rho'_0 - 2 \rho_0 \left(\frac{\mathring{v}'}{\mathring{v}}\right)|_0
  + \partial_s \left(\frac{\mathring{v}'}{\mathring{v}} \right)|_0 + \J_0 = 0
$$
or, equivalently,
\begin{equation}\label{rho-prime}
   (n-1) \rho_0' -  \rho_0 \tr (h_{(1)}) +  \tr (h_{(2)}) - \frac{1}{2} \tr (h_{(1)}^2) + \J_0 = 0.
\end{equation}
By \eqref{h-adapted} and $\tr(L^2) = \tr (\mathring{L}^2) + n H^2$, this gives
\begin{align*}
    0 & = (n-1) \rho_0' + 2n H^2 + \tr (\mathring{L}^2) - \Ric_{00} - 2 \tr (L^2) + \J_0 \\
       & = (n-1) \rho_0' - \Ric_{00} - |\mathring{L}|^2 + \J_0 \\
       & = (n-1) \rho_0' - (n-1) \Rho_{00}-  |\mathring{L}|^2 .
\end{align*}
This reproves the second part of Lemma \ref{rho-01}.

\begin{remark}\label{B1}
For $n=1$, Theorem \ref{obstruction-ex} states that
$$
   \B_1 = - \partial_s \left( \frac{\mathring{v}'}{\mathring{v}}\right)|_0
  + 2 \rho_0 \left( \frac{\mathring{v}'}{\mathring{v}}\right)|_0 - \J_0
$$
or, equivalently,
$$
   \B_1 = -\tr (h_{(2)}) + \frac{1}{2} \tr (h_{(1)}^2) + \rho_0 \tr (h_{(1)}) - \J_0.
$$
By $\lo = 0$, $\tr(h_{(1)}) = 2H$, $\tr(h_{(1)}^2) = 4 H^2$, $\tr(h_{(2)}) = -\Ric_{00}$
and $\rho_0 = -H$, this formula simplifies to
$$
    \B_1 = \Ric_{00} - \J_0 = \Ric_{00} - K_0 = 0
$$
using $\Ric = K g$, $K$ being the Gauss curvature of $g$.
\end{remark}

We continue discussing the second-order derivative of $\rho$ in general dimensions
$n \ge 3$. The arguments are parallel to the discussion in Section \ref{B2-details}.
Proposition \ref{rec-rho} for $k=2$ gives
\begin{equation*}
   -(n\!-\!2) \rho_0'' = \partial_s^2 \left( \frac{\mathring{v}'}{\mathring{v}}\right)|_0
   - 4 \rho_0 \partial_s \left( \frac{\mathring{v}'}{\mathring{v}}\right)|_0
   - 4 \rho_0' \left( \frac{\mathring{v}'}{\mathring{v}}\right)|_0 +2 \J'_0.
\end{equation*}
This formula is equivalent to
\begin{align}\label{rho-second}
   & -(n\!-\!2) \rho_0'' \notag \\
   & = \tr ( 3 h_{(3)} - 3 h_{(1)} h_{(2)} + h_{(1)}^3)
   - 4 \rho_0 \tr (h_{(2)}) + 2 \rho_0 \tr (h_{(1)}^2) - 2 \rho_0' \tr (h_{(1)}) + 2 \J_0'.
\end{align}
In order to make that formula explicit, we use Proposition \ref{h-low-order} and in particular
Corollary \ref{h3-trace}. We obtain
\begin{align}\label{rho-second-sum}
    &  -(n\!-\!2) \rho_0'' = -\nabla_{\partial_s}(\Ric)_{00} + 2 \partial_s(\J)
  - 4 \lo^{ij} R_{0ij0} + \Delta H - H \Ric_{00} - 3 H|\lo|^2  - 2n H \rho_0' \notag \\
    & - 6 \tr (L^2 \lo) + 6 L^{ij} R_{0ij0} + 8 \tr (L^3) + 4 H |\lo|^2 - 4 H \Ric_{00} - 8 H |L|^2.
\end{align}
Now, by the obvious relation
$$
   -\nabla_{\partial_s}(\Ric)_{00} + 2 \partial_s(\J) = - \nabla_{\partial_s} (G)_{00} - (n-2) \partial_s(\J)
$$
for the Einstein tensor $G = \Ric - \frac{1}{2} \scal g$ and the identity
\eqref{Einstein-00}, we get
\begin{align*}
   -(n\!-\!2) \rho_0''  & = \delta^h (\Ric(\partial_s,\cdot)) - \lo^{ij} \Ric_{ij} + (n\!+\!1) H \Ric_{00}
  - H \scal - (n\!-\!2) \J_0' - 2n H \rho_0' \notag \\
   & - 4 \lo^{ij} R_{0ij0} + \Delta H - H \Ric_{00} - 3 H|\lo|^2  \notag \\
   & - 6 \tr (L^2 \lo) + 6 L^{ij} R_{0ij0} + 8 \tr (L^3) + 4 H |\lo|^2 - 4 H \Ric_{00} - 8 H |L|^2 .
\end{align*}
The decomposition \eqref{KN} shows that
\begin{align*}
   L^{ij} R_{0ij0} & = L^{ij} (W_{0ij0} + \Rho_{ij} + \Rho_{00} g_{ij}) \\
  & =  L^{ij} W_{0ij0} + L^{ij} \Rho_{ij} + n H \Rho_{00} \\
  & = L^{ij} W_{0ij0} + \lo^{ij} \Rho_{ij} + H (\J - \Rho_{00}) + n H \Rho_{00} \\
  & = L^{ij} W_{0ij0} + \lo^{ij} \Rho_{ij} + H \J + (n-1) H \Rho_{00} \\
  & =  \lo^{ij} W_{0ij0} + \lo^{ij} \Rho_{ij} + H \Ric_{00}.
\end{align*}
Similarly, we find
$$
   \lo^{ij} R_{0ij0} =  \lo^{ij} W_{0ij0} + \lo^{ij} \Rho_{ij}.
$$
The latter two results and the formula \eqref{rho-1} for $\rho_0'$ imply
\begin{align*}
   -(n\!-\!2) \rho_0'' & = \delta^h (\Ric(\partial_s,\cdot)) - \lo^{ij} \Ric_{ij} + (n\!+\!1) H \Ric_{00}
   - H \scal - (n\!-\!2) \J_0' \notag \\
   & - 4  \lo^{ij} W_{0ij0} - 4 \lo^{ij} \Rho_{ij} + \Delta H - H \Ric_{00} - 3 H|\lo|^2  \notag \\
   & - 6 \tr (L^2 \lo) + 6 \lo^{ij} W_{0ij0} + 6 \lo^{ij} \Rho_{ij} + 6 H \Ric_{00} \notag \\
   & + 8 \tr (L^3) +  4 H |\lo|^2 - 4 H \Ric_{00} - 8 H |L|^2 - 2n H \left(\Rho_{00} + \frac{|\lo|^2}{n\!-\!1} \right).
\end{align*}
Simplification gives
\begin{align*}
   -(n\!-\!2) \rho_0'' & = \delta^h (\Ric(\partial_s,\cdot)) + \Delta H + 2  \lo^{ij} W_{0ij0} + 2 \lo^{ij} \Rho_{ij}
  - \lo^{ij} \Ric_{ij} \notag \\
   & + (n\!+\!2) H \Ric_{00} - 2n H \Rho_{00} - H \scal   \notag \\
   & + H|\lo|^2  - 8 H |L|^2 - 6 \tr (L^2 \lo) + 8 \tr (L^3) - \frac{2n}{n\!-\!1} H |\lo|^2 - (n\!-\!2) \J_0'.
\end{align*}
Further simplification using
$$
   |L|^2 = |\lo|^2 + n H^2, \quad \tr (L^2 \lo) = 2 H |\lo|^2 \quad \mbox{and}
   \quad \tr (L^3) = \tr (\lo^3) +3 H |\lo|^2 + n H^3
$$
yields
\begin{align}\label{rho2-a}
    -(n\!-\!2) \rho_0'' & = \delta^h (\Ric(\partial_s,\cdot)) + \Delta H + 2 (\lo^{ij} W_{0ij0} +  \tr (\lo^3))
    - (n\!-\!3) \lo^{ij} \Rho^g_{ij} \notag \\
    & + \frac{(n\!-\!2)(n\!+\!1)}{n\!-\!1} H \Ric_{00} - \frac{n\!-\!2}{n\!-\!1} H \scal
    + \frac{3n\!-\!5}{n\!-\!1} H |\lo|^2 - (n\!-\!2) \J_0' .
\end{align}
Next, we apply the basic identity \eqref{FW-relation}. It implies that
\begin{equation}\label{F-L-trace}
   \lo^{ij} W_{0ij0} + \tr (\lo^3) \stackrel{!}{=}  (n-2) \lo^{ij} \JF_{ij}.
\end{equation}
Now, combining the formula \eqref{ddL} for $\Delta H$ and \eqref{F-L-trace} with \eqref{rho2-a}, yields
\begin{align*}
   -(n\!-\!2) \rho_0'' & = \frac{n\!-\!2}{n\!-\!1} \delta^h (\Ric(\partial_s,\cdot)) + \frac{1}{n\!-\!1} \delta \delta (\lo)
   + 2 (n\!-\!2) \lo^{ij} \JF_{ij} -  (n\!-\!3) \lo^{ij} \Rho^g_{ij} \\
   & + \frac{(n\!-\!2)(n\!+\!1)}{n\!-\!1} H \Ric_{00} - \frac{n\!-\!2}{n\!-\!1} H \scal
   + \frac{3n\!-\!5}{n\!-\!1} H |\lo|^2 - (n\!-\!2) \J_0'.
\end{align*}
Hence
\begin{align*}
   \rho_0'' & = - \frac{1}{(n\!-\!1)(n\!-\!2)} \delta \delta (\lo) - \frac{1}{n\!-\!1} \delta^h (\Ric(\partial_s,\cdot))
  - 2 \lo^{ij} \JF_{ij} + \frac{n\!-\!3}{n\!-\!2} \lo^{ij} \Rho^g_{ij} \\
  & - \frac{n\!+\!1}{n\!-\!1} H \Ric_{00} + \frac{1}{n\!-\!1} H \scal
  -  \frac{3n\!-\!5}{(n\!-\!1)(n\!-\!2)} H |\lo|^2 + \J_0' .
\end{align*}
Finally, we use the defining relation $\iota^* \Rho^g = \JF + \Rho^h - H \lo - 1/2 h H^2$ of $\JF$
to replace in this identity the Schouten tensor $\Rho^g$ by the Schouten tensor $\Rho^h$. Thus,
we have proved the following result.

\begin{lem}\label{rho-two} For $n \ge 3$, it holds
\begin{align}\label{rho-2-ex}
   \rho_0'' & = - \frac{1}{(n\!-\!1)(n\!-\!2)} \delta \delta (\lo) - \frac{1}{n\!-\!1} \delta^h (\Ric^g(\partial_s,\cdot))
   - \frac{n\!-\!1}{n\!-\!2} (\lo,\JF) + \frac{n\!-\!3}{n\!-\!2} (\lo,\Rho^h) \notag \\
   & - \frac{n\!+\!1}{n\!-\!1} H \Ric_{00} + \frac{1}{n\!-\!1} H \scal - \frac{n\!+\!1}{n\!-\!1} H |\lo|^2 + \J_0'.
\end{align}
\end{lem}

In Section \ref{Q-low}, we shall connect this result with the holographic formula for $\QC_3$.

\begin{remark}\label{GW-rho-2}
Lemma \ref{rho-two} is equivalent to \cite[Remark 6.9]{GW-LNY} and \cite[Remark 3.11]{GW-Willmore},
up to the sign of the term $(\lo,\Rho^h)$. In fact, the quoted result calculates the restriction of 
$\nabla_{\NV}^2(\rho)$ to $M$. By Example \ref{trans-low-2}, it corresponds to $\rho''_0+ 2 H \rho_0'$. 
But \eqref{rho-2-ex} implies
\begin{align*}
   \rho''_0+ 2 H\rho_0' & = (\cdots)  -\frac{n\!+\!1}{n\!-\!1} H \Ric_{00} + \frac{1}{n\!-\!1} H \scal
   - \frac{n\!+\!1}{n\!-\!1} H |\lo|^2 + \J_0' \\
   & + 2 H \left(\frac{1}{n\!-\!1} \Ric_{00} - \frac{1}{n\!-\!1} \J_0 + \frac{|\lo|^2}{n\!-\!1}\right)  + \J_0' \\
   & = (\cdots) - H \Ric_{00} + 2 H \J_0 - H|\lo|^2 + \J_0' \\
   & = (\cdots) - H ((n\!-\!1) \Rho_{00} + |\lo|^2) + H \J_0 + \J_0',
\end{align*}
where $(\cdots)$ indicates the four terms in the first line of \eqref{rho-2-ex}. This proves the claim. The
present alternative proof rests on the recursive formula in Proposition \ref{rec-rho} and the explicit
formula for $h_{(3)}$ in Proposition \ref{h-low-order}.
\end{remark}

\begin{remark}
The proof of Proposition \ref{h-low-order} shows that the calculation of the coefficient $h_{(2)}$ involves
$h_{(1)}$ and $\rho_0$. Once $h_{(2)}$ has been determined, we calculated $\rho_0'$ using \eqref{rho-prime}.
Similarly, the calculation of $h_{(3)}$ in the proof of Lemma \ref{h-low-order} involves the lower-order coefficients
of $h_s$ and $\rho_0$, $\rho_0'$. Once $h_{(3)}$ has been determined, the recursive formula \eqref{rho-second}
yields $\rho_0''$. A continuation of that iterative process yields the higher-order coefficients, at least in principle.
\end{remark}

We finish this section with some comments concerning the functions $\rho$ and the
singular Yamabe obstructions in Examples \ref{ball}--\ref{subsphere}. In Example
\ref{ball}, it holds $\mathring{v}'/\mathring{v} = - n/(1-2s)$. Hence the
differential equation \eqref{magic-rec} reduces to
$$
   - s \rho' + n \rho - n(1-2s \rho)/(1-2s) = 0.
$$
One readily checks that $\rho=1$ is the unique solution with initial value $\rho(0)
= 1$. The vanishing of the obstruction is reproduced by the formula
\eqref{obstruction-magic}. Indeed, we obtain
$$
   (n+1)! \B_n = 2n \partial_s^n (1/(1-2s))|_0 - 4 n^2 \partial_s^{n-1} (1/(1-2s))|_0 =
   2n 2^n n! - 4 n^2 2^{n-1} (n-1)! = 0.
$$
Similarly, in Example \ref{subsphere}, it holds $\mathring{v}'/\mathring{v} = - ns/(1-s^2)$ and
the differential equation \eqref{magic-rec} reduces to
$$
   - s\rho' + n\rho - n s (1-2s\rho)/(1-s^2) + \tfrac{n+1}{2} s = 0.
$$
One easily checks that $\rho = \frac{1}{2} s$ is the unique solution with the
initial value $\rho(0)=0$. Again, the vanishing of the obstruction is reproduced by
$$
   (n+1)! \B_n = 2 \partial_s^n(ns/(1-s^2))|_0 - 4 \binom{n}{2} \partial_s^{n-2}(ns/(1-s^2))|_0
  = 2n n! - 4n \binom{n}{2} (n-2)! = 0
$$
for odd $n$ and trivially for even $n$.

\subsection{The obstruction $\B_2$ for general backgrounds}\label{B2-details}


In the present section, we prove the equivalence of both formulas for the singular
Yamabe obstruction $\B_2$ displayed in \eqref{B2} and derive the second of these
formulas from Theorem \ref{B-form}.

The Codazzi-Mainardi identity states that
\begin{equation}\label{CME}
   \nabla^h_Y (L)(X,Z) - \nabla^h_X(L)(Y,Z) = R^g (X,Y,Z,N)
\end{equation}
for $X,Y,Z \in \mathfrak{X}(M)$ if $L(X,Y) = - h(\nabla^g_X(Y),N)$ for some unit
normal vector field $N$ \cite[Theorem 1.72]{Besse}. Now, we decompose the curvature
tensor $R$ as
\begin{align}\label{KN}
   & R(X,Y,Z,W) = W(X,Y,Z,W)  \\
   & + \Rho(Y,Z) g (X,W) - \Rho(X,Z) g(Y,W) - \Rho(Y,W)g(X,Z) + \Rho(X,W)g(Y,Z), \notag
\end{align}
where $W$ is the trace-free Weyl tensor (see \eqref{RW-deco}), and take traces in
\eqref{CME} in the arguments $X,Z$. Then
$$
   n \langle dH, Y \rangle - \langle \delta^h (L), Y \rangle  = -(n-1) \Rho(N,Y).
$$
Hence
$$
   \delta^h (\lo) - (n-1) dH = (n-1) \Rho(N,\cdot)
$$
(see
[Lemma 6.25.2]{J1}). Now, taking a further divergence, yields
\begin{equation}\label{ddL}
   \delta^h \delta^h (\lo) - (n-1) \Delta H = (n-1) \delta^h (\Rho(N,\cdot)).
\end{equation}
For $n=2$, this proves the equivalence of both formulas in \eqref{B2}.

We continue by showing that the second formula for $\B_2$ in \eqref{B2} is a special
case of Theorem \ref{B-form}. First, we observe that the formula
$$
    - 3 \B_2 = \partial_s^2 \left( \frac{\mathring{v}'}{\mathring{v}} \right)|_0
    - 4 \rho_0  \partial_s \left (\frac{\mathring{v}'}{\mathring{v}} \right)|_0
    - 4 \rho_0' \left(\frac{\mathring{v}'}{\mathring{v}}\right)|_0 + 2 \J_0'
$$
in Theorem \ref{B-form} is equivalent to
$$
   -3 \B_2  =  \tr ( 3 h_{(3)} - 3 h_{(1)} h_{(2)} + h_{(1)}^3)
   - 4 \rho_0 \tr (h_{(2)}) + 2 \rho_0 \tr (h_{(1)}^2) - 2 \rho_0' \tr (h_{(1)}) + 2 \J_0'.
$$
In order to make that sum explicit, we use the results in Proposition \ref{h-low-order}.
By Corollary \ref{h3-trace} for $n=2$ and $\rho_0 = - H$, we obtain
\begin{align}\label{B2-sum}
    -3 \B_2  & = -\nabla^g_{\partial_s}(\Ric)_{00} + 2 \partial_s(\J)
    - 4 \lo^{ik} R_{0ik0} + \Delta H - H \Ric_{00} - 3 H|\lo|^2  \\
    & - 6 \tr (L^2 \lo) + 6 L^{ij} R_{0ij0} + 8 \tr (L^3) + 4 H |\lo|^2 - 4 H \Ric_{00} - 8 H |L|^2 - 4 H \rho_0'. \notag
\end{align}

Now we prove the identity
\begin{equation}\label{Einstein-00}
   \nabla^g_{\partial_s}(G)_{00}
   = -\delta^h (\Ric(\partial_s,\cdot)) + \lo^{ij} \Ric_{ij} - (n+1) H \Ric_{00} + H \scal
\end{equation}
for the Einstein tensor $G \st \Ric - \frac{1}{2} \scal g$. Note that $G= \Ric - 2
\J g$ in dimension $n=3$. It is well-known that the second Bianchi identity implies the
relation $2 \delta^g (\Ric) = d\scal$. Hence
\begin{align*}
   \nabla^g_{\partial_s}(\Ric)(\partial_s,\partial_s)
   & = \delta^g(\Ric)(\partial_s) - g^{ij} \nabla^g_{\partial_i}(\Ric)(\partial_j,\partial_s) \\
   & =  \frac{1}{2} \langle d \scal,\partial_s \rangle \\
   & - g^{ij} \partial_i (\Ric(\partial_j,\partial_s))
   + g^{ij} \Ric(\nabla^g_{\partial_i}(\partial_j),\partial_s) + g^{ij} \Ric(\partial_j,\nabla^g_{\partial_i}(\partial_s)) \\
   & = \frac{1}{2} \langle d \scal,\partial_s \rangle \\ & - h^{ij} \partial_i (\Ric(\partial_j,\partial_s))
   + h^{ij} \Ric (\nabla^h_{\partial_i}(\partial_j) - L_{ij} \partial_s,\partial_s)
   + h^{ij} \Ric(\partial_j,\nabla^g_{\partial_i}(\partial_s))  \\
   & = \frac{1}{2} \langle d \scal,\partial_s \rangle - \delta^h (\Ric(\partial_s,\cdot)) - n H \Ric_{00}
   + h^{ij} \Ric(\partial_j,\nabla^g_{\partial_i}(\partial_s))
\end{align*}
on $M$. Therefore, using $\nabla_{\partial_i}^g(\partial_s) = L_i^k \partial_k$, we obtain
\begin{align*}
   \nabla^g_{\partial_s}(G)_{00} & = -\delta^h (\Ric(\partial_s,\cdot)) - n H \Ric_{00} + h^{ij} L_i^k \Ric_{jk} \\
   & = -\delta^h (\Ric(\partial_s,\cdot)) - n H\Ric_{00} + L^{ij} \Ric_{ij}.
\end{align*}
This implies \eqref{Einstein-00}.

Next, we observe that the decomposition \eqref{KN} yields
\begin{align*}
   L^{ij} R_{0ij0} & = L^{ij} (W_{0ij0} + \Rho_{ij} + \Rho_{00} g_{ij}) \\
   & = L^{ij} \Rho_{ij} + 2 H \Rho_{00} \\
   & = \lo^{ij} \Rho_{ij} + H (\J - \Rho_{00}) + 2 H \Rho_{00} = \lo^{ij} \Rho_{ij} + H \J + H \Rho_{00} \\
   & = \lo^{ij} \Rho_{ij} + H \Ric_{00}
\end{align*}
since the Weyl tensor $W$ vanishes in dimension $3$. Similarly, we find
\begin{equation*}
   \lo^{ij} R_{0ij0} = \lo^{ij} \Rho_{ij}.
\end{equation*}

These results and the formula \eqref{rho-01} for $\rho_0'$ show that the sum
\eqref{B2-sum} simplifies to (recall that $n=2$)
\begin{align*}
   & \delta^h (\Ric(\partial_s,\cdot)) - \lo^{ij} \Ric_{ij} + 3 H \Ric_{00}
  - H \scal - 4 \lo^{ij} \Rho_{ij} + \Delta H - H \Ric_{00} - 3 H |\lo|^2 \\
   & - 6 \tr (L^2 \lo) + 6 \lo^{ij} \Rho_{ij} + 6 H \Ric_{00} + 8 \tr(L^3)
   + 4 H |\lo|^2 - 4 H \Ric_{00} - 8 H |L|^2 \\
   & - 4H \Ric_{00} + H \scal - 4 H |\lo|^2 \\
   & =  \delta^h (\Ric(\partial_s,\cdot)) + \Delta H + \lo^{ij} \Ric_{ij},
\end{align*}
up to terms that are at least quadratic in $L$. In order to deal with these terms, we note that
$$
   |L|^2 = |\lo|^2 + 2 H^2, \quad \tr (L^2 \lo) = 2 H |\lo|^2 \quad \mbox{and}
  \quad \tr (L^3) = 3 H |\lo|^2 + 2 H^3
$$
using $\tr(\lo^3)=0$. It follows that the remaining terms are
$(- 3 - 12 + 24 + 4 - 8 - 4) H |\lo|^2 = H |\lo|^2$. Summarizing we obtain
$$
   -3 \B_2 = \delta^h (\Ric(\partial_s,\cdot)) + \Delta H + \lo^{ij} \Ric_{ij}+ H |\lo|^2.
$$
This proves the second formula for $\B_2$ in \eqref{B2}.

\begin{remark}\label{GW-myst}
The identity \eqref{Einstein-00} is equivalent to \cite[Lemma B.7]{GW-LNY} or
\cite[Lemma A.7]{GW-Willmore}. In addition to the derivation of the above formula
for the singular Yamabe obstruction $\B_2$, in these references, this crucial identity is used to
prove a formula for the second normal derivative of $\rho$ and a formula for
$\QC_3$. The calculations lead to the same results as here. Our discussion of the second normal
derivative of $\rho$ and the explicit formula for $\QC_3$ is contained in Sections
\ref{TC+B} and \ref{Q-low}. It uses the identity \eqref{Einstein-00}.
\end{remark}

\subsection{The obstruction $\B_3$ for conformally flat backgrounds}\label{B3-flat-back}

We first evaluate Theorem \ref{B-form} for the obstruction $\B_3$ of a three-manifold $M \hookrightarrow \R^4$. 
The resulting formula is equivalent to a formula in \cite{GW-LNY} (see Lemma \ref{B-relation}). This will imply a 
formula for $\B_3$ in a conformally flat background.

\begin{prop}\label{B3-evaluation}
The Yamabe obstruction of a hypersurface $(M^3,h)$ in the four-dimensional flat space $\R^4$ is given by the formula
\begin{equation}\label{B3-final}
   12 \B_3 = \Delta (|\lo|^2) + 12 |dH|^2 + 6 (\lo,\Hess (H)) - 2 |\lo|^4 + 6 \tr (\lo^4) + 6 H \tr (\lo^3)
\end{equation}
or, equivalently,
\begin{equation}\label{B3-final2}
   12 \B_3 = \Delta (|\lo|^2) + 12 |dH|^2 + 6 (\lo,\Hess (H)) + |\lo|^4 + 6 H \tr (\lo^3).
\end{equation}
Here the Hessian, the Laplacian, scalar products, and traces are taken with respect to the metric $h$ on $M$.
\end{prop}

Newton's identity
$$
   24 \sigma_4 (L) = \tr(L)^4 - 6 \tr(L)^2 |L|^2 + 3 |L|^4 + 8 \tr(L) \tr(L^3) - 6 \tr(L^4)
$$
for the elementary symmetric polynomial $\sigma_4(L)$ of the eigenvalues of $L$ (or rather of the shape
operator) implies
\begin{align*}
   24 \sigma_4(L) & = n(n-1)(n-2)(n-3) H^4 - 6 (n-2)(n-3) H |\lo|^2 + 8(n-3) H \tr(\lo^3) \\
   & + 3 (|\lo|^4 - 2 \tr(\lo^4)).
\end{align*}
Since $\sigma_4(L) = 0$ in dimension $n=3$, we get

\begin{cor}\label{Newton-C3} $|\lo|^4 = 2 \tr(\lo^4)$ for $n=3$.
\end{cor}

Corollary \ref{Newton-C3} implies the equivalence of \eqref{B3-final} and \eqref{B3-final2}.

We derive formula \eqref{B3-final} as a consequence of the identity
\begin{equation}\label{3-flat}
   12\B_3 = - \partial_s^3 \left( \frac{\mathring{v}'}{\mathring{v}}\right)|_0
   + 6   \rho_0 \partial_s^2 \left( \frac{\mathring{v}'}{\mathring{v}}\right)|_0
   + 12 \rho_0'  \partial_s \left( \frac{\mathring{v}'}{\mathring{v}}\right)|_0
   + 6 \rho_0'' \left( \frac{\mathring{v}'}{\mathring{v}}\right)|_0 - 3 \partial_s^2(\bar{\J})|_0
\end{equation}
in Theorem \ref{B-form}. Note that, for a flat background, the last term vanishes. By the relations
\begin{align*}
    \left( \frac{\mathring{v}'}{\mathring{v}}\right)|_0 & = \frac{1}{2} \tr (h_{(1)}), \\
    \partial_s \left( \frac{\mathring{v}'}{\mathring{v}}\right)|_0  & = \frac{1}{2} \tr (2 h_{(2)} - h_{(1)}^2), \\
    \partial_s^2 \left( \frac{\mathring{v}'}{\mathring{v}}\right)|_0  & = 2! \frac{1}{2} \tr (3 h_{(3)}
    - 3 h_{(1)} h_{(2)} + h_{(1)}^3), \\
    \partial_s^3 \left( \frac{\mathring{v}'}{\mathring{v}}\right)|_0 & = 3! \frac{1}{2} \tr (4 h_{(4)} - 4 h_{(1)} h_{(3)}
   - 2 h_{(2)}^2 + 4 h_{(1)}^2 h_{(2)} - h_{(1)}^4),
\end{align*}
the identity \eqref{3-flat} is equivalent to
\begin{align}\label{B3-flat}
 12 \B_3 & = - \tr( 12 h_{(4)} - 12 h_{(1)} h_{(3)} - 6 h_{(2)}^2 + 12 h_{(1)}^2 h_{(2)} - 3 h_{(1)}^4) \notag \\
  & + 6 \rho_0  \tr (3 h_{(3)} - 3 h_{(1)} h_{(2)} + h_{(1)}^3) \notag \\
  & + 6 \rho_0'  \tr (2 h_{(2)} - h_{(1)}^2) \notag \\
  & + 3 \rho_0''  \tr (h_{(1)}).
\end{align}
In order to evaluate that sum, we use the formulas for $h_{(k)}$ ($k \le 3$) in
Proposition \ref{h-low-order} and the formulas for the first two normal derivatives
of $\rho$ in Section \ref{TC+B}. In addition, it remains to determine the
coefficient $\tr(h_{(4)})$. The following result even provides a closed formula for
the Taylor coefficient $h_{(4)}$ of $h_s$ (for a flat background).

\begin{lem}\label{h4} Let $n=3$. Then
\begin{align}\label{h4-ex}
    12 h_{(4)}  & = - 9 H \Hess (H) + L \Hess (H) + \Hess (H) L - 6 dH \otimes dH - \gamma \notag \\
    & - \Hess (|\lo|^2) + 4 \lo^2 |\lo|^2 + 15 H^2 ({(\lo^2)}_\circ + H \lo) - H \lo |\lo|^2 + 3 L \rho_0'',
\end{align}
where
\begin{equation}\label{gamma-form}
    \gamma_{jk} \st  2 h^{lm} (\nabla_j(L)_{km} + \nabla_k(L)_{jm} - \nabla_m(L)_{jk}) \partial_l(H)
\end{equation}
and
\begin{equation}\label{rho-pp}
    \rho_0'' = - \Delta (H) - 2 \tr (\lo^3) - 2 H |\lo|^2.
\end{equation}
\end{lem}

\begin{proof}
We use the same notation as in the proof of Proposition \ref{h-low-order}. We expand the curvature
components $R_{0jk}^0$ of the metric
$$
   \eta^*(g) = a^{-1} ds^2 + h + h_{(1)} s + h_{(2)} s^2 + h_{(3)} s^3 + h_{(4)} s^4 + \cdots
$$
into power series of $s$. We recall that $a = \eta^*(|\NV|^2)$. In order to simplify the notation, we
write $g$ for the metric $\eta^*(g)$ and identify $\eta^*(\rho)$ with $\rho$. By assumption, it holds
$a = 1- 2s \rho$. As usual, the $0$-components refer to $\partial_s$. Since $g$ is flat, the components
$R_{0jk}^0$ vanish. Now, for the above metric, we find the Christoffel symbols
$$
   \Gamma_{ij}^0 = - \frac{1}{2} g^{00} g_{ij}',
   \quad  \Gamma_{0j}^0 = \frac{1}{2} g^{00} \partial_j (g_{00}),
   \quad \Gamma_{00}^0 = \frac{1}{2} g^{00} g_{00}',
   \quad \Gamma_{0k}^l = \frac{1}{2} g^{rl} g_{kr}',
$$
where $'$ denotes the derivative in $s$. We recall that $g^{00} = a = 1- 2s \rho$ (by assumption) and
$g_{00} = a^{-1} = 1 + 2s\rho + \cdots$. Hence
\begin{align}\label{R-adapted-2}
   0 \stackrel{!}{=} R_{0jk}^0  & = \frac{1}{2} \Gamma_{jk}^l g^{00} (g_{00})_l
  + \frac{1}{4} g^{00} g^{rl} g_{kr}' g_{jl}' - \frac{1}{2} ((g^{00})' g_{jk}' + g^{00} g_{jk}'') \notag \\
  & - \frac{1}{2} ((g^{00})_j (g_{00})_k + g^{00} (g_{00})_{kj})
  - \frac{1}{4} (g^{00})^2 g_{jk}' g_{00}' - \frac{1}{4} (g^{00})^2 (g_{00})_j (g_{00})_k.
\end{align}
Now we display the Taylor expansions of all $6$ terms in \eqref{R-adapted-2} up to order $s^2$.
Using these results, it is easy to see that the coefficients of $s$ reproduce the result of
the earlier calculation of $h_{(3)}$ in the proof of Proposition \ref{h-low-order} (Remark \ref{h3-rev}).

First, we observe that the Christoffel symbols $\Gamma_{jk}^l$ for $g$ restrict to the Christoffel
symbols $\Gamma_{jk}^l$ for $h$. Moreover, we recall the general variation formula
$$
   \delta (\Gamma_{jk}^l) = \frac{1}{2} g^{lm} (\nabla_j (\delta(g))_{km}
   + \nabla_k (\delta(g))_{jm} - \nabla_{m}(\delta(g))_{jk}).
$$
Let $(\Gamma_{jk}^l)' = \delta (\Gamma_{jk}^l)$ denote the variation of the Christoffel symbols for the variation
$g = h + 2 L s + \cdots$. By $g_{00} = 1 + 2 s \rho + \cdots$ and $\rho_0 = - H$, it follows that the first term in \eqref{R-adapted}
contributes by
\begin{align}\label{term0}
   s^2 \frac{1}{2} (\Gamma_{jk}^l)' (-2 \partial_l(H)) & = - s^2 (\Gamma_{jk}^l)' \partial_l(H) \notag \\
   & = - s^2 h^{lm}  (\nabla_j(L)_{km} + \nabla_k(L)_{jm} - \nabla_{m}(L)_{jk}) \partial_l(H) \notag \\
   & = - s^2 \frac{1}{2} \gamma_{jk}.
\end{align}
The remaining contributions of the first term in \eqref{R-adapted} match with the contributions by the second part
of the fourth term to
\begin{align}\label{term1}
   -\frac{1}{2} s [\Hess (2\rho_0)] - \frac{1}{2} s^2 \left[\Hess (2\rho_0' + 4 \rho_0^2)
   - 2 \rho_0 \Hess (2\rho_0) \right] + \cdots
\end{align}
with $\Hess$ defined with respect to $h$.  Next, the second term contributes by $1/4$ times\footnote{Here we
use the fact that $[h_{(1)},h_{(2)}] = 0$ for a flat background.}
\begin{align}\label{term2}
  h_{(1)}^2  & + s \big[4h_{(1)} h_{(2)} - h_{(1)}^3 - 2 h_{(1)}^2 \rho_0\big] \notag \\
  & + s^2 \big[3 h_{(1)} h_{(3)} + 3 h_{(3)} h_{(1)} + 4 h_{(2)}^2 - 5 h_{(1)}^2 h_{(2)} + h_{(1)}^4  \notag \\
  & + 2 h_{(1)}^3 \rho_0 - 4 h_{(1)} h_{(2)} \rho_0 - 4 h_{(2)} h_{(1)} \rho_0 - 2 h_{(1)}^2 \rho_0'\big] + \cdots.
\end{align}
Finally, we find
\begin{itemize}
\item Term three contributes by $-1/2$ times
\begin{align}\label{term3}
   2 h_{(2)} - 2 h_{(1)} \rho_0 & + s \big[6h_{(3)} - 8h_{(2)} \rho_0 - 4 h_{(1)} \rho_0'\big] \notag\\
   & + s^2 \big[12 h_{(4)} - 18 h_{(3)}\rho_0 - 12 h_{(2)} \rho_0' - 3 h_{(1)} \rho_0''\big] + \cdots
\end{align}
\item The first part of term four contributes by $-1/2$ times
\begin{equation}\label{term4}
   s^2 (-4) \partial_j(\rho_0) \partial_k(\rho_0)
\end{equation}
\item Term five contributes by $-1/4$ times
\begin{equation}\label{term5}
   2 h_{(1)} \rho_0 + s \big[4 h_{(2)} \rho_0 + 4 h_{(1)} \rho_0' \big]
   + s^2 \left[ 6 h_{(3)} \rho_0 + 8 h_{(2)} \rho_0' + 3 h_{(1)} \rho_0'' \right] + \cdots
\end{equation}
\item Term six contributes by $-1/4$ times
\begin{equation}\label{term6}
     s^2 4  \partial_j(\rho_0) \partial_k(\rho_0)
\end{equation}
\end{itemize}

Summarizing the coefficients of $s^2$ in \eqref{term0}-\eqref{term6} we obtain
\begin{align}\label{h4-ev}
   0 = & - \frac{1}{2} (\gamma + \Hess (2\rho_0' + 4 \rho_0^2) - 4 \rho_0 \Hess (\rho_0)) \notag \\
   & + \frac{1}{4} \big[3 h_{(1)} h_{(3)} + 3 h_{(3)} h_{(1)} + 4 h_{(2)}^2 -  5 h_{(1)}^2 h_{(2)} + h_{(1)}^4
  + 2 h_{(1)}^3 \rho_0 - 8 h_{(1)} h_{(2)} \rho_0 - 2 h_{(1)}^2 \rho_0'\big]  \notag\\
   & - \frac{1}{2} \big[12 h_{(4)} - 18 h_{(3)}\rho_0 - 12 h_{(2)} \rho_0' - 3 h_{(1)} \rho_0'' \big] \notag\\
   & - \frac{1}{4} \left[ 6 h_{(3)} \rho_0 + 8 h_{(2)} \rho_0' + 3 h_{(1)} \rho_0'' \right] \notag\\
   & + 2 d \rho_0 \otimes d \rho_0 \notag\\
   & - d  \rho_0 \otimes d\rho_0.
\end{align}
Note that the third and the fourth line can be summarized to
$$
   -\frac{1}{4} (24 h_{(4)} - 30 h_{(3)} \rho_0 - 16 h_{(2)} \rho_0' - 3 h_{(1)} \rho_0'').
$$
Therefore, we find the formula
\begin{align}\label{h4-gen}
   24 h_{(4)} & = - 4 \Hess (\rho_0') - 8 \Hess (\rho_0^2) + 8 \rho_0 \Hess(\rho_0) - 2 \gamma \notag \\
   & + \big[3 h_{(1)} h_{(3)} + 3 h_{(3)} h_{(1)} + 4 h_{(2)}^2 -  5 h_{(1)}^2 h_{(2)} + h_{(1)}^4
  + 2 h_{(1)}^3 \rho_0 - 8 h_{(1)} h_{(2)} \rho_0 - 2 h_{(1)}^2 \rho_0'\big]  \notag \\
  & +30 h_{(3)} \rho_0 + 16 h_{(2)} \rho_0' + 3 h_{(1)} \rho_0'' \notag \\
  & + 4 d \rho_0 \otimes d \rho_0.
\end{align}

Now we apply the known formulas for $h_{(j)}$ ($j \le 3$) (Proposition \ref{h-low-order}) and
$\rho_0,\rho_0',\rho_0''$ (Lemma \ref{rho-01}, Lemma \ref{rho-two}) to make that sum fully explicit.
First, we prove the remarkable simplification
\begin{align}\label{rem-sim}
   & 3 h_{(1)} h_{(3)} + 3 h_{(3)} h_{(1)} + 4 h_{(2)}^2 -  5 h_{(1)}^2 h_{(2)} + h_{(1)}^4
  + 2 h_{(1)}^3 \rho_0 - 8 h_{(1)} h_{(2)} \rho_0 - 2 h_{(1)}^2 \rho_0' \notag \\
   & = 2 L \Hess(H) + 2 \Hess(H) L.
\end{align}
In fact, we calculate
\begin{align*}
  & 3 h_{(1)} h_{(3)} + 3 h_{(3)} h_{(1)} \\
  & = 2 L \Hess (H) + 2 \Hess(H) L - 6 H L^2 \lo + 2 L^2 |\lo|^2 - 6 H L \lo L + 2 L^2 |\lo|^2
\end{align*}
and
\begin{align*}
   2 h_{(1)}^3 \rho_0 - 8h_{(1)} h_{(2)} \rho_0 - 2 h_{(1)}^2 \rho_0' = -16 H L^3 + 16 H L^2 \lo - 4 L^2 |\lo|^2.
\end{align*}
The sum of these two results gives
$$
   2 L \Hess (H) + 2 \Hess(H) L + 4 H (\lo^3 + 2 H \lo^2 + H^2 \lo) - 16 H (\lo^3 + 3 H \lo^2 + 3 H^2 \lo + H^3 \id).
$$
Moreover, we get
\begin{align*}
   & 4 h_{(2)}^2 -  5 h_{(1)}^2 h_{(2)} + h_{(1)}^4  \\
   & = 4 \lo^4 + 8 H \lo^3 + 4 H^2 \lo^2 - 20 \lo^4-60 H \lo^3 - 60 H^2 \lo^2 - 20 H^3 \lo \\
   & + 16 \lo^4 + 64 H \lo^3 + 96 H^2 \lo^2 + 64 H^3 \lo + 16 H^4 \id \\
   & = 12 H \lo^3 + 40 H^2 \lo^2 + 44 H^3 \lo + 16 H^4 \id.
\end{align*}
Summing these identities proves \eqref{rem-sim}.

The above results imply
\begin{align*}
   12 h_{(4)} & = - \Hess (|\lo|^2)  - 4 \Hess (H^2) + 4 H \Hess (H) + 2 d H \otimes dH - \gamma \\
   & + L \Hess(H) + \Hess(H) L + \alpha \\
   & = - \Hess (|\lo|^2) + L \Hess(H) + \Hess(H) L - 8 H \Hess (H) - 6 dH \otimes dH - \gamma + \alpha,
\end{align*}
where
\begin{align*}
   & \alpha \st 15 h_{(3)} \rho_0 + 8 h_{(2)} \rho_0' + 3/2 h_{(1)} \rho_0'' \\
   & = - 5 H \Hess (H) + 15 H^2 L \lo - 5 H L |\lo|^2 + 4 L \lo |\lo|^2 + 3 L \rho_0'' \\
   & = - 5 H \Hess (H) + 15 H^2 \lo^2 + 15 H^3 \lo - 5 H \lo |\lo|^2 - 5 H^2 |\lo|^2 \id + 4 \lo^2 |\lo|^2
   + 4 H \lo |\lo|^2 \\
   & + 3 L \rho_0'' \\
   & = -5 H \Hess (H) + 4 \lo^2 |\lo|^2 + 5 H^2 (3 \lo^2 - |\lo|^2 \id) - H \lo |\lo|^2 + 15 H^3 \lo + 3 L \rho_0''.
\end{align*}
Summarizing the last two results implies the first assertion.

The formula for $\rho_0''$ in $n=3$ is a direct consequence of the formula for $\rho_0''$ for general $n$
(Lemma \ref{rho-two}) using $(\lo,\JF) = \tr (\lo^3)$ and $\delta \delta (\lo) = 2 \Delta (H)$ (by Codazzi-Mainardi).
\end{proof}

Lemma \ref{h4} implies

\begin{cor}\label{trace-h4}
$$
   12 \tr (h_{(4)}) = - \Delta (|\lo|^2) -  9 H \Delta (H) + 2 (L,\Hess(H)) - 12 |dH|^2 + 4 |\lo|^4 + 9 H \rho_0''.
$$
\end{cor}

\begin{proof} It only remains to prove that
\begin{equation}\label{trace-gamma}
   \tr (\gamma) = 6 (dH,dH).
\end{equation}
But
\begin{align*}
   \tr (\gamma) & = 2 h^{jk} h^{lm} (\nabla_j(L)_{km} + \nabla_k(L)_{jm} - \nabla_m(L)_{jk}) \partial_l(H) \\
   & = 2 h^{lm}  (\delta(L)_m + \delta (L)_m - \nabla_m (\tr(L))) \partial_l(H) \\
   & = 6 h^{lm} \partial_m(H) \partial_l(H) \\
   & = 6 (dH,dH)
\end{align*}
by $\delta(L) = 3 dH$ (Codazzi-Mainardi). The proof is complete.
\end{proof}

We proceed with the evaluation of \eqref{B3-flat}.

\begin{lem}\label{h4-supp}
\begin{align*}
   & \tr(12 h_{(1)} h_{(3)} + 6 h_{(2)}^2 - 12 h_{(1)}^2 h_{(2)} + 3 h_{(1)}^4) \\
   & = 8 (L,\Hess(H)) + 6 \tr (\lo^4) + 8 |\lo|^4 + 36 H \tr (\lo^3) + 126 H^2 |\lo|^2 +  144 H^4.
\end{align*}
\end{lem}

\begin{proof} By the known formulas for the coefficients $h_{(k)}$ for $k \le 3$, we find
\begin{align*}
  & \tr(12 h_{(1)} h_{(3)} + 6 h_{(2)}^2 - 12 h_{(1)}^2 h_{(2)} + 3 h_{(1)}^4) \\
  & = 8 \tr (L \Hess (H)) - 24 H \tr (L^2 \lo) + 8 \tr(L^2) |\lo|^2 + 6 \tr (L^2 \lo^2)
  - 48 \tr (L^3 \lo) + 48 \tr (L^4).
\end{align*}
The latter sum equals the sum of $8 \tr (L \Hess (H))$ and
\begin{align*}
   & -24 H \tr (\lo^3 + 2H \lo^2) + 8 \tr(\lo^2 + 2H \lo + H^2 \id) |\lo|^2 + 6 \tr( (\lo^2 + 2H \lo + H^2) \lo^2) \\
   & - 48 \tr (\lo^4 + 3 H \lo^3 + 3 H^2 \lo^2) + 48 \tr(\lo^4 + 4 H \lo^3 + 6 H^2 \lo^2 + H^4 \id).
\end{align*}
The result follows by simplification.
\end{proof}

The following result evaluates the lower-order terms in \eqref{B3-flat}.

\begin{lem}\label{B3-flat-low}
\begin{align*}
  & 6 \rho_0  \tr (3 h_{(3)} - 3 h_{(1)} h_{(2)} + h_{(1)}^3) + 6 \rho_0'  \tr (2 h_{(2)} - h_{(1)}^2)
  + 3 \rho_0''  \tr (h_{(1)}) \notag \\
  & = - 6 H \Delta (H) -  6 |\lo|^4 - 12 H \tr (\lo^3) - 108 H^2 |\lo|^2 - 144 H^4 + 18 H \rho_0''.
\end{align*}
\end{lem}

\begin{proof} By $3 \tr (h_{(3)}) = \Delta (H)$, the sum equals
\begin{align*}
   & -6 H (\Delta (H) - 6 \tr (L^2 \lo) + 8 \tr (L^3)) + 6 |\lo|^2 \tr (L \lo - 2 L^2) + 18 H \rho_0'' \\
   & = -6 H \Delta (H) + 36 H \tr (L^2 \lo) - 48 H \tr (L^3) - 6 |\lo|^4  - 36 H^2 |\lo|^2 + 18 H \rho_0'' \\
   & = -6 H \Delta (H) + 36 H (\tr(\lo^3) + 2H |\lo|^2) - 48 ( H\tr(\lo^3) + 3H^2 |\lo|^2 + 3 H^4) \\
   & - 6 |\lo|^4  -36 H^2 |\lo|^2 + 18 H \rho_0''.
\end{align*}
Simplification completes the proof.
\end{proof}

Now we summarize the above results. We obtain
\begin{align*}
   12 \B_3 & = \Delta (|\lo|^2) + 9 H \Delta (H) - 2 (L,\Hess(H)) + 12 |dH|^2 - 4 |\lo|^4 - 9 H \rho_0'' \\
   & +  8 (L,\Hess(H)) + 6 \tr (\lo^4) + 8 |\lo|^4 + 36 H \tr (\lo^3) + 126 H^2 |\lo|^2 + 144 H^4 \\
   & - 6 H \Delta (H) -  6 |\lo|^4 - 12 H \tr (\lo^3) - 108 H^2 |\lo|^2 - 144 H^4 + 18 H \rho_0'' \\
   & = \Delta (|\lo|^2) + 12 |dH|^2 +  6 (\lo,\Hess(H)) - 2 |\lo|^4 + 6 \tr (\lo^4)  \\
   & + 9 H \Delta (H) + 18 H^2 |\lo|^2 + 24 H \tr (\lo^3) + 9 H \rho_0''.
\end{align*}
The relation $\rho_0'' = - \Delta (H) - 2 \tr (\lo^3) - 2 H |\lo|^2$ implies the assertion. This completes
the proof of \eqref{B3-final}.
\medskip

\begin{rem}\label{h3-rev}
The linear terms in the expansions \eqref{term1}-\eqref{term6} of Christoffel symbols show that
\begin{align*}
   0 & = \Hess (H) + \frac{1}{4} (2 h_{(1)} h_{(2)} + 2 h_{(2)} h_{(1)} - h_{(1)}^3 - 2 \rho_0 h_{(1)}^2) \\
   & - \frac{1}{2} ( 6h_{(3)} - 8 \rho_0 h_{(2)} - 4 \rho_0' h_{(1)}) \\
   & - \frac{1}{4} (4 \rho_0 h_{(2)} + 4 \rho_0' h_{(1)}).
\end{align*}
Hence
\begin{align*}
   0 & =  \Hess(H)  + h_{(1)} h_{(2)} - \frac{1}{4} h_{(1)}^3 + \frac{1}{2} H h_{(1)}^2
  - 3 h_{(3)} - 3 H h_{(2)} + \rho_0' h_{(1)} \\
   & = \Hess(H) - 3 h_{(3)} - 3 H L \lo + 2 \rho_0' L
\end{align*}
using the formulas for $h_{(1)}$ and $h_{(2)}$ in Proposition \ref{h-low-order}.
This reproduces the formula \eqref{h-adapted-cubic} for $h_{(3)}$ for a flat background.
\end{rem}

We round up this section with a discussion of the relation of the formula for $\B_3$ in Proposition \ref{B3-evaluation}
to alternative formulas in the literature. In \cite[(6)]{GW-announce} and \cite[Proposition 4.18]{GW-LNY}, it is stated
that for a conformally flat background $\B_3$ equals $\BB_3$, where        \index{$\BB_3$}
\begin{equation}\label{GGHW-B3}
   6 \BB_3 \st |\nabla \lo|^2 + 2 (\lo,\Delta (\lo)) + 3/2 (\delta(\lo),\delta(\lo)) - 2 \J |\lo|^2 + |\lo|^4
\end{equation}
with $\J = \J^h$. In the flat case, using $\delta(\lo) = 2 dH$ (Codazzi-Mainardi) and the Gauss identity
$$
   \J = \frac{3}{2} H^2 - \frac{1}{4} |\lo|^2,
$$
this formula reads
\begin{equation}\label{GW-B3}
   6 \BB_3 = |\nabla \lo|^2 + 2 (\lo, \Delta (\lo)) + 6 |dH|^2 - 3H^2 |\lo|^2 + 3/2 |\lo|^4.
\end{equation}

\begin{lem}[{\cite[Remark 12]{JO-Y}}]\label{Id-basic}
For $n=3$, it holds
\begin{equation}\label{Id-1}
    \frac{1}{2} \Delta (|\lo|^2) = 3 (\lo,\Hess(H)) + |\nabla \lo|^2 + 3 H \tr(L^3) - |L|^4
\end{equation}
and
\begin{equation*}\label{Id-2}
   \delta \delta (\lo^2) = 4 (\lo,\Hess(H)) + |\nabla \lo|^2 + 2 |dH|^2 + 3H \tr(L^3) - |L|^4.
\end{equation*}
Hence we have the difference formula
\begin{equation}\label{basic-div}
   \frac{1}{2} \Delta (|\lo|^2) - \delta \delta (\lo^2) = - (\lo,\Hess(H)) - 2 |dH|^2.
\end{equation}
\end{lem}

As a consequence, we obtain

\begin{lem}\label{B-relation} $\BB_3 = \B_3$.
\end{lem}

\begin{proof}
On the one hand, we apply the identity \eqref{Id-1} to calculate
\begin{align*}
   12 \B_3 & = 6 (\lo,\Hess(H)) + 2 |\nabla \lo|^2 + 6 H \tr(L^3) - 2 |L|^4  \\
   & + 12 |dH|^2 + 6 (\lo,\Hess(H)) - 2 |\lo|^4 + 6 \tr (\lo^4) + 6 H \tr (\lo^3)  \\
   & = 12 (\lo,\Hess(H)) + 2 |\nabla L|^2 + 6 |dH|^2 \\
   & + 6 H \tr(L^3) + 6 H \tr(\lo^3) - 2 |L|^4 - 2 |\lo|^4 + 6 \tr(\lo^4)
\end{align*}
using the relation $|\nabla L|^2 = |\nabla \lo|^2 + 3 |dH|^2$. Hence
$$
  6 \B_3 = 6 (\lo,\Hess(H)) + |\nabla L|^2 + 3 |dH|^2 + 3 H \tr(L^3) + 3 H \tr(\lo^3) - |L|^4 - |\lo|^4 + 3 \tr(\lo^4).
$$
On the other hand, we use the identity $\Delta (|\lo|^2) = 2 (\lo,\Delta(\lo)) + 2 |\nabla \lo|^2$ and
\eqref{Id-1} to find
\begin{align*}
   6 \BB_3 & = \Delta (|\lo|^2) - |\nabla \lo|^2 + 6 |dH|^2- 3 H^2 |\lo|^2 + 3/2 |\lo|^4 \\
   & = 6 (\lo,\Hess(H))  + |\nabla L|^2 + 3 |dH|^2 + 6 H \tr (L^3) - 3 H^2 |\lo|^2 -2 |L|^4 + 3/2 |\lo|^4.
\end{align*}
Hence the difference $6 (\BB_3 - \B_3)$ equals
\begin{align*}
   & 3 H \tr(L^3) - 3 H \tr(\lo^3) - |L|^4 + |\lo|^4  + 3/2 |\lo|^4 - 3 H^2 |\lo|^2 - 3 \tr(\lo^4) \\
   & = 3 H (3 H \tr (\lo^2) + 3 H^3) - 6H^2|\lo|^2 - 9H^4  - 3 H^2 |\lo|^2 + 3 (1/2 |\lo|^4 - \tr (\lo^4)) \\
   & = 3 ((1/2 |\lo|^4 - \tr (\lo^4)) \\
   & = 0
\end{align*}
by Corollary \ref{Newton-C3}. This proves the assertion.
\end{proof}

Finally, we show that the formula for $\B_3$ established in Proposition \ref{B3-evaluation} implies

\begin{lem}[{\cite{GGHW}}]\label{B3-CF}
For a conformally flat background, it holds
\begin{equation}\label{B3-GW-0}
   6 \B_3 = 3 \delta \delta ((\lo^2)_\circ) + 3 (\Rho,(\lo^2)_\circ)  + |\lo|^4.
\end{equation}
\end{lem}

Note that both $\B_3$ and the right-hand side of \eqref{B3-GW-0} are conformally
invariant. In fact, the operator $T: b \mapsto \delta \delta (b) + (\Rho,b)$ acting
on trace-free symmetric bilinear forms $b$ on $M^3$ is well-known to be conformally
invariant in the sense that $e^{4\varphi} \hat{T} (b)  = T (b)$. In \eqref{B3-GW-0},
the operator $T$ acts on the trace-free part $(\lo^2)_\circ$ of $\lo^2$. The
conformal invariance $\hat{\lo}^2 = \lo^2$ implies the conformal invariance of the
trace-free part of $\lo^2$ (Section \ref{CTLL}). This shows the claimed conformal
invariance. In particular, the right-hand side of \eqref{B3-GW-0} has the same
conformal transformation law as $\B_3$.

In more explicit terms, formula \eqref{B3-GW-0} reads
\begin{align}\label{B3-GW}
   6 \B_3 & = 3 \delta \delta (\lo^2) - \Delta (|\lo|^2) + 3 (\Rho,\lo^2) - |\lo|^2 \J + |\lo|^4,
\end{align}
and it suffices to verify \eqref{B3-GW} in the flat case.

\begin{lem}\label{B3-flat-equiv} For a flat background, the formulas \eqref{B3-final2} and \eqref{B3-GW}
are equivalent.
\end{lem}

\begin{proof} The assertion is equivalent to the vanishing of the sum
\begin{align*}
    & 6 \delta \delta (\lo^2) - 2 \Delta (|\lo|^2) + 6 (\Rho,\lo^2) - 2 |\lo|^2 \J + 2 |\lo|^4 \\
    & - \Delta (|\lo|^2) - 12 |dH|^2 - 6 (\lo,\Hess (H)) - |\lo|^4 - 6 H \tr (\lo^3).
\end{align*}
But the identity
\begin{equation}\label{Fial}
    \JF = \iota^* \bar{\Rho} - \Rho + H \lo + \frac{1}{2} H^2 h \stackrel{!}{=} \lo^2 - \frac{1}{4} |\lo|^2 h
\end{equation}
for the Fialkov tensor $\JF$ (see \eqref{FW-relation}) and the Gauss identity $\J = 3/2 H^2 - 1/4 |\lo|^2$
(see \eqref{G2}) imply
\begin{equation}\label{JP}
    6(\lo^2,\Rho) - 2 |\lo|^2 \J =  6 H \tr(\lo^3) - |\lo|^4.
\end{equation}
Hence the above sum simplifies to
$$
   6  \delta \delta (\lo^2) - 3 \Delta (|\lo|^2)  - 12 |dH|^2 - 6 (\lo,\Hess (H)).
$$
By \eqref{basic-div}, this sum vanishes. The proof is complete.
\end{proof}

\begin{rem}\label{B3-general-proof}
Alternatively, one may derive formula \eqref{B3-GW-0} for conformally flat backgrounds by direct
evaluation of the definition of the obstruction $\B_3$. For details, we refer to \cite{JO-Y}.
\end{rem}

\subsection{Variational aspects}\label{var-aspects}

Here we relate the obstructions $\B_2$ (for general backgrounds) and $\B_3$ (for
conformally flat backgrounds) to singular Yamabe {\em energy functionals}. The
discussion illustrates the general results of \cite{GW-reno, Graham-Yamabe} and connects
with the classical literature.

We first consider the classical situation of a variation of the Willmore functional. For a closed surface $f: M^2
\hookrightarrow \R^3$, we consider a normal variation $f_t: M^2 \hookrightarrow \R^3$ of $f$:
$$
   f_t (x) = f (x) + t u(x) N_0,
$$
where $u \in C^\infty(M)$ and $N_0$ is the unit normal of $M$. The variation field of $f_t$ is
$u N_0$. We set $M_t = f_t (M)$ and let          \index{$\W_2$ \quad Willmore functional}
\begin{equation}\label{W2-flat}
   \W_2(f) \st \int_{f(M)} |\lo|^2 dvol_h,
\end{equation}
where the metric $h$ is induced by the Euclidean metric on $\R^3$. We often identify $M$ with $f(M)$. Set
$$
   \var (\W_2)[u] \st (d/dt)|_0 (\W_2(M_t)).
$$
In order to calculate that variation, we recall the well-known variation formulas
\begin{align*}
    \var (h)[u] & = 2 u L, \\
    \var (L)[u] & = -\Hess(u) + u L^2, \\
    2 \var (H)[u] & = -\Delta (u) - u |L|^2
\end{align*}
and
$$
   \var(dvol_h)[u] = 2 u H dvol_h,
$$
where $L^2$, $|L|^2$ and $\Delta (u)$ are defined by the metric $h$ on $M$. First, we note that
$$
   \var( |\lo|^2)[u] = 2 (\lo ,\var(\lo)[u]) - 4 u (L,\lo^2);
$$
the second term comes from raising $2$ indices: $|\lo|^2 = \tr (\lo^2) = h^{ia} h^{jb} \lo_{ij}\lo_{kb}$.
Hence
\begin{align*}
   \var (|\lo|^2)[u]  & = 2 (\lo,\var(L)[u] - H  \var (h)[u]) - 4 u \tr(\lo^3 + H \lo^2) \\
   & = - 2 ((\lo,\Hess(u) - u L^2) + 2 H u (\lo,L)) + 4 u H |\lo|^2
\end{align*}
using $\tr(\lo^3)=0$. Now partial integration gives
\begin{align*}
   \var(\W_2)[u]  & = - \int_M u \left[ 2 \delta \delta (\lo) -2  (\lo,L^2) + 4 H (\lo,L)
   + 4 H |\lo|^2 - 2 H |\lo|^2 \right] dvol_h;
\end{align*}
the last term comes from the variation of the volume. Simplification yields
\begin{equation}\label{W2}
   \var (\W_2)[u] = -\int_M u ( 2 \Delta (H) + 2 H |\lo|^2) dvol_h
\end{equation}
using $\delta \delta (\lo) = \Delta(H)$ and again $\tr (\lo^3)=0$. This proves the classical result that in
a flat background, the Euler-Lagrange equation of the Willmore functional $\W_2$ is
$$
   \Delta(H) + H |\lo|^2 = 0.
$$
By $|\lo|^2 = 2 (H^2- K)$, where $K$ is the Gauss curvature, the  Euler-Lagrange equation of
the Willmore functional $\W_2$ reads
$$
   \Delta (H) + 2 H(H^2-K) = 0.
$$
This equation is known as the Willmore equation. It was already mentioned in \cite{Thom} and Schadow (1922). 
We refer to \cite[Section 7.4]{Will} for more details.

The variation formula \eqref{W2} implies the special case
\begin{equation}\label{AB-2}
    \var(\A_2)[u] = \frac{3}{2} \int_{M^2} u \B_2 dvol_h
\end{equation}
of the variation formula
\begin{equation}\label{var-form-1}
   \var(\A_n)[u] = (n+2)(n-1) \int_M u \LO_n dvol_h = \frac{(n+1)(n-1)}{2} \int_{M^n}  u \B_n dvol_h
\end{equation}
for $\A_n = \int_M v_n dvol_h$. The first equality in \eqref{var-form-1} was proved in
\cite[Theorem 3.1]{Graham-Yamabe}. The variational formula in terms of $\B_n$ was
established in \cite[Section 5]{GW-reno} by different arguments. For the second
equality, we refer to \eqref{obstruction-two}. In fact, \eqref{v2-2} shows that
$$
   \A_2 = \int_{M^2} v_2 dvol = \int_{M^2} \left(-\frac{1}{2} \J + \frac{1}{4} |\lo|^2 \right) dvol_h.
$$
Hence
$$
   \var(\A_2)[u] = \frac{1}{4} \var(\W_2)[u] = -\frac{1}{2} \int_M u (\Delta (H) + H |\lo|^2) dvol_h
$$
by Gauss-Bonnet.  On the other hand, we have
$$
   \B_2 = -\frac{1}{3} (\Delta (H) + H|\lo|^2).
$$
This implies \eqref{AB-2}.

These results generalize as follows to closed surfaces $M^2 \hookrightarrow (X^3,g)$
in general backgrounds. For the following discussion, we also refer to \cite[Section
5.1]{V}.

We consider normal variations with a variation field of the form $u N_0$ with a unit
normal field $N_0$. The following formula is well-known (see \cite[Section 3]{Ros},
\cite[Theorem 3.2]{Huisken}, \cite[Theorem 4.1]{V}). It can be proved by calculation
in geodesic normal coordinates. For $u=1$, it plays a central role in \cite[Chapter
3]{Gray}. It holds
\begin{equation}\label{VL}
   \var(L_{ij})[u] = -\Hess_{ij}(u) + u (L^2)_{ij} - u \bar{R}_{0ij0}.
\end{equation}
Hence
\begin{align*}
   \var(|\lo|^2)[u] & = 2 (\lo,\var(L)[u] - H \var(h)[u]) - 4 u \tr (\lo^3 + H \lo^2) \\
   & = 2 (\lo, - \Hess (u) + u L^2) - 2 u \lo^{ij} R_{0ij0} - 4 H u (\lo,L) - 4 H |\lo|^2.
\end{align*}
Now simplification and partial integration gives
$$
   \var(\W_2)[u] = - \int_M u (2 \delta \delta (\lo) + 2 H |\lo|^2 + 2 \lo^{ij} R_{0ij0}) dvol_h.
$$
Since the Weyl tensor vanishes in dimension $3$, we have $\lo^{ij} R_{0ij0}
= \lo^{ij} (\Rho_{ij} + \Rho_{00} h_{ij}) = \lo^{ij} \Rho_{ij}$. Thus the integrand
is given by letting the operator
\begin{equation}\label{op-2}
   b \mapsto \delta \delta (b) + (\iota^*(\Rho), b) + H (\lo,b)
\end{equation}
act on $b = \lo$. The above operator is well-known to be conformally covariant on trace-free symmetric
bilinear forms (see \cite[Section 5.1.4]{V}). This implies the conformal invariance of the integrand.

The above calculation shows that the Euler-Lagrange equation of the Willmore functional
$\W_2$ is
\begin{equation}\label{EL}
   \delta \delta (\lo) + H |\lo|^2 + \lo^{ij} \bar{R}_{0ij0} = 0.
\end{equation}
By Codazzi-Mainardi, $\delta \delta (\lo)$ equals $\Delta (H) + \delta (\bar{\Rho}_{0})$. Thus
we obtain
$$
   \Delta (H) + \delta (\overline{\Ric}_{0}) + H |\lo|^2 + \lo^{ij} \bar{R}_{0ij0} = 0.
$$

We also observe that the left-hand side of \eqref{EL} coincides with
$$
   \delta \delta (\lo) + H |\lo|^2 + (\lo,\iota^* (\bar{\Rho}))
   = \delta \delta (\lo) + H |\lo|^2 + (\lo,\iota^* (\overline{\Ric}))
$$
since  $\lo^{ij} \bar{R}_{0ij0} = \lo^{ij} \overline{\Ric}_{ij}$. In fact, since the Weyl tensor
vanishes in dimension $3$, it holds
$$
    \bar{R}_{0ij0} = \bar{\Rho}_{ij} + \bar{\Rho}_{00} h_{ij}
$$
and we obtain
$$
   \lo^{ij} \bar{R}_{0ij0} = \lo^{ij} \bar{\Rho}_{ij} = \lo^{ij} \overline{\Ric}_{ij}.
$$
It follows that the variation of the Willmore functional $\W_2$ yields the Yamabe obstruction $\B_2$:
$$
   \var (\A_2)[u] = - \frac{1}{4} \var (\W_2) [u] = -\frac{3}{2} \int u \B_2 dvol_h
$$
confirming \eqref{AB-2} for general backgrounds.

\index{$\W_3$ \quad Willmore functional}

We continue with an analogous discussion of the variation of
\begin{equation}\label{W3}
   \W_3 \st \int_{M^3} \tr(\lo^3) dvol_h
\end{equation}
for variations of a closed three-manifold $M^3 \hookrightarrow X^4$ in a conformally flat background $(X,g)$.
We first determine the variation of $\W_3$ for a general background metric $g$. Here we use the variation
formulas
\begin{align}\label{var-form}
    \var (h)[u] & = 2 u L, \notag \\
    \var (L)[u] & = - \Hess(u) + u L^2 - u \bar{R}_{0 \cdot \cdot 0}, \notag \\
    3 \var (H)[u] & = - \Delta (u) - u |L|^2 - u \overline{\Ric}_{00}
\end{align}
and
$$
   \var(dvol_h)[u] = 3 u H dvol_h.
$$
Note that the operator $\Delta (u) + u |L|^2 + u \overline{\Ric}_{00}$ is the Jacobi operator
appearing in the second variation formula for the area of minimal surfaces \cite{CM}.

First, we observe that
$$
   \var(\tr (\lo^3))[u] = 3 (\lo^2,\var(\lo)[u]) - 6 u (L,\lo^3);
$$
the second term comes from raising $3$ indices: $\tr (\lo^3) = h^{ia} h^{jb} h^{kc} \lo_{ij}\lo_{kb} \lo_{ca}$.
Hence
\begin{align*}
   & \var (\tr (\lo^3))[u] \\
   & = 3  (\lo^2,\var(\lo)[u]) - 6 u (L,\lo^3) \\
   & = 3 (\lo^2,\var (L)[u] - \var (H)[u] h - H \var(h)[u]) - 6 u (L,\lo^3) \\
   & = 3 (\lo^2,-\Hess(u) + u L^2) - 3 u (\lo^2,\bar{R}_{0 \cdot\cdot 0}) + (\Delta(u) + u |L|^2
  + u \overline{\Ric}_{00}) |\lo|^2 \\
   & - 6 u H (\lo^2,L) - 6 u (L,\lo^3).
\end{align*}
Now partial integration and simplification yields
\begin{align}\label{var-W3-g}
  & \var\big(\int_M \tr(\lo^3) dvol_h\big)[u] \notag \\
  & = \int_M u \big[ - 3 \delta \delta (\lo^2) + \Delta (|\lo|^2) - 3 \tr (\lo^4)
  + |\lo|^4 - 3 H \tr(\lo^3) \big] dvol_h \notag \\
  & +  \int_M u ( - 3 (\lo^2,\bar{R}_{0 \cdot \cdot 0}) + |\lo^2| \overline{\Ric}_{00}) dvol_h.
\end{align}

Now, for a conformally flat background, we reformulate this variation formula in a conformally invariant way.
The following result is also covered in \cite[Section 4]{GGHW} using a different method.

\begin{lem}\label{var-W3} Assume that $n=3$ and that the Weyl tensor of $(X,g)$ vanishes. Then
\begin{equation*}\label{var-W3-CI}
   \var \left( \int_M \tr(\lo^3) dvol_h \right)[u]
   = - 3 \int_M u \left(\delta \delta ((\lo^2)_\circ) + (\Rho, (\lo^2)_\circ) + \frac{1}{3} |\lo|^4\right) dvol_h.
\end{equation*}
\end{lem}

\begin{proof} By \eqref{var-W3-g}, the claim is equivalent to the identity
\begin{align*}
   & \delta \delta ((\lo^2)_\circ) + (\Rho, (\lo^2)_\circ) + \frac{1}{3} |\lo|^4 \\
   & = \delta \delta (\lo^2) - \frac{1}{3} \Delta (|\lo|^2)
   + \tr (\lo^4) - \frac{1}{3}|\lo|^4 + H \tr(\lo^3)
  + (\lo^2,\bar{R}_{0 \cdot \cdot 0}) - \frac{1}{3} |\lo^2| \overline{\Ric}_{00}.
\end{align*}
Note that
$$
\delta \delta ((\lo^2)_\circ) = \delta \delta (\lo^2) - \frac{1}{3} \Delta (|\lo|^2)
$$
and
\begin{align*}
   (\Rho, (\lo^2)_\circ) = (\Rho,\lo^2) - \frac{1}{3} (\Rho,|\lo|^2 h)
  = (\Ric,\lo^2) - \J |\lo|^2 - \frac{1}{3} \J |\lo|^2  = (\Ric,\lo^2) - \frac{4}{3} \J |\lo|^2.
\end{align*}
Moreover, it holds
\begin{align*}
    (\lo^2,\bar{R}_{0 \cdot \cdot 0}) - \frac{1}{3} |\lo^2| \overline{\Ric}_{00}
    & = (\lo^2,\bar{\Rho}) + \bar{\Rho}_{00} |\lo|^2 - \frac{1}{3} |\lo^2| \overline{\Ric}_{00}  \\
    & = (\lo^2,\bar{\Rho}) + \frac{1}{3} \bar{\Rho}_{00} |\lo|^2 - \frac{1}{3} \bar{\J} |\lo|^2
\end{align*}
(by the vanishing of the Weyl tensor). Hence the claim reduces to
\begin{align*}
    & (\Ric,\lo^2) - \frac{4}{3} \J |\lo|^2 + \frac{1}{3} |\lo|^4 \\
    & = \tr (\lo^4) - \frac{1}{3}|\lo|^4 + H \tr(\lo^3)
   + (\bar{\Rho},\lo^2) + \frac{1}{3} \bar{\Rho}_{00} |\lo|^2 - \frac{1}{3} \bar{\J} |\lo|^2 .
\end{align*}
Now the Gauss equations
\begin{align*}
    \Ric_{ij} = \overline{\Ric}_{ij} - \bar{R}_{0ij0} + 3 H L_{ij} - (L^2)_{ij} \quad \mbox{and} \quad
    \J - \overline{\J} = - \bar{\Rho}_{00} - \frac{1}{4} |\lo|^2 + \frac{3}{2} H^2
\end{align*}
 (see \eqref{GRicci} and \eqref{G1}) imply
\begin{align*}
    (\Ric,\lo^2) - \frac{4}{3} \J |\lo|^2 & = (\overline{\Ric},\lo^2) - (\bar{\Rho},\lo^2) - \bar{\Rho}_{00} |\lo|^2
   + 3 H (L,\lo^2) - (L^2,\lo^2) \\
   & - \frac{4}{3} \left(\bar{\J} - \bar{\Rho}_{00} - \frac{1}{4} |\lo|^2 + \frac{3}{2} H^2\right) |\lo|^2 \\
   & = 2 (\bar{\Rho},\lo^2) + \bar{\J} |\lo|^2 - (\bar{\Rho},\lo^2) - \bar{\Rho}_{00} |\lo|^2 + \frac{4}{3} \bar{\Rho}_{00} |\lo|^2
   - \frac{4}{3} \bar{\J} |\lo|^2 \\
   & + 3 H(L,\lo^2) - (L^2,\lo^2) + \frac{1}{3} |\lo|^4 - 2 H^2 |\lo|^2 \\
   & =  (\bar{\Rho},\lo^2) + \frac{1}{3} \bar{\Rho}_{00} |\lo|^2 - \frac{1}{3} \bar{\J} |\lo|^2 \\
   & + 3 H(L,\lo^2) - (L^2,\lo^2) + \frac{1}{3} |\lo|^4 - 2 H^2 |\lo|^2.
\end{align*}
Hence it suffices to prove that
$$
   3 H(L,\lo^2) - (L^2,\lo^2) + \frac{2}{3} |\lo|^4 - 2 H^2 |\lo|^2
  =  \tr (\lo^4) - \frac{1}{3} |\lo|^4 + H \tr (\lo^3).
$$
By simplification, this identity is equivalent to
$$
   \frac{2}{3}|\lo|^4 - \tr (\lo^4) = \tr (\lo^4) - \frac{1}{3} |\lo|^4,
$$
i.e., $|\lo|^4 = 2 \tr (\lo^4)$ (see Corollary \ref{Newton-C3}). The proof is complete.
\end{proof}

In terms of
\begin{equation}\label{A3}
  \A_3 = \int_{M^3} v_3 dvol_h = -\frac{2}{3} \int_{M^3} \tr (\lo^3) dvol_h = -\frac{2}{3} \W_3
\end{equation}
we obtain

\begin{cor}\label{var-A3}
$
   \var (\A_3)[u] = 4 \int_{M^3} u \B_3 dvol_h.
$
\end{cor}

\begin{proof} Combine \eqref{B3-GW-0} with Lemma \ref{var-W3}.
\end{proof}

This is a special case of \eqref{var-form}.

Note that the second equality in \eqref{A3} follows by combining \eqref{F-L-trace} with
\eqref{vQ3} and \eqref{Q3-ex}.

Another proof of Corollary \ref{var-A3} (even for general backgrounds) has been given in \cite{GGHW} using a method
that rests on a certain distributional calculus. Finally, the above classical style arguments have been extended to the general
case in \cite{JO-Y}.

\subsection{Low-order extrinsic $Q$-curvatures}\label{Q-low}

Here we discuss the low-order extrinsic $Q$-curvatures $\QC_2$ and $\QC_3$ from the perspective of
their holographic formulas.

\begin{example}\label{Q2}
We consider $\QC_2$ in general dimensions. The holographic formula
\eqref{Q-holo-form-gen} states that
$$
   -\QC_2 = \frac{1}{n-3} (2v_2 + (n-2) (v\J)_0) + \frac{1}{n-1} \T_1^*\left(\frac{n}{2}-1\right) (v_1)
$$
for even $n$. Using $v_1 = (n-1)H$ and the formula \eqref{v2n} for $v_2$ as well as
$\T_1(\lambda)=-\lambda H$ (Lemma \ref{sol-1}), we obtain
$$
  - \QC_2 = \J_0 -  \rho_0'  +  \frac{n}{2}H^2 = \J_0 - \Rho_{00} - \frac{|\lo|^2}{n-1} + \frac{n}{2} H^2.
$$
In particular, we see that the holographic formula makes sense for all $n \ge 2$. By
the hypersurface Gauss identity \eqref{G2}, the above formula simplifies to
\index{$\QC_2$ \quad extrinsic $Q$-curvature of order $2$}
\begin{equation}\label{Q2-gen}
   -\QC_2 = \J^h - \frac{|\lo|^2}{2(n-1)}.
\end{equation}
Note that this result fits with the formula        \index{$\PO_2$ \quad second-order conformal Laplacian}
\begin{equation}\label{PO2}
   \PO_2 = \Delta_h - \left(\frac{n}{2}-1\right)\left (\J^h - \frac{|\lo|^2}{2(n-1)}\right)
\end{equation}
\cite[Proposition 8.5]{GW-LNY}. Independently, the latter formula for $\PO_2$
will be derived below from the solution operator $\T_2(\lambda)$ (Lemma \ref{sol-2}
and the discussion following it). We also recall that $\QC_2 = - Q_2$ in the
Poincar\'e-Einstein case. Finally, we note that the formula for $\QC_2$ is singular
for $n=1$ and $\Res_{n=1}(\QC_2) = \frac{1}{2} |\lo|^2 = 0$.
\end{example}

\begin{example}\label{Q3}
The holographic formula \eqref{Q-holo-form-gen} for $\QC_3$ states that
\begin{align*}
   \tfrac{1}{4} \QC_3 & = \tfrac{1}{n-5} (6 v_3 + (n\!-\!3) (v\J)_1)
   + \tfrac{1}{n-3} \T_1^*(\tfrac{n-3}{2})
   (4 v_2 + (n\!-\!1) (v\J)_0) + \tfrac{1}{n-1} \T_2^*(\tfrac{n-3}{2})(2v_1)
\end{align*}
for even $n$. As a byproduct of the following discussion, we will see that the
fractions in that formula do not prevent its validity in odd dimensions. First, we
note that the above formula is equivalent to
\begin{align*}
   \frac{1}{4} \QC_3 & = \frac{1}{n\!-\!5} (6 v_3 + (n\!-\!3) \J'_0 + (n\!-\!3) v_1 \J_0) \\
   & + \frac{1}{n\!-\!3} \T_1^*\left(\frac{n-3}{2}\right) (4v_2 + (n\!-\!1) \J_0)
   + \frac{1}{n\!-\!1} \T_2^*\left(\frac{n-3}{2}\right) (2v_1).
\end{align*}
Now the formulas
\begin{align*}
   v_1 & = (n\!-\!1) H \\
   2 v_2 & = (n\!-\!3)(n\!-\!1) H^2 - \iota^* (\J) - (n\!-\!3) \iota^* \nabla_\NV(\rho)  & \mbox{(by \eqref{v2n})} \\
   6 v_3 & = (n\!-\!5)(n\!-\!3)(n\!-\!1) H^3 - (3n\!-\!7) H \iota^* (\J) - (n\!-\!5)(3n\!-\!7) H \iota^* \nabla_\NV(\rho) \\
  & - (n\!-\!5) \iota^* \nabla_\NV^2(\rho) - 2 \iota^* \nabla_\NV(\J) & \mbox{(by \eqref{v3-inter})}
\end{align*}
show that the fractions are reduced. Indeed, it holds
\begin{align*}
    6 v_3 + (n\!-\!3) \J'_0 + (n\!-\!3) v_1 \J_0 & = 0 \qquad \mbox{if $n=5$}, \\
    4 v_2 + (n\!-\!1) \J_0 & = 0 \qquad \mbox{if $n=3$}
\end{align*}
(these are the relations mentioned in Remark \ref{co-ex}). Then we obtain
\begin{align}\label{Q3-ex-holo}
    \tfrac{1}{4} \QC_3 & = (n\!-\!3)(n\!-\!1) H^3 - (3n\!-\!7) H \iota^* \nabla_\NV(\rho) + \iota^* \nabla_\NV(\J)
    - \iota^* \nabla_\NV^2(\rho) + (n\!-\!2) H \iota^* (\J) \notag \\
    & + \T_1^*\left(\frac{n\!-\!3}{2}\right) (\iota^* (\J) - 2 \iota^* \nabla_\NV(\rho)
   + 2(n\!-\!1)H^2) + \T_2^*\left(\frac{n\!-\!3}{2}\right) (2H).
\end{align}
A calculation using Lemma \ref{rho-01}, Lemmas \ref{sol-1}--\ref{sol-2} and the hypersurface Gauss
identity \eqref{G2} shows that
\begin{equation}\label{Q3-rho}
   \frac{1}{4} \QC_3 = \Delta H + H \iota^* \J + \iota^* \nabla_\NV (\J) - (n\!-\!1) H \iota^* \nabla_\NV(\rho)
   - \iota^* \nabla_\NV^2(\rho).
\end{equation}
In particular, for $n=3$, we find
\begin{align*}
   \frac{1}{4} \QC_3 = \Delta H + H\iota^* (\J)  + \iota^* \nabla_\NV(\J) - 2 H \iota^* \nabla_\NV(\rho)
   -\iota^* \nabla_\NV^2(\rho).
\end{align*}
By comparison with Example \ref{v3-ex}, we see that
$$
   12 v_3 = - \QC_3 + 4\Delta H
$$
and consequently
\begin{equation}\label{vQ3}
   12 \int_{M^3} v_3 dvol_h = - \int_{M^3} \QC_3 dvol_h.
\end{equation}
This confirms Theorem \ref{LQ} for $n=3$.
\end{example}

\begin{remark}
Formula \eqref{Q3-rho} also follows from the last display in the proof of \cite[Proposition 8.5]{GW-LNY}
or \cite[Proposition 5.4]{GW-Willmore} which evaluates the composition of three Laplace-Robin operators.
\end{remark}

We continue with the discussion of the holographic formula for $\QC_3$ in general
dimensions $n \ge 3$. In fact, we prove that an evaluation of \eqref{Q3-rho} shows
that the explicit formula \eqref{rho-2-ex} for $\rho_0''$ is equivalent to a simple
formula for $\QC_3$.

\begin{prop}\label{Q3-fully}
Assume that $n \ge 3$. Then the formula \eqref{rho-2-ex} for $\rho_0''$ is equivalent to
\index{$\QC_3$ \quad extrinsic $Q$-curvature of order $3$}
\begin{align}\label{Q3F}
   \frac{1}{4} \QC_3  & = \frac{1}{n\!-\!2} (\delta \delta (\lo) - (n\!-\!3) (\lo, \Rho^h) + (n\!-\!1) (\lo,\JF)) \\
   & = \frac{1}{n\!-\!2} (\delta \delta (\lo) - 2(n\!-\!2) (\lo,\Rho^h) + (n\!-\!1) (\lo,\Rho^g) +
   (n\!-\!1) H |\lo|^2). \notag
\end{align}
\end{prop}

Note that \eqref{Q3F} fits with the explicit formula for $\PO_3$ in Proposition \ref{P3F}.

The proof of Proposition \ref{Q3-fully} will show that the displayed formula for $\QC_3$ follows by combining
the holographic formula for $\QC_3$ (in the form \eqref{Q3-rho}) with the formulas for the first two normal
derivatives of $\rho$.

Note also that in dimension $n=3$, the above formula reads
\index{$\QC_3$ \quad critical extrinsic $Q$-curvature of order $3$}
\begin{equation}\label{Q3-ex}
   \QC_3 = 4 \delta \delta (\lo) + 8 (\lo,\JF).
\end{equation}
It immediately follows from that expression that the integral $\int_{M^3} \QC_3 dvol_h$ is conformally
invariant as a functional of $g$.


\begin{proof}
It suffices to prove the equivalence of \eqref{rho-2-ex} and the first identity. For this,
we make explicit the equality of \eqref{Q3-rho} and \eqref{Q3F}. In terms of adapted coordinates,
it states the equality
\begin{equation*}
    \Delta H + H \J_0 + \J_0' - (n\!-\!1) H \rho_0' -  \rho_0'' - 2 H \rho_0' =
    \frac{1}{n\!-\!2} \delta \delta (\lo) - \frac{n\!-\!3}{n\!-\!2} \lo^{ij} \Rho^h_{ij}
   + \frac{n\!-\!1}{n\!-\!2} \lo^{ij} \JF_{ij}.
\end{equation*}
Here we used that $\iota^* \nabla_\NV^2(\rho)$ corresponds to $\rho_0'' - 2 \rho_0
\rho_0' = \rho_0'' + 2 H \rho_0'$ (see Example \ref{trans-low-2}). By the identity \eqref{ddL} for $\Delta H$,
this relation is equivalent to
\begin{equation*}
   \rho_0'' = \frac{1}{n\!-\!1} \delta \delta (\lo)  -\frac{1}{n\!-\!2} \delta \delta (\lo)
   + \frac{n\!-\!3}{n\!-\!2} \lo^{ij} \Rho^h_{ij} - \frac{n\!-\!1}{n\!-\!2} \lo^{ij} \JF_{ij}
   + H \J_0 - (n+1) H \rho_0' + \J_0'.
\end{equation*}
Now
$$
   H \J_0 - (n\!+\!1) H \rho_0' = - \frac{n\!+\!1}{n\!-\!1} H \Ric_{00} + \frac{1}{n\!-\!1} H \scal
  - \frac{n\!+\!1}{n\!-\!1} H |\lo|^2
$$
shows that this expression for $\rho_0''$ coincides with the one given in  \eqref{rho-2-ex}.
The proof is complete.
\end{proof}

The above formula for $\QC_3$ is singular for $n=2$. But the second formula in \eqref{Q3F} shows that
\begin{equation}\label{QB2}
   \Res_{n=2} (\QC_3) = 4 (\delta \delta (\lo) + (\lo, \Rho^g) + H |\lo|^2),
\end{equation}
up to the term $-8(n-2) (\lo,\Rho^h) = -8(\lo,\Ric^h)$ (which vanishes in dimension $n=2$ by
$\Ric^h = K h$, $K$ being the Gauss curvature). The right-hand side of \eqref{QB2} is proportional
to the obstruction $\B_2$. From that perspective, its conformal invariance follows from the conformal
covariance of $\PO_3$. The above argument to derive $\B_2$ from the constant term of $\PO_3$ is
due to \cite[Remark 8.6]{GW-LNY}. For the general relation between the singular Yamabe obstruction and
the super-critical $Q$-curvature $Q_{n+1}$, we refer to Theorem \ref{QB-residue}.

\subsection{The pair $(\PO_3,\QC_3)$ }\label{P3}

The following explicit formula for $\PO_3$ was first proven in \cite[Proposition 8.5]{GW-LNY}
by evaluation of the relevant composition of three Laplace-Robin
operators.\footnote{See also the identical \cite[Proposition 5.4]{GW-Willmore}.}

\index{$\PO_3$ \quad extrinsic conformal Laplacian of order $3$}
\begin{prop}\label{P3F}
For $n \ge 3$, the operator
$$
   \PO_3 = 8 \delta (\lo d) + \frac{n\!-\!3}{2} \frac{4}{n\!-\!2} (\delta \delta (\lo)
  - (n\!-\!3) (\lo,\Rho^h) + (n\!-\!1) (\lo,\JF))
$$
is conformally covariant:
$$
   e^{\frac{n+3}{2} \varphi} \circ \PO_3(\hat{g}) = \PO_3(g) \circ e^{\frac{n-3}{2}\varphi}, \; \varphi \in C^\infty(X).
$$
\end{prop}

The formula for the leading term of $\PO_3$ will be derived in Lemma \ref{LT-P3}.

\begin{proof} We first calculate
\begin{align*}
   & e^{\frac{n+3}{2} \varphi} \hat{\delta} (\hat{\lo} d)(e^{-\frac{n-3}{2} \varphi} u)
   = \delta (e^{\frac{n-3}{2}\varphi} \lo d)(e^{-\frac{n-3}{2}\varphi} u) +
   \frac{n\!-\!3}{2} e^{\frac{n-3}{2}\varphi} i_{\grad(\varphi)}(\lo d)(e^{-\frac{n-3}{2} \varphi} u) \\
   & = \delta (\lo d)(u) - \frac{n\!-\!3}{2} \delta (\lo d\varphi \cdot u) +
   \frac{n\!-\!3}{2} i_{\grad(\varphi)}(\lo du) - \left(\frac{n\!-\!3}{2}\right)^2 (d\varphi,\lo d\varphi) u \\
   & = \delta (\lo d)(u) - \frac{n\!-\!3}{2} \delta (\lo d\varphi) u - \left(\frac{n\!-\!3}{2}\right)^2
   (d\varphi,\lo d\varphi) u
\end{align*}
using $\hat{\lo} = e^{-\varphi} \lo$ for the endomorphism $\lo$ corresponding to the
trace-free second fundamental form and the transformation law $e^{(a+2)\varphi}
\hat{\delta} e^{-a\varphi} = \delta + (n\!-\!2\!-\!a) i_{\grad(\varphi)}$ on
$\Omega^1(M)$. Next, the term with the Schouten tensor yields
$$
   (\lo,\Rho^h) - (\lo,\Hess(\varphi)) + (\lo,d\varphi \otimes \varphi).
$$
Finally, it holds
\begin{align*}
   & e^{3\varphi} \hat{\delta} \hat{\delta} (\hat{\lo}) = \delta(e^\varphi \hat{\delta} (e^\varphi \lo)) +
   (n\!-\!3) i_{\grad(\varphi)} e^{\varphi} \hat{\delta} (e^\varphi \lo) \\
   & = \delta \delta (\lo) + (n\!-\!1) \delta (\lo d\varphi) + (n\!-\!3) i_{\grad(\varphi)} \delta (\lo)
   + (n\!-\!3)(n\!-\!1) i_{\grad(\varphi)} i_{\grad(\varphi)} (\lo)
\end{align*}
using $\hat{\lo} = e^{\varphi} \lo$ and the transformation law
$$
   e^{(a+2)\varphi} \hat{\delta} (e^{-a\varphi} b) = \delta (b) +
   (n\!-\!2\!-\!a) i_{\grad(\varphi)}(b) - \tr(b) d\varphi
$$
on symmetric bilinear forms $b$. Now simplification proves the claim.
\end{proof}

\begin{remark}\label{P3-tractor}
The operator $\frac{n-2}{4} \PO_3$ differs from the tractor calculus operator
$$
    D_M^A L_{AB} D_X^B
$$
by the contribution of $(\lo,\JF)$ \cite[Section 6.3]{Grant}. Here $L_{AB}$ is
a tractor calculus version of the second fundamental form. This identification again
implies its conformal covariance.
\end{remark}

Proposition \ref{P3F} yields an explicit formula for $\QC_3$ (see
\eqref{Q3}). As a direct cross-check of that formula, we note that in
dimension $n=3$
\begin{align*}
   e^{3\varphi} \hat{\QC}_3  & = 4 e^{3\varphi} \hat{\delta} \hat{\delta} (\hat{\lo})
   + 8 e^{3\varphi} (\hat{\lo}, \hat{\JF})_{\hat{h}} \\
   & = 4 \delta e^\varphi \hat{\delta} (e^{\varphi} \lo) +8 (\lo,\JF)_h \\
   & = 4 (\delta \delta (\lo) + 2 \delta (\lo d)(\varphi)) + 8 (\lo,\JF)_h \\
   & = \QC_3 + \PO_3 (\varphi).
\end{align*}
This also confirms Theorem \ref{CTL-Q} for $n=3$.

Similar arguments provide an elementary proof of the conformal invariance of the obstruction $\B_2$.

\subsection{Low-order solution operators}\label{sol-low}

Under the assumption $\SCY$, we make the low-order solution operators
$\T_1(\lambda)$ and $\T_2(\lambda)$ explicit in general dimensions. For that
purpose, we use the ansatz $u = s^\lambda f + s^{\lambda+1} \T_1(\lambda) f +
s^{\lambda+2} \T_2(\lambda) f + \dots$ for an approximate solution of the equation
$-\Delta_{\sigma^{-2}g}(u)=\lambda(n-\lambda)u$ in adapted coordinates.

\begin{lem}\label{sol-1} $\T_1(\lambda) = - \lambda H$.
\end{lem}

\begin{proof} By the assumption $\SCY$, the coefficients in \eqref{L-op} expand as
$a = 1-2s\rho + O(s^{n+1}) = 1 + 2s H + \cdots$ (using $\iota^* (\rho) = -H$) and
$\tr (h_s^{-1} h'_s) = 2 \tr(L) + \cdots$. Now we let the operator \eqref{L-op}
act on $u$. The expansion of the result starts with
$$
   s^2 \partial_s^2 (s^\lambda) f - (n-1) s\partial_s (s^\lambda) f \stackrel{!}{=}
   -\lambda(n-\lambda) s^\lambda f.
$$
The next term in the expansion reads
\begin{align*}
   & s^2 \partial_s^2(s^{\lambda+1}) \T_1(\lambda)(f)+ 2 s^3 H \partial_s^2(s^\lambda) f
   + s^2 \tr(L) \partial_s(s^\lambda) f \\
   & - (n-1) s \partial_s(s^{\lambda+1}) \T_1(\lambda)(f) - 2(n-1) s^2 H \partial_s(s^\lambda) f
   + H s^2 \partial_s(s^\lambda) f \\
   & = (\lambda\!-\!n\!+\!1)(\lambda\!+\!1) s^{\lambda+1} \T_1(\lambda)(f)
  + (2\lambda\!-\!n\!+\!1) \lambda H s^{\lambda+1} f.
\end{align*}
This result equals $-\lambda(n-\lambda) s^{\lambda+1} \T_1(\lambda)(f)$ iff
$\T_1(\lambda)(f)=-\lambda H$.
\end{proof}

It is worth emphasizing that $\T_1(\lambda)$ is regular in $\lambda$. In fact, this
property does not hold for general asymptotically hyperbolic metrics \cite[Section 4]{Gu05}.

As a consequence of Lemma \ref{sol-1}, we find
\begin{align}\label{D1-exp}
   \D_1^{res}(\lambda) & = (2\lambda\!+\!n\!-\!1)(\iota^* \partial_s + (v_1 +
   \T_1^*(\lambda\!+\!n\!-\!1)) \iota^*) \notag \\
   & = (2\lambda\!+\!n\!-\!1)(\iota^* \partial_s -\lambda H \iota^*).
\end{align}
This formula obviously confirms Theorem \ref{residue-product} for the operator
$\iota^* L(\lambda)$.

The formula for $\T_2(\lambda)$ is a bit more complicated.

\begin{lem}\label{sol-2} Assume that $n \ge 2$. Then
\begin{align*}
   \T_2(\lambda) & = \frac{-\Delta_h + \lambda \J^h}{2(2\lambda\!-\!n\!+\!2)} +
   \frac{\lambda}{2} \Rho_{00} - \frac{\lambda}{2(2\lambda\!-\!n\!+\!2)}
   \left(\frac{n\!-\!1\!-\!2\lambda}{n\!-\!1} - \frac{1}{2(n\!-\!1)} \right) |\lo|^2 \\
   & + \frac{\lambda}{2(2\lambda\!-\!n\!+\!2)} \left((\lambda\!+\!1)(2\lambda\!-\!n\!+\!3)
   - \frac{n}{2}\right) H^2.
\end{align*}
\end{lem}

\begin{proof} By $\SCY$, the coefficients in \eqref{L-op} expand as
$$
   a = 1-2s\rho + \cdots = 1 + 2s H - 2s^2 \left(\Rho_{00} + \frac{|\lo|^2}{n-1}\right) + \cdots,
$$
up to a remainder in $O(s^{n+1})$, and
$$
   \tr (h_s^{-1} h'_s) = 2 \tr(L) - s (2\Ric_{00} + 2 |\lo|^2 + 4n H^2) + \cdots.
$$
Hence the equality of the coefficients of $s^{\lambda+2}$ in the expansion of the
eigenequation $\Delta_{s^{-2}\eta^*( g)} u = -\lambda(n-\lambda) u$ yields the
relation
\begin{align*}
   & 2(2\lambda\!-\!n\!+\!2) \T_2(\lambda)(f) + (\lambda\!+\!1)(2\lambda\!-\!n\!+\!3) H \T_1(\lambda)(f) \\
   & = \left[-2\lambda(\lambda\!-\!1) \rho_0' - \lambda (\Ric_{00} + |\lo|^2) + 2(n\!-\!1)
   \lambda \rho_0' -2\lambda \rho_0' + \Delta \right](f) = 0.
\end{align*}
This condition for $\T_2(\lambda)$ simplifies to
\begin{align*}
   & 2(2\lambda\!-\!n\!+\!2) \T_2(\lambda)(f) - \lambda(\lambda\!+\!1)(2\lambda\!-\!n\!+\!3) H^2 f \\
   & = \lambda \left[ -2(n\!-\!1\!-\!\lambda) \rho_0' + \Ric_{00} + |\lo|^2 \right] f - \Delta f \\
   & = \lambda \left [-2(n\!-\!1\!-\!\lambda)\left(\Rho_{00} + \frac{|\lo|^2}{n-1}\right)
   + ((n\!-\!1) \Rho_{00} + \J_0) + |\lo|^2 \right] f - \Delta f \\
   & = -\lambda \left((n\!-\!1\!-\!2\lambda) \Rho_{00} + \frac{n\!-\!1\!-\!2\lambda}{n\!-\!1} |\lo|^2
   - \J_0 \right) f - \Delta f
\end{align*}
using Lemma \ref{rho-01}. Finally, \eqref{G2} implies
\begin{align*}
   & 2(2\lambda\!-\!n\!+\!2) \T_2(\lambda)(f) \\
   & = - \lambda \left((n\!-\!2\!-\!2\lambda) \Rho_{00} - \J^h + \left(\frac{n\!-\!1\!-\!2\lambda}{n\!-\!1}
   - \frac{1}{2(n\!-\!1)} \right) |\lo|^2 + \frac{n}{2} H^2 \right) f \\
   & + \lambda(\lambda\!+\!1)(2\lambda\!-\!n\!+\!3) H^2 f - \Delta f.
\end{align*}
This completes the proof.
\end{proof}

In particular, Lemma \ref{sol-2} implies
$$
   -4 \Res_{\frac{n}{2}-1}(\T_2(\lambda)) = \Delta_h - \left( \frac{n}{2}-1\right)
   \left(\J^h - \frac{ |\lo|^2}{2(n-1)} \right) = \PO_2(g).
$$
This is a special case of Theorem \ref{power-spec}. It follows that
$$
   -\QC_2(g) = \J^h - \frac{1}{2(n-1)} |\lo|^2.
$$

For $n=2$, Lemma \ref{sol-2} implies
$$
   \T_2(0)(1) = \frac{1}{4} \J^h - \frac{1}{8} |\lo|^2.
$$
For the function $\QC_2(\lambda)$ defined in \eqref{T-1}, it follows that
$\QC_2(0)= - \J^h + \frac{1}{2} |\lo|^2 \stackrel{!}{=} \QC_2$.
This confirms \eqref{QT-critical} for $n=2$.

The above results yield an explicit formula for $\D_2^{res}(\lambda)$.

\begin{lem}\label{D2-exp}
\begin{align}\label{D2-final}
   \D_2^{res}(\lambda) & = (2\lambda\!+\!n\!-\!3 ) \Big[(2\lambda\!+\!n\!-\!2) \iota^* \partial_s^2 - 2
   (2\lambda\!+\!n\!-\!2)(\lambda\!-\!1) H \iota^* \partial_s \notag \\
   & - \Big[ \Delta_h + \lambda \J^h - \lambda(2\lambda\!+\!n\!-\!2) \Rho_{00}
   - \lambda \left (\frac{2\lambda\!+\!n\!-\!2}{2(n\!-\!1)} + \frac{2\lambda\!+\!n\!-\!1}{2(n\!-\!1)} \right) |\lo|^2 \notag \\
   & -(2\lambda\!+\!n\!-\!2)(\lambda\!-\!1/2) \lambda H^2\Big] \iota^* \Big].
\end{align}
\end{lem}

We omit the details of the calculation.

This result fits with the formula for $\iota^* L(\lambda-1) L(\lambda)$ in \cite[Lemma 7.9]{GW-LNY},
i.e.,
$$
   \D_2^{res}(\lambda) \circ \eta^* = \iota^* L(\lambda-1) L(\lambda).
$$
In order to see this, it only remains to express the normal derivatives in the
variable $s$ by iterated gradients. But $\iota^* \nabla_\NV$ corresponds to
$\iota^*\partial_s$ and $\iota^* \nabla_\NV^2$ corresponds to $\iota^* (\partial_s^2 +
2 H \partial_s)$ (see Example \ref{trans-low-2}).

Note that the prefactor in \eqref{D2-final} implies that $\D_2^{res}(-\frac{n-3}{2}) = 0$.
The conformally covariant term in brackets (for $\lambda=-(n-3)/2$ and up to the contribution by $|\lo|^2$)
has been used in \cite[Theorem 1.1]{Case} as a boundary operator associated to the Paneitz operator on $X$.
For general $\lambda$, it appears in \cite[Proposition 3.3]{Case} (up to the term containing $|\lo|^2$).
Concerning the classification of such boundary operators, we refer to \cite[Remark 3.9]{Case}.

Lemma \ref{D2-exp} immediately shows that $\D_2^{res}(-\frac{n}{2}+1)=
\PO_2(g)\iota^*$. In addition, we find the remarkable identity
\begin{equation}\label{factor-big}
   \D_2^{res}\left(\frac{-n+1}{2}\right) \circ \eta^* = 2 \iota^* P_2(g),
\end{equation}
where $P_2(g)$ is the Yamabe operator of $g$. In fact, we calculate
\begin{align*}
   & \D_2^{res}\left(\frac{-n+1}{2}\right) \\
   & = 2 \left(\iota^* \partial_s^2 + (n\!+\!1) H \iota^* \partial_s + \Delta_h \iota^*
   - \frac{n\!-\!1}{2} \left[\J^h + \Rho_{00} - \frac{n}{2} H^2 +
   \frac{1}{2(n\!-\!1)} |\lo|^2 \right] \iota^* \right) \\
   & = 2 \left(\iota^* \partial_s^2 + (n\!+\!1) H \iota^* \partial_s +
   \Delta_h \iota^* - \frac{n\!-\!1}{2} \iota^* \J^g \right)
\end{align*}
using \eqref{G2}. Hence
\begin{align*}
    \D_2^{res}\left(\frac{-n\!+\!1}{2}\right) \circ \eta^* & = 2 \left(\iota^* \nabla_\NV^2 + (n\!-\!1) H \iota^* \nabla_\NV +
    \Delta_h \iota^* - \frac{n\!-\!1}{2}  \iota^* \J^g \right) \\
    & = 2 \iota^*\left(\Delta_g - \frac{n\!-\!1}{2}\J^g \right) = 2 \iota^* P_2(g)
\end{align*}
using $\iota^* \nabla_\NV^2 = (\iota^* \partial_s^2 + 2 H \iota^* \partial_s) \circ \eta^*$ (Example
\ref{trans-low-2}) and the identity \eqref{LR}.

For Poincar\'e-Einstein metrics, the relation \eqref{factor-big} is one of the
identities in a second set of so-called factorization identities \cite[Theorem 3.2]{J2}. It is an open
problem whether the higher-order residue families in the general case continue to satisfy such identities.

Finally, we note that Lemma \ref{D2-exp} implies that
\begin{equation}\label{Q2-van}
   \QC^{res}_2(0) \st \D_2^{res}(0)(1) =0.
\end{equation}
This is a special case of the following conjecture.

\begin{conj}\label{van-Q}
$\QC_N^{res}(0) \st \D_N^{res}(0)(1) \stackrel{!}{=} 0$ for $N \ge 1$.
\end{conj}

This vanishing result is well-known for residue families of even order in the Poincar\'e-Einstein case
\cite[Theorem 1.6.6]{BJ}.

A conformally covariant second-order family of differential operators which
interpolates between the GJMS operators $\iota^* P_2(g)$ and $P_2(h) \iota^*$ (with
$h = \iota^* (g))$ was given in \cite[Theorem 6.4.1]{J1}. This result suggests restating
the above formula for $\D_2^{res}(\lambda)$ in the perhaps more enlightening
form
\begin{align}\label{L2-nice}
    \iota^* L(\lambda-1) L(\lambda) & = (2\lambda\!+\!n\!-\!3) \Big[ (2\lambda\!+\!n\!-\!2) \iota^* P_2(g)
    - (2\lambda\!+\!n\!-\!1) \PO_2(g) \iota^* \notag \\
    & -(2\lambda\!+\!n\!-\!1)(2\lambda\!+\!n\!-\!2) H \iota^* \nabla_\NV \notag \\
    & + 2 \left(\lambda\!+\!\frac{n\!-\!1}{2}\right)
    \left(\lambda\!+\!\frac{n\!-\!2}{2}\right) (\QC_2(g) + \iota^* Q_2(g) + \lambda H^2) \iota^* \Big],
\end{align}
where
$$
    P_2(g) = \Delta_g - \frac{n\!-\!1}{2} \J^g = \Delta_g - \frac{n\!-\!1}{2} Q_2(g)
$$
is the Yamabe operator of $g$ and
$$
    \PO_2(g) = \Delta_h - \left(\frac{n}{2}-1\right)\left (\J^h - \frac{|\lo|^2}{2(n-1)}\right)
   = \Delta_h + \left(\frac{n}{2}-1\right) \QC_2(g)
$$
is the extrinsic Yamabe operator on $M$ (see \eqref{PO2}).\footnote{Note that we use the
conventions $\PO_2(g)(1) = \frac{n-2}{2}
\QC_2(g)$ on $M$ and $P_2(g)(1) = -\frac{n-1}{2} Q_2(g)$ on $X$.} In order to prove that formula, it is
enough to relate the cubic polynomial in square brackets to the corresponding cubic polynomial in \eqref{D2-final}.
It is easy to relate the terms with derivatives. Moreover, the coincidence of the zeroth order terms
follows from the Gauss identity - we omit the calculation.

\eqref{L2-nice} immediately shows that $\D_2^{res}(g;-\frac{n}{2}+1) = \PO_2(g) \iota^*$ and
\begin{equation*}
   \D_2^{res}\left(g;\frac{-n+1}{2}\right) \circ \eta^* = 2 \iota^* P_2(g).
\end{equation*}

The formula \eqref{L2-nice} leads to a simple formula for the $Q$-curvature polynomial
$$
   \QC_2^{res}(g;\lambda) \st \D_2^{res}(g;\lambda)(1).
$$

\begin{lem}\label{Q2-curvature-simple} It holds
\begin{align}\label{Q2-simple}
    & \QC_2^{res}(g;\lambda) = (2\lambda\!+\!n\!-\!3) \lambda \notag \\
    & \times \left[ (2\lambda\!+\!n\!-\!2) \iota^* Q_2(g) + (2\lambda\!+\!n\!-\!1) \QC_2(g)
   + 2\left(\lambda\!+\!\frac{n\!-\!1}{2}\right) \left(\lambda\!+\!\frac{n\!-\!2}{2}\right) H^2 \right].
\end{align}
\end{lem}

This result clearly implies \eqref{Q2-van}. It suggests considering the quadratic polynomial in brackets
as the actual interesting object. In the critical dimension $n=2$, we also find $\dot{\D}_2^{res}(0)(1)
= \dot{\QC}_2^{res}(0) = -\QC_2$.

One should compare Lemma \ref{Q2-curvature-simple} with \cite[(6.4.11)]{J1}.

Note that in the Poincar\'e-Einstein case (i.e., if $g = r^2 g_+$), it holds $\iota^* Q_2(g) = \iota^* \J^g
= \J^h = Q_2(h)$, $\QC_2(g) = - Q_2(h)$ (Remark \ref{Q-GJMS}) and $H=0$. Hence \eqref{Q2-simple}
reduces to
$$
    (2\lambda\!+\!n\!-\!3) \lambda \left[(2\lambda\!+\!n\!-\!2) Q_2(h) - (2\lambda\!+\!n\!-\!1) \QC_2(h) \right]
    = (2\lambda\!+\!n\!-\!3) \lambda Q_2(h)
$$
(see also Remark \ref{rel-def}).

\index{$\QC_N^{res}(\lambda)$}

It is an open problem whether for $N \ge 3$ the $Q$-curvature polynomial
$\QC_N^{res}(\lambda) \st \D_N^{res}(\lambda)(1)$ similarly can be reduced to a lower degree
polynomial and whether these polynomials admit a recursive description as in the Poincar\'e-Einstein case
\cite{J2, J4}.

Next, we determine the leading part of $\PO_3$ from the leading part of the solution
operator $\T_3(\lambda)$. $\PO_3$ is an operator of second-order. Let $\LT(\PO_3)$
denote its terms that contain one or two derivatives. The same notation will be
used for other second-order operators.

\begin{lem}\label{LT-P3} It holds
$$
   \LT( \Res_{\lambda=\frac{n-3}{2}}(\T_3(\lambda)) = \frac{1}{3} \delta (\lo d).
$$
Hence
$$
   \LT(\PO_3) = 8 \delta (\lo d).
$$
\end{lem}

\begin{proof}
The leading part of the solution operator $\T_3(\lambda)$ is determined by the equation
\begin{align*}
  & 3(2\lambda\!-\!n\!+\!3) \LT(\T_3(\lambda)) + 2 (\lambda\!+\!1)(\lambda\!+\!2) H \LT(\T_2(\lambda))
  + (\lambda+2) \tr(L) \LT(\T_2(\lambda)) \\ & + (\lambda+2) H \LT(\T_2(\lambda))
  - 2 (n\!-\!1)(\lambda+2) H \LT(\T_2(\lambda)) - (dH,d\cdot)_h \\
  & + \LT(\Delta \T_1(\lambda)) + \Delta_h' = 0,
\end{align*}
where $\Delta_{h_s} = \Delta_h + s \Delta_h' + \frac{s^2}{2!} \Delta_h'' + \cdots$. Simplification yields
\begin{align*}
   & -3(2\lambda\!-\!n\!+\!3) \LT(\T_3(\lambda)) \\
  & = (\lambda\!+\!2)(2\lambda\!-\!n\!+\!5) H \LT(\T_2(\lambda)) - (dH,d\cdot)_h + \LT(\Delta \T_1(\lambda)) + \Delta'.
\end{align*}
But the first variation $\Delta'$ of the Laplacian with respect to the
variation $h_s$ of $h$ is given by \eqref{var-1}. Using $\LT( \T_2(\frac{n-3}{2}))= \frac{1}{2} \Delta$ (Lemma \ref{sol-2}),
it follows that $$-6 \LT( \Res_{\lambda=\frac{n-3}{2}}(\T_3(\lambda))$$ equals
\begin{align*}
   & \frac{n\!+\!1}{2} H \Delta - (dH,d\cdot)_h - \frac{n\!-\!3}{2} \LT(\Delta \circ H) -2H \Delta + (n\!-\!2)
   (dH,d\cdot)_h - 2 \delta (\lo d) \\
   & = -(dH,d\cdot)_h - (n\!-\!3) (dH,d\cdot)_h + (n\!-\!2) (dH,d\cdot)_h - 2 \delta (\lo d) \\
   & = - 2 \delta (\lo d).
\end{align*}
This completes the proof.
\end{proof}

Similar arguments prove

\begin{lem}\label{LT-P5} It holds
$$
  \LT( \Res_{\lambda=\frac{n-5}{2}}(\T_5(\lambda))
  = \frac{1}{30} (\Delta \delta (\lo d) + \delta(\lo ) \Delta).
$$
Hence
$$
   \LT(\PO_5) = 192 (\Delta \delta (\lo d) + \delta(\lo d) \Delta).
$$
\end{lem}

These results are special cases of Proposition \ref{LT-P-odd}.

\subsection{Low order cases of Theorem \ref{w-holo}}\label{w-LR}


The following result illustrates Theorem \ref{w-holo} for $k=1,2$. 

\begin{lem} Assume that $\sigma$ satisfies $\SCY$. Then
\begin{equation*}
   \iota^* L(-n\!+\!1)  \left( \left( \frac{r}{\sigma} \right)^{n-1} \right) = - \frac{(n\!-\!1)^2}{2} H
\end{equation*}
for $n \ge 1$ and
\begin{align*}
   & \iota^* L(-n\!+\!1) L(-n\!+\!2) \left( \left( \frac{r}{\sigma} \right)^{n-2} \right) \\
   & = \frac{(n\!-\!1)(n\!-\!2)}{6} \left( -\frac{n\!-\!5}{n\!-\!1} |\lo|^2 - 2 (n\!-\!2) \iota^* \J^g
   + 2 (n\!-\!5) \J^h + \frac{(n\!-\!2)(n\!-\!3)}{2} H^2 \right)
\end{align*}
for $n \ge 2$. Here $L(\lambda)$ is short for $L(g,\sigma;\lambda)$.
\end{lem}

By \cite[(4.5)]{G-vol}, it holds $w_1 = \frac{n-1}{2} H$ and the right-hand side of the second identity
equals $2(n\!-\!1)(n\!-\!2) w_2$. Hence
$$
   w_1 = - \frac{1}{n-1} \iota^* L(-n\!+\!1)\left( \left( \frac{r}{\sigma} \right)^{n-1} \right)
$$
for $n \ge 2$ and
$$
   w_2
  = \frac{1}{2(n\!-\!1)(n\!-\!2)} \iota^* L(-n\!+\!1) L(-n\!+\!2) \left( \left(\frac{r}{\sigma} \right)^{n-2} \right)
$$
for $n \ge 3$.

\begin{proof} We expand in geodesic normal coordinates. First, we note that
$$
   \left( \frac{r}{\sigma} \right)^{n-1} = 1 - (n\!-\!1) \sigma_{(2)} r + \cdots
$$
and recall that $\sigma_{(2)} = \frac{1}{2} H$. Now we calculate
\begin{align*}
   \iota^* L(-n\!+\!1) (1) = -(n\!-\!1)^2 H \quad \mbox{and} \quad
   \iota^* L(-n\!+\!1) (H r) = - (n\!-\!1) H
\end{align*}
using $\rho_0 = -H$. This proves the first identity. Similarly, we find
$$
   \left( \frac{r}{\sigma} \right)^{n-2} = 1 - (n\!-\!2) \sigma_{(2)} r +
   \left(-(n\!-\!2) \sigma_{(3)} + \binom{n\!-\!1}{2} \sigma_{(2)}^2\right) r^2 + \cdots.
$$
Now $\grad_g(\sigma) = \partial_r + H r \partial_r + \cdots$ and we obtain
\begin{align*}
   & \iota^* L(-n\!+\!1)L(-n\!+\!2)(1) \\
   & = -(n\!-\!1) \iota^* (\partial_r + (n\!-\!1) H) ((-n\!+\!3)
   (\partial_r + H r \partial_r - (n\!-\!2) \rho) - r (\Delta_g - (n\!-\!2) \J^g))(1) \\
   & = -(n\!-\!1)(n\!-\!2) ((n\!-\!3) \partial_r (\rho)|_0 + \J^g|_0 - (n\!-\!1)(n\!-\!3) H^2).
\end{align*}
Note that this result coincides with $2(n\!-\!1)(n\!-\!2)v_2$, where $v_2$ is as in \eqref{v2n}.
Moreover, we get
\begin{align*}
   \iota^* L(-n\!+\!1)L(-n\!+\!2)(\sigma_{(3)} r^2) &
   =  (n\!-\!1)(n\!-\!2) (-2(n-3) \sigma_{(3)} - 2 \sigma_{(3)}) \\
   & = 2(n\!-\!1)(n\!-\!2) \sigma_{(3)}
\end{align*}
and
\begin{equation*}
    \iota^* L(-n\!+\!1)L(-n\!+\!2)(\sigma_{(2)}^2 r^2) = 2(n\!-\!1)(n\!-\!2) \sigma_{(2)}^2.
\end{equation*}
Finally, we obtain
$$
   \iota^* L(-n\!+\!1)L(-n\!+\!2)(\sigma_{(2)} r) = \frac{1}{2}(n\!-\!1)(n\!-\!2)(2n\!-\!3) H^2
$$
using
$$
   \Delta_g (\sigma_{(2)} r) = \frac{1}{2} \tr (h_r^{-1} h_r') \sigma_{(2)} + r \Delta_h (\sigma_{(2)})
   = n H \sigma_{(2)} + O(r).
$$
We omit the details. Note that only in the latter calculation the contribution $H r\partial_r$ in $\grad_g(\sigma)$
plays a role. Now, using Lemma \ref{g12}, Lemma \ref{sigma23} and Lemma \ref{rho-01}, these results imply
the second assertion.
\end{proof}



\printindex

\end{document}